\documentclass[a4paper]{amsart}
\usepackage[margin=2.23cm]{geometry}
\usepackage{amsmath,amsthm,amssymb}
\usepackage[T1]{fontenc}
\usepackage{url}
\usepackage{amsmath,amsthm,amssymb}
\usepackage{mathrsfs}
\usepackage{verbatim}
\usepackage{graphics}
\usepackage{graphicx}
\usepackage{subfigure}
\usepackage{hyperref}
\usepackage{mathtools}
\usepackage{color}
\usepackage{enumitem}
\allowdisplaybreaks[1]

\newtheorem{theorem}{Theorem}[section]
\newtheorem{proposition}[theorem]{Proposition}
\newtheorem{lemma}[theorem]{Lemma}
\newtheorem{corollary}[theorem]{Corollary}

\theoremstyle{remark}
\newtheorem{remark}[theorem]{Remark}
\newtheorem{definition}{Definition}

\numberwithin{equation}{section}

\newcommand{\Var}[2]{{\operatorname{Var}}\left. \left( #1 \right)\right|_{#2}}
\def\be{\begin{equation}}
\def\ee{\end{equation}}
\def\bes{\begin{equation*}}
\def\ees{\end{equation*}}

\newcommand{\vep}{\varepsilon}
\newcommand{\R}{{\mathbb{R}}}

\newcommand{\Z}{{\mathbb{Z}}}
\newcommand{\N}{{\mathbb{N}}}

\newcommand{\var}{\operatorname{Var}}
\newcommand{\ac}{\operatorname{AC}}
\newcommand{\bv}{\operatorname{BV}}

\newcommand{\ol}{\operatorname{LG}}
\newcommand{\logs}{\operatorname{LSG}}
\newcommand{\Int}{\operatorname{Int}}

\newcommand{\SL}{\operatorname{SL}}

\newcommand{\lv}{\mathcal{LV}}
\newcommand{\bl}{\mathcal{L}}
\newcommand{\as}{\mathcal{AS}}

\makeatletter
\newcommand{\customlabel}[2]{\protected@write\@auxout{}{\string\newlabel{#1}{{#2}{\thepage}{#2}{#1}{}}}\hypertarget{#1}{#2}}
\makeatother

\begin{document}
\title[On  Birkhoff integrals for locally Hamiltonian flows]
{On the asymptotic growth of Birkhoff integrals for locally Hamiltonian flows and ergodicity of their extensions}
\author[K. Fr\k{a}czek]{Krzysztof Fr\k{a}czek}
\address{Faculty of Mathematics and Computer Science, Nicolaus
Copernicus University, ul. Chopina 12/18, 87-100 Toru\'n, Poland}
 \email{fraczek@mat.umk.pl}
\author[C.\ Ulcigrai]{Corinna Ulcigrai}
\address{Institut f\"ur Mathematik, Universit\"at Z\"urich, Winterthurerstrasse 190,
CH-8057 Z\"urich, Switzerland}
\email{corinna.ulcigrai@math.uzh.ch}\date{}

  \keywords{}
%\thanks{Research partially supported by the SNSF (Swiss National Science Foundation) and by the Narodowe Centrum Nauki Grant 2017/27/B/ST1/00078.}

\begin{abstract}
We consider smooth area-preserving flows (also known as \emph{locally Hamiltonian flows}) on surfaces of genus $g\geq 1$ and study ergodic integrals of smooth observables along the flow trajectories. We show that these integrals display a \emph{power deviation spectrum} and describe the cocycles that lead the pure power behaviour, { giving a new proof of results by Forni ({\sf Annals 2002}) and Bufetov ({\sf Annals 2014}) and generalizing them to observables which are non-zero at fixed points. This in particular completes the proof of the original formulation of the Kontsevitch-Zorich conjecture}. Our proof is based on building suitable \emph{correction operators} for cocycles with logarithmic singularities over a full measure set of interval exchange transformations (IETs), in the spirit of Marmi-Moussa-Yoccoz work on piecewise smooth cocycles over IETs. In the case of symmetric singularities, exploiting former work of the second author ({\sf Annals 2011}), we prove a tightness result for a finite codimension class of observables. We then apply the latter result to prove the existence of ergodic infinite extensions for a full measure set of locally Hamiltonian flows  with non-degenerate saddles in any genus $g\geq 2$.
\end{abstract}

\maketitle

\section{Introduction and main results}
In this paper we give a contribution to the study of ergodic theory of smooth area-preserving flows on higher genus surfaces (also known as locally Hamiltonian flows) as well as to the infinite ergodic theory of flow extensions.  The class of surface flows that we work with is introduced in \S~\ref{sec:locHamintro}. We study in particular \emph{deviations of ergodic averages}, by proving the existence of a \emph{power deviation} spectrum for the ergodic integrals along the flow. % of observables which are non-zero at singularities.
This extends and gives a new proof of results by Forni \cite{Fo2} and Bufetov \cite{Bu} for observables with compact support \emph{outside} a neighbourhood of the fixed points of the flow, to observables which have full support and are \emph{non-zero} at singularities. We then use our result to show the existence of infinite extensions of such flows which are \emph{ergodic} with respect to the natural infinite invariant measure. This result  generalizes to higher genus a classical result by Krygin \cite{Kr} in genus one and extends  a previous result in higher genus by the authors (see \cite{Fr-Ul}, where we showed the existence of ergodic extensions in any genus, but only for flows with self-similar foliations) to a  full measure set of flows.

\subsection{Locally Hamiltonian flows}\label{sec:locHamintro}
Let $M$ be a compact, connected, orientable (smooth) surface and let $g$ denote its genus.  We will assume throughout that $g\geq 1$. We will consider smooth flows  on  $M$ preserving a smooth measure $\mu$ (i.e.\ absolutely continuous measure with smooth positive density), see   \S~\ref{sec:locHam}. These flows, also known in the literature  as  \emph{multi-valued Hamiltonian}, are \emph{locally Hamiltonian} flows:
indeed, the flow $\psi_\R:= (\psi_t)_{t\in \R}$ is \emph{locally} Hamiltonian in the sense that around any point in $M$ one can find coordinates $(x_1,x_2)$ on $M$ in which $\psi_\R$ is locally given by
%The form $\eta$ determines a flow  $\psi_\R$ %given by a \emph{multi-valued Hamiltonian} as follows.
 % a closed surface of genus $g\geq 2$ with a fixed area form and a closed differential $1$-form $\eta$ on it.
%since $\eta $ is closed, one can locally write $\eta= \ud H$ for some real-valued function $H$ and
 the solution to the  equations
 \[
 \begin{cases}\dot{x}_1&={\partial H}/{\partial x_2},\\ \dot{x}_2& =-{\partial H}/{\partial x_1}\end{cases}
 \]
  for some smooth  real-valued Hamiltonian function $H$. A \emph{global}  Hamiltonian $ H$ cannot be in general defined (see \cite{NZ:flo}, \S~1.3.4), but one can think of  $\psi_\R$ as globally given by a \emph{multi-valued} Hamiltonian function. We will assume throughout this paper that the fixed points of $\psi_\R$ are \emph{non-degenerate} (also called \emph{Morse} fixed points), namely  that for every fixed point $p$ the local Hamiltonian $H$ is a Morse function at $p$.

The interest in the study of multi-valued Hamiltonians and the associated flows in higher genus ($g\geq 1$) and, in particular, in their ergodic and mixing properties, was highlighted
by Novikov \cite{No} in connection with problems arising in solid-state
physics as well as in pseudo-periodic topology (see e.g.\ the survey \cite{Zo:how} by A.~Zorich).
%\footnote{Novikov \cite{No} and his school in the $1990s$  advocated the study of locally Hamiltonian flows as model to describe the motion of an electron in a metal under a magnetic field in the semi-classical approximation (the surface appears here as Fermi energy level surface).}
The simplest examples of locally Hamiltonian flows  with singularities on a torus, i.e.~flows with one center and one simple saddle (see Figure~\ref{Arnoldtorus}), were studied by V.~Arnold in \cite{Arn} and are nowadays often called \emph{Arnold flows}\footnote{More precisely, referring to the decomposition described in \S~\ref{sec:generic}, we call \emph{Arnold flow} the restriction to a minimal component obtained by removing the center and the disk filled by periodic orbits around it (called \emph{island}), which, as Arnold shows in  \cite{Arn}, is always bounded by a saddle loop.\label{Arnoldflowfootnote}}.

 \begin{figure}[h!]
 \subfigure[An Arnold flow ($g=1) $ \label{Arnoldtorus}]{ \includegraphics[width=0.3\textwidth]{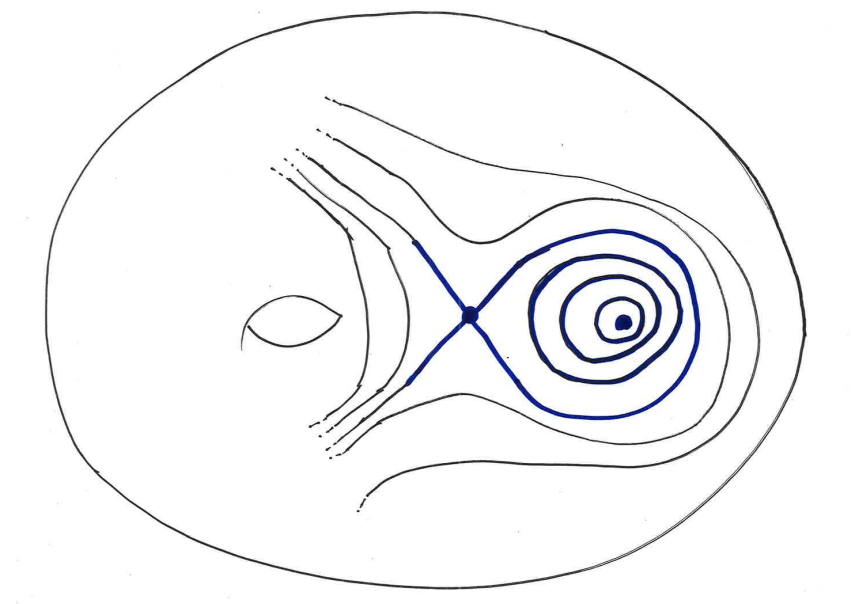}}
 \hspace{6mm}
 \subfigure[A flow on a surface of $g=3$\label{g3}]{ \includegraphics[width=0.56\textwidth]{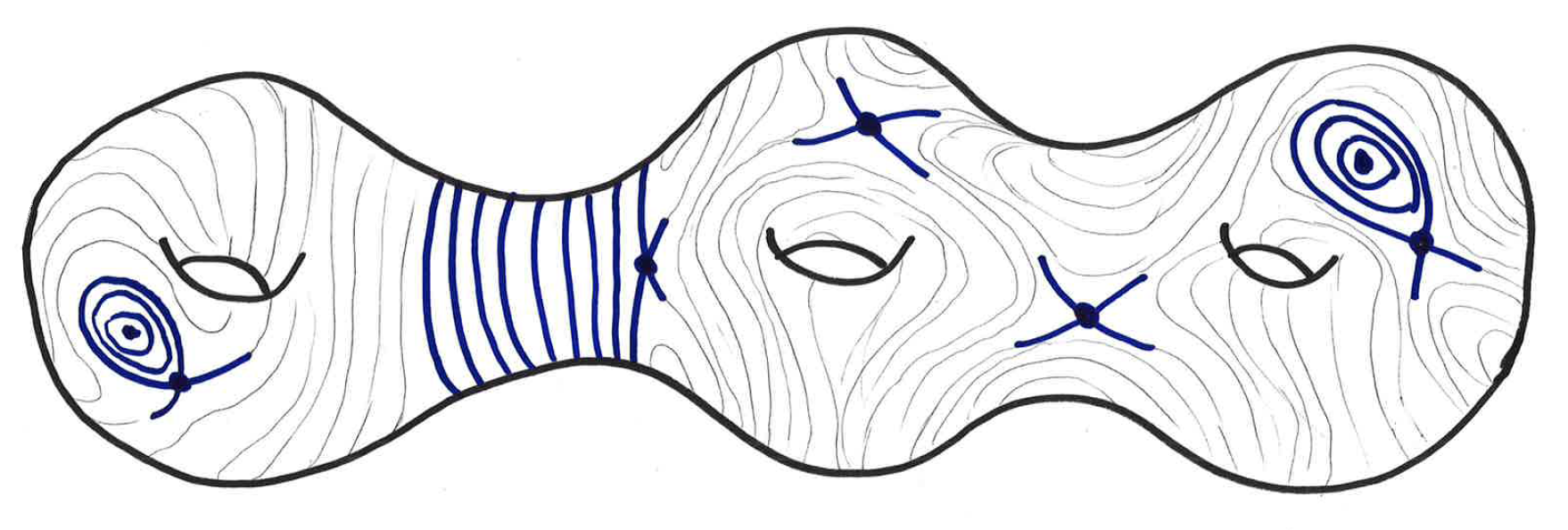}}
%\subfigure[on $\mathbb{R}^2/\mathbb{Z}^2$ \label{Arnoldflowsquare}]{\includegraphics[width=0.33\textwidth]{LocHamSquare}   }	
 \caption{Examples of locally Hamiltonian flows on a surfaces. \label{locHamflows}}
\end{figure}

On the space of locally Hamiltonian flows, one can define a \emph{topology} (see \S~\ref{sec:generic}) as well as a  measure class
(the \emph{Katok fundamental class}, see \S~\ref{sec:ergodicity}).  Our understanding of the \emph{typical} chaotic properties (in the measure theoretical sense) of these flows has advanced a lot in the last forty years.  While results concerning orbit properties, such as minimality or ergodicity, were known first, since they can be simply deduced\footnote{One can show  (see for example~\cite{Zo:how}) that every \emph{minimal} locally Hamiltonian flow on $M$ (as well as the restriction of a locally Hamiltonian flow to one of its
minimal components (see \S~\ref{sec:generic}) has the same \emph{trajectories} (up to time-reparametrization) as a  \emph{translation flow}.  \label{ftnt:rel} Thus, one can infer properties which depend only on trajectories as sets and not on their time-parametrization,  such as minimality and ergodicity, from the known properties of typical translation flows.} %, see \cite{Kea:int, Ve:int}.}
 from {classical results which were proved using Teichm{\"u}ller dynamics (see below as well as \S~\ref{sec:mixing})}, results on finer chaotic properties such as (weak or multiple) mixing, or recently spectral and disjointness results, were proved only in the last twenty years, since  they depend on the movement along trajectories (i.e.\ on time-reparametrization) and  require more delicate estimates exploiting the locally Hamiltonian parametrization of the orbits.
%\footnote{Properties of spectral nature (like mixing, weak mixing, multiple mixing or the nature of the maximal spectral type) are indeed sensitive  to time-parametrizations of the orbits.  For example translation flows are known to be \emph{never} mixing \cite{Ka:int}, while locally Hamiltonian flows can (and often typically are) mixing and even mixing of all orders (see section \S~\ref{sec:mixing}).}.
  We summarize some of the known results in \S~\ref{sec:mixing} below.

\smallskip
In the classification of chaotic behaviour in locally Hamiltonian flows it is crucial to distinguish between two  open sets (complementary, up to measure zero, see \S~\ref{sec:ergodicity} for more details): in the first open set, which we will denote by $\mathscr{U}_{min}$, the typical flow is \emph{minimal}, in the sense that the orbits of all points which are not fixed points are \emph{dense} in $M$. On the other open set, that we call $\mathscr{U}_{\neg min}$, the flow is not minimal, but one can decompose the surface into a finite number of subsurfaces with boundary $M_i$, $i=1,\dots ,N$ such that for each $i$ either $M_i$ is a \emph{periodic component}, i.e.~the interior of $M_i$ is foliated into closed orbits of $\psi_\R$ (in  Figure~\ref{locHamflows} (b) one can see three periodic components, namely two disks and one cylinder, all foliated by closed orbits), or $M_i$ is such that  the restriction of  $\psi_\R$ to $M_i$ is minimal in the sense above (two such subsurfaces are visible in the example in Figure~\ref{locHamflows} (b)). The latter are called \emph{minimal components} and there are at most $g$ of them (where $g$ is the genus of $M$), see \S~\ref{sec:mixing}.

{ The study of locally Hamiltonian flows is intertwined with the study another famous class of flows on surfaces, namely
\emph{translation} (\emph{linear}) \emph{flows}\footnote{Translation flows are \label{ftnt:tf} unit speed \emph{linear flows}  on \emph{translation surfaces}, namely surfaces which are locally Euclidean outside a finite number of conical singularities with cone angles of angle $2\pi k, k\in \mathbb{N}$. On these surfaces, one has a well defined notion of direction and for each $\theta\in S^1$ one can define a directional flow $h^\theta_\R=(h_t^\theta)_{t\in\R}$ which moves points along lines in direction $\theta$ at unit speed.} on translation surfaces, which are at the center of Teichm{\"u}ller dynamics. Each minimal component of a locally Hamiltonian flow $\psi_\R$ indeed can be seen as a \emph{time-reparametrization} (or a \emph{time-change}) of a translation flow. Notice though that the time-change is \emph{singular} at the fixed points $\mathrm{Fix}(\psi_\R)$ of $\psi_\R$ (see \S~\ref{sec:reductionsf} and Remark~\ref{rk:singularrep} for a more precise description of the relation).}
One of the results which can be inferred from classical results on translation flows (proved through Teichm{\"u}ller dynamics) is that  the \emph{typical} flow (in the measure theoretical sense) in $\mathscr{U}_{min}$ is \emph{ergodic}  (with respect to $\mu$)
and the typical flow  in $\mathscr{U}_{\neg min}$ is  ergodic when restricted to each minimal component (see \S~\ref{sec:mixing}); it also follows that the associated foliation into flow trajectories (or equivalently any Poincar{\'e} map of the flow) is \emph{uniquely} ergodic (i.e.~there is an unique invariant probability transverse measure, the transverse measure induced by $\mu$). Notice, though, that any locally Hamiltonian flow $\psi_\R$ with $Fix(\psi_\R)\neq \emptyset$ is \emph{not} uniquely ergodic (as a smooth flow on a compact manifold): indeed, in the presence of singularities, there are always trivial invariant measures (Dirac deltas) supported at singularities. The  presence of such measures and their effect on ergodic integrals  plays a key role in this work.

\subsection{Power deviations and asymptotic behaviour of ergodic averages.}\label{sec:dev}
Let $\psi_\R$ denote either a  locally Hamiltonian flow on $M$ in  $\mathscr{U}_{min}$ or the restriction of $\psi_\R$ in $\mathscr{U}_{\neg min}$   to
a minimal component $M_i$, that by abusing the notation we will again denote by $M$ here, and assume that $\psi_\R$ is ergodic (and the associated foliation is uniquely ergodic).
Thus, for every smooth observable  $f: M\to \R $
%\footnote{\color{red} vanishing on a neighbourhood of fixed points $\mathrm{Fix}(\psi_\R)$}
 and for {\emph{almost every}\footnote{{Equidistribution of almost very point follows simply by ergodicity and Birkhoff ergodic theorem. \emph{Unique} ergodicity yields a stronger conclusion \emph{if} the observable if supported \emph{outside} the set of fixed points $Fix(\psi_\R)$: in this case equidistribution, namely \eqref{nu}, holds for any \emph{regular} $p$, i.e.\ any $p$ such that its forward orbit is $(\psi_t(p))_{t\geq 0}$ is dense). One can show though, that this is not the case for observables $f$ which are non-zero at some fixed points, {namely there are regular points for which equidistribution does not hold.}} }} initial point $p\in M$, the ergodic averages of $f$ converge to the spatial averages, i.e.\
\begin{equation}\label{nu}
\lim_{T\to + \infty} \frac{I_T(f,p)}{T} = \int_M f  \mathrm{ d}\,  \mu, \ \text{where}\ I_T(f,p)=I_T(f,p,\psi_\R):= \int_0^T f(\psi_t(p)) \mathrm{ d}\,   t .
\end{equation}
With \emph{deviations of ergodic averages} one refers to the study of the \emph{oscillations} of the ergodic integrals $I_T(f,p)$  (or the related Birkhoff sum over an interval exchange map obtained as Poincar{\'e} section) of an observable $f:M \to \mathbb{R}$ of zero mean $\int_M f(p) \, \mathrm{d}  \mu=0$ over the orbit of (typical) point $p \in M $.  A distinctive phenomenon %of locally Hamiltonian flows and the related translation flows
first discovered experimentally by A.~Zorich
 in the $1990$s (see \cite{Zor} and also \cite{Ko:Lya, Zor})
is that  deviations of ergodic averages have \emph{polynomial nature}, in the following sense: for a typical flow, for  suitable classes of observables, one can find an exponent $\nu=\nu(f)$ with $0<\nu<1$  such that,   $I_T(f,p)\sim O(T^\nu)$ for every regular point $p$, where we use the notation
%$|I_T(f,x)|\leq C\, T^\nu$ for some $C>0$ and the bound is optimal in the sense that
\begin{equation}\label{bigOdef}
I_T(f,p)\sim O(T^\nu)\qquad \Leftrightarrow \qquad \limsup_{T\to \infty}\frac{\log I_T(f,p)}{\log T}=\nu.
\end{equation}
Kontsevich and Zorich explained this phenomenon heuristically  using renormalization and  conjectured that, at least in the case of
%\footnote{{\color{red} Konstevich and Zorich's initial interest focused on
locally Hamiltonian flows with non-degenerate fixed points\footnote{This is the framework proposed in the paper \cite{Ko:Lya}, where Kontsevich (based on joint work with Zorich) formulates the conjecture on the existence of the deviation spectrum (which later became known as \emph{Kontsevich-Zorich conjecture}). They first state the result for homology classes (or equivalently characteristic functions over interval exchange transformations) and then suggest that the phenomenon should hold more generally if one considers, for simplicity, locally Hamiltonian flows with Morse saddles and the space of smooth functions.},  there is
 %existence of  exponents $0< \nu_{g}\leq \cdots \leq \nu_{2}< \nu_1:=1 $ and a corresponding
 a  full \emph{deviation spectrum}, namely there are exactly $g$ positive
exponents  $0< \nu_{g}< \cdots < \nu_{2}< \nu_1:=1 $ and a corresponding filtration
of $H_{g+1} \subset H_{g}\subset \cdots \subset H_1$   of the space of smooth functions such that
if $f\in H_i\backslash H_{i+1}$, with $1\leq i\leq g$, then $I_T(f,p)\sim O(T^{\nu_i})$ (see \cite{Ko:Lya}). Zorich gave in \cite{Zo} a rigorous proof of this phenomenon for ergodic integrals of a special class of functions $f:M\to \R$, those  which represent cohomology classes\footnote{In the setting of \cite{Zo}, this class of functions reduces  to the study of Birkhoff sums of piecewise constant functions over interval exchange maps.}.
Forni proved most of this conjecture in \cite{Fo2} (with the exception of \emph{simplicity}, namely the strict inequalities between $\nu_g<\nu_{g-1} <\dots <\nu_1$, which was later proved by Avila and Viana in \cite{Av-Vi}, while the positivity of $\nu_g>0$ is a crucial part of \cite{Fo2})  %, showing the existence of $0< \nu_{g}\leq \cdots \leq \nu_{2}< \nu_1:=1 $ for a class of
%this conjecture in , showed the existence of a \emph{power deviation spectrum}
for smooth observables and typical flows in the closely related class of translation flows on translation surfaces (see footnote \ref{ftnt:tf}). In the setting of locally Hamiltonian flows, he considers the minimal case $\psi_\R\in \mathcal{U}_{min}$ and has the further assumption that the (smooth)  observable $f$ is \emph{compactly supported} outside of a neighbourhood of the finite set of fixed points $\mathrm{Fix}(\psi_\R)$ (or, more generally, in the Sobolev regularity setting, that at least the function $f$
 %and its weak $m_\sigma-1$ first derivatives
 vanishes on $\mathrm{Fix}(\psi_\R)$, {see \cite{Fo2} as well as \cite{Fo21})}. %, where $m_\sigma$ is the multiplicity of the saddle $\sigma$).
 We comment below on the consequences of this assumption (see Remark~\ref{rk:nologsing}).

% (or \emph{singularities})  $\mathrm{Fix}(\psi_\R)\subset M$ of $\psi_\R$, namely, he showed that
%for a typical flow $\psi_\R$ there are of the space of weakly differentiable observables with $L^2$ %partial derivative compactly supported  in $M\setminus \mathrm{Fix}(\psi_\R)$, such that if $f\in H_i$, with $1\leq i\leq g$, then $I_T(f,p)\sim O(T^{\nu_i})$ (while if $f\in H_{g+1}$, then $\limsup_{T\to \infty} \log |I_T(f, p)|/ \log T =0$ ) for every regular point $p\in M$.
% is a \emph{coboundary} for $T$ with bounded transfer function, i.e. $f=g\circ T - T$ where the \emph{transfer function} $g:[0,1] \to \mathbb{R}$ is bounded, then Birkhoff sums $M_nf (p)$ are uniformly bounded (and in particular $\nu=0$).
The \emph{power spectrum} of ergodic integrals is related in { \cite{Zo, Fo2} to \emph{Lyapunov exponents} of the Kontsevich-Zorich cocycle (so that in particular the strict inequalities $\nu_g< \nu_{g-1}< \cdots < \nu_1$ hold in view of  the simplicity of the Lyapunov spectrum, which is the result later  shown by Avila-Viana in \cite{Av-Vi} work);  the filtration is described by Forni in \cite{Fo2} } in terms of kernels of what we nowadays call \emph{Forni's invariant distributions}. We refer the interested reader to \cite{Zo:how, FoICM, Fo21, Av-Vi} for surveys of this phenomenon; in \cite{FoICM} other instances of \emph{parabolic} flows for which deviations can be studied via  \emph{renomalization} are also mentioned.

A finer analysis of the behaviour of Birkhoff sums or integrals, beyond the \emph{size} of oscillations, appears in the work {\cite{Bu} by Bufetov, as well as in the work \cite{MMY3} by Marmi, Moussa and Yoccoz}. %In \cite{MMY3}, motivated by the study of wandering intervals in affine i.e.m., S.~M., P.~Moussa and J.C.Y. introduced an object called \emph{limit shape} and used it to describe the \emph{shape} of ergodic sums (see \S~3.4 and \S~3.7.3 in  \cite{MMY3}). Roughly speaking these are obtained by looking at suitably rescaled Birkhoff sums, where time is renormalized according to the leading Laypunov exponent of the Kontsevich-Zorich cocycle, whereas the range of the sum  is renormalized using one of the other positive exponents, according to the choice of $f$. After this double rescaling one obtains a sequence of shapes exponentially converging (in the Hausdorff metric)  to the graph of a H\"older function.
 In \cite{Bu}, Bufetov studies limit theorems for ergodic integrals of translation flows (and describe weak limit distributions) in terms of objects that he calls \emph{H\"older cocycles} (or, in the  more general context of Markov compacta, \emph{finitely-additive measures}) and turn out to be \emph{dual} to Forni's invariant distributions (see \cite{Bu} for details).
 In particular, he shows that for a full measure set of translation flows $h_\R:=(h_t)_{t\in\mathbb{R}}$ (with respect to the Masur-Veech measure),
 %from one of his main results for translation flows, one can deduce that for observables $f$ compactly supported in $S\backslash \mathrm{Fix}(\psi_t)$,
 there exists  $g-1$ cocycles\footnote{Here $\Phi_i(t,x)$ is a \emph{cocycle} over the flow $h_\R$ in the sense that $\Phi_i(t+s,x)= \Phi_i(t,x) + \Phi_i(s,h_t(x))$} $\Phi_i(t,x): \R\times M \to \R$, for $i=2,\dots, g$ (closely related to the \emph{limit shapes} introduced { independently at the same time by Marmi, Moussa and Yoccoz} in \cite{MMY3}), each of which has a \emph{pure power} growth, i.e.\ such that $|\Phi_i(T,x)|\sim O(T^{\nu_i})$ (in the sense of \eqref{bigOdef} above), which, together with the trivial cocycle $\Phi_1(t,x)=t$, encode the \emph{asymptotic behaviour} of the ergodic integrals along the flow, by providing an \emph{asymptotic expansion} up to subpolynomial terms, i.e.\ such that
 \[
 I_T(f,p,h_\R)=\int_{0}^T f(h_s(p)) \mathrm{d} s = c_1 T + c_2 \Phi_2(T,p) + \dots + c_g \Phi_g(T,p) + err(f,T,p),
 %\qquad \text{where}\  |Err(p,t)|\leq
 \]
where the \emph{error term} $err(f,T,p)$ is subpolynomial, i.e.\ for any $\epsilon>0$ there exists $C_\epsilon>0$ such that $|err(f,T,p)|\leq C_\epsilon T^\epsilon$. The constant of the linear leading term is $c_1=\int f \mathrm{d}\omega$, where $\omega $  is the underlying translation surface area form, and the other coefficients can be computed evaluating invariant distributions $D_i$ for $i=1,\dots,g$, i.e. $c_i=D_i(f)$.

 % We remark that limit shapes and H\"older cocycles, despite having been introduced independently, are intrinsically related: limit shapes are essentially \emph{graphs} of  H\"older cocycles along flow leaves. { Let us remark that similar results can also be proved for horocycle flows on negatively curved surfaces, see \cite{BuFo} were the existence of H\"older cocycles is proved in this context.}

\subsection{Ergodicity of extensions.}
A classical way to visualize and study the behaviour of ergodic averages of an observable $f:M\to \R$ along the flow $\psi_\R$ on $M$ is to
consider the flow on $M\times \R$ given by coupling $\psi_\R$ with the differential equation on $\R$
\[
\frac{\mathrm{d}\,\! y}{\mathrm{d}\,\! t}= f(\psi_t(p)), \qquad y\in \R,\ \quad t\in \R.
\]
One can see that the solution is given by the flow $\Phi^f_\R:=(\Phi^f_t)_{t\in\mathbb{R}}$  on $M\times \R$ given by the formula
\begin{equation}\label{extension}
\Phi^f_t(p,y)=\left(\psi_t(p), \ y+\int_0^{t}f(\psi_s (p))\,\mathrm{d}s\right),\qquad p\in M,\ y\in \R, \ t\in \R.
\end{equation}
Thus, the  flow  {$\Phi^f_\R$} is a \emph{skew product} and provides an \emph{extension}  to $M\times \R$ of the flow $\psi_\R$ on $M$ (i.e.~it  \emph{projects} on the $M$ coordinate to the flow $\psi_\R$). The motion in the $\R$ fiber is determined by the oscillations of
 the \emph{ergodic integrals} of $f$ along $\psi_\R$.
Notice that  $\Phi^f_\R$ preserves the \emph{infinite} product
measure $ \mu\times Leb$, where $\mu$ is the invariant measure
for $\psi_\R$ and $Leb$ denotes the Lebesgue measure on
${\R}$.

The study of these type of skew products goes back to Poincar{\'e} \cite{P} and his work on differential equations on $\mathbb{R}^3$ (in the case when $\psi_\R$ is a smooth flow on the torus); the study of infinite skew product extensions in greater generality became later a central topic in \emph{infinite ergodic theory}, see for example the monographs \cite{Aa, Sch}.
 A basic question is whether the flow $\Phi^f_\R$ is ergodic (see \S~\ref{sec:extensionsdef}) or, if not, what  is a description
of ergodic components.  A necessary condition for ergodicity is that $f$ has zero mean, i.e.~$\int_M  f \mathrm{d}\mu=0$, since otherwise $\Phi^f_\R$ has a \emph{drift} and is not even \emph{recurrent} (see \S~\ref{sec:extensionsdef}).
In the setting of extensions, a property completely opposite to
ergodicity is \emph{reducibility}. If the skew product on $M\times \R$ is  \emph{reducible} (see \S~\ref{sec:extensionsdef} for the definition), $M\times \R$ is foliated into invariant sets for $\Phi^f_\R$, on which the dynamics is conjugated to $\psi_\R$ on $M$.

%We anticipate (see \S~\ref{sec:reductionsp} for details) that
Taking a suitably chosen Poincar{\'e} section %(of the form $\gamma \times \R$ where $\gamma\subset M$ is a transversal for $\psi_R$,
(see \S~\ref{sec:reductionsp} for details), the ergodicity of $\Phi^f_\R$ is equivalent to the ergodicity of a skew product automorphism $T_\varphi$ of the strip $I\times \R$, where $I=[0,1)$, of the form
\begin{equation}\label{IETskew}
T_\varphi (x,y)= (T(x), y+ \varphi(x)), \qquad x\in I, \quad y\in \R,
\end{equation}
where  $T:I\to I$ is a \emph{rotation} (i.e.~the map $T(x)=x+\alpha \mod 1$) when $M$ is a torus ($g=1$), or more in general, for any $g\geq 1$, an \emph{interval exchange transformation} (see \S~\ref{sec:IETs}), while $\varphi:I\to \overline{\R}$ is a function with
\emph{singularities} (i.e.~points where the function blows up) which are are \emph{logarithmic} (see \S~\ref{sec:sing} for the precise definition) whenever $\psi_\R$ has only non-degenerate saddles (while \emph{polynomial} in presence of a degenerate saddle).

{
\begin{remark}\label{rk:nologsing}
 Notice also that if $f$ is compactly supported in $M \backslash \mathrm{Fix}(\psi_\R)$ (or, more generally, it vanishes on $\mathrm{Fix}(\psi_\R)$, see \S~\ref{sec:cocycleproperties}, in particular Proposition~\ref{prop:oper},
% in the case of $\psi_\R$ with non-degenerate saddles},
 then the function $\varphi$ in \eqref{IETskew} is piecewise absolutely continuous (or even piecewise smooth), in particular does not have logarithmic singularities. Thus, the singularities are a combined  effect of the nature of the locally Hamiltonian parametrization, \emph{together with} the assumption that (the jet of) $f$ does not vanish identically zero at $\mathrm{Fix}(\psi_\R)$.
\end{remark}}
{We stress that the problem of {ergodicity of skew product extensions over IETs} is currently actively researched, but still widely open. See for example \cite{Hu-We, Co-Fr, Fr-Ul_Ehr, Fr-Hu,  Ho, Ra-Tr1, Ra-Tr2, Ch-R} for some results in particular settings.}

\smallskip
In the \emph{genus one} case, the  existence of ergodic skew products was first discovered by Krygin, in \cite{Kr}, in the case where the flow $\psi_\R$ has no singularities. Ergodicity of extensions of typical
 \emph{Arnold flows}\footnote{Recall that an Arnold flow is the restriction to the minimal component of a locally Hamiltonian flow in genus one with one saddle and one center, see Figure~\ref{Arnoldtorus}.} (or, correspondingly, of skew products of the form \eqref{IETskew} where $T$ is a rotation and $\varphi$ has one \emph{asymmetric} logarithmic singularity, see \S~\ref{sec:sing} for definitions),
 was proved  by Fayad and Lema\'nczyk in \cite{FL:ont}, where they proved ergodicity for a full measure set of rotation numbers. This case is particularly delicate since the underlying  Arnold flows are mixing; in a related easier case (namely the case when $T$ is a rotation but $\varphi$ in \eqref{IETskew} has one  \emph{symmetric} logarithmic singularity, see \S~\ref{sec:sing}), ergodicity was proved previously by Lema\'nczyk and the first author, see \cite{Fr-Le}.

Very little is understood in the  case of infinite skew product extensions (i.e.~extensions by a non-compact fiber, for which the natural invariant measure is infinite)  of locally Hamiltonian flows in higher genus $g\geq 2$, { even
 in the case when $f:M\to \R$ has compact support in $M\backslash\mathrm{Fix}(\psi_\R)$ and the cocycle $\varphi$ is piecewise-smooth (see Remark~\ref{rk:nologsing}) or even piecewise-constant. Some specific results for piecewise constant or piecewise absolutely continuous cocycles over IETs with $d>2$} were proved for example in \cite{Co-Fr, Fr-Ul_Ehr,  Fo1, Ma-Mo-Yo}.

We considered the case of a locally Hamiltonian flow $\psi_\R$ with non-degenerate saddles and a  general observable $f:M\to \R$ and, correspondingly, of a cocycle $\varphi$ with \emph{logarithmic} (symmetric) \emph{singularities}
%which are non-zero at singularities $\mathrm{Fix}(\psi_\R)$
in our previous joint work \cite{Fr-Ul},
where we showed the \emph{existence} of ergodic extensions in any genus, but for a very restrictive class of locally Hamiltonian flows.
More precisely, in  \cite{Fr-Ul} we could treat only the special (measure zero) class of locally Hamiltonian flows in $\mathscr{U}_{min}$ for which the Poincar{\'e} section can be chosen to be a \emph{self-similar} interval exchange transformation\footnote{These IETs are also known as \emph{periodic-type} IETs in the literature, see for example \cite{SU}. In \cite{Fr-Ul} we further assume that the periodic-type IET is of \emph{hyperbolic type}, see \cite{Fr-Ul} for details. Explicit examples of locally  Hamiltonian flow  of
{hyperbolic} periodic type  were constructed in  \cite{Co-Fr}.} and restrict the observable $f$ to belong to an infinite dimensional (but finite codimension $g$) space.  For extensions of flows in this special class, though, we could provide a complete description of the  ergodic behavior and  prove a dichotomy between ergodicity and reducibility.
One of the main results of this paper is to show that
 this dichotomy actually holds also for a \emph{full measure} set of such {minimal} locally Hamiltonian flows  (see the Main Theorem~\ref{main1} below).

 \subsection{Main results}
One of the main results of this paper is that  infinite \emph{ergodic} extensions  \emph{exist} in any genus $g\geq 1$ for a \emph{full measure} set of (minimal) locally Hamiltonian flows
with non-degenerate fixed points (with respect to the Katok fundamental class for each \emph{stratum}, see \S~\ref{sec:generic}). More precisely, we are able to extend the result previously proved in \cite{Fr-Ul} only for a \emph{measure zero class} of self-similar IETs to a \emph{full measure} set of locally Hamiltonian flows, by proving the following \emph{dichotomy} for the dynamics of the extensions:

\begin{theorem}[{\bf Ergodic or reducible extensions of locally Hamiltonian flows}]\label{main1}
For  a \emph{full measure} set of locally Hamiltonian flows $\psi_\R$ with non-degenerate saddles in $\mathcal{U}_{min}$, for any $\epsilon>0$, for any $f$ in a infinite dimensional (finite codimension) subspace $K \subset \mathscr{C}^{2+\epsilon}(M)$,
 we have the following dichotomy:
\begin{itemize}
\item If $\sum_{x\in \mathrm{Fix}(\psi_\R)} |f(x)|\neq 0$ then the extension $\Phi^f_{\R}$ is
ergodic;
\item If $\sum_{x\in \mathrm{Fix}(\psi_\R)} |f(x)| = 0$ then the extension $\Phi^f_{\R}$ is
reducible.
\end{itemize}
\end{theorem}
We will comment later on the full measure set, which is explicitly described by a new \emph{Diophantine-type} condition (see \S~\ref{sec:UDCdefsec} for the definition) as well as on the infinite dimensional (invariant) subspace $K$ (which will be defined as the kernel of $g$ invariant
distributions, %$\mathscr{C}^{2+\epsilon}(S)$-distributions. A similar space
 see \S~\ref{devspectrum:sec}).
%A more precise statement of this result is given in \S~\ref{}, as Theorem~\ref{}. ADD

The proof of this ergodicity result takes as starting point our results on deviations of ergodic averages\footnote{In particular, to prove ergodicity we need to show a form of \emph{tightness} of Birkhoff sums, which, combined with enough \emph{oscillations} thanks to the presence of logarithmic singularities, allows to apply classical \emph{essential values} (see \cite{Sch}).} of $f$, which is of independent interest and we now state.  As it is clear from the dichotomy, to produce ergodic extensions one needs to study observables $f:M\to \R$ which \emph{do not vanish} at (at least one) the saddle points\footnote{Since we are here assuming that $\psi_\R \in \mathcal{U}_{min}$ has only non-degenerate fixed points, $\mathrm{Fix}(\psi_\R)$ consists of simple saddles only, see \S~\ref{sec:generic}.} in $\mathrm{Fix}(\psi_\R)$.

\smallskip
%\subsection{Locally Hamiltonian flows deviation spectrum}
For ergodic integrals of  (typical) minimal locally Hamiltonian flows in {$\mathcal{U}_{min}$ (see Theorem~\ref{mainthm:BS}), as well as for minimal components of (typical) locally Hamiltonian flows in $\mathcal{U}_{\neg min}$ (see Theorem~\ref{mainthm:BSN})}, we give asymptotic descriptions of the deviation spectrum, as follows.
%reproves (using correction operators, see the nex%t section \S~\ref{sec:correction} for details) and extends  both the work of Forni and Bufetov for smooth observables with full support, namely non necessarily vanishing on the set $\mathrm{Fix}(\psi_\R)$ of singularities and also provides a different and more concrete construction of the cocycles which lead the asymptotic beahviour. We prove the following:
{
\begin{theorem}[{\bf Asymptotic power spectrum of ergodic integrals (minimal case)}]\label{mainthm:BS}
%Let $g\geq 2$ be the genus of $M$. %and consider locally Hamiltonian flows with $c\geq 0$  centers and $s\geq 1$ simple saddles (where $2g=s+c-1$).
For a full measure set of locally Hamiltonian flows on $M$ in $\mathscr{U}_{min}$ with non-degenerate saddles,
%with $c$ centers and $s$ simple saddles (where $2g=s+c-1$),
%for any minimal component $M'$ of $\psi_\R$,
there exist a \emph{power spectrum} $0<\nu_g<\cdots < \nu_2<\nu_1:=1$,
where $g$ is the genus of the surface $M$ and,
% sAssume that $\eta \in U_{s_i,s_b}$ satisfies UDC.
for any $\epsilon>0$, \emph{invariant distributions} $D_i:C^{2+\epsilon}(M)\to\R$, $i=1,\ldots,g$,
such that, for every $f\in C^{2+\epsilon}(M)$, we have the \emph{asymptotic expansion}:
%\emph{non-vanishing} on the set $\mathrm{Fix}(\psi_\R)$ of fixed points of $\psi_\R$
\begin{equation}\label{eq:asexp}
\int_0^Tf(\psi_t(x))\,dt=\sum_{i=1}^gD_i(f)u_i(T,x)+ \sum_{\sigma\in \mathrm{Fix}(\psi_\R)}f(\sigma)u_\sigma(T,x)+err_b(f,T,x),
\end{equation}
where, for $1\leq i\leq g$, $u_i$ are smooth cocycles  $u_i:\R\times M\to\R$ over the flow $\psi_\R$ such that
\begin{equation}
\label{eq:explambdai}
\limsup_{T\to+\infty}\frac{\log\|u_i(T,\,\cdot\,)\|_{L^{\infty}(M)}}{\log T}={\nu_i}, % \text{ for }1\leq i\leq g,
\end{equation}
while, for $\sigma\in \mathrm{Fix}(\psi_\R)$, $u_\sigma$ are smooth cocycles $u_\sigma:\R\times M\to\R$ over $\psi_\R$ which grow sub-polynomially pointwise and in $L^p$ norm for every $p\geq 1$, i.e.
 such that
\begin{equation}\label{eq:expusigma}
\limsup_{T\to+\infty}\frac{\log|u_\sigma(T,x)|}{\log T}=0\text{ for a.e. }x\in M, \quad \text{and} \quad \limsup_{T\to+\infty}\frac{\log\|u_\sigma(T,\,\cdot\,)\|_{L^{p}(M)}}{\log T}=0,  \text{for all}\ p\geq 1,
\end{equation}
and $err_b$  is a uniformely bounded \emph{error} term,  i.e.\
\begin{equation}\label{eq:boundederr}\sup_{t\in\R} \|err_b(f,t,\,\cdot\,)\|_{L^\infty}<+\infty.
\end{equation}
Furthermore, for every $\sigma\in \mathrm{Fix}(\psi_\R)$ and for $\mu$-almost every $x\in M$, the values of the cocyle  $t\mapsto u_\sigma(t,x)$
are \emph{equidistributed} on $\R$, i.e.\ for any pair of intervals $J_1,J_2\subset\R$ we have
\begin{equation}\label{equid:usigma}\lim_{T\to+\infty}\frac{Leb\{t\in[0,T]:u_\sigma (t,x)\in J_1\}}{Leb\{t\in[0,T]:u_\sigma( t,x)\in J_2\}}=\frac{|J_1|}{|J_2|}.\end{equation}
Finally, if we set
\begin{equation}\label{errdef} err(f,t,x):=\sum_{\sigma\in \mathrm{Fix}(\psi_\R)}f(\sigma)u_\sigma(t,x)+err_b(f,t,x), \end{equation}
as soon as $f$ does not vanish identically on $\mathrm{Fix}(\psi_\R)$, for $\mu$-almost every $x\in M$ also the values of the cocyle  $t\mapsto err(f,t,x)$
are \emph{equidistributed} on $\R$.
% so that for any pair of intervals $J_1,J_2\subset\R$ we have
%\[\lim_{T\to+\infty}\frac{Leb\{t\in[0,T]:err(f,t,x)\in J_1\}}{Leb\{t\in[0,T]:err(f,t,x)\in J_2\}}=\frac{|J_1|}{|J_2|}.\]
\end{theorem}}
{ Main Theorem~\ref{mainthm:BS} completes in particular the proof of the Kontsevich-Zorich conjecture, in its original formulation for smooth functions over locally Hamiltonian flows with non-degenerate saddles (as formulated in \cite{Ko:Lya}, see the above \S~\ref{sec:dev})}. The result should be seen as a generalization (for smooth\footnote{The class of functions {considered by Forni \cite{Fo2, Fo21}  and Bufetov  \cite{Bu} is more general:  smoothness is not required, but only a Sobolev condition in \cite{Fo2} (see also \cite{Fo21} for a more general result on the cohomological equation)} and a \emph{weak Lipschitz property} in Bufetov's work, see \cite{Bu} for details.}
 %In this sense the result is not a strict generalization. Furthermore, we assume that the observable is \emph{non identically zero} on $\mathrm{Fix}(\psi_\R)$, but this assumption is only used for the last conclusion, which exploits the ergodicity result stated by the Main Theorem 1; the asymptotic deviation spectrum can be proved also for smooth observables with compact support on $M\setminus  \mathrm{Fix}(\psi_\R)$.
 functions) of both the results by Forni \cite{Fo2}  (since it proves the existence of a  power deviation spectrum) and Bufetov  \cite{Bu} (since we show the existence of asymptotic cocycles).
 While the observables in both Forni's \cite{Fo2} and Bufetov's \cite{Bu} works vanish on $\mathrm{Fix}(\psi_\R)$, {  we allow the observables to be non-zero at  singularities in $\mathrm{Fix}(\psi_\R)$.  This leads to the presence in the  asymptotic expansion of    $k$ new cocycles, where $k$ is the cardinality of $\mathrm{Fix}(\psi_\R)$, one for each saddle $\sigma \in \mathrm{Fix}(\psi_\R)$. We will call these $u_\sigma$ \emph{singular cocycles}, since they describe the fluctuations of the ergodic averages due to the presence of singularities. }
 While these cocycles $u_\sigma$ have sub-polynomial deviations, as shown by \eqref{eq:expusigma}, they are \emph{not} uniformly bounded. % (since if they were, it would contradict the equidistribution in \eqref{equid:usigma}).

\smallskip
\noindent{\it Comparison to Forni's and Bufetov's works.}
To further compare the result with Forni's \cite{Fo2} and Bufetov's \cite{Bu} works, let us consider the global  error  term  $ err (f, t, \cdot )$  defined as in \eqref{errdef} combining the bounded error $err_b (f, t, \cdot )$ together with the cocycles $u_\sigma$, $\sigma\in \mathrm{Fix}(\psi_\R)$. Then one can see that  $ err (f, t, \cdot )$ has always \emph{sub-polynomial} pointwise growth (in view of \eqref{eq:expusigma} combined with \eqref{eq:boundederr}), but we have a \emph{dichotomy}: on one hand, if $f$ does  vanish identically on $\mathrm{Fix}(\psi_\R)$,  $ err (f, t, \cdot )$ coincides with  $ err_b (f, t, \cdot )$ and is uniformly bounded. { In this case,  the $g$ cocycles $u_i$, which lead the power growth, can be shown \emph{a posteriori} to coincide with the \emph{Bufetov functionals} in \cite{Bu} up to a bounded error}. On the other hand, as soon as $f$ does \emph{not} vanish identically on $\mathrm{Fix}(\psi_\R)$,  $ err(f, t, x)$  \emph{cannot} be controlled \emph{uniformly}:
% neither as a function of $x$ nor $t$: one one hand, for every $t\in \R$, $err(f,t,\cdot)$ is unbounded (in view of the presence of logarithmic singularities\footnote{One can see that the function $x\mapsto err(f,t,\cdot)$ has logarithmic singularities (and therefore is unbounded) from the definition of $f_e$ (see ) and by Proposition~\ref{}.}); on the other hand,
for $\mu$-almost every $x$,  the function
% (i.e.~in the sup norm $\Vert \cdot \Vert_{\infty}$),
$t \mapsto  err (f, t, x )$ is unbounded,
in view of the equidistribution of $ err (f, t, \cdot )$ in this case (see the final part of Theorem~\ref{mainthm:BS},  which
follows directly  from the ergodicity of the extensions proved in the Main Theorem~\ref{main1}, more precisely from an application of the \emph{ratio} ergodic theorem in infinite ergodic theory).

This novel phenomenon is an effect of the presence of infinite \emph{tails}, due to the assumption that $f$ is non-zero at (some) singularities and the slowing down of trajectories near Hamiltonian saddles. We are nevertheless able to control the error term  $ err (f, t, \cdot )$  \emph{pointwise} almost everywhere (in view of \eqref{eq:expzero1})  and \emph{in average}, in any $L^p$ norm with $p\geq 1$, in view of \eqref{eq:expzero2}.
%The second part of the result, {namely, the \emph{equidistribution} of {both the singularities cocycles as well as the error function} when $f$ does non-vanish on $\mathrm{Fix}(\psi_\R)$,
% This  is a crucial difference with the results of \cite{Fo2,Bu}, as well as a novel phenomenon due to the non-vanishing of $f$ at fixed points,} is that the error term  { err (f, T, \cdot )} (defined in \eqref{errdef} combining the bounded error together with the cocycles) This is
%{Let us remark also that, even in the special case in which  $f$ has compact support on $M\backslash \mathrm{Fix}(\psi_\R)$, which is a standing assumption in the work of Bufetov \cite{Bu}, our result improves on Bufetov's result: Bufetov indeed only shows that the error term $err(f, t, \cdot)$ grows subpolynomially in $t$, while the expansion \eqref{eq:asexp} given by Theorem~\ref{mainthm:BS} shows in this special case that $err(f, t, x)$  is uniformely bounded (in view of \eqref{eq:boundederr}, since in this case $err=err_b$). }

\medskip
\noindent {{\it Minimal components in the non-minimal setting.} { Another novelty of our work is that,  while Forni and Bufetov in \cite{Fo2,Bu} study only minimal flows, we prove the existence of an asymptotic expansion also for ergodic integrals of non-mimimal flows in $\mathcal{U}_{\neg min}$. More precisely, we prove the following result for a minimal component $M_0\subset M$ of a  \emph{typical} flow on $\mathcal{U}_{\neg min}$.

\begin{theorem}[{\bf Asymptotic power spectrum for non minimal components}]\label{mainthm:BSN}
For a full measure set of locally Hamiltonian flows on $M$ in $\mathscr{U}_{\neg min}$ with non-degenerate saddles,
%with $c$ centers and $s$ simple saddles (where $2g=s+c-1$),
for any minimal component $M_0\subset M$ of $\psi_\R$, if $g_0$ denotes the genus of $M_0$, there exist  a \emph{power spectrum} $0<\nu_{g_0}<\cdots < \nu_2<\nu_1:=1$
 and, for any $\epsilon>0$, $g_0$ \emph{invariant distributions} $D_i:C^{2+\epsilon}(M_0)\to\R$, $i=1,\ldots,g_0$,  and $g_0$ smooth cocycles $u_i:\R\times M_0\to\R$, for $i=1,\ldots,g_0$, each of which satisfies \eqref{eq:explambdai},
%and $k_0$ smooth cocycles $u_\sigma:\R\times M\to\R$ for  $\sigma\in \mathrm{Fix}(\psi_\R)\cap M_0$ (where $k_0$ is the number of saddles of $\varphi_\R$ in the interior of $M_0$) {\color{red} CHECK counting!}
%each of which has pure power behaviour dictated by $\nu_i$, namely$$$u_i(t,x)$$$
such that   for every $f\in C^{2+\epsilon}(M_0)$ we have an  asymptotic expansion
\begin{equation*}
\int_0^T f(\psi_t(x))\,dt=\sum_{i=1}^{g_0} D_i(f)u_i(T,x)+err(f,T,x),
%\limsup_{T\to+\infty}\frac{1}{\log T}\log\|u_i(T,\,\cdot\,)\|_{L^{\infty}(M_0)}={\nu_i} \text{ for }1\leq i\leq g;
\end{equation*}
where, { if $f$ vanishes on $\mathrm{Fix}(\psi_\R)\cap M_0$,   the error term  $err(f, T, \cdot)$ satisfies
\begin{equation}\label{eq:experrb}
\limsup_{T\to+\infty}\frac{\log\|err(f,T,\,\cdot\,)\|_{L^{\infty}(M_0)}}{\log T}\leq 0,
\end{equation}
while if  $f$ is not identically zero on $\mathrm{Fix}(\psi_\R)\cap M_0$ then
\begin{equation}
\label{eq:expzero1}
\limsup_{T\to+\infty}\frac{\log|err(f,T,x)|}{\log T}= 0\text{ for $\mu$-almost every }x\in M_0,
\end{equation}
and furthermore
\begin{equation}
\label{eq:expzero2}
\limsup_{T\to+\infty}\frac{\log\|err(f,T,\,\cdot\,)\|_{L^p(M_0)}}{\log T}= 0\text{ for  every }p\geq 1.
\end{equation}}
%Moreover, there exist smooth cocycles $u_\sigma:\R\times M\to\R$ for  $\sigma\in \mathrm{Fix}(\psi_\R)\cap M_0$
%such that
%\begin{gather}
%\nonumber
%err(f,T,x)=\sum_{\sigma\in \mathrm{Fix}(\psi_\R)\cap M'}f(\sigma)u_\sigma(T,x)+err_b(f,T,x);\\
%\label{eq:expusigma}
%\limsup_{T\to+\infty}\frac{\log|u_\sigma(T,x)|}{\log T}=0\text{ for a.e. }x\in M', \quad \limsup_{T\to+\infty}\frac{\log\|u_\sigma(T,\,\cdot\,)\|_{L^{p}(M')}}{\log T}=0
%\end{gather}
%for every for  $\sigma\in \mathrm{Fix}(\psi_\R)\cap M'$ and $p\geq 1$, and
%\begin{gather}
%\label{eq:experrb}
%\limsup_{T\to+\infty}\frac{\log\|err_b(f,T,\,\cdot\,)\|_{L^{\infty}}}{\log T}\leqgather}
\end{theorem}
Notice that in this case, when restricting to a minimal component of $\psi_\R\in \mathcal{U}_{\neg min}$, we only claim that $err(f, T, \cdot )$ grows sub-polynomially {(which is the same type of estimate proved by Bufetov for the error term in the symmetric case). This result is in particular an extension of Bufetov's work \cite{Bu} to the restriction to a  minimal component in the non minimal case  $\psi_\R\in \mathcal{U}_{\neg min}$. }
%{\color{red} What is the true nature of the error in this case? do we want to make a conjecture or a remark? you were saying that one can still have $err_b$ but is not bounded, but has non-zero sum of jumps? how does it grow?} }

\smallskip
 { Thus, Theorems~\ref{mainthm:BS} and~\ref{mainthm:BSN} complete the study of deviations of ergodic averages of smooth functions over locally Hamiltonian flows with \emph{non-degenerate} saddles. The study of locally Hamiltonian flows with \emph{degenerate-saddles} leads to other new phenomena and additional  polynomial terms in the asymptotic expansion and is treated in an upcoming paper by M.~Kim and the first author \cite{Fr-Kim}. }

\medskip
\medskip
\noindent{\it On the proof and the Diophantine-like conditions.}
The proof {of the asymptotic expansion in Theorem~\ref{mainthm:BS}, which will be proved at the same time than Theorem~\ref{mainthm:BSN},} follows a completely different approach to both Forni's \cite{Fo2} and Bufetov's \cite{Bu} works and is inspired by Marmi-Moussa-Yoccoz work  \cite{Ma-Mo-Yo} on solving the cohomological equation for (Roth-type) interval exchange transformations (and the follow up work \cite{Ma-Yo} by Marmi and Yoccoz). We comment in detail on this strategy below in \S~\ref{proofmethod}.

An advantage of this different approach is that it allows to give a \emph{description} of the full measure set of locally Hamiltonian flows for which the result holds in terms of a \emph{Diophantine-type} condition. Furthermore, it  also provides a  different  construction of the cocycles which describe the asymptotic behaviour of ergodic integrals in terms of the correction operators.

The full measure \emph{Diophantine-like} conditions (which are different for Theorem~\ref{mainthm:BS} and Theorem~\ref{mainthm:BSN} respectively)  are expressed more precisely on the interval exchange transformations which arise as Poincar{\'e} sections of the flows. We introduce (in \S~\ref{sec:DC}) two such conditions, both of which we show to be of full measure.
{ The first,  that we call \emph{Uniform Diophantine Condition} (or \ref{UDC}), is used to prove the existence of the asymptotic expansion in both Theorem~\ref{mainthm:BS} and Theorem~\ref{mainthm:BSN} up to a subpolynomial error. In the case of minimal flows in $\mathcal{U}_{min}$, to improve the estimates on the error and show in particular that the error is equidistributed (see the second part of Theorem~\ref{mainthm:BS}), we need to assume  a more restrictive condition, namely the \emph{Symmetric Uniform Diophantine Condition} (or \ref{SUDC}). For this result indeed  we also need to crucially exploit the cancellations proved by the second author in \cite{Ul:abs} to prove typical absence of mixing and these require further assumptions on the IET to hold.

Both {Diophantine-like} conditions} expressed in terms of  the matrices of
the Rauzy-Veech cocycle, which often plays the role of \emph{multi-dimensional continued fraction} in the study of IETs.
These conditions, similarly to the Roth-type condition for IETs introduced by
Marmi-Moussa and Yoccoz in \cite{Ma-Mo-Yo} (and its variations, see for example
\cite{Ma-Mo-Yo, Ma-Mo-Yo:lin, Ma-Yo, Ma-Ul-Yo}), impose constraints both on the growth of the matrices of (an
acceleration of) the cocycle, as well as requests on the hyperbolic behaviour of the matrix product, in the form of Oseledets genericity requests.
In addition, we require \emph{effective} Oseledets control, which in turns allow to control certain \emph{Diophantine series} (see \S~\ref{sec:Dseries}). We point out that similar conditions also appear in the recent work \cite{Gh-Ul} on rigidity of generalized interval exchanges.

\subsection{Correction of cocycles with logarithmic singularities}\label{proofmethod}
We comment now on the methods and the proofs. First of all we work  with Poincar{\'e} maps, both to study the flow $\psi_\R$ and its extensions $\Phi^f_\R$; it is well known that Poincar{\'e} maps of area-preserving flows, in suitably chosen coordinates, are \emph{interval exchange transformations} (for short IETs), namely,  piecewise-isometries of the interval $I=[0,1)$ (the definition is recalled in \S~\ref{sec:IETs}). Moreover, any minimal locally Hamiltonian flow admits a representation as \emph{special flow} over the IET $T:I\to I$ which arise as Poincar{\'e} map (see \S~\ref{sec:specialflowdef} for definitions). The \emph{roof function} $r:I\to \overline{\R^+}$ which arise from this representation has \emph{singularities} at the discontinuities of $T$, which, in case of simple (non-degenerate) saddles, are of \emph{logarithmic type} (formally defined in \S~\ref{sec:sing}), i.e.~as $x\to x_i^{\pm}$ approaches a discontinuity $x_i\in I$ of $T$ from the right or left, $r(x)$ blows up as $C_i^{\pm}|\log (x-x_i)|$ , where the constants $C_i^\pm $ are \emph{positive} and are globally \emph{symmetric}, namely $\sum {C}_i^+= \sum {C}_i^-$, for typical flows in $\mathcal{U}_{min}$, while \emph{asymmetric} for minimal components of typical flows in $\mathcal{U}_{\neg min}$.

Fix now an observable $f:M\to \R$ which is non-zero on $\mathrm{Fix}(\psi_\R)$. To study ergodic integrals, we build the extension $\Phi^f_\R$ on $M\times \R$ (given by \eqref{extension}). Choosing a  Poincar{\'e} section for the extension which projects on $I$, namely of the form $I\times \R$, the Poincar{\'e} first return map of $\Phi^f_\R$ (in suitable coordinates) turns out to be a \emph{skew product} over the IET $T$ of the form \eqref{IETskew}, in which the \emph{cocycle} $\varphi$ has \emph{logarithmic singularities} (where the constants $C_i^\pm$ here can be positive or negative, { or zero if the function is zero on $\mathrm{Fix}(\psi_\R)$, in which case there are no singularities, see Remark~\ref{rk:nologsing}).} We have now reduced the study of ergodic integrals and ergodicity of extensions to the study of Birkhoff sums of cocycles with logarithmic singularities over IETs and ergodicity of skew products over IETs with logarithmic singularities.

%In the seminal work \cite{Ma-Mo-Zo}  (in turns inspired by Forni's work \cite{Fo1} on the cohomological equation for surface flows), Marmi-Moussa and Yoccoz studied piecewise absolutely continuous cocycles over IETs and built
Under the new Diophantine-type conditions that we introduce in \S~\ref{sec:DC},  for every function with logarithmic  singularities, we prove the existence of a \emph{correction} operator, namely an operator which, removing the projection on a finite dimensional space (which corresponds morally to the projection on the unstable space of renormalization), allows to get a better control of the behaviour of Birkhoff sums of functions with  logarithmic singularities (see Theorem~\ref{operatorcorrection} for the precise statement). The result provides an extension %(and an independent proof)
of the main result in the work of Marmi-Moussa-Yoccoz \cite{Ma-Mo-Yo}. In the latter, in order to solve  the cohomological equation for IETs of Roth-type, correction operators are constructed for (piecewise) \emph{absolutely continuous} cocycles.

While the main steps of our correction procedure are inspired by the construction introduced  in \cite{Ma-Mo-Yo} (and later developed in \cite{Ma-Yo}), there are considerable differences and difficulties. Notably, while the authors of \cite{Ma-Mo-Yo} were interested in controlling the growth of the sequence of Birkhoff sums
$(S(k)\varphi)_{k\in \mathbb{N}}$ of a piecewise absolutely continuous functions $\varphi$ using the \emph{uniform norm} (in order to keep them bounded after correction and be able to apply Gottschalk-Hedlund theorem, see \cite{Ma-Mo-Yo} for details), for functions with \emph{logarithmic singularities}, the uniform norm \emph{cannot} be used (since functions with logarithmic singularities are always unbounded). The key idea to treat cocycles with logarithmic singularities in this paper is to exploit instead the $L^1$-norm and to build correction operators which allow to bound or control the $L^1$-norm of the sequence $(S(k)\varphi)_{k\in \mathbb{N}}$. The use of the  $L^1$-norm has already appeared in our previous work  \cite{Fr-Ul}, where we had considered the correction problem\footnote{In \cite{Fr-Ul}, for IETs of hyperbolic periodic type, we build correction operators for cocycles with \emph{symmetric} logarithmic singularities and we then exploit the result to build ergodic extensions, but we do not work out the full deviation spectrum and asymptotic cocycles formalism.} for the (measure zero set of) IETs of hyperbolic periodic type.  It turns out that to  extend the result to  almost every IET  requires once again  changes in the basic  step of construction, as well as the introduction of the above mentioned delicate Diophantine-type condition on the IET.
We refer the interested reader to \S~\ref{correction:sec} (and in particular the outline of the strategy to build the correction operators given  in \S~\ref{sec:outline}) for  further details on the differences and the steps in the construction of the correction operators.

The construction of the asymptotic cocycles $u_i:\R\times M\to \R$ which lead to understanding the behaviour of ergodic integrals (see the statement of Main Theorem~\ref{mainthm:BS}) is strictly connected to the  finite dimensional space of \emph{corrections}. Indeed, corrections can be realized by subtracting \emph{piecewise constant cocycles}, which, through the correspondence between extensions and skew-products, allow to define the asymptotic cocycles $u_i$.

In the case of  minimal locally Hamiltonian flows in $\mathcal{U}_{min}$ (which give rise to \emph{symmetric} logarithmic singularities), we also exploit the delicate \emph{cancellations} among contributions of singularities which were proved by the second author in \cite{Ul:abs} and, introducing the \ref{SUDC} Diophantine-type condition, we are able to prove that, after corrections, a subsequence of Birkhoff sums $(S(k)\varphi)_{k\in\mathbb{N}}$ is \emph{tight}. Tightness, combined with partial rigidity of the IET in the base (a result which dates back to Katok \cite{Ka}) and the \emph{presence} of logarithmic singularities (which comes from the assumption that $f$ is non identically zero on $\mathrm{Fix}(\psi_\R)$), allows to apply a quite standard ergodicity criterium based on the existence of \emph{essential values} (see Proposition \ref{thm:erg1} for the precise incarnation of the criterium which we use in this paper). This allows to prove ergodicity of the corresponding extensions.

 \subsection*{Structure of the paper}  In \S~\ref{sec:background} we recall basic definitions and background material on locally Hamiltonian flows and their extensions. We also summarize their typical ergodic properties and explain the reductions to special flows and skew products over IETs. In \S~\ref{sec:IETDC}, after recalling the required definitions and properties of the Rauzy-Veech induction procedure and the associated cocycle, we define the two Diophantine conditions (the \ref{UDC} and the \ref{SUDC} conditions) and prove that they have full measure.

 In \S\S~\ref{sec:cocycles}, \ref{renormalization:sec} and \ref{correction:sec} we study cocycles with logarithmic singularities over IETs. After  giving definitions and proving elementary properties in \S~\ref{sec:cocycles}, we proceed in \S~\ref{renormalization:sec} at investigating the renormalization process induced on such cocycles by performing Rauzy-Veech induction. The correction operators are constructed in \S~\ref{correction:sec} (where the above mentioned Theorem~\ref{operatorcorrection} about existence and properties of the correction operators is proved).

The asymptotic deviation spectrum (see the first part of Main Theorem~\ref{mainthm:BS}) is proved in \S~\ref{devspectrum:sec}, where the asymptotic of ergodic integrals is recovered from the cocycles associated to the correction operators. In \S~\ref{sec:ergskew} we state the ergodicity criterium that we then apply to prove ergodicity of extensions. After discussing also the reducibility case, we then prove Main Theorem~\ref{main1}, as well as the second part of Main Theorem~\ref{mainthm:BS}. Some technical but standard proofs in this part are relegated to the  Appendix (in particular the proofs of the ergodicity criterium and of a cohomological reduction result which is needed for the reducibility part).

%Both limit shapes and H\"older cocycles are associated to functions which display \emph{truly polynomial} deviations, i.e. for which the exponent $\nu$ in \eqref{nu} is strictly positive.
%\footnote{{ sto un po' imrbogliando, perche' in realta' voi le fate solo per il secondo esponente... comunque e' positivo! quindi forse basta avere una frase deliberatamente vaga...}.}
%{  More precisely, from the} work of  Forni { \cite{For2}} and Avila-Viana \cite{AV}  it follows that for a typical i.e.m. $T$ with  $d$ exchanged  subintervals,  the extended Kontsevich-Zorich cocycle  has $g$ positive Lyapunov exponents, $g$ negative and $s-1$ zero ones, where $d=2g+s-1$ and $g$ and $s$ can be computed from the combinatorics of $T$ ($g$ is the genus and $s$ is the number of marked points of any translation surface which suspends $T$, see \S~\ref{sec:transl}). For typical i.e.m. $T$, functions which are \emph{coboundaries} with bounded transfer functions  %(i.e. can be written as $g - g \circ T$ where $g:I \to \mathbb{R}$ is bounded)
%  and hence have \emph{bounded} Birkhoff sums, can be associated to  the \emph{stable} space of the Kontsevich-Zorich cocycle, which correspond to \emph{negative} Lyapunov exponents.

\section{Definitions, background material and reductions}\label{sec:background}
In this section we recall some basic definitions and background material concerning locally Hamiltonian flows (\S~\ref{sec:locHam}) and their extensions (\S~\ref{sec:extensionsdef}), including a brief summary in \S~\ref{sec:mixing} of our current knowledge of their typical chaotic properties. We also the definition of special flows (see \S~\ref{sec:specialflowdef}) and skew-products (in \S~\ref{sec:IETextensions}) over interval exchange transformations (defined in \S~\ref{sec:IETs}). We finally recall in \S~\ref{sec:reductions} the representation of locally Hamiltonian flows to special flows (see \S~\ref{sec:reductionsf})  with logarithmic singularities (defined in \S~\ref{sec:sing}) and the reduction of the study of their extensions to skew products over IETs, see \S~\ref{sec:reductionsp}.

\subsection{Locally Hamiltonian flows}\label{sec:locHam}
Let $(M, \omega)$ be a surface with a fixed smooth area form  $\omega$. A \emph{smooth area preserving flow}     $\psi_\R=(\psi_t)_{t\in\mathbb{R}}$ on $M$ is  a smooth flow on $M$ which preserves the measure $\mu$ associated to  $\omega$. These flows are also called \emph{locally Hamiltonian flows} or \emph{multi-valued Hamiltonian flows} in the literature, in view of their interpretation as flows locally given by Hamiltonian equations, see the introduction.

It turns out that such smooth area preserving flows on $M$  are in one-to-one correspondence  with smooth closed real-valued differential $1$-forms as follows.  Given a smooth, closed, real-valued differential $1$-form $\eta$, let $X$ be the vector field determined by $\eta = i_X \omega$ where $i_X$ denotes the contraction operator, i.e. $i_X \omega =\omega( \eta, \cdot )$ and consider the flow $\psi_\R$ on $M$ given by $X$. Since $\eta$ is closed, the transformations $\psi_t$, $t \in \mathbb{R}$, are  area-preserving. Conversely, every smooth area-preserving flow can be obtained in this way.

Let $\mathrm{Fix}(\psi_\R)$ denote the set of \emph{fixed points} (also called \emph{singularities}) of the flow $\psi_\R$. We will always require that $\mathrm{Fix}(\psi_\R)$ is a \emph{finite set}, so in particular singularities are \emph{isolated}. Remark that when $g\geq 2$,  $\mathrm{Fix}(\psi_\R)$ is always not empty, thus singularities are isolated. Since $\psi_\R$ is area-preserving,  \emph{singularities} in $\mathrm{Fix}(\psi_\R)$, as shown in Figure~\ref{sing_types}, can be either centers (Fig.~\ref{center}), simple saddles (Fig.~\ref{simplesaddle}) or multi-saddles (i.e.~saddles with $2k$ pronges, $k\geq 2$, see Fig.~\ref{multisaddle} for $k=3$). %. Examples of flow trajectories are shown in Figure~\ref{locHamflows}.
 For $g=1$, i.e.~on a torus, if there is a singularity then there has to be another one and we get an Arnold flow as in Figure~\ref{Arnoldtorus}.

 \begin{figure}[h!]
  \subfigure[center \label{center}]{
  \includegraphics[width=0.23\textwidth]{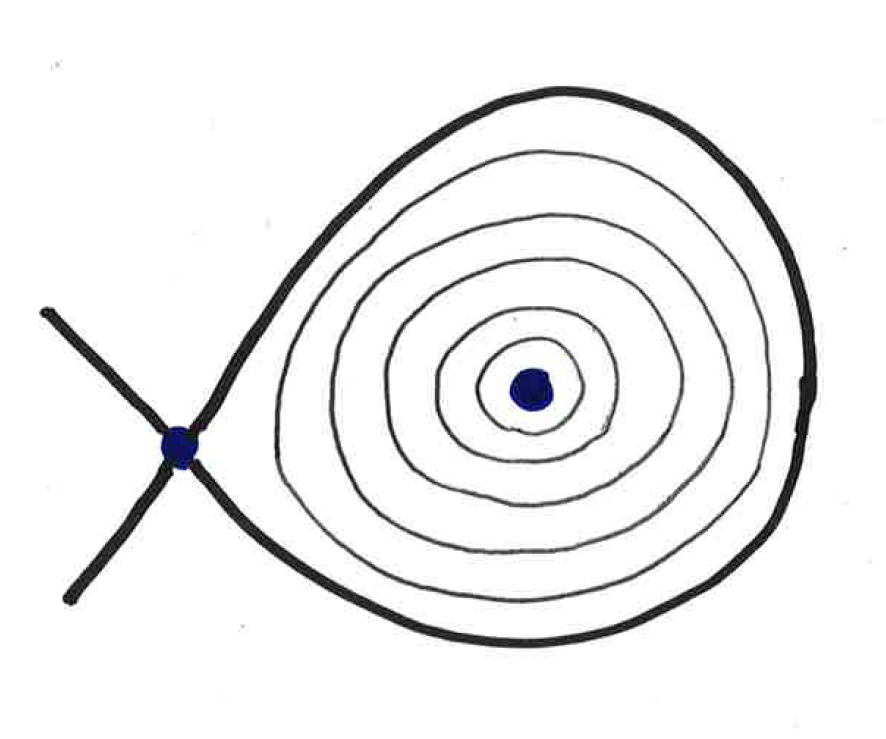} 	} \hspace{6mm} \subfigure[simple saddle \label{simplesaddle}]{ \includegraphics[width=0.18\textwidth]{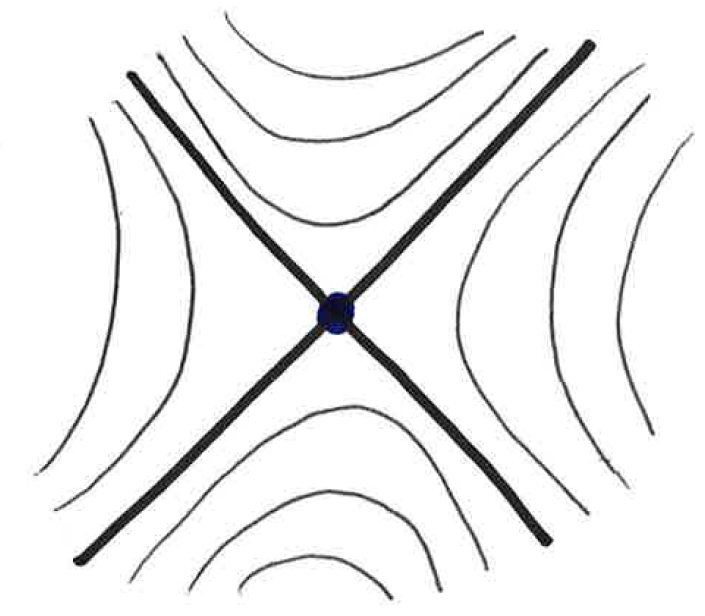}}\hspace{9mm}
\subfigure[multisaddle \label{multisaddle}]{
\includegraphics[width=0.18\textwidth]{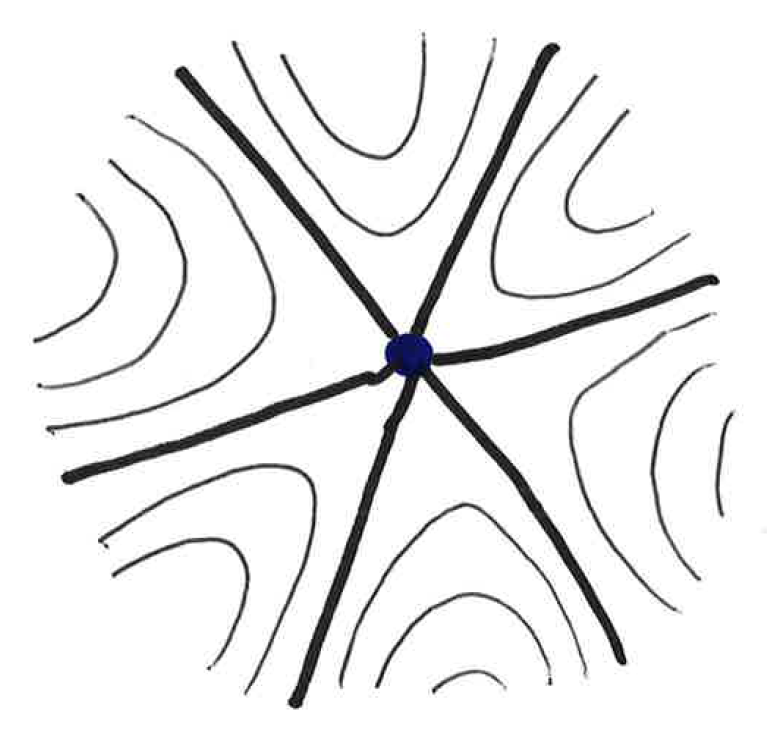}	\hspace{9mm}  }
{\caption{Type of singularities of a locally Hamiltonian flow.\label{sing_types}}}
\end{figure}

%Recall
  We call  \emph{saddle connection} a flow trajectory from a saddle to a saddle and a \emph{saddle loop} a saddle connection from a saddle to the same saddle (see Fig.~\ref{components}).  A \emph{periodic component} is either a (maximal) punctured disk or a (maximal) cylinder filled with closed (i.e.~periodic) trajectories (see Fig.~\ref{island} and Fig.~\ref{cylinder} respectively). A \emph{minimal component} is a subsurface $M'\subset M$, possibly with boundary, such that any trajectory different than a fixed point is \emph{dense} in $M'$.
Periodic and minimal components are bounded by union of saddle connections.
  %	From the point of view of topological dynamics (as proved independently  by

 \begin{figure}[h!]
 \subfigure[periodic island \label{island}]{
 \includegraphics[width=0.25\textwidth]{Island} 	} \hspace{6mm}
 \subfigure[periodic cylinder \label{cylinder}]{
	\includegraphics[width=0.25\textwidth]{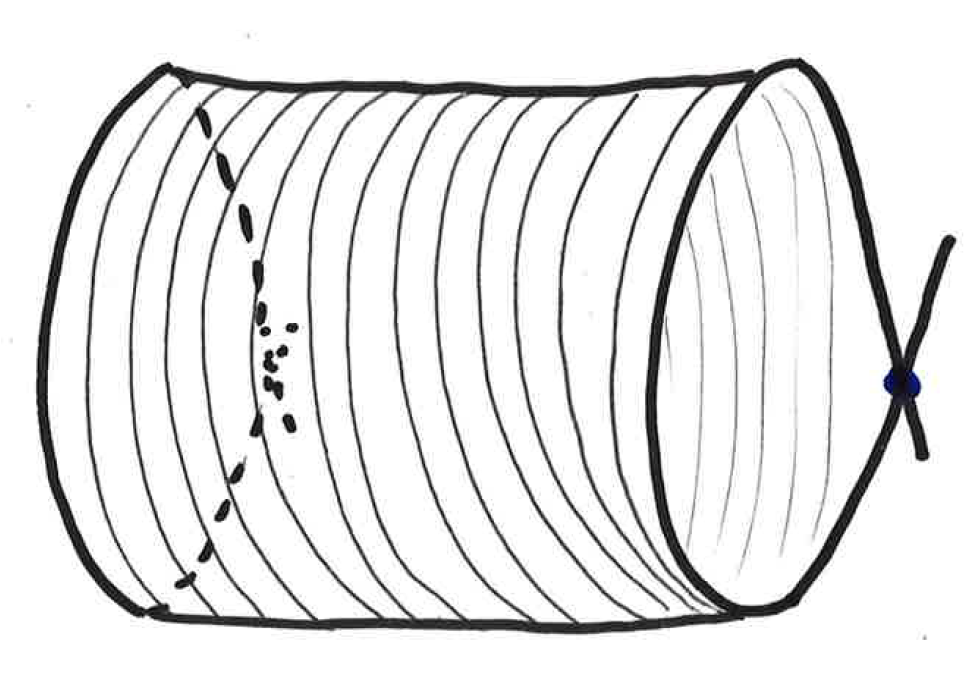}} \hspace{6mm}
\subfigure[$g=2$ minimal component  \label{minimalcomp}]{\includegraphics[width=0.33\textwidth]{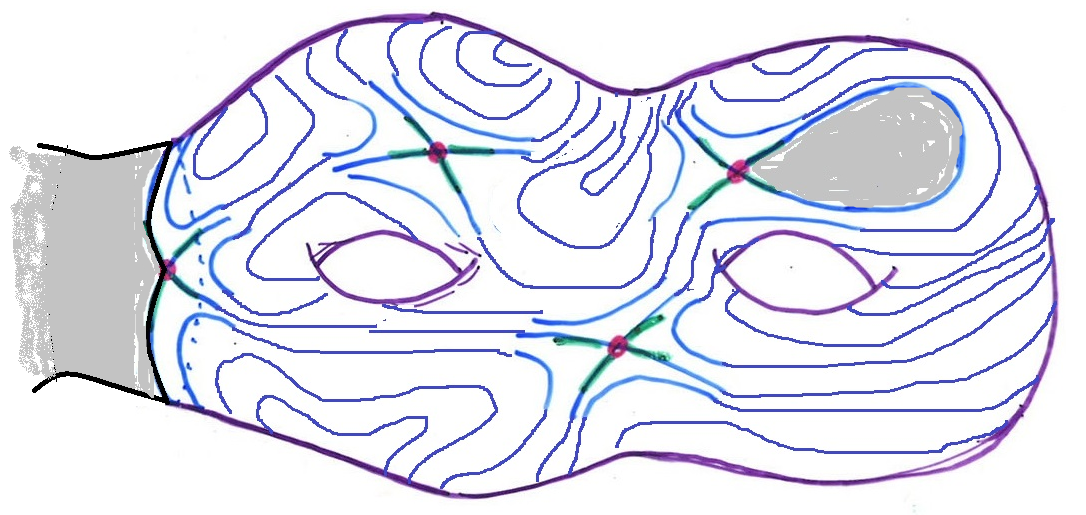}   }	
 \caption{Periodic and minimal components% in the decomposition of a locally Hamiltonian flow
.\label{components}}
\end{figure}

%subsurfaces (with boundary) on which the restriction of $\psi_\R$ either foliates into closed (i.e.~periodic) orbits and up to $g$ subsurfaces (recall that $g$ is the genus of $S$) on which (the restriction of) $\psi_\R$

%( Decomposition in periodic components filled by closed orbits (two islands around centers and a cylinder in blue in the Figure) and minimal components two in the example)ads
%( Decomposition in periodic components filled by closed orbits (two islands around centers and a cylinder in blue in the Figure) and minimal components two in the example)
	
%	into  {periodic components} and {minimal components}. A  is a subsurface (possibly with boundary) on which all orbits are closed and periodic. These can be for example.  \emph{Minimal components} are subsurfaces (possibly with boundary) on which the flow is \emph{minimal} in the sense that all semi-infinite trajectories are dense
% As proved independently by Maier \cite{Maier} and Zorich \cite{Zo:how}, $S$ can be decomposed The latter  are known as \emph{minimal components} of  $\psi_\R$.
%Also in the Nineties it was realized ( and independently Maier \cite{Maier}) that

%Minimal components of a locally Hamiltonian flow (and in particular minimal such flows, for which $S$ is in itself a minimal component), in suitably chosen coordinates, have the same orbits (up to \emph{time-reparametrization}, see \S\ref{sec:examples}) than  \emph{linear flows} discussed in  \S\ref{sec:linear} (see e.g.~\cite{Zo:how}).

\subsubsection{Open sets, genericity and minimality} \label{sec:generic}
Let us denote by $\mathcal{F}$ the set of smooth closed $1$-forms on $M$ (i.e.~locally Hamiltonian flows) with
isolated zeros.
 One can define a \emph{topology} on $\mathcal{F}$ by considering perturbations of closed smooth $1$-forms by (small) closed  smooth $1$-forms\footnote{Let $\eta$, $\eta'$ be two smooth closed $1$-forms. We say
that $\eta'$ is an $\epsilon$-perturbation of $\eta$ if for any $x\in M$ there exists coordinates on a simply connected neighbourhood  $U$ of $x$, such that  $\eta |_U=dH $ and $(\eta'-\eta)|_U= dh$ where $\Vert h\Vert_{C^\infty}<\epsilon \Vert H\Vert_{C^\infty}$.}. We say that a condition is \emph{generic} (in the sense of Baire) if it holds for
flows described by an open and dense set of forms with respect to this topology.
%  Thus, all zeros of $\eta$ correspond to either centers or simple saddles.   This condition is generic (in the Baire cathegory sense) in the space of perturbations of closed smooth $1$-forms by closed smooth $1$-forms.

Let $\mathcal{A} \subset \mathcal{F}$ be the subset of Morse $1$-forms (adopting the notation introduced by Ravotti \cite{Rav}), namely forms which are locally the differential of a \emph{Morse function} (i.e.~ a function that has \emph{non-degenerate} zeros, so that the Hessian at every fixed point is non-degenerate). The set $\mathcal{A}$ of Morse  $1$-forms is then \emph{generic}.
Locally Hamiltonian flows corresponding to forms in $\mathcal{A}$ have only \emph{non-degenerate fixed points}, i.e.\ \emph{centers}  and \emph{simple saddles} (see Figures~\ref{center} and ~\ref{simplesaddle}), as opposed to degenerate \emph{multi-saddles} (as in Fig.~\ref{multisaddle}).
We denote by $\mathcal{A}_{s,c}$ the set of $1$-forms in $\mathcal{A}$ with $s$ saddle points and $c$ centers. By the Poincare-Hopf Theorem, $c-s=2-2g$. Furthermore,   each $\mathcal{A}_{s,l}$ is open and their union $\mathcal{A}$ is dense in $\mathcal{F}$ (see e.g.\ Lemma~2.3 in \cite{Rav}).

For every $1$-form in $\mathcal{A}$, the surface $M$ splits into periodic components and (up to $g$) minimal components (as proved independently by Maier \cite{Ma:tra}, Levitt \cite{Le:feu} and Zorich \cite{Zo:how}). Notice that if there is a unique minimal component (which is equal to the whole surface $M$), then $c=0$ (since if there is a center is associated to a periodic component) and $s=2g-2$.
%More generally, if the restriction of the flow $\psi_\R$ to a minimal component $M'\subset M$ contains $s'$ saddles in its interior,
 %and there are  $s_b$ saddles belonging to the boundary of $M'$,
%then the genus $g'$ of the component $M'$ is $g'=s_i/2+1$. { CHECK FORMULA or REMOVE}

Moreover, one can show that if the flow  $\psi_\R$ given by a closed $1$-form $\eta$ has a \emph{saddle loop homologous to zero} (i.e.\ the saddle loop is a \emph{separating} curve on the surface), then the saddle loop is persistent under small perturbations (see \S~2.1 in \cite{Zo:how} or Lemma 2.4 in \cite{Rav}). In particular, the set of locally Hamiltonian flows which have at least one saddle loop is an open set, which consists of \emph{non-minimal} flows. The set  $\mathscr{U}_{\neg min}$ mentioned in the introduction is an open and dense set of this open set (where the open condition guarantees  \emph{asymmetry} in the special flow representation recalled in \S~\ref{sec:reductionsf}, we refer to \cite{Rav} for the precise definition, see Notation 3.3 in \S~3.1 of \cite{Rav}). The set $\mathscr{U}_{min}$ is given by the interior   (which one can show to be non-empty) of the complement of  $\mathscr{U}_{\neg min}$, i.e.\ the set of locally Hamiltonian flows without  saddle loops homologous to zero\footnote{Note that saddle loops \emph{non-homologous} to zero (as well as saddle connections) disappear after arbitrarily small perturbations; therefore neither the set of 1-forms with saddle loops (or more generally saddle connections) non-homologous to zero,  nor its complement are open (see \cite{Rav} for details).}. %In particular, flows in  $\mathscr{U}_{min}$ do not have saddle loops, nor centers.

%For simplicity, we focus on $U_{s_i,s_b}\subset A$ the open set of $1$-forms $\eta$ for which all periodic orbits are homologous to zero (there is only one minimal component $M(\eta)$),

\subsubsection{Measure class and typicality}\label{sec:ergodicity}
Let us fix an open set $\mathcal{A}_{s,c}$ of closed %, Morse
$1$-forms with $c$ centers and $s$ (simple) saddles.
 A \emph{measure-theoretical notion of  typical} on $\mathcal{A}_{s,c}$ can be defined  on each  $\mathcal{A}_{s,c}$
 %(and restricted to any open subset, such as the subset $\mathcal{A}_{s,c}$' of \emph{Morse} closed $1$-forms)
 as follows, by using the \emph{Katok fundamental class} (introduced by Katok in \cite{Ka0}, see also \cite{NZ:flo}),
 i.e.\ the cohomology class of the 1-form $\eta$  which defines the flow.
%Assume first for simplicity that $M$ contains a unique minimal component. In this case,
%let $k$ be the cardinality of  the set $\mathrm{Fix}(\psi_\R)$ of fixed points of the flow induced by %$\eta$. Let
Let $\gamma_1, \dots, \gamma_n$ be a base of the relative homology $H_1(M, \mathrm{Fix}(\psi_\R), \mathbb{R})$, where $n=2g+s+c-1$,
%Let $\gamma_1,\ldots, \gamma_d$ be a basis of $H_1(M,\mathrm{Fix}(\psi_\R), \Z)$  and
and consider the period map
\[\Theta(\eta)=\Big(\int_{\gamma_1}\eta,\ldots,\int_{\gamma_n}\eta\Big)\in\R^n.\]
%The image of  $\eta$ by the period map $Per $ is $Per(\eta) = (\int_{\gamma_1} \eta, \dots, \int_{\gamma_n} \eta) \in \mathbb{R}^{n}$.
The map $\Theta$ is well defined in a neighbourhood of $\eta$ in $\mathcal{A}_{s,c}$ and one can show that it is a complete isotopy invariant (see \cite{Ka0}, or also  Prop.\ 2.7 in \cite{Rav}).

The pull-back $Per_* Leb$ of the Lebesgue measure class (i.e.\ class of sets with zero measure) by the period map gives the desired measure class on closed $1$-forms in  $\mathcal{A}_{s,c}$. When we use the expression \emph{typical}  below (or typical in $\mathcal{U}_{min}$ or  $\mathcal{U}_{\neg min}$)  we mean full measure in each $\mathcal{A}_{s,c}$  with respect to this  measure class on each $\mathcal{A}_{s,c}$ (or on each open subset of $\mathcal{A}_{s,c}$  contained in the union $\mathcal{U}_{min}$ or  $\mathcal{U}_{\neg min}$).

%($d=2g+s_i+s_b-1$)
%Then denote by $\mathrm{Fix}(\psi_\R)=\mathrm{Fix}(\eta)$ the set of saddles in the closer
%of $M(\eta)$.
% We equip $U_{s_i,s_b}$
%with the measure class (class of sets with zero measure) given by the $\Theta$-pullback of the %Lebesgue measure class on $\R^d$,
%i.e.\  the measure class on $U_{s_i,s_b}$ consists of $\Theta$-preimages of zero Lebesgue %measure sets on   $\R^d$.

\subsubsection{Ergodicity and reducibility of extensions}\label{sec:extensionsdef}
Let $\Phi^f_\R:=(\Phi^f_t)_{t\in\mathbb{R}}$  on $M\times \R$ denotes the extension of  an ergodic flow $\psi_\R$ on $M$ by $f:M\to \R$  given by the formula \eqref{extension}. Recall that, if $\psi_\R$ preserves a measure $\mu$,  $\Phi^f_\R$ preserves the (infinite) measure $\mu\times Leb$.
The flow $\Phi^f_\R$ is \emph{recurrent} if $\mu\times Leb$-almost every point is recurrent. A result by Atkinson \cite{At} (which holds for $1$-dimensional extensions of ergodic flows) shows that $\Phi^f_\R$ is recurrent if and only if $\int_M f\mathrm{d}\mu=0$.

We recall that $\Phi^f_{\R}$  is \emph{ergodic} with respect to the (infinite) measure $\mu\times Leb$ if
for any measurable set $A$ which is invariant, i.e.~such that
$\mu\times Leb(A)=\mu\times Leb(\Phi^f_t A) $ for all $t\in \mathbb{R}$, either $\mu\times Leb(A)=0$ or $\mu\times Leb(A^c)=0$, where $A^c$ denotes the complement.

Remark that if $f = 0$, the phase space $M \times \R$
for the corresponding trivial extension given by $\Phi^f_t (x,y)= (\psi_t (x),y) $ is foliated in
invariant sets of the form $M \times \{y \}$, $y \in \R$. In this sense, the dynamics is reduced
to the dynamics of the surface flow $\psi_\R$. We say that $\Phi^f_\R$
is (topologically)
reducible if it is isomorphic to $\Phi^0_\R$
and the isomorphism $G : M \times \R \to M \times \R$
is of the form $G(x,y) = (x,y + g(x))$, where $g : M \to \R$ is continuous.
 So the reducibility of $\Phi^f_\R$ is equivalent to asking that
 $$\int_0^t f(\psi_s x)\,ds=g(x)-g(\psi_t x)$$
for every regular point $x\in M$ and any $t\in\R$.
In this case,
the phase space is again foliated into invariant sets for
$\Phi^f_\R$ of the form $\{(x, y+ g(x)),\ x\in M \}$, $\
y \in \R $. On each leaf the action of $\Phi^f_{\R}$ is
conjugated to $\psi_{\R}$ on $M$.

\subsubsection{Typical chaotic properties of locally Hamiltonian flows}\label{sec:mixing}
Let us briefly summarize the key chaotic properties of locally Hamiltonian flows and some of the recent works on this topic. We already recalled in the introduction, in view of the relation between locally Hamiltonian flows and translation flows {(see also Remark~\ref{rk:singularrep}),} the seminal works by Keane \cite{Kea:int} and Masur \cite{Ma:erg} and Veech \cite{Ve:erg}
show that a full measure set of locally Hamiltonian flows in $\mathcal{U}_{min}$ are minimal and ergodic  and that almost every flow in $\mathcal{U}_{\neg min}$, the restriction to each minimal component is ergodic (and in both cases the underlying foliation in uniquely ergodic).

Mixing depends crucially on the type of singularities of the flow. For a (non-generic) locally Hamiltonian flow with at least one \emph{degenerate saddle} (see e.g.\ Figure~\ref{multisaddle}), mixing was proved in the $1970s$ (by Kochergin in \cite{Ko:mix}). When $\eta \in\mathcal{A}$ and all saddles are \emph{simple}, one has the following dichotomy: in $\mathscr{U}_{min}$, the typical locally Hamiltonian flow is \emph{weakly mixing, but it is \emph{not} mixing} in view of work \cite{Ul:wea, Ul:abs} by the second author  (see also \cite{Ko:abs, Ko:abs2} and \cite{Sch:abs} for previous special cases of this result). There exist nevertheless  exceptional mixing flows, see the work by \cite{CW}, which produces sporadic  examples in $g=5$. If $\eta \in \mathcal{U}_{\neg min}$, the restriction of the typical locally Hamiltonian flow $\psi_\mathbb{R}$ on each of its minimal components is mixing (as proved by Ravotti \cite{Rav} extending previous work by the second author \cite{Ul:mix}). Ravotti also shows in \cite{Rav} \emph{subpolynomial} bounds for the speed of mixing.
%\label{thm:absence}\footnote{Absence of mixing for typical flows for any $g\geq 2$ was proved in \cite{Ul:abs}, while weak mixing is proved also for minimal components of locally Hamiltonian flow with simple saddles in \cite{Ul:wea}. Let us remark that a result in this direction for $g=1$ was already proved by Kocergin in \cite{Ko:abs} (see also \cite{Ko:abs2}) (in the language of special flows over rotations, which does not have a direct implication for locally Hamiltonian flows but suggested that the absence of mixing could hold also when rotations are replaced by IETs and hence in higher genus. Absence of mixing in the special case of $S$ with $g=2$ and locally Hamiltonian flows with two isomorphic simple saddles was shown by Scheglov \cite{Sch:abs}.}]

Further recent work (see \cite{KKU}) also shows that locally Hamiltonian flows in $\mathcal{U}_{\neg min}$
display a \emph{quantitative shearing} property inspired by the \emph{Ratner property} which plays a crucial role in the theory of unipotent flows (or more precisely a variation introduced in ~\cite{FK} to deal with the presence of singularities).
 From this property, one can deduce  that the restriction of a typical locally Hamiltonian flow $\psi_\R$ in $\mathscr{U}_{\neg min}$  on its minimal components is not only mixing, but \emph{mixing of all orders}, see \cite{KKU}.   Arnold flows in genus one were also recently shown (by A.~Kanigowski and M.~Lema{\'{n}}czyk and the second author, see \cite{KLU}) to typically have \emph{disjointess\footnote{The notion of \emph{disjointness} in ergodic theory was introduced in the $1970s$ by H.~Furstenberg, see in particular \cite{Fur}.} of rescalings}, a property which in particular implies Sarnak M\"obius orthogonality conjecture \cite{Sa} to hold (see \cite{KLU} for details  and \cite{surveyMoebius} for a nice survey on the conjecture and progress toward it).

The spectral theory of locally Hamiltonian flows is still largely not understood. Examples\footnote{These examples are known as \emph{Blokhin examples} and are essentially built glueing genus one flows. This allows to study them using (special flows over)  rotations. On the other hand, they are highly non typical.} of locally Hamiltonian flows on surfaces of any genus $\geq 1$ with  singular continuous spectrum were build by  M. Lema{\'{n}}czyk and the first author (see \cite[Theorem 1]{Fr-Le}). For some flows in genus one with a \emph{degenerate} singularity (sometimes known as \emph{Kochergin flows}),  Forni, Fayad and Kanigowski could recently, prove in  \cite{FFK} that the spectrum is \emph{countably Lebesgue}.
The first typical spectral result  for surfaces of higher genus, namely $g \geq 2$ was recently proved by Chaika, Kanigowski and the authors, who showed in \cite{CFKU} that a {\it typical} locally Hamiltonian flow on a \emph{genus two} surface with two isomorphic {\it simple saddles} has {\emph{purely singular}}  {spectrum}.
%The first result, to the best of our knowledge, for higher genus surfaces, i.e. for $g\geq 2$, is the following result for genus two (see Figure~\ref{onlysaddles}), which goes in a completely opposite direction.

\subsection{IETs, special flows and extension.}\label{sec:basicobjects}
Let us now introduce the notation that we will use for interval exchange transformations (\S~\ref{sec:IETs}) and recall the definition of two basic constructions, special flows (\S~\ref{sec:specialflowdef}) and extensions  of IETs (\S~\ref{sec:IETextensions}).

\subsubsection{Interval exchange transformations.}\label{sec:IETs}
 Let $\mathcal{A}$ be a $d$-element alphabet and let
$\pi=(\pi_0,\pi_1)$ be a pair of bijections
$\pi_\vep:\mathcal{A}\to\{1,\ldots,d\}$ for $\vep=0,1$. We adopt
the notation from \cite{Vi0}. Denote by
$\mathcal{S}^0_{\mathcal{A}}$ the subset of irreducible pairs,
i.e.\ such that
$\pi_1\circ\pi_0^{-1}\{1,\ldots,k\}\neq\{1,\ldots,k\}$ for $1\leq
k<d$.
%We will denote by $\pi^{sym}_d$ any pair $(\pi_0,\pi_1)$
%such that $\pi_1\circ\pi_0^{-1}(j)=d+1-j$ for $1\leq j\leq d$.
\smallskip

\noindent For any
$\lambda=(\lambda_\alpha)_{\alpha\in\mathcal{A}}\in
\R_{>0}^{\mathcal{A}}$ let
\[|\lambda|=\sum_{\alpha\in\mathcal{A}}\lambda_\alpha,\quad
I=\left[0,|\lambda|\right)\] and define
\[I_{\alpha}=[l_\alpha,r_\alpha),\text{ where
}l_\alpha=\sum_{\pi_0(\beta)<\pi_0(\alpha)}\lambda_\beta,\;\;\;r_\alpha
=\sum_{\pi_0(\beta)\leq\pi_0(\alpha)}\lambda_\beta.\] Then
$|I_\alpha|=\lambda_\alpha$. Denote by $\Omega_\pi$ the matrix
$[\Omega_{\alpha\,\beta}]_{\alpha,\beta\in\mathcal{A}}$ given by
\[\Omega_{\alpha\,\beta}=
\left\{\begin{array}{cl} +1 & \text{ if
}\pi_1(\alpha)>\pi_1(\beta)\text{ and
}\pi_0(\alpha)<\pi_0(\beta),\\
-1 & \text{ if }\pi_1(\alpha)<\pi_1(\beta)\text{ and
}\pi_0(\alpha)>\pi_0(\beta),\\
0& \text{ in all other cases.}
\end{array}\right.\]
Given $(\pi,\lambda)\in
\mathcal{S}^0_{\mathcal{A}}\times\R_{>0}^\mathcal{A}$ let
$T_{(\pi,\lambda)}:[0,|\lambda|)\rightarrow[0,|\lambda|)$
 stand for  the {\em interval
exchange transformation} (IET) on $d$ intervals $I_\alpha$,
$\alpha\in\mathcal{A}$, which are rearranged according to the
permutation $\pi^{-1}_1\circ\pi_0$, i.e.\
$T_{(\pi,\lambda)}x=x+w_\alpha$ for $x\in I_\alpha$, where
$w=\Omega_\pi\lambda$.

\smallskip
\noindent {\it Keane condition.} Let $End(T)$ stand for the set of end points of the intervals
$I_\alpha:\alpha\in\mathcal{A}$.  A pair ${(\pi,\lambda)}$
satisfies the {\em Keane condition} if $T_{(\pi,\lambda)}^m
l_{\alpha}\neq l_{\beta}$ for all $m\geq 1$ and for all
$\alpha,\beta\in\mathcal{A}$ with $\pi_0(\beta)\neq 1$. Keane \cite{Kea:int} showed that an  IET with an irreducible permutation that satisfy the Keane condition is \emph{minimal}.
\smallskip

\noindent We record here two remarks that will be useful later.

\begin{remark}\label{rk:endpointspartition}
Note that for every $\alpha\in\mathcal{A}$ with $\pi_0(\alpha)\neq
1$ there exists $\beta\in\mathcal{A}$ such that $\pi_0(\beta)\neq
d$ and $l_\alpha=r_\beta$. It follows that
\begin{equation*}%\label{zbzero}
\{l_\alpha:\alpha\in\mathcal{A},\;\pi_0(\alpha)\neq 1\}=
\{r_\alpha:\alpha\in\mathcal{A},\;\pi_0(\alpha)\neq d\}.
\end{equation*}
\end{remark}

\begin{remark}\label{rk:endpoints}
Denote by $\widehat{T}_{(\pi,\lambda)}:(0,|I|]\to(0,|I|]$   the
exchange of the intervals
$\widehat{I}_\alpha:=(l_\alpha,r_\alpha]$, $\alpha\in\mathcal{A}$,
i.e.\ $T_{(\pi,\lambda)}x=x+w_\alpha$ for $x\in
(l_\alpha,r_\alpha]$.
Note that for every $\alpha\in\mathcal{A}$
with $\pi_1(\alpha)\neq 1$ there exists $\beta\in\mathcal{A}$ such
that $\pi_1(\beta)\neq d$ and
$T_{(\pi,\lambda)}l_\alpha=\widehat{T}_{(\pi,\lambda)}r_\beta$.
\end{remark}

%It
%follows that
%\begin{equation}\label{zbjeden}
%\{T_{(\pi,\lambda)}l_\alpha:\alpha\in\mathcal{A},\;\pi_1(\alpha)\neq
%1\}=
%\{\widehat{T}_{(\pi,\lambda)}r_\alpha:\alpha\in\mathcal{A},\;\pi_1(\alpha)\neq
%d\}.
%\end{equation}

\subsubsection{Special flow definition}\label{sec:specialflowdef}
Let $T:I\to I$ be an (ergodic) IET and let $r:I\to\R_{>0}\cup\{+\infty\}$ be an integrable function such that $\underline{r}=\inf_{x\in I}r(x)>0$. The
\emph{special flow} over $T$ under the \emph{roof function} $r$ is the flow $T^r_\R:=(T^r_t)_{t\in\R}$ acting on
\[I^r:=\{(x,s)\in I\times \R:0\leq s<r(x)\},\]
so that $T^r_t(x,s)=(x,s+t-r^{(n)}(x))$, where $r^{(n)}(x)$ denote the Birkhoff sums cocycle\footnote{Here $r^{(n)}(x)$ denotes the additive cocycle defined by $r^{(n)}(x):= \sum_{0\leq k<n}r(T^kx)$ if $n\geq 0$ and $r^{(n)}(x):= -\sum_{n\leq k<0} r(T^k(x))$ if $n<0$.} associated to $r$ and $n$ is the unique integer number with $r^{(n)}(x)\leq s+t<r^{(n+1)}(x)$. It describes the motion of a point in $(x,s)\in I^r\subset I\times \R$ along vertical trajectories, modulo the identification of each point $(x,r(x))$, $x\in I$, with the point $(T x, 0)$.

\subsubsection{Skew product extensions}\label{sec:IETextensions}
Given an IET $T:I\to I$ and a function $\varphi:I\to \R$   the  \emph{extension}  of $T$ by $\varphi$ is the skew-product map $T_\varphi:I\times \R\to I\times \R$ defined as in \eqref{IETskew} by $T_\varphi(x,y)= (T(x), y+ \varphi(x))$. Notice that, for $n\geq 0$, the iterates of $T_\varphi$ have the form
\[
T^n_\varphi(x,y)= (T^n(x), y+ \varphi^{(n)}(x)), \qquad \text{where}\quad  \varphi^{(n)}(x):=
%\begin{cases}
\sum_{k=0}^{n-1} \varphi(T^k(x)).
%& \text{if}\ n>0, \\ 0 & \text{if}\ n=0, \\ ADD & \text{if}\ n<0.\end{cases}.
\]
Remark that the Birkhoff sums $ \varphi^{(n)} (\cdot )$ are a (additive) \emph{cocycle} over $T$ in view of the \emph{cocycle relation} $ \varphi^{(m+n)} (x)= \varphi^{(m)} (T^n x )+ \varphi^{(n)} (x )$.

%\smallskip
%\noindent
%We call $Z(k)$ (resp.~$Q(k,l)$) the \emph{accelerated cocycle} (respectively the \emph{accelerated cocycle matrix products}) along the  (accelerating) sequence $(n_k)_{k\in\mathbb{N}}$.

\subsection{Reduction to special flows and skew-product presentations}\label{sec:reductions}
We recall two classical results that show that locally Hamiltonian flows and their extensions can be reduced respectively to the study of special flows and skew-product extensions over IETs, with roof functions or, respectively, cocycles, with logarithmic singularities.

\subsubsection{Logarithmic singularities}\label{sec:sing}
We say that a   function (or cocycle) $\varphi: I\to \R$
 for an IET $T_{(\pi,
\lambda)}$ has \emph{logarithmic singularities} if there exist
constants $C_\alpha^+,C_\alpha^-\in \mathbb{R}$, $\alpha \in
\mathcal{A}$,  and a function $g_\varphi$ absolutely continuous on the interior of each interval $I_\alpha$, $\alpha\in\mathcal{A}$ (i.e.~with the notation that we will introduce later, a function $g_\varphi \in \ac(\sqcup_{\alpha\in \mathcal{A}} I_{\alpha})$) such that
\begin{equation}\label{def:logsing}
\begin{aligned}
\varphi(x)=&-\sum_{\alpha\in\mathcal{A}}C^+_\alpha\log\big(|I|\{(x-l_\alpha)/|I|\}\big)
- \sum_{\alpha\in\mathcal{A}} C^-_\alpha\log\big(|I|\{(r_\alpha-x)/|I|\}\big)+g_\varphi(x).
\end{aligned}
\end{equation}
We refer to Figure~\ref{logsing} for some examples. We say that the logarithmic singularities are \emph{of geometric
type} if at least one among $ C_{\pi_0^{-1}(d)}^{-}$ and $
C_{\pi_1^{-1}(d)}^{-}$ is zero and  at least one among $
C_{\pi_0^{-1}(1)}^{+}$ or $C_{\pi_1^{-1}(1)}^{+}$ is zero (as shown in the examples in  Figure~\ref{logsing}). We
denote by $\ol(\sqcup_{\alpha\in \mathcal{A}} I_{\alpha})$ the
space of  functions with logarithmic singularities of geometric
type.
We define also the subspace
$\logs(\sqcup_{\alpha\in \mathcal{A}} I_{\alpha})$ $\subset \ol
(\sqcup_{\alpha\in \mathcal{A}} I_{\alpha})$ of functions
satisfying the symmetry condition
\begin{equation}
\label{zerosymweak}
 \sum_{\alpha\in\mathcal{A}} C^-_{\alpha}-
\sum_{\alpha\in\mathcal{A}}C^+_{\alpha}=0.
\end{equation}

 \begin{figure}[h!]
 \subfigure[Roof function $r\in \ol (\sqcup_{\alpha\in \mathcal{A}} I_{\alpha})$ \label{logroof}]{
 \includegraphics[width=0.4\textwidth]{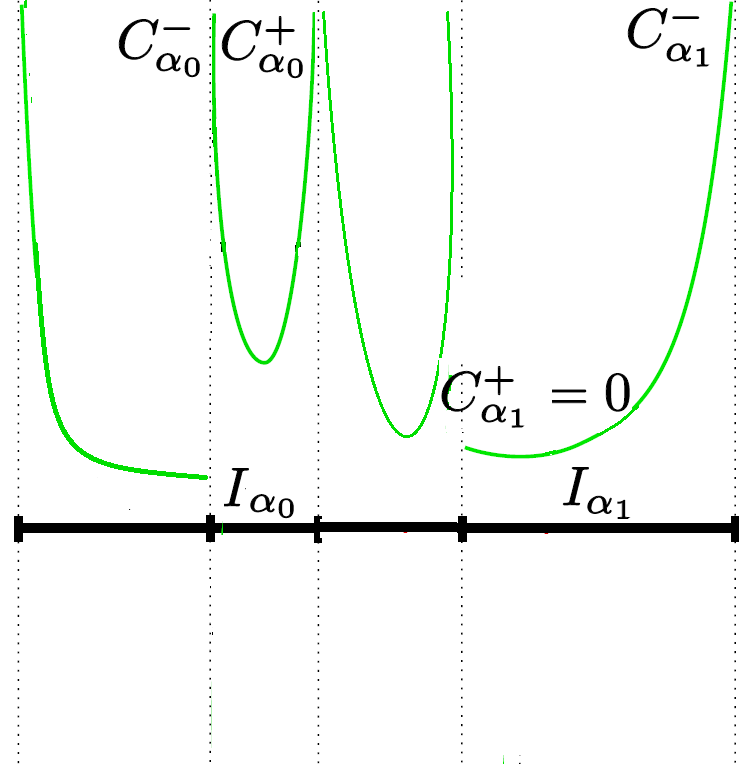} 	} \hspace{6mm}
 \subfigure[Cocycle $\varphi \in \ol(\sqcup_{\alpha\in \mathcal{A}} I_{\alpha})$ \label{logcocycle}]{
  \includegraphics[width=0.4\textwidth]{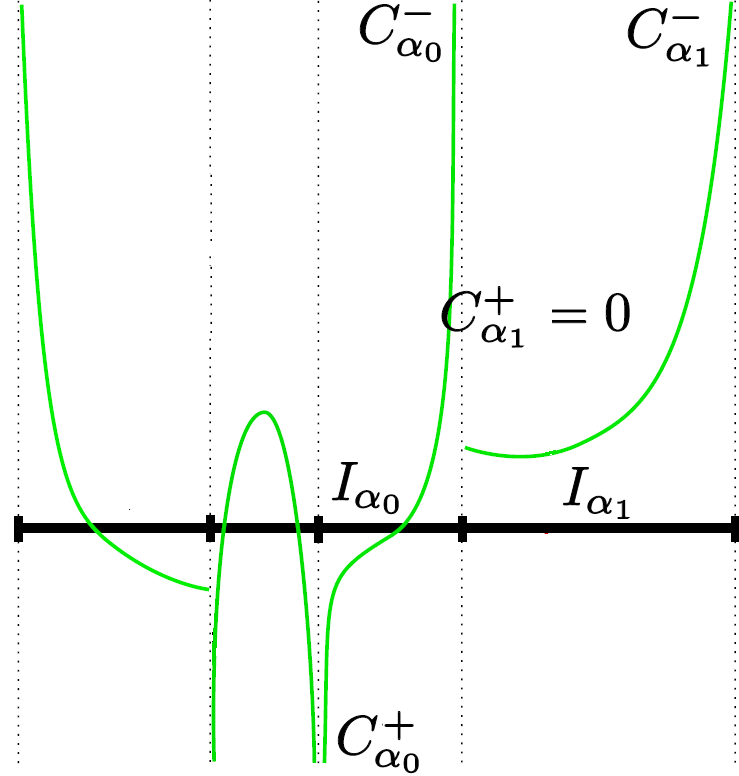} 	} \hspace{6mm} 	
 \caption{Examples  of functions with geometric logarithmic singularities in $\ol
(\sqcup_{\alpha\in \mathcal{A}} I_{\alpha})$.\label{logsing}}
\end{figure}

\subsubsection{Special flow representations of locally Hamiltonian flows}\label{sec:reductionsf}
It is well known that locally Hamiltonian flows can be represented as special flows as follows (see for example \cite{Ul:abs, Rav, Co-Fr, Fr-Ul}).
Consider either a minimal locally Hamiltonian flow $\psi_\R$  on $M$ or the restriction of
 a locally Hamiltonian flow on $M$ to a minimal component $M'\subset M$.  Let $\eta$ be the associated closed $1$-form and assume that $\eta \in\mathcal{A}$, i.e.\ $\eta$ is Morse. Then $\psi_\R$ can be shown to be (measure theoretically) \emph{isomorphic} to a \emph{special flow} $T^r: I^r\to I^r$ over an interval exchange transformation $T:I\to I$ of $d\geq 1$ intervals and under a roof $r\in \ol(\sqcup_{\alpha\in \mathcal{A}} I_{\alpha})$.
 %of the form $r=\varphi+ g$ where $g$ is absolutely continuous on every $I_\alpha$ and $\varphi $ has \emph{logarithmic singularities} (see Figure~\ref{logroof}); moreover, the singularities are  of \emph{geometric type} (see \cite{Fr-Ul} for this latter point).
The number of exchanged intervals is $d=2g+s-1$ in the case when $\psi_\R$ is minimal and $s$ is the number of simple saddles, or, for a minimal component $M'$,
$d=2g'+s'-1$, where $g'$ is the genus of $M'$ and  $s'$ is the number of saddles in the closure of $M'$.
% and $s_b$ the number of saddles which belong to the boundary of the subsurface $M'$.
%If $\eta$ belongs to the open set $\mathcal{A}_{s_i,s_b}$ (i.e.~$\psi_\R$ has $s_i+s_b$ simple saddles, see \S~\ref{sec:generic}), then the number $d$ of intervals exchanged by $T$ is  $d=2g+s_i+s_b-1$.
	Furthermore, if $\eta \in \mathcal{U}_{min}$, the logarithmic singularities are \emph{symmetric}, i.e.\ $\varphi \in \logs(\sqcup_{\alpha\in \mathcal{A}} I_{\alpha})$ (while they are  \emph{asymmetric} for special flows representations of minimal components of typical $\eta \in \mathcal{U}_{\neg min}$).

{
\begin{remark}\label{rk:singularrep}
We recall for contrast that also \emph{translation flows} can be seen as special flows over an interval exchange map, but under a  roof function $r$ which is \emph{piecewise-constant} (and constant on each continuity interval of the IET). One can therefore see from these special representations that minimal (components of) locally Hamiltonian flows are time-changes of translation flows via a \emph{singular} reparametrization.\end{remark}}
%\begin{itemize}
%\item If $\eta \in \mathcal{A}_{s_i,s_b}$, then
%\item If $\eta \in \mathcal{U}_{min}$, $\varphi \in \logs(\sqcup_{\alpha\in \mathcal{A}} I_{\alpha})$, i.e.\ the singularities are \emph{symmetric}.
%\item
%\end{itemize}

%Fix a segment $\gamma\subset M(\eta)$ transverse to the flow, containing no fixed points and
%whose endpoints lie on outgoing separatrices of saddles. Let us say that the parametrization of $\gamma$  is standard if $\gamma: I \to M$ (where
%$I$ is an interval starting at zero) is parametrized so that $\eta(d\gamma ) = 1$. It is well known
%(see for example \cite[Section 4.4]{Yo} that, in the standard parametrization, the Poincaré
%first return map $T : I\to I$  of the flow $(\psi_t)_{t\in\R}$ to $\gamma$ is an IET. The number of exchanged intervals is
%$d=2g+s_i+s_b-1$. Suppose that $T=T_{(\pi,\lambda)}$ for some $(\pi,\lambda)\in \mathcal{S}^0_{\mathcal A}\times \R^{\mathcal{A}}_{>0}$ with $\#\mathcal A=d$.

\medskip
\subsubsection{Reduction to skew products.}\label{sec:reductionsp}
The study of (ergodic properties of) extensions can be reduced to the study of skew-products over IETs as follows.

\begin{proposition}[Reduction of ergodicity of extensions to skew products]\label{prop:redtoskew}
Consider a Morse closed one-form $\eta \in \mathcal{A}$ on $M$ and let $\psi_\R$ on $M$ be the associated locally Hamiltonian flow. Consider its minimal component $M'\subset M$. For every $C^{2+\epsilon}$-map $f:M'\to\R$ ($\epsilon>0$), the extension $\Phi^f_\R$ of $\psi_\R$ on $M'$ has a Poincar{\'e} map which, in suitable coordinates, is given by
%on a section $\gamma \times \R$ of the form which is measure theoretically isomorphic (via the conjugacy which sends the transverse measure $\nu$ induced by $\mu$ to $\gamma$ to the Lebesgue measure on $I$) to
a skew-product of the form
\begin{equation}\label{skform}
(x,y)\mapsto T_{\varphi_f}(x,y):= (Tx, y+\varphi_f(x)), \qquad (x,y)\in I\times \R.
\end{equation}
where $T=T_{(\pi,\lambda)}$ with $\pi$ irreducible  % which satisfies the Keane condition
 and the cocycle  $\varphi_f:I\to\R$
has logarithmic singularities, i.e $\varphi_f\in \ol(\sqcup_{\alpha\in \mathcal{A}} I_{\alpha})$, where $(I_\alpha)_{\alpha\in \mathcal{A}}$ are intervals exchanged by $T$.  % ADD
%As a Corollary, we therefore have the following reduction.
%\begin{corollary}\label{cor:redtoskew}

Moreover, the extension $\Phi^f_\R$ on $M'\times \R$ is ergodic with respect to $\mu\times Leb$ if and only if $T_{\varphi_f}: I\times \R\to I\times \R$ is ergodic with respect to the (restriction of) the $2$-dimensional Lebesgue measure on $I\times \R$.
%\end{corollary}
\end{proposition}
\noindent We give here only a brief sketch of the proof, referring to the proof in \cite{Fr-Ul} for details.
%The reduction was mostly proved  we here only briefly sketch the main points of the proof and add some additional information on the cocycles which arise from this reduction (which we will need later).
\begin{proof}%[Proof of Proposition~\ref{prop:redtoskew}]
Fix a segment $\gamma\subset M'\subset M$ transverse to the flow $\psi_\R$, containing no fixed points and
whose endpoints lie on outgoing separatrices of saddles.
%Let us say that the parametrization of $\gamma$  is standard if $\gamma: I \to M'$ (where
%$I$ is an interval starting at zero) is parametrized so that $\eta(d\gamma ) = 1$.
It is well known
(see for example \cite[Section 4.4]{Yo}) that one can choose a  parametrization $t\in I \to \gamma(t)$ of $\gamma$ by the unit interval $I=[0,1)$ so that  the Poincar{\'e}
first return map $T : I\to I$  of the flow $\psi_\R$ to $\gamma$ is an IET, which is minimal by assumption. It follows that $\pi$ is irreducible.

Denote by $r:I\to\R_{>0}$ the first return time map for the flow $(\psi_t)_{t\in\R}$ on $M'$. Then the isomorphism between the restriction of $\psi_\R$ to $M'$ and a special flow $T^r$ on $I^r$
 is given by
\[I^r\ni (x,r)\mapsto \psi_r(x)\in M'.\]
As recalled in the previous \S~\ref{sec:reductionsf}, $r\in {\ol}(\sqcup_{\alpha\in \mathcal{A}}
I_{\alpha})$ and moreover,  if $\psi_\R\in \mathcal{U}_{min}$, i.e.~$M'=M$, then $r\in {\logs}(\sqcup_{\alpha\in \mathcal{A}}
I_{\alpha})$, see e.g.\ \cite{Rav}.

Consider now the extension $\Phi^f_\R$ of $\psi_\R$ on $M'$ given by a bounded function $f:M'\to \R$. The Poincar{\'e} map of $\Phi^f_\R$
on $M'\times \R$ to the section $\gamma \times \R$ in the parametrization by $I\times \R$  is by construction an extension of the Poincar{\'e} map $T$ of $\psi_\R$ to $I$, with return time function $r(x,y)=r(x)$ (i.e.\ the return time only depends on the return to $I$ in the first coordinate, by definition of the section which has full fiber). Moreover, if we  consider
 the cocycle
\begin{equation}\label{Poincarecocyle}
\varphi_f(x):=\int_0^{r(x)}f(\psi_t(x))\,dt \end{equation}
(which gives the value of the ergodic integrals of $f$ along the trajectory from $x$ until the first return time to the section), one can then see that the first return Poincar{\'e} map of the extension $\Phi^f_\R$ has the form \eqref{skform}.  If $f$ is a $C^{2+\epsilon}$-map, from the explicit expression \eqref{Poincarecocyle} and the properties of $r$, one can then show  that also $\varphi_f\in \ol(\sqcup_{\alpha\in \mathcal{A}}I_\alpha)$ (see \cite{Fr-Ul} for details) and  $\varphi_f \in {\logs}(\sqcup_{\alpha\in \mathcal{A}}
I_{\alpha})$ if $\eta \in \mathcal{U}_{min}$.

The  final statement is simply a consequence that ergodicity of a minimal flow is equivalent to ergodicity of its Poincar{\'e} map with respect to the induced measure, together  with the remark that, under the isomorphism described above, the measure induced on the section $\gamma \times \R$ by the invariant measure $\mu\times Leb$  is mapped to the Lebesgue measure on $I\times \R$.
%Thus,   is a trivial consequence of the fact that a system is ergodic if and only if its Poincar{\'e} map is and the measure theoretical isomorphism between the Poincar{\'e} map and the skew-product \eqref{skform}.
\end{proof}
%\begin{remark} One can also show (refer again to \cite{Fr-Ul}) that the map $f\mapsto \varphi_f$ is a \emph{bounded linear operator} from the space ADD
%\end{remark}

\noindent
The following result shows that not only ergodicity, but also reducibility of the extension $\Phi^f_\R$  can be reduced to a property of the skew product $T_{\varphi_f}$ given by Proposition~\ref{prop:redtoskew}.

\begin{proposition}[Reduction of reducibility to skew products, \cite{Fr-Ul}]\label{prop:reduc}
For every minimal locally Hamiltonian flow $\psi_\R$ on $M$ with non-degenerate saddles and any $f\in C^{2+\epsilon}(M)$  vanishing on $\mathrm{Fix}(\psi_\R)$,  the associated flow $\Phi^f_\R$ is reducible if and only if the cocycle $\varphi_f:I\to\R$ is a coboundary with a bounded transfer map having at least one continuity point, i.e.\ there exists
a bounded $g:I\to\R$ such that $\varphi_f=g-g\circ T$ and $g$ has at least one continuity point.
\end{proposition}
\noindent The statement of the Proposition is proved in the proof\footnote{Note that the statement of  Lemma~6.3 in \cite{Fr-Ul} claims incorrectly that reducibility requires the existence of  transfer function continous at \emph{every} point, while a the \emph{existence} of a point of continuity is sufficient. Nevertheless, the \emph{proof} of Lemma~6.3 in \cite{Fr-Ul} is correct and gives a proof of the statement of Proposition~\ref{prop:reduc} here above.}  of Lemma~6.3 in \cite{Fr-Ul}.

\section{Rauzy-Veech induction and Diophantine-type conditions}\label{sec:IETDC}
In this section we define the Diophantine-type condition on IETs which we will use to prove our main results on deviations of ergodic averages and ergodicity of extensions. The condition is described in terms of  Rauzy-Veech induction, an algorithm introduced by Rauzy and Veech in \cite{Ra, Ve:erg} which is now a well established tool to study IETs as well to impose Diophantine conditions on them (see e.g.~\cite{Av-Fo, Av-Vi, Bu,  Ma-Mo-Yo, Ma-Mo-Yo:lin, Ul:abs, Ul:mix, Zor} and many more). We first recall some basic background material concerning Rauzy-Veech induction in \S~\ref{sec:RV}. The condition, that we call \emph{Uniform Diophantine Condition}, or for short \ref{UDC}, is defined in \S~\ref{sec:DC} (see Definition~\ref{def:UDC} in \S~\ref{sec:UDCdefsec}). In \S~\ref{sec:fullmeasure} we also prove that this condition is satisfied by a full measure set of IETs (see Theorem~\ref{thm:UDC}).

\subsection{Rauzy-Veech induction}\label{sec:RV}
We recall here some basic definitions and notation related to Rauzy-Veech induction that  will be used throughout the paper, including
how it acts on Rokhlin towers (\S~\ref{sec:Rokhlintowers}) and on Birkhoff sums (\S~\ref{sec:RVBS}), as well as   the definition of natural extension (\S~\ref{sec:natextension}). We recall also Oseledets theorem (\S~\ref{sec:Oseledets}). % and Rauzy-Veech cocycle(s) in \S~\ref{sec:RVcocycle}.

%\subsubsection{Definition of the induction}

\subsubsection{Elementary step of RV induction}\label{sec:RVelementary}
 Let $T=T_{(\pi,\lambda)}$,
$(\pi,\lambda)\in\mathcal{S}^0_{\mathcal{A}}\times\R_{>0}^{\mathcal{A}}$
be an IET satisfying Keane's condition. Then
$\lambda_{\pi_0^{-1}(d)}\neq\lambda_{\pi_1^{-1}(d)}$. Let
\[\tilde{I}=\left[0,\max\left({l}_{\pi_0^{-1}(d)},{l}_{\pi_1^{-1}(d)}\right)\right)\]
and denote by $\mathcal{R}(T)=\tilde{T}:\tilde{I}\to\tilde{I}$ the
first return map of $T$ to the interval $\tilde{I}$. Set
\[\vep(\pi,\lambda)=\left\{
\begin{array}{ccl}
0&\text{ if }&\lambda_{\pi_0^{-1}(d)}>\lambda_{\pi_1^{-1}(d)},\\
1&\text{ if }&\lambda_{\pi_0^{-1}(d)}<\lambda_{\pi_1^{-1}(d)}.
\end{array}
\right.\]  Let us consider a pair
$\tilde{\pi}=(\tilde{\pi}_0,\tilde{\pi}_1)\in\mathcal{S}^0_{\mathcal{A}}$,
where
\begin{eqnarray*}\tilde{\pi}_\vep(\alpha)&=&\pi_\vep(\alpha)
\text{ for all }\alpha\in\mathcal{A}\text{ and }\\
\tilde{\pi}_{1-\vep}(\alpha)&=&\left\{
\begin{array}{cll}
\pi_{1-\vep}(\alpha)& \text{ if
}&\pi_{1-\vep}(\alpha)\leq\pi_{1-\vep}\circ\pi^{-1}_\vep(d),\\
\pi_{1-\vep}(\alpha)+1& \text{ if
}&\pi_{1-\vep}\circ\pi^{-1}_\vep(d)<\pi_{1-\vep}(\alpha)<d,\\
\pi_{1-\vep}\pi^{-1}_\vep(d)+1& \text{ if
}&\pi_{1-\vep}(\alpha)=d.\end{array} \right.
\end{eqnarray*}
As it was shown by Rauzy in \cite{Ra}, $\tilde{T}$ is also an IET
on $d$-intervals
\begin{equation}\label{def:Theta}
\tilde{T}=T_{(\tilde{\pi},\tilde{\lambda})}\text{ with }
\tilde{\lambda}=A^{-1}(\pi,\lambda)\lambda,
\end{equation}
where
\[A(T)=A(\pi,\lambda)=I+E_{\pi_{\vep}^{-1}(d)\,\pi_{1-\vep}^{-1}(d)}\in\SL(\Z^{\mathcal{A}}).\]
Moreover,
\begin{equation}\label{eq:omega}
A^t(\pi,\lambda)\Omega_{\pi}A(\pi,\lambda)=\Omega_{\tilde{\pi}}.
\end{equation}
It follows that
$\ker\Omega_{{\pi}}=A(\pi,\lambda)\ker\Omega_{\tilde{\pi}}$.
Thus taking
$H(\pi)=\Omega_{{\pi}}(\R^{\mathcal{A}})=\ker\Omega_{{\pi}}^\perp$
we get
\begin{equation}\label{omega}
H(\tilde{\pi})=A^t(\pi,\lambda)H(\pi).
\end{equation}
Moreover,
$\dim H(\pi)=2g$ and $\dim \ker\Omega_{{\pi}}=\kappa(\pi)-1$, where
$\kappa(\pi)$ is the number of singularities and $g$ is the genus of
the translation surfaces  associated to $\pi$.

\subsubsection{Renormalized induction}\label{sec:RVinduction}

Let $\mathcal{G}\subset \mathcal{S}^0_{\mathcal A}$ be any Rauzy class, i.e.\
a minimal subset of $\mathcal{S}^0_{\mathcal A}$ for which $\mathcal{G}\times\R_{>0}^{\mathcal{A}}$
is $\mathcal{R}$-invariant. Let
\[\Delta^{\mathcal{A}}:=\{\lambda\in\R_{>0}^{\mathcal{A}}:|\lambda|=1\}.\]
Then we can define the normalized Rauzy-Veech renormalization
\[\widetilde{\mathcal{R}}:\mathcal{G}\times\Delta^{\mathcal{A}}\to\mathcal{G}\times\Delta^{\mathcal{A}},\quad \widetilde{\mathcal{R}}(\pi,\lambda)=(\tilde{\pi},\tilde{\lambda}/|\tilde{\lambda}|).\]
Veech in \cite{Ve:erg} proved the existence of an $\widetilde{\mathcal{R}}$-invariant ergodic measure $\mu_{\mathcal{G}}$
($\widetilde{\mathcal{R}}$ is recurrent with respect to $\mu_{\mathcal{G}}$)
which is equivalent to the product of the counting measure on $\mathcal{G}$ and the Lebesgue measure on $\Delta^\mathcal{A}$.

\smallskip

For every $T$ satisfying the Keane condition, the IET $\tilde{T}$ fulfills the Keane condition as
well. Therefore we can iterate the renormalization procedure and
generate a sequence of IETs $(\mathcal{R}^n(T))_{n\geq 0}$.
For every $n\geq 1$ let
\[A^{(n)}(T)=
A(T)\cdot A(\mathcal{R}(T))\cdot\ldots\cdot A(\mathcal{R}^{n-1}(T)).\]

%\subsubsection*{Products of matrices and return times}
In what follows, the norm of a vector is defined as the sum of
the absolute value of coefficients and for any matrix
$B=[B_{\alpha\beta}]_{\alpha,\beta\in\mathcal{A}}$ we set
$\|B\|=\max_{\alpha\in\mathcal{A}}\sum_{\beta\in\mathcal{A}}|B_{\alpha\beta}|$.

\smallskip
\subsubsection{Accelerations.} \label{sec:accelerations}
Let $T:I\to I$ be an arbitrary  IET satisfying Keane's condition. Let $(n_k)_{k\geq 0}$ be an increasing sequence of integer numbers
with $n_0=0$, called an \emph{accelerating sequence}. For every $k\geq 0$ let $T^{(k)}:=\mathcal{R}^{n_k}(T):I^{(k)}\to I^{(k)}$. Denote by $(\pi^{(k)},\lambda^{(k)})$ the pair defining
$T^{(k)}$ and  by
$\lambda^{(k)}=(\lambda^{(k)}_\alpha)_{\alpha\in\mathcal{A}}=(|I^{(k)}_\alpha|)_{\alpha\in\mathcal{A}}$ the
vector which determines $T^{(k)}$.

\smallskip
\noindent In view of \eqref{def:Theta}, letting $Z(k+1):=A^{(n_{k+1}-n_k)}(\mathcal{R}^{n_k}(T))^t$ for $k\geq 0$ we have
\[\lambda^{(k)}=Z(k+1)^t \lambda^{(k+1)}\text{ for all }k\geq 0.\]
We use the notation from \cite{Ma-Mo-Yo}, but adopt the convention later introduced in \cite{Ma-Yo}.
%{ Change; the notation is from MMY, but the convention and type of cocycle from MY...} \cite{Ma-Mo-Yo}.
For each $0\leq k<l$ let
\[Q(k,l)=Z(l)\cdot Z(l-1)\cdot\ldots\cdot Z(k+2)\cdot Z(k+1)=A^{(n_{l}-n_k)}(\mathcal{R}^{n_k}(T))^t.\]
Then $Q(k,l)\in SL_\mathcal{A}(\Z)$ and
\[\lambda^{(k)}=Q(k,l)^t \lambda^{(l)}.\]
It follows that
\begin{equation}\label{neq:dli}
|I^{(k)}|\leq |I^{(l)}|\|Q(k,l)\|.
\end{equation}
We will write $Q(k)$ for $Q(0,k)$.

\smallskip
\noindent We say that $Z(k)$, $k\in\N$ (resp.~$Q(k,l)$) are the \emph{matrices} (resp.~the \emph{product matrices}) of the acceleration of $A$ along the (accelerating) sequence $(n_k)_{k\in\N}$

\subsubsection{Rokhlin towers}\label{sec:Rokhlintowers}
By definition,
$T^{(l)}:I^{(l)}\to I^{(l)}$ is the first return map of
$T^{(k)}:I^{(k)}\to I^{(k)}$ to the interval $I^{(l)}\subset
I^{(k)}$. Moreover, $Q_{\alpha\beta}(k,l)$ is the time spent by
%{ Can we check the convention here? I usuall think it's $\beta \alpha$}
any point of $I^{(l)}_{\alpha}$ in $I^{(k)}_{\beta}$ until it
returns to $I^{(l)}$.  It follows that
\[Q_{\alpha}(k,l)=\sum_{\beta\in\mathcal{A}}Q_{\alpha\beta}(k,l)\]
is the first return time of points of $I^{(l)}_{\alpha}$ to
$I^{(l)}$.

\smallskip
\noindent The map $T^{(k)}:I^{(k)}\to I^{(k)}$ can be then represented as a \emph{Rokhlin skyscraper} as follows. For every $\alpha\in \mathcal{A}$, we say that the set
$$
\{ (T^{(k)})^i (I^{(l)}_{\alpha}), \qquad 0\leq i < Q_{\alpha}(k,l)\}
$$
is called a \emph{Rokhlin tower}. Notice that the $Q_{\alpha}(k,l)$ sets part of it are \emph{disjoint} intervals called \emph{floors} of the tower and that, for $ 0\leq i < Q_{\alpha}(k,l)$, $T^{(k)}$ acts on the $i^{th}$ floor $(T^{(k)})^i (I^{(l)}_{\alpha})$ mapping it to the $(i+1)^{th}$ one. The union of all Rokhlin towers over $\alpha\in\mathcal{A}$ gives $I^{(k)}$.
%\begin{remark}\label{rem:nk}
%Note that $Q(k,l)$ belongs to $SL_2(\mathcal A)$ and has nonnegative coefficients. Moreover,
%for each $\alpha, \beta \in\mathcal A$ and $k\geq 0$, the sequence $(Q_{\alpha\beta}(k,l))_{l>k}$ is a nondecreasing.  It
%follows from \cite{Ma-Mo-Yo} that for fixed $k$ and $l$ large enough all coefficients
%$Q(k,l)$ are positive. Therefore, for every IET satisfying Keane's condition
%we define a sequence of integers $(n_k)_{k\geq 0}$ as follows: $n_0 = 0$ and $n_{k+1}$ is the smallest
%integer such that all coefficients of $Q(n_k , n_{k+1})$ are strictly positive.
%\end{remark}

\smallskip
\subsubsection{Special Birkhoff sums.}\label{sec:RVBS}
We deal with the \emph{special Birkhoff sums} operators $S(k,l):  L^1(I^{(k)})\to L^1(I^{(l)})$ for $0\leq k<l$  defined by
\[S(k,l)f(x)=\sum_{0\leq j<Q_\alpha(k,l)}f((T^{(k)})^jx)\quad\text{ if }\quad x\in I^{(l)}_\alpha.\]

Let $T=T^{(0)}$  be an IET satisfying Keane's condition. For every $k\geq 0$
let $\Gamma^{(k)}\subset L^1(I^{(k)})$ be the subspace of  functions on
$I^{(k)}$ which are constant on each $I^{(k)}_{\alpha}$,
$\alpha\in\mathcal{A}$.  Then for $0\leq k<l$ we have
$S(k,l)\Gamma^{(k)}=\Gamma^{(l)}$.
Let us identify every function $\sum_{\alpha\in \mathcal{A}}
h_{\alpha}\chi_{I^{(k)}_{\alpha}}\in \Gamma^{(k)}$ with the
vector $h=(h_{\alpha})_{\alpha\in
\mathcal{A}}\in\R^{\mathcal{A}}$. Clearly $\Gamma^{(k)}$ is
isomorphic to $\mathbb{R}^{\mathcal{A}}$.
Under the identification,
 the
operator $S(k,l)$ is the linear automorphism of
$\R^{\mathcal{A}}$ whose matrix in the canonical basis is
$Q(k,l)$.
In view of \eqref{omega} for $0\leq k<l$ we have
\[Q(k,l)H(\pi^{(k)})= H(\pi^{(l)}).\]

For every $k\geq 0$ let
\begin{align*}
\Gamma^{(k)}_s:=\{h\in\Gamma^{(k)}:\exists_{\sigma>0}\exists_{C>0}\forall_{l>k}\, \|Q(k,l)h\|\leq C\|Q(k,l)\|^{-\sigma}\}.
\end{align*}
The space $\Gamma^{(k)}_s$ is a subspace of $H(\pi^{(k)})$ and for every $l>k$ we have
\[Q(k,l)\Gamma^{(k)}_{s}=\Gamma^{(l)}_s.\]
Therefore, the restriction operator and the quotient operators of $Q(k, l)$
\begin{gather*}
Q_s(k,l):\Gamma^{(k)}_s\to\Gamma^{(l)}_s,\quad
Q_\flat(k,l):\Gamma^{(k)}/\Gamma^{(k)}_s\to\Gamma^{(l)}/\Gamma^{(l)}_s,\quad
Q_\sharp(k,l):H(\pi^{(k)})/\Gamma^{(k)}_s\to H(\pi^{(l)})/\Gamma^{(l)}_s
\end{gather*}
are well defined and are invertible.
Arguments presented in Section 3.2 in \cite{Ma-Yo} shows that if $\dim\Gamma^{(0)}_s=g$ then
\begin{equation}\label{eq:unstst}
\|Q_\sharp(k,l)^{-1}\|=\|Q_s(k,l)\|.
\end{equation}

\subsubsection{The natural extension}\label{sec:natextension}
Rauzy-Veech induction is not intertible, but it can be extended to an invertible induction on the space of \emph{zippered rectangles} (as described in the seminar paper by Veech \cite{Ve:erg}).
We recall briefly the construction. We refer the reader who needs more background to the lecture notes by Yoccoz \cite{Yo} or Viana \cite{Vi0}.

\smallskip
\noindent For every $\pi\in S^{\mathcal A}_0$ let
\[\Theta_\pi:=\Big\{\tau\in \R^{\mathcal A}:\sum_{\pi_0(\alpha)\leq k}\tau_\alpha>0,\; \sum_{\pi_1(\alpha)\leq k}\tau_\alpha<0\text{ for }1\leq k<d\Big\}.\]
For every $\tau\in \Theta_\pi$ let $h=h(\tau)=\Omega\tau\in\R^{\mathcal{A}}_{>0}$.
For every Rauzy class $\mathcal{G}\subset S^{\mathcal A}_0$ let
\begin{equation}\label{XGdef}
X(\mathcal{G})=\bigcup_{\pi\in\mathcal{G}}\{(\pi,\lambda,\tau)\in\{\pi\}\times\Delta^\mathcal{A}\times\Theta_\pi:\langle\lambda,\Omega_\pi\tau\rangle=1\}.\end{equation}
For every $(\pi,\lambda,\tau)\in X(\mathcal{G})$ denote by $M(\pi,\lambda,\tau)$ the translation surface arising in the  zippered rectangles process.
Then $M(\pi,\lambda,\tau)$ is zippered from the rectangles $I_\alpha\times[0,h_\alpha]$, $\alpha\in\mathcal{A}$ such that
the points $\sum_{\pi_0(\alpha)\leq k}(\lambda_\alpha+i\tau_\alpha)$, $0\leq k\leq d$ are its singular points. Moreover, the IET $T$ is the first return map to $I\subset M(\pi,\lambda,\tau)$ for the vertical flow on $M(\pi,\lambda,\tau)$.

The map $\widehat{\mathcal{R}}:X(\mathcal{G})\to X(\mathcal{G})$ given by
\[\widehat{\mathcal{R}}(\pi,\lambda,\tau)=\Big(\tilde{\pi},\frac{A^{-1}(\pi,\lambda)\lambda}{|A^{-1}(\pi,\lambda)\lambda|}, |A^{-1}(\pi,\lambda)\lambda|A^{-1}(\pi,\lambda)\tau\Big)\]
is an invertible map and  is the natural extension of $\widetilde{\mathcal{R}}$. Denote by $\widehat{\mu}_{\mathcal{G}}$ the natural extension
of the measure ${\mu}_{\mathcal{G}}$. Then $\widehat{\mu}_{\mathcal{G}}$ is $\widehat{\mathcal{R}}$-invariant and $\widehat{\mathcal{R}}$ is recurrent and ergodic with respect to $\widehat{\mu}_{\mathcal{G}}$.

\subsubsection{Oseledets splitting}\label{sec:Oseledets}
Let us extend the cocycle $A:\mathcal{G}\times\Lambda^\mathcal{A}\to SL_\mathcal{A}(\Z)$ to $\widehat{A}:X(\mathcal{G})\to SL_\mathcal{A}(\Z)$ by
\[\widehat{A}(\pi,\lambda,\tau):=A(\lambda,\tau)\]
and let us consider the cocycle $\widehat{A}:\Z\times X(\mathcal{G})\to SL_\mathcal{A}(\Z)$
\[\widehat{A}^{(n)}(\pi,\lambda,\tau)=\left\{
\begin{array}{cl}
\widehat{A}(\pi,\lambda,\tau)\cdot \widehat{A}(\widehat{\mathcal{R}}(\pi,\lambda,\tau))\cdot\ldots\cdot \widehat{A}(\widehat{\mathcal{R}}^{n-1}(\pi,\lambda,\tau))\!&\!\text{if } n\geq 0\\
\widehat{A}(\widehat{\mathcal{R}}^{-1}(\pi,\lambda,\tau))\cdot \widehat{A}(\widehat{\mathcal{R}}^{-2}(\pi,\lambda,\tau))\cdot\ldots\cdot \widehat{A}(\widehat{\mathcal{R}}^{n}(\pi,\lambda,\tau))\!&\!\text{if } n<0.
\end{array}
\right.\]
Then
\begin{equation}\label{eq:An>0}
\widehat{A}^{(n)}(\pi,\lambda,\tau)={A}^{(n)}(\pi,\lambda)\text{ if }n\geq 0.
\end{equation}
\noindent Let $Y\subset X(\mathcal{G})$ be a subset with $0<\widehat{\mu}_{\mathcal{G}}(Y)<+\infty$. For a.e.\ $(\pi,\lambda,\tau)\in Y$
let $r(\pi,\lambda,\tau)\geq 1$ by the first return time of $(\pi,\lambda,\tau)$ for the map $\widehat{\mathcal{R}}$.
 Denote by $\widehat{\mathcal{R}}_Y:Y\to Y$
the induced map and by $\widehat{A}_Y: Y\to SL_\mathcal{A}(\Z)$ the induced cocycle, i.e.\
\[\widehat{\mathcal{R}}_Y(\pi,\lambda,\tau)=\widehat{\mathcal{R}}^{r(\pi,\lambda,\tau)}(\pi,\lambda,\tau),\quad
\widehat{A}_Y(\pi,\lambda,\tau)=\widehat{A}^{(r(\pi,\lambda,\tau))}(\pi,\lambda,\tau)\]
for a.e.\ $(\pi,\lambda,\tau)\in Y$. Let $\widehat{\mu}_Y$ be the restriction of $\widehat{\mu}_{\mathcal{G}}$ to $Y$.
Then $\widehat{\mathcal{R}}_Y$ is an ergodic measure-preserving invertible map on $(Y,\widehat{\mu}_Y)$.

\smallskip
Suppose that $\log\|\widehat{A}_Y\|$ and $\log\|\widehat{A}^{-1}_Y\|$ are integrable. Then, by Oseledets theorem, symplecticity of $\widehat{A}_Y$ (see \cite{Zor})
and simplicity of spectrum (see \cite{Av-Vi}), there exists $\lambda_1>\ldots>\lambda_g>0$ such that
for a.e.\ $(\pi,\lambda,\tau)\in Y$ we have a \emph{Oseledets splitting}
\[\R^{\mathcal A}=\bigoplus_{-g\leq i\leq g}\Gamma_i(\pi,\lambda,\tau)\]
for which
\begin{align*}
\lim_{n\to\pm\infty}\frac{1}{n}\log\|\widehat{A}_Y^{(n)}(\pi,\lambda,\tau)^tv\|=\lambda_i&\text{ if }v\in \Gamma_i(\pi,\lambda,\tau)\text{ and }i>0\\
\lim_{n\to\pm\infty}\frac{1}{n}\log\|\widehat{A}_Y^{(n)}(\pi,\lambda,\tau)^tv\|=-\lambda_i&\text{ if }v\in \Gamma_i(\pi,\lambda,\tau)\text{ and }i<0\\
\lim_{n\to\pm\infty}\frac{1}{n}\log\|\widehat{A}_Y^{(n)}(\pi,\lambda,\tau)^tv\|=0&\text{ if }v\in \Gamma_0(\pi,\lambda,\tau),
\end{align*}
\[\dim \Gamma_i(\pi,\lambda,\tau)=1\text{ if }i\neq 0, \quad \dim \Gamma_0(\pi,\lambda,\tau)=\kappa-1 .\]
Furthermore, we have that
\[H(\pi)=\bigoplus_{i\neq 0}\Gamma_i(\pi,\lambda,\tau).\]
We denote by $\Gamma_s(\pi,\lambda,\tau)$ and $\Gamma_u(\pi,\lambda,\tau)$ the \emph{stable} and \emph{unstable} spaces, which are given respectively by
\begin{equation}\label{stabledef}
\Gamma_s(\pi,\lambda,\tau):=\bigoplus_{-g\leq i\leq -1}\Gamma_i(\pi,\lambda,\tau)\text{ and }
\Gamma_u(\pi,\lambda,\tau):=\bigoplus_{1\leq i\leq g}\Gamma_i(\pi,\lambda,\tau).
\end{equation}
Notice that both $\Gamma_s(\pi,\lambda,\tau)$ and $\Gamma_u(\pi,\lambda,\tau)$ have exactly dimension $g$. We say in this case that the Oseledets splitting is \emph{of hyperbolic type}.

\subsubsection{Veech bases for the kernel $\ker \Omega_{\pi}$}\label{sec:kernel}
In \cite{Ve:erg,Ve2}, Veech explicitly defines a bases for  $\ker \Omega_{\pi}$  for every $\pi$ in a given Rauzy class. We recall the construction (which uses the classical notation for the permutation describing the IETs, also called \emph{monodromy}, namely the permutation $\pi_1\circ\pi^{-1}_0$). Let us first define the \emph{extended permutation}  $p:\{0,1,\ldots,d,d+1\}\to\{0,1,\ldots,d,d+1\}$ to be the
permutation
\[p(j)=\left\{
\begin{matrix}
\pi_1\circ\pi^{-1}_0(j)&\text{ if }&1\leq j\leq d\\
j&\text{ if }&j=0,d+1.
\end{matrix}
\right.
\]
Following Veech (see \cite{Ve:erg,Ve2}), denote by $\sigma=\sigma_\pi$ the
corresponding permutation on $\{0,1,\ldots,d\}$,
\[\sigma(j)=p^{-1}(p(j)+1)-1\text{ for }0\leq j\leq d.\]
Notice that (recalling Remark~\ref{rk:endpoints} and the definition just before of $\widehat{T}$), we have
$\widehat{T}_{(\pi,\lambda)}r_{\pi_0^{-1}(j)}={T}_{(\pi,\lambda)}r_{\pi_0^{-1}(\sigma
j)}$ for all $j\neq 0,p^{-1}(d)$.

\smallskip
\noindent Denote by $\Sigma(\pi)$ the set
of orbits for the permutation $\sigma$. Let $\Sigma_0(\pi)$ stand
for the subset of orbits that do not contain zero. Then
$\Sigma(\pi)$ corresponds to the set of singular points of any
translation surface associated to $\pi$ and hence
$\#\Sigma(\pi)=\kappa(\pi)$.

\smallskip
\noindent For every $\mathcal{O}\in\Sigma(\pi)$ denote by
$b(\mathcal{O})\in\R^{\mathcal{A}}$ the vector given by
\begin{equation} \label{bdef}
b(\mathcal{O})_{\alpha}=\chi_{\mathcal{O}}(\pi_0(\alpha))-\chi_{\mathcal{O}}(\pi_0(\alpha)-1)
\text{ for }\alpha\in\mathcal{A},
\end{equation}
where $\chi_{\mathcal{O}} (j)=1$ iff $j \in \mathcal{O}$ and $0$ otherwise.
Moreover, for every $\mathcal{O}\in \Sigma(\pi)$, we denote by
\begin{equation} \label{defAO}
\mathcal{A}^-_{\mathcal{O}}=\{
\alpha\in\mathcal{A}, \ \pi_0(\alpha)\in \mathcal{O}\}, \qquad
\mathcal{A}^+_{\mathcal{O}}=\{ \alpha\in\mathcal{A}, \
\pi_0(\alpha)-1\in \mathcal{O}\} .
\end{equation}
If $\alpha\in \mathcal{A}_{\mathcal{O}}^+$ (respectively $\alpha\in
\mathcal{A}_{\mathcal{O}}^-$) then the left (respectively right)
endpoint of $I_\alpha$ belongs to a separatrix of the saddle
represented by $\mathcal{O}$.
\begin{lemma}[see \cite{Ve2}]\label{bcharh}
For every irreducible pair $\pi$ we have:
\begin{enumerate}[label=\textrm{(\roman*)}]
\item  $\sum_{\mathcal{O}\in\Sigma(\pi)}b(\mathcal{O})=0$;
\item the vectors
$b(\mathcal{O})$, $\mathcal{O}\in\Sigma_0(\pi)$ are linearly
independent;
\item  the linear subspace generated by
$\{ b(\mathcal{O})$, $\mathcal{O}\in\Sigma_0(\pi)\}$  is equal to
$\ker\Omega_\pi$.
\end{enumerate}
Moreover, $h\in H(\pi)$ if and only if $\langle
h,b(\mathcal{O)}\rangle=0$ for every $\mathcal{O}\in\Sigma(\pi)$.
\end{lemma}
\noindent Veech also describes how these bases change under Rauzy-Veech induction:
\begin{lemma}[see Veech, \cite{Ve2}]\label{invario}
Suppose that
$T_{(\tilde{\pi},\tilde{\lambda})}=\mathcal{R}(T_{(\pi,\lambda)})$.
Then there exists a bijection
$\xi:\Sigma(\pi)\to\Sigma(\tilde{\pi})$ such that
$$A(\pi,\lambda)^{-1}b(\mathcal{O})=b(\xi\mathcal{O}), \qquad \text{for\ all}\ \ \mathcal{O}\in\Sigma(\pi).$$
\end{lemma}

\subsubsection{The boundary operator}\label{sec:bddef}
The following operator $\partial_\pi$ is known by \emph{boundary operator} (as a special case of the more general operator introduced in \cite{Ma-Mo-Yo}, see \S~\ref{bd:ac}). Let $\Sigma(\pi)$  and $\mathcal{A}_{\mathcal{O}}^\pm$ be as in the previous subsection.
\begin{definition}
Let $\partial_\pi:\R^{\mathcal{A}}\to\R^{\Sigma(\pi)}$ stand for the
linear transformation which maps a vector $h\in \R^{\mathcal{A}}$ to the vector in $\R^{\Sigma(\pi)}$ whose coordinates  $(\partial_\pi h)_\mathcal{O}, \mathcal{O}\in \Sigma(\pi)$ are given by
\[(\partial_\pi h)_\mathcal{O} :=\langle h,b(\mathcal{O)}\rangle=
\sum_{\alpha\in\mathcal{A}_{\mathcal{O}}^-}h_\alpha-\sum_{\alpha\in\mathcal{A}_{\mathcal{O}}^+}h_\alpha, \qquad \text{for} \ \ \mathcal{O}\in\Sigma(\pi). \]
\end{definition}
\noindent One sees (in light of Remark~\ref{rk:endpointspartition}) that the image of $\partial_\pi$ is:
\begin{equation}\label{eq:pariso}
\partial_\pi(\R^{\mathcal{A}})=\Big\{(x_{\mathcal{O}})_{\mathcal{O}\in\Sigma(\pi)}:\sum_{\mathcal{O}\in\Sigma(\pi)}x_{\mathcal{O}}=0\Big\}.
\end{equation}

\begin{remark}\label{rk:constboundary}
We can identify a vector $h\in \R^{\mathcal{A}}$ with a piecewise constant function $g_h$, which gives the constant value $h_\alpha$ to the subinterval $I_\alpha$. Then the operator $\partial$ can be thought of as acting on piecewise constant functions and producing, as a value at $\mathcal{O}\in \Sigma(\pi)$, the \emph{sum of jumps} of the function $g_h$ at the endpoints corresponding to the singularity labelled by $\mathcal{O}$.
\end{remark}
\noindent Two extensions of this operator (viewed as in the previous remark as an operator on functions) will be defined later, to functions  piecewise absolutely continuous on each $I_\alpha$ (\S~\ref{bd:ac}) and to functions with logarithmic singularities (\S~\ref{bd:ls}).

\subsubsection{Boundary operator estimate} \label{sec:bd:estimate}
Let $H(\pi):=\ker\partial_\pi$. Denote by $p_{H(\pi)}:\R^{\mathcal{A}}\to H(\pi)$ the orthogonal projection
on $H(\pi)$ with respect to the standard scalar product on $\R^{\mathcal{A}}$.

\begin{lemma}\label{bd:estimate}
For any $h\in \mathbb{R}^\mathcal{A}$, we have
\begin{equation}\label{eq:projnorm}
\|p_{H(\pi)}h\|\leq \sqrt{d}\|h\|.
\end{equation}
Moreover, for any Rauzy class $\mathcal{G}\subset \mathcal{S}^0_{A}$ there exists a positive constant $C_{\mathcal{G}}$ such that
for every $\pi\in\mathcal{G}$ and $h\in \R^{\mathcal{A}}$ we have
\begin{equation}\label{eq:proj}
\|h-p_{H(\pi)}h\|\leq C_{\mathcal{G}}\|\partial_\pi h\|.
\end{equation}
\end{lemma}
\begin{proof}
Let $H(\pi)^{\perp}\subset\R^{\mathcal{A}}$ be the orthogonal complement of $H(\pi)$.
By Lemma~\ref{bcharh}, $\partial_\pi:H(\pi)^{\perp}\to\R^{\Sigma(\pi)}$ is a linear isomorphism. It follows that there exists $C_\pi>0$ such that
\[\|h\|\leq C_\pi\|\partial_\pi h\| \quad\text{ for all }\quad h\in H(\pi)^{\perp}.\]
Hence \eqref{eq:proj} holds with $C_{\mathcal{G}}=\max\{C_\pi:\pi\in\mathcal{G}\}$. Denote by $\|\,\cdot\,\|_2$ the Euclidean norm on $\R^{\mathcal{A}}$. Since $\|h\|_2\leq \|h\|\leq\sqrt{d}\|h\|_2$ and $p_{H(\pi)}$ is
an orthogonal projection, we have
\[
\|p_{H(\pi)}h\|\leq\sqrt{d}\|p_{H(\pi)}h\|_2\leq\sqrt{d}\|h\|_2\leq\sqrt{d}\|h\|.
\]
\end{proof}

\subsection{The Uniform Diophantine-type Condition and its full measure}\label{sec:DC}
We will now define the Diophantine-type condition that we will use. First, it is convenient to introduce an acceleration of Rauzy-Veech induction which produces times which we call \emph{Rokhlin-balanced}. We then define the condition and prove that it has full measure.

\subsubsection{The Rokhlin-balanced acceleration}
The following  acceleration of Rauzy-Veech induction  produces  times of the Rauzy-Veech algorithm where the corresponding Rokhlin towers
(see \ref{sec:Rokhlintowers}) are \emph{balanced} in the sense that all bases have comparable lengths (see \eqref{def:UDC-d} in Definition~\ref{def:balanced})
and all the towers travel together for a long enough time (see \eqref{def:UDC-g} in Definition~\ref{def:balanced}). We call  these times \emph{Rokhlin-balanced}.
\begin{definition}[{Rokhlin-balance}]\label{def:balanced}
Let us say that an accelerating  sequence $(n_k)_{k\geq 0}$ is \emph{Rokhlin-balanced} if there exist  constants $\kappa>1$  and $0<\delta<1$ such that the following two conditions hold for every $k\in \mathbb{N}$:
\begin{align}
 \tag{B1} &  |I^{(k)}| \leq \kappa\,  |I^{(k)}_\alpha|\qquad  \text{ for all } k\geq 1 \text{ and } \alpha\in \mathcal A; \label{def:UDC-d} \\
 \tag{B2} & \text{for every $k\geq 1$ there exists a natural } \text{number $0<p_{k}\leq \min_{\alpha\in\mathcal{A}}Q_\alpha(k)$ such that}  \label{def:UDC-g}  \\
& \{T^iI^{(k)}:0\leq i<p_k\} \text{ is a Rokhlin tower of intervals with measure greater than  $\delta|I|$.} \nonumber
\end{align}
We say that an IET is \emph{Rokhlin-balanced} if it satisfies Keane's condition and it admits  a Rokhlin balanced accelerating sequence $(n_k)_{k\geq 0}$.
\end{definition}
\begin{remark}
Notice that by conditions \eqref{def:UDC-d} and \eqref{def:UDC-g},  for every $\alpha\in\mathcal A$ and $k\geq 1$ we have
\begin{gather}
\label{eq:lg}
\|Q(k)\||I^{(k)}|\leq \kappa \sum_{\alpha\in\mathcal{A}}Q_\alpha(k)|I^{(k)}_\alpha|=\kappa|I|,\text{ and }\\
\label{eq:qlambda}
Q_\alpha(k)\lambda^{(k)}_\alpha\geq \frac{1}{\kappa}p_k|I^{(k)}|\geq \frac{\delta}{\kappa}|I|.
\end{gather}
so that each Rokhlin tower of a balanced acceleration induction time has measure uniformly bounded below.
\end{remark}
%In view of \eqref{def:UDC-c}, there exists $C>0$ and $m\geq 1$ such that
%\begin{equation}\label{eq:qkl}
%\|Q(k,l)\|\leq C\|Q(l)\|^m\text{ for all }0\leq k\leq l.
%\end{equation}

\smallskip

Let us show that for almost every IET one can find a Rokhlin-balanced sequence by considering returns of Rauzy-Veech induction to special compact sets (for  the parameter space of the natural extension, see \S~\ref{sec:natextension}). Let us recall that $X(\mathcal{G})$ denotes the domain of the natural extension of the Rauzy-Veech induction  (see  \eqref{XGdef} in \S~\ref{sec:natextension}).
\begin{lemma}\label{lemma:Rokhlinbalance}
Let $\pi$ be irreducible. For Lebesgue-almost every choice of $\lambda$, the IET $T=T_{(\pi,\lambda)}$ is Rokhlin-balanced. Furthermore,  for every $ 0<\delta< 1 $ one can define a set $Y=Y(\delta) \subset X(\mathcal{G})$ such that a Rokhlin-balanced accelerating sequence with constant $\delta$ is given by returns of the natural extension of Rauzy-Veech induction to $Y$.
\end{lemma}
\begin{proof}
Fix $0< \delta< 1 $.
Let us consider a subset $Y=Y(\delta)\subset X(\mathcal{G})$ which satisfies:
\begin{enumerate}[label=(\roman*)]
\item \label{defYi} its projection $Y_0$ on $\mathcal{G}\times\Lambda^\mathcal{A}$ is  precompact
with respect to the Hilbert metric;
\item \label{defYii} %there exists $0<\delta<1$ such that
for every $(\pi,\lambda,\tau)\in Y$ we have
\begin{equation*}
\min \Big\{\Big\{\sum_{\pi_0(\alpha)\leq k}\tau_\alpha:1\leq k<d\Big\}\cup\{h_\alpha(\tau):\alpha\in\mathcal A\}\Big\}>\delta\max\{h_\alpha(\tau):\alpha\in\mathcal A\};
\end{equation*}
\end{enumerate}
Let $R>0$ be such that $Y_0 \subset \mathcal{G}\times \overline{B}_H((1/d,\ldots,1/d),R)$, where $\overline{B}_H((1/d,\ldots,1/d),R)$ is the closed ball
(with respect to the Hilbert metric $d_H$)
of radius $R$ and center at the center of the simplex $\Lambda^{\mathcal A}$.

\smallskip
\noindent {\it Balance at visit times.} Consider any sequence $(n_k)_{k\geq 1}$ which corresponds to visits to the set $Y$.
By definition, for every $k$ belonging to this subsequence, $(\pi^{(k)}, \lambda^{(k)}, \tau^{(k)})\in  Y$. It follows that $d_H\big(\lambda^{(k)},(1/d,\ldots,1/d)\big)\leq R$. Therefore
$$\max_{\alpha\in\mathcal A}|I^{(k)}_\alpha|/\min_{\alpha\in\mathcal A}|I^{(k)}_\alpha|\leq e^R,$$ which implies the condition \eqref{def:UDC-d} for $\kappa:= e^R$.

\noindent As $(\pi^{(k)}, \lambda^{(k)}, \tau^{(k)})=\widehat{\mathcal{R}}^{n_k}(\pi,\lambda,\tau))\in Y$, by condition \ref{defYii} in the choice of $Y$,
taking
\[t^{(k)}:=\min\Big\{\Big\{\sum_{\pi^{(k)}_0(\alpha)\leq l}\tau^{(k)}_\alpha:1\leq l<d\Big\}\cup\{h^{(k)}_\alpha:\alpha\in\mathcal A\}\Big\}\quad (h^{(k)}=h(\tau^{(k)}))\]
we have that $I^{(k)}\times[0,t^{(k)}]$ is a rectangle (without singular points inside) in the translation surface $M(\pi^{(k)}, \lambda^{(k)}, \tau^{(k)})$ ($=M(\pi, \lambda, \tau)$) and its area is greater than
\[t^{(k)}\sum_{\alpha\in\mathcal A}\lambda^{(k)}_\alpha>\delta\max_{\alpha\in\mathcal A}h^{(k)}_\alpha\sum_{\alpha\in\mathcal A}\lambda^{(k)}_\alpha\geq \delta\langle\lambda^{(k)},h^{(k)}\rangle=\delta|I|.\]
This gives \eqref{def:UDC-g} with $p_k:=[t^{(k)}/\max_{\alpha\in\mathcal{A}}h_\alpha(\tau)]$ and $\delta:=\tfrac{\delta^2}{2}<\tfrac{\delta}{2}\tfrac{\min_{\alpha\in\mathcal{A}}h_\alpha(\tau)}{\max_{\alpha\in\mathcal{A}}h_\alpha(\tau)}$.

\smallskip
\noindent {\it Typical Rokhlin balance.} It now follows from Poincar{\'e} recurrence theorem (and absolute continuity and finiteness of the Veech invariant measure, see \cite{Ve:erg}) that almost every IET visits $Y(\delta)$ infinitely often and hence is Rokhlin-balanced.
\end{proof}

\subsubsection{The Uniform Diophantine Condition definition}\label{sec:UDCdefsec}
The Diophantine-type condition that we will use in the main theorems is the following.
\begin{definition}[\customlabel{UDC}{UDC}]\label{def:UDC}
An IET $T:I\to I$ satisfying Keane's condition, satisfies the \emph{Uniform Diophantine Condition} \text{UDC} if $T$ is  Rokhlin-balanced (in the sense of Definition~\ref{def:balanced}), and
for every $\tau>0$ there exist  constants $0<c<C$, a Rokhlin-balanced accelerating sequence  $(n_k)_{k\geq 0}$
%corresponding to an accelerated cocycle $A_Y$
and an increasing sequence of integers
 $(r_n)_{n\geq 0}$ with $r_0=0$ and $r_n/n\to \alpha>0$, so that:
\begin{align}
\tag{O}\label{def:O}
&\text{$T$ is \emph{Oseledets\ generic}, i.e.\ there exists an extension $(\pi, \lambda, \tau)$ of $T=T_{(\pi,\lambda)}$}\\
&\text{such that it admits an Oseledets splitting of hyperbolic type, as in \S~\ref{sec:Oseledets};}\nonumber
\end{align}
%\begin{align}
%\tag{$O$} \quad  & T\  \text{is\ \emph{Oseledets\ generic}}, \text{i.e.}\ there exsits \label{def:O} \ \text{admits} \ \\ &  \text{an %Oseledets}\ \text{splitting of hyperbolic type, as in \S~\ref{sec:Oseledets};}\nonumber%
%\end{align}
% $T$ is hyperbolic (in the sense of Definition~\ref{def:Rokhlinbalanced}) for the accelerated cocycle $A_Y$
 and, furthermore,  the matrices $Z(k) $ and product matrices $Q(k,l)$ of the acceleration along the subsequence $(n_k)_{k\in \N}$ (see \S~\ref{sec:accelerations})
 satisfy the following conditions:
 %there exist  constants $\kappa>1$ and $\lambda_1>\ldots>\lambda_g>0$ such that
\begin{align}
\tag{UDC1} &\|Q_s(k,l)\|\leq Ce^{-\lambda(l-k)} \text{ for all }0\leq k\leq l,\text{ where }\lambda=\lambda_g/2; \label{def:UDC-a}\\
\tag{UDC2}\label{def:UDC-b}
&\|Z(k+1)\|\leq Ce^{\tau|k-r_n|} \text{ for all }k\geq 0\text{ and }n\geq 0;
\\
\tag{UDC3} \label{def:UDC-c}
& ce^{\lambda_1  k}\leq\|Q(k)\|\leq Ce^{\lambda_1 (1+\tau) k} \text{ for all }k\geq 0;
\end{align}
\end{definition}

\begin{remark}\label{rem:QQ}
By conditions \eqref{def:UDC-b} and \eqref{def:UDC-c}, there exists $C'>0$ such that
\begin{equation}\label{eq:zk1}
\|Z(k+1)\|=O(\|Q(k)\|^\tau).
\end{equation}
Then using arguments from Section~1.3.1 in \cite{Ma-Mo-Yo}, one can show that
\begin{equation}\label{eq:minqk}
\|Q(k)\|=O(\min_{\alpha\in\mathcal A}Q_\alpha(k)^{1+\tau}).
\end{equation}
Thus, the \ref{UDC} condition implies condition $(a)$ of the Roth-type Diophantine condition defined in \cite{Ma-Mo-Yo}. The other two conditions (as well as the last assumption of the \emph{restricted} Roth-type condition\footnote{In \cite{Ma-Mo-Yo:lin}, Marmi, Moussa and Yoccoz introduced a more restrictive (but still full measure) Diophantine-type condition, that they called \emph{restricted Roth-type}: in addition to all the properties of \emph{Roth-type}, one requests in this case that the stable space has exactly dimension $g$. This holds for IETs which satisfy the \ref{UDC} in view of the Oseledets genericity assumption \eqref{def:O}, since we require that the splitting is of hyperbolic type, which means exactly that there are $g$ positive exponents, see \S~\ref{sec:Oseledets}).}) also hold, in view of the Oseledets genericity assumption \eqref{def:O} (see for example~Remark 3.4 in \cite{Ma-Yo}). Thus IETs which satisfy the  \ref{UDC} are in particular of (restricted) Roth-type.
%{ Check restricted and where it is defined, it should have to do with hyperbolicity}.
%In view of \eqref{def:UDC-c}, there exists $C>0$ and $m\geq 1$ such that
%\begin{equation}\label{eq:qkl}
%\|Q(k,l)\|\leq C\|Q(l)\|^m\text{ for all }0\leq k\leq l.
%\end{equation}
\end{remark}

\subsubsection{Full measure of the \ref{UDC}}\label{sec:fullmeasure}
Let us  show that the \ref{UDC} condition has full measure.
\begin{theorem}\label{thm:UDC}
Almost every IET satisfies the \ref{UDC} Diophantine condition.
\end{theorem}

\begin{proof} We split the proof in several steps.

\smallskip
\noindent {\it Construction of a good recurrence set.} Let us consider a subset $Y\subset X(\mathcal{G})$ which satisfies  the assumptions \ref{defYi} and \ref{defYii} in the proof of Lemma~\ref{lemma:Rokhlinbalance}, which guarantees that visits to $Y$ give a Rokhlin-balanced sequence, and furthermore
such that:
\begin{enumerate}[label=(\roman*)]
\setcounter{enumi}{2}
%\item[$(i)$] its projection $\underline{Y}$ on $\mathcal{G}\times\Lambda^\mathcal{A}$ is  precompact
%with respect to the Hilbert metric;
%\item[$(ii)$] there exists $0<\delta<1$ such that for every $(\pi,\lambda,\tau)\in Y$ we have
%\begin{equation}\label{eq:bigrect}
%\min\{\sum_{\pi_0(\alpha)\leq k}\tau_\alpha:1\leq k<d\}\cup\{h_\alpha(\tau):\alpha\in\mathcal A\}>\delta\max\{h_\alpha(\tau):\alpha\in\mathcal A\};
%\end{equation}
\item $\widehat{\mu}(Y)$ is finite, so  $\widehat{\mu}_Y:= \widehat{\mu}/\widehat{\mu}(Y)$ is a probability measure;
\item the functions $\log\|\widehat{A}_Y\|$ and $\log\|\widehat{A}_Y^{-1}\|$ are integrable with respect to $\widehat{\mu}_Y$.
\end{enumerate}
%Let $R>0$ be such that $\underline{Y}\subset \mathcal{G}\times \overline{B}_H((1/d,\ldots,1/d),R)$, where $\overline{B}_H((1/d,\ldots,1/d),R)$ is the closed ball
%(with respect to the Hilbert metric $d_H$)
%of radius $R$ and center at the center of the simplex $\Lambda^{\mathcal A}$.
Let
$\lambda_1>\ldots>\lambda_g>0$ the positive Lyapunov exponents of the corresponding accelerated cocycle, which are $g$ and distinct in view of \cite{Fo1} and \cite{Av-Vi}. Let $\lambda:=\lambda_g/2$ and $\kappa=de^R$. Fix $0<\tau<\lambda_g/2$. Since for $\widehat{\mu}_Y$-a.e.\ $(\pi,\lambda,\tau)\in Y$ we have
\begin{align*}
\lim_{n\to+\infty}\frac{1}{n}\log\|\widehat{A}^{(n)}_Y(\pi,\lambda,\tau)^t\upharpoonright_{\Gamma_s(\pi,\lambda,\tau)}\|=-\lambda_g,
\end{align*}
the map from $Y$ to $\mathbb{R}$ given by
\[(\pi,\lambda,\tau)\mapsto\sup_{n\geq 0}e^{(\lambda_g-\tau)n}\|\widehat{A}^{(n)}_Y(\pi,\lambda,\tau)^t\upharpoonright_{\Gamma_s(\pi,\lambda,\tau)}\|\]
is a.e.\ defined and measurable.
Therefore, there exists a subset $K\subset Y$ with $\widehat{\mu}_Y(K)/\widehat{\mu}_Y(Y)>1-\tau/2$ and a constant $C>0$ such that if $(\pi,\lambda,\tau)\in K$
then for every $n\geq 0$ we have
\begin{equation}\label{eq:ac}
\|\widehat{A}^{(n)}_Y(\pi,\lambda,\tau)^t\upharpoonright_{\Gamma_s(\pi,\lambda,\tau)}\|\leq Ce^{-(\lambda_g-\tau)n}\leq Ce^{-\lambda n}.
\end{equation}

\smallskip
\noindent {\it First acceleration.}
Let us consider the induced map $\widehat{\mathcal{R}}_K:K\to K$ and the induced cocycle $\widehat{A}_K:K\to SL_\mathcal{A}(\Z)$.
Then $\widehat{\mathcal{R}}_K(\pi,\lambda,\tau)=\widehat{\mathcal{R}}^{r_K(\pi,\lambda,\tau)}_Y(\pi,\lambda,\tau)$, where $r_K(\pi,\lambda,\tau)\geq 1$
is the first return time of $(\pi,\lambda,\tau)\in K$ to $K$ for the map  $\widehat{\mathcal{R}}_Y$. Let $r_K^{(n)}:=\sum_{0\leq i<n}r_k\circ \widehat{\mathcal{R}}_K$ for every $n\geq 0$.
Then
\[\frac{r_K^{(n)}}{n}\to\frac{\widehat{\mu}_Y(Y)}{\widehat{\mu}_Y(K)}\text{ a.e.\ on }K\]
and furthermore
\[\widehat{A}_K^{(n)}=\widehat{A}_Y^{(r_K^{(n)})}\text{ for every }n\geq 0.\]
In view of \eqref{eq:ac}, for every $(\pi,\lambda,\tau)\in K$ we have
\begin{equation}\label{eq:acK}
\|\widehat{A}^{(n)}_K(\pi,\lambda,\tau)^t\upharpoonright_{\Gamma_s(\pi,\lambda,\tau)}\|\leq Ce^{-\lambda r_K^{(n)}(\pi,\lambda,\tau)}\leq Ce^{-\lambda n}
\end{equation}
and for a.e.\ $(\pi,\lambda,\tau)\in K$ we have
\begin{equation}\label{eq:maxlap}
\lim_{n\to+\infty}\frac{1}{n}\log\|\widehat{A}^{(n)}_K(\pi,\lambda,\tau)\|=\lambda_1\frac{\widehat{\mu}_Y(Y)}{\widehat{\mu}_Y(K)}\in(\lambda_1 ,\lambda_1 (1+\tau)).
\end{equation}

\smallskip
\noindent {\it Second acceleration.}
Since the functions $\log\|\widehat{A}_K\|$ and $\log\|\widehat{A}_K^{-1}\|$ are integrable, for a.e.\ $(\pi,\lambda,\tau)\in K$ we have $\log\|\widehat{A}_K(\widehat{\mathcal{R}}_K^n(\pi,\lambda,\tau))\|/n\to 0$
as $|n|\to+\infty$,
also the map from $K$ to $\R$ given by
\[(\pi,\lambda,\tau)\mapsto\sup_{n\in\Z}e^{-\tau |n|}\|\widehat{A}_K(\widehat{\mathcal{R}}_K^{n}(\pi,\lambda,\tau))\|\]
is a.e.\ defined and measurable.
Therefore, there exists a subset $K'\subset K$ with $\widehat{\mu}_K(K')>0$ and a constant $C'>0$ such that if $(\pi,\lambda,\tau)\in K'$
then for every $n\in\Z$ we have
\begin{equation}\label{eq:arn}
\|\widehat{A}_K(\widehat{\mathcal{R}}_K^{n}(\pi,\lambda,\tau))\|\leq C'e^{\tau |n|}.
\end{equation}
Moreover, for a.e.\ $(\pi,\lambda,\tau)\in K'$ there exists an increasing sequence of non-negative integer numbers $(r_n(\pi,\lambda,\tau))_{n\geq 1}$ such that $r_1(\pi,\lambda,\tau)=0$ and
\begin{equation}\label{eq:rn}
\widehat{\mathcal{R}}_K^{r_n(\pi,\lambda,\tau)}(\pi,\lambda,\tau)\in K'\text{ for all }n\geq 0\text{ and }\frac{r_n(\pi,\lambda,\tau)}{n}\to \frac{\widehat{\mu}_K(K)}{\widehat{\mu}_K(K')}=:\alpha>0.
\end{equation}
Let $K''\subset K'$ be a subset of $(\pi,\lambda,\tau)\in K'$ for which \eqref{eq:maxlap} and \eqref{eq:rn} hold. Then $\widehat{\mu}_{\mathcal{G}}(K'')=\widehat{\mu}_{\mathcal{G}}(K')>0$.
By the ergodicity of $\widehat{\mathcal{R}}$, for a.e.\ $(\pi,\lambda,\tau)\in X(\mathcal{G})$
\begin{equation}\label{eq:n1}
\text{there exists }n_1(\pi,\lambda,\tau)\geq 0\text{ such that }\widehat{\mathcal{R}}^{n_1(\pi,\lambda,\tau)}(\pi,\lambda,\tau)\in K''.
\end{equation}
By Fubini argument, there exists a measurable subset $\Xi\subset \mathcal{G}\times\Lambda^{\mathcal A}$ such that $\mu_\mathcal{G}(\mathcal{G}\times\Lambda^{\mathcal A}\setminus \Xi)=0$
and for every $(\pi,\lambda)\in \Xi$ there exists $\tau\in\Theta_\pi$ such that $(\pi,\lambda,\tau)\in X(\mathcal{G})$ satisfies \eqref{eq:n1}.

\smallskip
\noindent {\it Full measure.}
We can now show that every $(\pi,\lambda)\in \Xi$ satisfies the \ref{UDC}. Suppose that $(\pi,\lambda)\in \Xi$ and $(\pi,\lambda,\tau)\in X(\mathcal{G})$ satisfies \eqref{eq:n1}. Then the corresponding acceleration sequence $(n_k)_{k\geq 0}$ is defined by setting $n_0:=0$ and then defining $n_k$ inductively such that, for every $k\geq 1$,
%\[ n_0:=0; \qquad \text{ for}\ k\geq 1, \ n_k \ \text{is\ such\ that}\
\[ \widehat{\mathcal{R}}^{n_k}(\pi,\lambda,\tau)=\widehat{\mathcal{R}}_K^{k-1}\widehat{\mathcal{R}}^{n_1(\pi,\lambda,\tau)}(\pi,\lambda,\tau).\]
\noindent Let us now consider the cocycle matrices $Z(k)$, $k\in \N$, of the acceleration along the sequence $(n_k)_{k\in\N}$, as defined in \S~\ref{sec:accelerations}, as well as their products $Q(k,l)$, $k,l\in \N$ (see again \S~\ref{sec:accelerations}).
By definition of $Q$ and \eqref{eq:An>0}, for $1\leq k\leq l$ we have
\begin{align*}
Q(k,l)&=\widehat{A}_K^{(l-k)}(\widehat{\mathcal{R}}_K^{k-1}(\widehat{\mathcal{R}}^{n_1}(\pi,\lambda,\tau)))^t\\
Q(0,l)&=\widehat{A}_K^{(l-1)}(\widehat{\mathcal{R}}^{n_1}(\pi,\lambda,\tau))^t\widehat{A}^{(n_1)}(\pi,\lambda,\tau)^t\\
\|Q_s(k,l)\|&=\|\widehat{A}_K^{(l-k)}(\widehat{\mathcal{R}}_K^{k-1}(\widehat{\mathcal{R}}^{n_1}(\pi,\lambda,\tau)))^t\upharpoonright_{\Gamma_s(\widehat{\mathcal{R}}_K^{k-1}(\widehat{\mathcal{R}}^{n_1}(\pi,\lambda,\tau)))}\|\\
\|Q_s(0,l)\|&\leq\|\widehat{A}_K^{(l-1)}(\widehat{\mathcal{R}}^{n_1}(\pi,\lambda,\tau))^t\upharpoonright_{\Gamma_s(\widehat{\mathcal{R}}^{n_1}(\pi,\lambda,\tau))}\|\|{A}^{(n_1)}(\pi,\lambda)^t\|.
\end{align*}
Since $\widehat{\mathcal{R}}_K^{k-1}(\widehat{\mathcal{R}}^{n_1}(\pi,\lambda,\tau))\in K$ for every $k\geq 1$, by \eqref{eq:acK}, for $0\leq k< l$ we have
\[\|Q_s(k,l)\|\leq Ce^\lambda\|{A}^{(n_1)}(\pi,\lambda)^t\| e^{-\lambda(l-k)},\]
which gives \eqref{def:UDC-a}.

\smallskip
\noindent Consider now the sequence $(r_n)_{n\geq 0}$ defined setting $r_0:=0$ and, for $n\geq 1$, \[ r_n:=r_n(\widehat{\mathcal{R}}^{n_1}(\pi,\lambda,\tau))+1.\]
%\begin{itemize}
%\item $r_0=0$;
%\item $r_n=r_n(\widehat{\mathcal{R}}^{n_1}(\pi,\lambda,\tau))+1$ for $n\geq 1$.
%\end{itemize}
As $\widehat{\mathcal{R}}^{n_1}(\pi,\lambda,\tau)\in K''$, by \eqref{eq:rn}, we have $r_n/n\to\alpha>0$ and
\[\widehat{\mathcal{R}}_K^{r_n-1}(\widehat{\mathcal{R}}^{n_1}(\pi,\lambda,\tau))=\widehat{\mathcal{R}}_K^{r_n(\widehat{\mathcal{R}}^{n_1}(\pi,\lambda,\tau))}\widehat{\mathcal{R}}^{n_1}(\pi,\lambda,\tau)\in K'\text{ for }n\geq 1.\]
Since $Z(k+1)=\widehat{A}_K(\widehat{\mathcal{R}}_K^{k-1}(\widehat{\mathcal{R}}^{n_1}(\pi,\lambda,\tau)))^t$ for $k\geq 1$
and $Z(1)=\widehat{A}^{(n_1)}(\pi,\lambda,\tau)^t$, by \eqref{eq:arn}, for every $n\geq 1$ and $k\geq 1$ we have
\begin{align*}
\|Z(k+1)\|&=\|\widehat{A}_K(\widehat{\mathcal{R}}_K^{k-1}(\widehat{\mathcal{R}}^{n_1}(\pi,\lambda,\tau)))\|=
\|\widehat{A}_K(\widehat{\mathcal{R}}_K^{k-r_n}(\widehat{\mathcal{R}}_K^{r_n-1}(\widehat{\mathcal{R}}^{n_1}(\pi,\lambda,\tau))))\|
\leq C'e^{\tau|k-r_n|}.
\end{align*}
For $k=0$, on the other hand, we have
\[\|Z(1)\|=\|{A}^{(n_1)}(\pi,\lambda)\|\leq \|{A}^{(n_1)}(\pi,\lambda)\|e^{\tau|r_n|}\]
for every $n\geq 0$. Moreover, as $r_1=1$, it follows that for every $k\geq 1$ we have
\[\|Z(k+1)\|\leq C'e^{\tau|k-r_1|}\leq C'e^{\tau|k-r_0|},\]
which gives \eqref{def:UDC-b} with $C=\max(C',\|{A}^{(n_1)}(\pi,\lambda)\|)$.

As $\widehat{\mathcal{R}}^{n_1}(\pi,\lambda,\tau)\in K''$, by \eqref{eq:maxlap}
\[\lim_{k\to+\infty}\frac{\log\|Q(k)\|}{k}=
\lim_{k\to+\infty}\frac{\log\|\widehat{A}^{(k-1)}_K(\widehat{\mathcal{R}}^{n_1}(\pi,\lambda,\tau))\|}{k}=\frac{\lambda_1\widehat{\mu}_Y(Y)}{\widehat{\mu}_Y(K)}\in(\lambda_1 ,\lambda_1 (1+\tau)),\]
which implies the condition \eqref{def:UDC-c}.
Finally, the sequence is a Rokhlin-balanced acceleration sequence by Lemma~\ref{lemma:Rokhlinbalance}, since the set $Y$ was chosen to satisfies the conditions \ref{defYi} and \ref{defYii} which guarantee Rokhlin-balance in the proof of Lemma~\ref{lemma:Rokhlinbalance}. This concludes the proof.
\end{proof}

\subsection{Diophantine series}\label{sec:Dseries}
In the proof of our main results, certain sums and series (defined in Definition~\ref{def:series}) which depend on the matrices of the (accelerated) cocycle will play a central role, both to control Birkhoff sums and to prove ergodicity. We here show that these quantities, under the \ref{UDC}, are first of all well defined and furthermore grow in a controlled way (see Proposition~\ref{prop:estKC}).
\begin{definition}\label{def:series}
For every IET $T:I\to I$ satisfying Keane's condition and  any accelerating sequence we define four sequences $(K_l)_{l\geq -1}$, $(K'_l)_{l\geq -1}$, $(C_k)_{k\geq 0}$, $(C'_k)_{k\geq 0}$:
\begin{align*}
&K_l(T):=\sum_{j\geq l}\|Z(j+1)\|\|Q_s(l,j+1)\|\text{ for }l\geq 0\text{ and }K_{-1}:=0;\\
&K'_l(T):=\sum_{j\geq l}\|Z(j+1)\|\|Q_s(l,j+1)\|\log\|Q(j)\|\text{ for }l\geq 0\text{ and }K'_{-1}:=0;\\
&C_k(T):=\sum_{0\leq l\leq k}\|Q_s(l,k)\| \big(\|Z(l)\|K_{l-1}(T)+K_l(T)\big)\text{ for }k\geq 0;\\
&C'_k(T):=\sum_{0\leq l\leq k}\|Q_s(l,k)\| \big(\|Z(l)\|K'_{l-1}(T)+K'_l(T)\big)\text{ for }k\geq 0.
\end{align*}
\end{definition}
\noindent Proposition~\ref{prop:estKC} below shows in particular that if $T$ satisfies the \ref{UDC} these quantities are finite and hence well defined for every pairs of integers $k\geq 0$, $l\geq -1$.
%\begin{remark} \label{rem:k+l}
%Note that for all $k,l\geq 0$ we have
%\begin{equation}\label{eq:ckl}
%K_k(T^{(l)})= K_{k+l}(T)\text{ and }C_k(T^{(l)})\leq C_{k+l}(T).
%\end{equation}
%\end{remark}

\begin{proposition}\label{prop:estKC}
For every IET $T:I\to I$ satisfying the \ref{UDC} all sequences $(K_l)_{l\geq -1}$, $(K'_l)_{l\geq -1}$, $(C_k)_{k\geq 0}$, $(C'_k)_{k\geq 0}$
are well defined and for every $0<\tau<\lambda/2$ there exists a constant $D>0$ such that
\begin{align}
\label{eq:K}
K_l(T)&\leq De^{\tau(r_n-l)}\text{ if } r_{n-1}\leq l\leq r_n \text{ for some } n\geq 0;\\
\label{eq:K'}
K'_l(T)&\leq D(l+1)e^{\tau l}\text{ for every } l\geq 0;\\
\label{eq:C}
C_{r_n}(T)&\leq D \text{ for every } n\geq 1;\\
\label{eq:C'}
C'_k(T)&\leq D(k+1)e^{2\tau k}\text{ for every } k\geq 0.
\end{align}
\end{proposition}

\begin{proof}
By \eqref{def:UDC-a} and \eqref{def:UDC-b}, for $r_{n-1}<l\leq r_n$ we have
\begin{align*}
K_l(T)&=\sum_{l+1\leq j\leq r_n}\|Z(j)\|\|Q_s(l,j)\|+\sum_{j> r_n}\|Z(j)\|\|Q_s(l,j)\|\\
&\leq C^2\sum_{l+1\leq j\leq r_n}e^{\tau(r_n-j+1)}e^{-\lambda(j-l)}+C^2\sum_{j> r_n}e^{\tau(j-1-r_n)}e^{-\lambda(j-l)}
\\
&\leq C^2e^{\tau(r_n-l)}\sum_{j\geq 1}e^{-\lambda j}+C^2e^{-\lambda(r_n-l+1)}\sum_{j\geq 0}e^{-(\lambda-\tau) j},
\end{align*}
which gives \eqref{eq:K}.

\medskip

By condition \eqref{def:UDC-c}, for all $j\geq l+1$ we have
\[\log\|Q(j)\|\leq \log C+\lambda_1 (1+\tau) j\leq C'j\leq C'(l+1)(j-l).\]
Therefore,
again by \eqref{def:UDC-a} and \eqref{def:UDC-b}, we have
\begin{align*}
K'_l(T)&\leq C'(l+1)\sum_{j\geq l+1}\|Z(j)\|\|Q_s(l,j)\|(j-l)\\
&\leq  C'C^2(l+1)\sum_{j\geq l+1}(j-l)e^{\tau j}e^{-\lambda(j-l)}
= C'C^2(l+1)e^{\tau l}\sum_{j\geq 1}je^{-(\lambda-\tau)j},
\end{align*}
which gives \eqref{eq:K'}.

\medskip

In view of \eqref{eq:K}, \eqref{def:UDC-a} and \eqref{def:UDC-b}, we have
\begin{align*}
C_{r_n}(T)&=\sum_{0\leq l\leq r_n}\|Q_s(l,r_n)\| \big(\|Z(l)\|K_{l-1}(T)+K_l(T)\big)\\
&\leq C^2D\sum_{0\leq l\leq r_n}e^{-\lambda(r_n-l)}\big(e^{\tau(r_n-l+1)}e^{\tau(r_n-l+1)}+e^{\tau(r_n-l)}\big)
\leq 2C^2De^{2\tau}\sum_{l\geq 0}e^{-(\lambda-2\tau)l},
\end{align*}
which gives \eqref{eq:C}.

\medskip

In view of \eqref{eq:K'}, \eqref{def:UDC-a} and \eqref{def:UDC-b}, for every $k\geq 0$ we have
\begin{align*}
C'_{k}(T)&=
\sum_{0\leq l\leq k}\|Q_s(l,k)\| \big(\|Z(l)\|K'_{l-1}(T)+K'_l(T)\big)\\
&\leq
C^2D\sum_{0\leq l\leq k}e^{-\lambda(k-l)} \big(le^{\tau l}e^{\tau (l-1)}+(l+1)e^{\tau l}\big)
\leq (k+1)2C^2De^{2\tau k} \sum_{j\geq 0}e^{-(\lambda-2\tau)j},
\end{align*}
%\endgroup
which gives \eqref{eq:C'}.
\end{proof}

\section{Cocycles with logarithmic singularities}\label{sec:cocycles}
We define in this section norms on the spaces of cocycles $\varphi: I \to \mathbb{R}$ with logarithmic singularities over IETs that we are interested in (in view of the reduction explained in \S~\ref{sec:reductionsp}). We first introduce (in \S~\ref{sec:BV}) the class of cocycles of bounded variation over a given IET, then move to cocycles with logarithmic singularities. The norms we introduce make the space of such cocycles a Banach space. We then prove several properties which will be used later in the proofs of the main results.

\subsection{Bounded variation and absolutely continuous cocycles}\label{sec:BV}
Let us denote by $\bv(\sqcup_{\alpha\in \mathcal{A}} I_{\alpha})$
the space of functions $\varphi:I\to \mathbb{R}$ such that
the restriction $\varphi:I_{\alpha}\to \mathbb{R}$ is of
bounded variation for every $\alpha\in \mathcal{A}$.

\subsubsection{Banach structure on bounded variation cocycles}
For every function $\varphi\in\bv(\sqcup_{\alpha\in \mathcal{A}}
I_{\alpha})$ and $x\in I$ we will denote by $\varphi_{+}(x)$ and
$\varphi_-(x)$ the right-handed  and left-handed limit of
$\varphi$ at $x$ respectively. Let us denote
by $\Var{{\varphi}}{J}$ the total variation of ${\varphi}$ on the interval
$J\subset I$. Then set
\begin{equation}\label{variation}
\var \varphi:=\sum_{\alpha\in
\mathcal{A}}\Var{\varphi}{\operatorname{Int}I_{\alpha}}.
\end{equation}
The space $\bv(\sqcup_{\alpha\in \mathcal{A}} I_{\alpha})$
is equipped with the Banach norm
$\|\varphi\|_{\bv}=\|\varphi\|_{\sup}+\var\varphi$.

\subsubsection{Piecewise absolutely continuous cocycles}
 Denote by $\ac(\sqcup_{\alpha\in
\mathcal{A}} I_{\alpha})$ the subspace of cocycles in $\bv(\sqcup_{\alpha\in
\mathcal{A}} I_{\alpha})$  which are absolutely continuous on the interior of
each $I_{\alpha}$, $\alpha\in \mathcal{A}$.

Denote by $\bv^1(\sqcup_{\alpha\in \mathcal{A}} I_{\alpha})$ the
space of functions $\varphi\in\ac(\sqcup_{\alpha\in \mathcal{A}}
I_{\alpha})$ such that $\varphi'\in \bv(\sqcup_{\alpha\in
\mathcal{A}} I_{\alpha})$. The space $\ac(\sqcup_{\alpha\in
\mathcal{A}} I_{\alpha})$  equipped with the $\bv$ norm is a Banach space and
$\bv^1(\sqcup_{\alpha\in \mathcal{A}} I_{\alpha})$ is its dense subspace.

\subsubsection{Boundary operator on cocycles}\label{bd:ac}
Let $\partial_\pi:\bv(\sqcup_{\alpha\in \mathcal{A}} I_{\alpha})\to\R^{\Sigma(\pi)}$
be the linear operator given by
\[(\partial_\pi \varphi)_\mathcal{O}:=
\sum_{\alpha\in\mathcal{A}_{\mathcal{O}}^-}\varphi_-(r_\alpha)-\sum_{\alpha\in\mathcal{A}_{\mathcal{O}}^+}\varphi_{+}(l_\alpha)\]
for $\mathcal{O}\in\Sigma(\pi)$. This is an extension of the operator defined
in \S~\ref{sec:bddef}  from piecewise constant cocycles (in view of Remark~\ref{rk:constboundary}) to bounded variation cocycles. It associates to each singularity the \emph{sum of jumps} at the discontinuities associated to that singularity (see also Remark~\ref{rk:constboundary}).

\smallskip
\noindent Remark that if $\varphi \in \ac(\sqcup_{\alpha\in
\mathcal{A}} I_{\alpha})$ then
\begin{equation}\label{eq:sofj}
\sum_{\mathcal{O}\in\Sigma(\pi)}(\partial_\pi \varphi)_\mathcal{O}=\int_I\varphi'(x)\,dx=:s(\varphi).
\end{equation}

\subsection{Cocycles with logarithmic singularities}\label{defls}
Consider the space $\ol( \sqcup_{\alpha\in \mathcal{A}} I_{\alpha})$ of cocycles with logarithmic singularities of geometric type on  $\sqcup_{\alpha\in \mathcal{A}} I_{\alpha}$, defined in  \S~\ref{sec:sing} (see  in particular \eqref{def:logsing} for the form of such cocycles),
as well as its subspace %$\logs(\sqcup_{\alpha\in \mathcal{A}} I_{\alpha})$ and
 $\logs(\sqcup_{\alpha\in \mathcal{A}} I_{\alpha}) $,  which consist
%respectively
of  cocycles with logarithmic singularities of geometric
type  (see  \S~\ref{sec:sing})
satisfying in addition also the symmetry condition \eqref{zerosymweak} (both also defined in \S~\ref{sec:sing}).
%Clearly we have the inclusions
%\[ \logs(\sqcup_{\alpha\in \mathcal{A}} I_{\alpha}) \subset \ol
%(\sqcup_{\alpha\in \mathcal{A}} I_{\alpha}) \subset \ol(\sqcup_{\alpha\in \mathcal{A}} I_{\alpha}). \]
%We say that a cocycle $\varphi:I\to\R$ for an IET $T_{(\pi,
%\lambda)}$ has \emph{logarithmic singularities} if there exist
%constants $C_\alpha^+,C_\alpha^-\in \mathbb{R}$, $\alpha \in
%\mathcal{A}$,  and $g_\varphi\in\ac(\sqcup_{\alpha\in \mathcal{A}} I_{\alpha})$, such that
%\begin{equation}\label{fform}
%\begin{aligned}
%\varphi(x)=&-\sum_{\alpha\in\mathcal{A}}C^+_\alpha\log\big(|I|\{(x-l_\alpha)/|I|\}\big)\\
%&- \sum_{\alpha\in\mathcal{A}} C^-_\alpha\log\big(|I|\{(r_\alpha-x)/|I|\}\big)+g_\varphi(x).
%\end{aligned}
%\end{equation}
%We say that the logarithmic singularities are \emph{of geometric
%type} if at least one among $ C_{\pi_0^{-1}(d)}^{-}$ and $
%C_{\pi_1^{-1}(d)}^{-}$ is zero and  at least one among $
%C_{\pi_0^{-1}(1)}^{+}$ or $C_{\pi_1^{-1}(1)}^{+}$ is zero. We
%denote by $\ol(\sqcup_{\alpha\in \mathcal{A}} I_{\alpha})$ the
%space of  functions with logarithmic singularities of geometric
%type.
%We define also the subspace
%$\logs(\sqcup_{\alpha\in \mathcal{A}} I_{\alpha})$ $\subset \ol
%(\sqcup_{\alpha\in \mathcal{A}} I_{\alpha})$ of functions
%satisfying the symmetry condition
%\begin{equation}
%\label{zerosymweak}
% \sum_{\alpha\in\mathcal{A}} C^-_{\alpha}-
%\sum_{\alpha\in\mathcal{A}}C^+_{\alpha}=0.
%\end{equation}
\noindent  We will also use the spaces
\begin{eqnarray*}
& \ol^{\bv}(\sqcup_{\alpha\in
\mathcal{A}} I_{\alpha}) :={\ol}(\sqcup_{\alpha\in \mathcal{A}}
I_{\alpha})+{\bv}(\sqcup_{\alpha\in \mathcal{A}} I_{\alpha})\nonumber \\
&  {\logs}^{\bv}(\sqcup_{\alpha\in \mathcal{A}} \nonumber
I_{\alpha}) :={\logs}(\sqcup_{\alpha\in \mathcal{A}}
I_{\alpha})+{\bv}(\sqcup_{\alpha\in \mathcal{A}} I_{\alpha}),
\end{eqnarray*}
consisting of all functions with logarithmic singularities
(respectively symmetric logarithmic singularities) of
geometric type  of the form \eqref{def:logsing} for which
we  require only that $g_\varphi\in\bv( \sqcup_{\alpha\in
\mathcal{A}} I_{\alpha})$. Notice that the space  $\bv$ ($\ac$ resp.) coincides with the
subspace of functions $\varphi\in {\ol}^{\bv}$ ($\ol$ resp.) as
in \eqref{def:logsing} such that $C^{\pm}_\alpha=0$ for all
$\alpha\in\mathcal{A}$.

%CHECK Therefore we have the following Remark that we will use later on:
%{ what is the remark to write here so that it can be used later when several times you use a trick with $g_{\varphi-g_\varphi}$? thanks}
%\begin{remark}\label{rk:gtrivial}
%If $g_\varphi=0$ (i.e.~$\varphi\in BV$, then $g_\varphi = \varphi$. In particular, for every  $g_\varphi\in\bv( \sqcup_{\alpha\in
%\mathcal{A}} I_{\alpha})$, if $g_\varphi=0$, $g_{\varphi-g_\varphi}=\varphi-g_\varphi$.
%\end{remark}}

\subsubsection{Norms and Banach space structures}
 We now define a norm on ${\ol}^{\bv}(\sqcup_{\alpha\in \mathcal{A}}
I_{\alpha})$ which makes it a Banach space.
\begin{definition}\label{LVdef}
For every $\varphi\in{\ol}^{\bv}(\sqcup_{\alpha\in \mathcal{A}}
I_{\alpha})$ of the form \eqref{def:logsing} set
\[\bl(\varphi) := \sum_{\alpha\in\mathcal{A}}(|C^+_\alpha|+|C^-_\alpha|),\quad
\;\lv(\varphi):=\bl(\varphi)+\var g_\varphi.\]
\end{definition}
\noindent
The space ${\ol}^{\bv}(\sqcup_{\alpha\in \mathcal{A}}
I_{\alpha})$ equipped with the norm
\[\|\varphi\|_{\lv}=\bl(\varphi)+\|g_\varphi\|_{\bv} \] becomes a
Banach space. Then, since  ${\ol}(\sqcup_{\alpha\in \mathcal{A}}
I_{\alpha})$ and ${\logs}(\sqcup_{\alpha\in \mathcal{A}}
I_{\alpha})$ are closed subspaces of ${\ol}^{\bv}(\sqcup_{\alpha\in \mathcal{A}}
I_{\alpha})$, they also inherit the Banach space structure. Moreover, for every $\varphi\in{\ol}^{\bv}(\sqcup_{\alpha\in \mathcal{A}}
I_{\alpha})$ we have
\begin{equation}\label{eq:norl1lv}
\frac{1}{|I|}\|\varphi\|_{L^1(I)}\leq (1+|\log|I||)\|\|\varphi\|_{\lv}.
\end{equation}
Indeed, since every $\varphi\in{\ol}^{\bv}(\sqcup_{\alpha\in \mathcal{A}}
I_{\alpha})$ is of the form \eqref{def:logsing}, we have
\[\frac{1}{|I|}\|\varphi\|_{L^1(I)}\leq \frac{\bl(\varphi)}{|I|}\int_I|\log x|\,\mathrm{d}x+\frac{\|g_{\varphi}\|_{L^1(I)}}{|I|}
\leq (1+|\log |I||)\bl(\varphi)+\|g_{\varphi}\|_{\sup}.\]

\smallskip
We can associate a value also to each \emph{saddle} in $\mathrm{Fix}(\psi_\R)$ individually as follows. Using the notation introduced in \S~\ref{sec:kernel}, let  $\mathcal{O}\in\Sigma(\pi)$ be a saddle and let $\mathcal{A}_\mathcal{O}^-, \mathcal{A}_\mathcal{O}^+$  be the sets of letters defined in \eqref{defAO}, associated respectively to right and left endpoints of intervals which correspond to this saddle. Then
\begin{equation}\label{zerosym}
\Delta_\mathcal{O}(\varphi):=\sum_{\alpha\in\mathcal{A}_\mathcal{O}^-}C^-_{\alpha}-
\sum_{\alpha\in\mathcal{A}_\mathcal{O}^+}C^+_{\alpha},
\end{equation}
is the value of the \emph{asymmetry at the saddle} labelled by $\mathcal{O}$. We also set
\[\as(\varphi):=\sum_{\mathcal{O}\in\Sigma(\pi)}|\Delta_\mathcal{O}(\varphi)|.\]
\noindent Comparing the above definition and \eqref{zerosym} with  Definition~\ref{LVdef}, one sees that
\begin{equation}\label{eq:asbl}
\as(\varphi)\leq \bl(\varphi).
\end{equation}

\subsubsection{Properties of the cocycles arising in the reduction}\label{sec:cocycleproperties}
As we saw in \S~\ref{sec:reductions}, the study of extensions of locally Hamiltonian flows can be reduced to the study of skew product extensions of IETs with logarithmic singularities (see Proposition~\ref{prop:redtoskew}). We now recall the properties of the cocycles which appear from this reduction, which were described in \cite{Fr-Ul} (see the proof of Theorem~6.1 and Proposition~6.1).
%, we have the following results  that precisely describes the properties of the

\smallskip
Let $M' \subset M$ be a minimal component of a locally Hamiltonian flow $\psi_\R$ with non-degenerate saddles. Fix a section $\gamma$ as in the proof of Proposition~\ref{prop:redtoskew} and consider the map that associate   $f\in C^{2+\epsilon}(M')$ to the cocycle $\varphi_f$ which appears in the skew-product presentation of the Poincar{\'e} map of the extension $\Phi^f_\R$ to $\gamma \times \R$ (see  Proposition~\ref{prop:redtoskew}).

\begin{proposition}[Properties of the skew-products cocycles, { see \cite{Fr-Ul} and in particular\footnote{{The statements are all part of Theorem~6.1 in \cite{Fr-Ul}, but $(i)$, namely the boundedness of $\varphi_{|f|}$ when $f\in C^{1}(M)$ and $f$ vanishes on $\mathrm{Fix}(\psi_\R)\cap M'$.  This last statement can be proved with the same arguments used in  \cite{Fr-Ul} to prove Theorem 6.1.}} Theorem 6.1}]\label{prop:oper}
For every $\epsilon>0$ the map  from $C^{2+\epsilon}(M)$ to ${\ol}(\sqcup_{\alpha\in \mathcal{A}}I_\alpha)$ which maps
\[f\mapsto \varphi_f\in {\ol}(\sqcup_{\alpha\in \mathcal{A}}
I_{\alpha})\]
is a bounded linear operator. Moreover, $g_{\varphi_f}'\in {\ol}(\sqcup_{\alpha\in \mathcal{A}}I_\alpha)$
and there exists $C>0$ such that
\[C^{-1}\sum_{\sigma\in \mathrm{Fix}(\psi_\R)\cap M'}|f(\sigma)|\leq\bl(\varphi)\leq C\sum_{\sigma\in \mathrm{Fix}(\psi_\R)\cap M'}|f(\sigma)|\text{ for every }f\in C^{2+\epsilon}(M).\]
{Furthermore:
\begin{itemize}
\item[(i)] if $f\in C^{1}(M)$ and $f(\sigma)=0$ for all $\sigma\in \mathrm{Fix}(\psi_\R)\cap M'$ then the map $\varphi_{|f|}:I\to\R$ is bounded;\vspace{.9mm}
\item[(ii)]
If  $\psi_\R\in \mathcal{U}_{min}$, so $M'=M$, then $\mathcal{AS}(\varphi_f)=0$ and $\partial_{\pi}(\varphi_f)=0$.
\end{itemize} }
\end{proposition}
%{\color{red}
%\begin{remark}\label{rem:|f|}

\subsection{Properties of cocycles with logarithmic singularities.}\label{prooflemlos}
We state and prove in this section a number of elementary properties of cocycles with logarithmic singularities which will be used in the construction of the correction operators.

\subsubsection{Control of tails of the derivatives growth} The derivative of a cocycle with logarithmic singularities has singularities which explode at most as $1/x$, as stated in the following Lemma.
{\begin{lemma}\label{rk:controlconstants}
Suppose that ${\ol}(\sqcup_{\alpha\in \mathcal{A}}I_{\alpha})$ and $g_{\varphi}=0$.
For every $\alpha\in\mathcal{A}$ denote by $m_\alpha$ the middle point of the interval $I_\alpha$, i.e.\ $m_\alpha:=\tfrac{1}{2}(l_\alpha+r_\alpha)$.
Then
\begin{align}\label{ass12}
\begin{split}
|\varphi'(x)(x-l_\alpha)|&\leq \bl(\varphi)\text{ for }x\in (l_\alpha,m_\alpha],\\
|\varphi'(x)(x-r_\alpha)|&\leq \bl(\varphi)\text{ for  }x\in [m_\alpha, r_\alpha).
\end{split}
\end{align}
\end{lemma}
\begin{proof}
Indeed, for every $x\in (l_\alpha,m_\alpha]$ and $\beta\in\mathcal{A}$ we have
\[\Big\{\frac{x-l_\beta}{|I|}\Big\}\geq \frac{x-l_\alpha}{|I|},\ \Big\{\frac{r_\beta-x}{|I|}\Big\}\geq \frac{r_\alpha-x}{|I|}\geq\frac{x-l_\alpha}{|I|}.\]
It follows that
\begin{align*}
|\varphi'(x)(x-l_\alpha)|&\leq \sum_{\beta\in\mathcal{A}}\frac{|C^+_\beta|(x-l_\alpha)}{|I|\{(x-l_\beta)/|I|\}}+
\sum_{\beta\in\mathcal{A}} \frac{|C^-_\beta|(x-l_\alpha)}{|I|\{(r_\beta-x)/|I|\}}
 \leq \sum_{\beta\in\mathcal{A}}(|C^+_\beta|+|C^-_\beta|)=\bl(\varphi).
\end{align*}
The second inequality of \eqref{ass12} follows by the same arguments.
\end{proof}}

\subsubsection{Control of mean value on subintervals.}
For every integrable function $f:I\to\R$ and a subinterval
$J\subset I$ let $m(f,J)$ stand for the \emph{mean value} of $f$ on $J$,
i.e.
\[m(f,J)=\frac{1}{|J|}\int_Jf(x)\,dx.\]

\begin{proposition}[Proposition 2.5 in \cite{Fr-Ul}]\label{lemlos}
If $\varphi\in{\ol}^{\bv}(\sqcup_{\alpha\in \mathcal{A}}
I_{\alpha})$ and $J\subset I_\alpha$ for some
$\alpha\in\mathcal{A}$, then
\begin{equation}\label{meanv1}
|m(\varphi,J)-m(\varphi,I_\alpha)|\leq
\lv(\varphi)\left(4+\frac{|I_\alpha|}{|J|}\right)
\end{equation}
and
\begin{equation}\label{meanv2}
\frac{1}{|J|}\int_J|\varphi(x)-m(\varphi,J)|\,dx\leq
8\lv(\varphi).
\end{equation}
\end{proposition}

{\begin{lemma}\label{lemma:control}
Let $\varphi\in{\ol}^{\bv}(\sqcup_{\alpha\in \mathcal{A}}
I_{\alpha})$. Then for every $x\in \Int I_\alpha$ we have
\begin{align}\label{eq:valuephi}
\begin{aligned}
|\varphi(x)-m(\varphi,[l_\alpha,m_\alpha])|&\leq \lv(\varphi)\Big(1+\log\frac{|I_\alpha|}{(x-l_\alpha)}\Big)\text{ if }x\in (l_\alpha,m_\alpha],\\
|\varphi(x)-m(\varphi,[m_\alpha, r_\alpha])|&\leq \lv(\varphi)\Big(1+\log\frac{|I_\alpha|}{(r_\alpha-x)}\Big)\text{ if  }x\in [m_\alpha, r_\alpha).
\end{aligned}
\end{align}
\end{lemma}

\begin{proof}
{\it Step 1:} First note that for any $C^1$-map $f:(x_0,x_1]\to\R$ such that $|f'(x)(x-x_0)|\leq C$ for $x\in(x_0,x_1]$, we have that  for all $t,s\in(x_0,x_1]$
\[|f(s)-f(t)|=|\int_t^s f'(u)\,du|\leq C|\int_t^s \frac{1}{u-x_0}\,du|=C|\log\frac{t-x_0}{s-x_0}|\]
and hence that
\begin{align}\label{eq:estlog}
\begin{split}
|f(s)-m(f,[x_0,x_1])|&\leq \frac{C}{x_1-x_0}\int_{x_0}^{x_1}|\log\frac{t-x_0}{s-x_0}|\,dt\\
&=C\Big(\log\frac{x_1-x_0}{s-x_0}+1-2\frac{x_1-s}{x_1-x_0}\Big)\leq C\Big(\log\frac{x_1-x_0}{s-x_0}+1\Big).
\end{split}
\end{align}

\noindent {\it Step 2:} Suppose now that $\varphi\in{\ol}(\sqcup_{\alpha\in \mathcal{A}}
I_{\alpha})$ with $g_\varphi=0$. In view of Lemma~\ref{rk:controlconstants} (see \eqref{ass12}), we can apply \eqref{eq:estlog} to $f=\varphi$ restricted to $I_\alpha$ and taking $C=\lv(\varphi)=\bl(\varphi)$. This gives \eqref{eq:valuephi} in the case $g_\varphi=0$.

\smallskip
\noindent {\it Step 3:} Consider now the general case. For every $g\in{\bv}(\sqcup_{\alpha\in \mathcal{A}}
I_{\alpha})$ and any interval $J\subset I_\alpha$, we have
\begin{equation}\label{eq:bvsup}
|g(x)-m(g,J)|\leq\var(g)\text{ for every }x\in J.
\end{equation}
Adding this equality to the  result of \textit{Step 2}, we obtain \eqref{eq:valuephi} for any $\varphi\in{\ol}(\sqcup_{\alpha\in \mathcal{A}}
I_{\alpha})$.
\end{proof}}

%\subsubsection{Pointwise bounds}
From  Lemma~\ref{lemma:control}  and \eqref{eq:bvsup}, we immediately get the following Corollary:
\begin{corollary}
Let $\varphi\in{\ol}^{\bv}(\sqcup_{\alpha\in \mathcal{A}}
I_{\alpha})$. Then for every $x\in \Int I_\alpha$ we have
\begin{equation}\label{eq:valuephi1}
|\varphi(x)|\leq\frac{2|I|}{|I_\alpha|}\frac{\|\varphi\|_{L^1(I)}}{|I|}+ \lv(\varphi)\Big(1+\log\frac{|I_\alpha|}{\min\{x-l_\alpha,r_\alpha-x\}}\Big).
\end{equation}
If additionally $\varphi\in{\bv}(\sqcup_{\alpha\in \mathcal{A}}
I_{\alpha})$ then
\begin{equation}\label{eq:supL1bv}
\|\varphi\|_{\sup}\leq \frac{|I|}{\min_{\alpha\in\mathcal{A}}|I_\alpha|}\frac{\|\varphi\|_{L^1(I)}}{|I|}+\var(\varphi).
\end{equation}
\end{corollary}

\subsubsection{Extenstion of the boundary opeartor}\label{bd:ls}
The operator $\partial_\pi:\bv(\sqcup_{\alpha\in \mathcal{A}} I_{\alpha}) \to\R^{\Sigma(\pi)}$ introduced in \S~\ref{bd:ac} can be extended to an operator $\partial_\pi:{\ol}^{\bv}(\sqcup_{\alpha\in \mathcal{A}}
I_{\alpha})\to\R^{\Sigma(\pi)}$ as follows.
\begin{definition}\label{Odef}
Let $\partial_\pi:{\ol}^{\bv}(\sqcup_{\alpha\in \mathcal{A}}
I_{\alpha})\to\R^{\Sigma(\pi)}$ be a linear operator given by
\begin{align*}
\partial_\pi(\varphi)_\mathcal{O}:=\lim_{x\to
0^+}\Big(\sum_{\alpha\in\mathcal{A}_\mathcal{O}^-}
\big(\varphi(r_\alpha-x)+C_\alpha^-\log x\big)-
\sum_{\alpha\in\mathcal{A}_\mathcal{O}^+}\big(\varphi(l_\alpha+x)+C_\alpha^+\log x\big)\Big).
\end{align*}
for every $\varphi\in{\ol}^{\bv}(\sqcup_{\alpha\in \mathcal{A}}
I_{\alpha})$ and $\mathcal{O}\in\Sigma(\pi)$.
\end{definition}

Let $a:= \min\{|I_\beta|:\beta\in\mathcal{A}\}/2$.
Then for every $\alpha\in\mathcal{A}$ and every $\varphi\in {\ol}^{\bv}(\sqcup_{\alpha\in \mathcal{A}}
I_{\alpha})$ there are $\underline{\varphi}_\alpha^+,\underline{\varphi}_\alpha^-:[0,a]\to\R$ functions
of bounded variation such that
\[\varphi(r_\alpha-x)=-C_\alpha^-\log x+\underline{\varphi}^-_\alpha(x), \
\varphi(l_\alpha+x)=-C_\alpha^+\log x+\underline{\varphi}^+_\alpha(x) \text{ for }x\in(0,a].\]
For every $\mathcal{O}\in\Sigma(\pi)$ let us consider the bounded variation map $D_{\mathcal{O}}:[0,a]\to\R$
given by
\begin{equation}\label{eq:D1}
D_{\mathcal{O}}(x):=\sum_{\alpha \in \mathcal{A}_\mathcal{O}^-} \underline{\varphi}^-_\alpha(x)- \sum_{
\alpha\in\mathcal{A}_\mathcal{O}^+} \underline{\varphi}^+_\alpha(x)\text{ for }x\in[0,a].
\end{equation}
Then for all $x\in(0,a]$ we have
\begin{equation}\label{eq:D2}
D_{\mathcal{O}}(x)=\sum_{\alpha\in\mathcal{A}_\mathcal{O}^-}
\big(\varphi(r_\alpha-x)+C_\alpha^-\log x\big)-
\sum_{\alpha\in\mathcal{A}_\mathcal{O}^+}\big(\varphi(l_\alpha+x)+C_\alpha^+\log x\big).
\end{equation}
As $D_{\mathcal{O}}$ is of bounded variation,  it follows that
\begin{equation}\label{Ousinggvalues}
\partial_\pi(\varphi)_\mathcal{O}=(D_{\mathcal{O}})_+(0)
\end{equation}
is well defined.

\subsection{Mean value projection}\label{sec:meanprojection}
%\subsection{Boundary operators}
%{ Merge with previous boundaries? explain if/how it is an extension?} ADD
If $\varphi\in L^1(I)$, we can consider the piecewise constant function that is constant and equal to the mean $m(\varphi,I_\alpha)$ on $I_\alpha$. Formally, we define the linear operator $\mathcal{M}:L^1(I)\to \R^{\mathcal{A}}$
given by
\[\mathcal{M}(\varphi)(x)=m(\varphi,I_\alpha)\text{ if }x\in I_\alpha.\]
%\text{ and }P(\varphi)=\varphi-C(\varphi).\]
\noindent This operator will play an important role in defining \emph{corrections} operators. In the rest of this subsection we prove the following Proposition, that gives an estimate on how the boundary operator $\partial_\pi$ changes when one projects using this mean value projection operator $\mathcal{M}$.
%Let us now consider a \emph{correction} operator that
\begin{proposition}\label{prop:partial}
For every  $\varphi\in{\ol}^{\bv}(\sqcup_{\alpha\in \mathcal{A}}I_{\alpha})$ we have
\begin{align}\label{eq:parCphi}
\begin{aligned}
\|\partial_\pi(\mathcal{M}\varphi)\|&\leq\|\partial_\pi(\varphi)\|+
\as(\varphi)\Big(1+\log \frac{2}{\min_{\beta\in\mathcal A}|I_\beta|}\Big)+
2d\lv(\varphi)\left(5+2\frac{|I|}{\min_{\beta\in\mathcal A}|I_\beta|}\right).
\end{aligned}
\end{align}
Furthermore, we also have that
\begin{align}\label{eq:parphi}
\|\partial_\pi(\varphi)\|\leq 2d\log \frac{2}{\min_{\beta\in\mathcal A}|I_\beta|}\|\varphi\|_{\lv}.
\end{align}
\end{proposition}

\begin{proof}
First suppose that $g_\varphi=0$. Then
the maps $\underline{\varphi}^{\pm}_\alpha:[0,a]\to\R$ ($a:= \min\{|I_\beta|:\beta\in\mathcal{A}\}/2$) are of class $C^1$  for all
$\alpha\in\mathcal{A}$ with
\begin{equation}\label{eq:derdel}
|\underline{\varphi}^\pm_\alpha(x)|\leq
\bl(\varphi)\log a^{-1}\quad\text{and}\quad|(\underline{\varphi}^\pm_\alpha)'(x)|\leq
\bl(\varphi)/a\quad\text{for}\quad x\in[0,a].
\end{equation}
In view of \eqref{eq:D1} and \eqref{Ousinggvalues}, it follows that for every $\mathcal{O}\in\Sigma(\pi)$ the map $D_{\mathcal{O}}$ is of class $C^1$ and we have
\begin{equation}\label{eq:Ophi}
|\partial_{\pi}(\varphi)_{\mathcal{O}}|=|D_{\mathcal{O}}(0)|\leq (\#\mathcal{A}^+_{\mathcal{O}}+\#\mathcal{A}^-_{\mathcal{O}})\bl(\varphi)\log a^{-1}
\end{equation}
and
\[|D_{\mathcal{O}}'(x)|\leq \frac{(\#\mathcal{A}^+_{\mathcal{O}}+\#\mathcal{A}^-_{\mathcal{O}})\bl(\varphi)}{a}\text{ for }x\in[0,a].\]
Therefore, for every $x\in[0,a]$,
\begin{align}
\begin{aligned}\label{eq:difdel}
|D_{\mathcal{O}}(0)-m(D_{\mathcal{O}},[0,a])|&\leq\frac{\int_0^a|D_{\mathcal{O}}(0)-D_{\mathcal{O}}(x)|\,dx}{a}
\leq \frac{\int_0^a\int_0^x|D_{\mathcal{O}}'(s)|\,ds\,dx}{a}\leq (\#\mathcal{A}^+_{\mathcal{O}}+\#\mathcal{A}^-_{\mathcal{O}})\bl(\varphi).
\end{aligned}
\end{align}
Moreover, by \eqref{eq:D2} and \eqref{zerosym}, we have
\begin{align*}
m(D_{\mathcal{O}},[0,a])=\Delta_\mathcal{O}(\varphi) m(\log,[0,a])
+\sum_{\alpha\in\mathcal{A},\pi_0(\alpha)\in\mathcal{O}}
m(\varphi,[r_\alpha-a,r_\alpha])-
\sum_{\alpha\in\mathcal{A},\pi_0(\alpha)-1\in\mathcal{O}}m(\varphi,[l_\alpha,l_\alpha+a]).
\end{align*}
In view of \eqref{meanv1}, for every $\alpha\in\mathcal A$ we have
\begin{align*}
&|m(\varphi,[r_\alpha-a,r_\alpha])-m(\varphi,I_\alpha)|\leq\bl(\varphi)\Big(4+\frac{|I_\alpha|}{a}\Big)
\\
&|m(\varphi,[l_\alpha,l_\alpha+a])-m(\varphi,I_\alpha)|\leq\bl(\varphi)\Big(4+\frac{|I_\alpha|}{a}\Big).
\end{align*}
As $m(\log,[0,a])=\log a-1$, it follows that
\begin{align*}
|\partial_{\pi}(\mathcal{M}\varphi)_{\mathcal{O}}&-m(D_{\mathcal{O}},[0,a])|
\leq |\Delta_\mathcal{O}(\varphi)|(1+\log a^{-1})
+\bl(\varphi)\Big(4(\#\mathcal{A}^+_{\mathcal{O}}+\#\mathcal{A}^-_{\mathcal{O}})+2\frac{|I|}{a}\Big).
\end{align*}
Together with \eqref{Ousinggvalues} and \eqref{eq:difdel}, this gives
\[|\partial_{\pi}(\mathcal{M}\varphi)_{\mathcal{O}}-\partial_{\pi}(\varphi)_{\mathcal{O}}|
\leq|\Delta_\mathcal{O}(\varphi)|(1+\log a^{-1})+\bl(\varphi)\Big(5(\#\mathcal{A}^+_{\mathcal{O}}+\#\mathcal{A}^-_{\mathcal{O}})+2\frac{|I|}{a}\Big).\]
As $\sum_{\mathcal{O}\in\Sigma(\pi)}(\#\mathcal{A}^+_{\mathcal{O}}+\#\mathcal{A}^-_{\mathcal{O}})=2\#\mathcal{A}=2d$ and $\as(\varphi)=\sum_{\mathcal{O}\in\Sigma(\pi)}|\Delta_\mathcal{O}(\varphi)|$, summing up these inequalities for all $\mathcal{O}\in\Sigma(\pi)$ we have
\begin{equation}\label{eq:pardif}
\|\partial_\pi(\mathcal{M}\varphi)-\partial_\pi(\varphi)\|\leq\as(\varphi)(1+\log a^{-1})+2d\bl(\varphi)\left(5+\frac{|I|}{a}\right).
\end{equation}
%Since $\varphi\in\ov{\ol}(\sqcup_{\alpha\in \mathcal{A}}I_{\alpha})$, we have
%\[\varphi=\sum_{\beta\in\mathcal A}(\varphi^+_\beta+\varphi^-_\beta)+g_{\varphi},\]
%where
%\[\varphi^+_\beta(x)=-C^+_\beta\log\big(|I|\{(x-l_\beta)/|I|\}\big),\quad\varphi^-_\beta(x)=-C^-_\beta\log\big(|I|\{(r_\beta-x)/|I|\}\big)\]
%and $g_\varphi$ is of bounded variation. Applying  \eqref{meanv1} to functions $\varphi^\pm_\beta$ and to the pair intervals $I_\alpha\subset I$, we obtain
%\begin{align*}
%|m(\varphi^\pm_\beta,I_\alpha)-m(\varphi^\pm_\beta,I)|\leq |C^\pm_\beta|\frac{5|I|}{|I_\alpha|}\leq |C^\pm_\beta|\frac{5|I|}{\min_{\gamma\in\mathcal A}|I_\gamma|}
%\end{align*}
%It follows that for every $\mathcal{O}\in\Sigma(\pi)$ we have
%\[\big|\mathcal{O}(C(\varphi^\pm_\beta))\big|=\Big|\sum_{\pi_0(\alpha)\in\mathcal{O}}m(\varphi^\pm_\beta,I_\alpha)-\sum_{\pi_0(\alpha)-1\in\mathcal{O}}m(\varphi^\pm_\beta,I_\alpha)\Big|\leq\frac{ 10|\mathcal{O}||C^\pm_\beta||I|}{\min_{\gamma\in\mathcal A}|I_\gamma|}.\]
%Hence
%\begin{equation}\label{eq:partC}
%\|\partial_\pi(C(\varphi^\pm_\beta))\|\leq 10d|C^\pm_\beta|\frac{ |I|}{\min_{\gamma\in\mathcal A}|I_\gamma|}
%\end{equation}

Now assume that $g_\varphi\neq 0$. Since $g_\varphi\in\bv(\sqcup_{\alpha\in \mathcal{A}}I_{\alpha})$, we have
\[|(g_\varphi)_{+}(l_\alpha)-m(g_\varphi,I_\alpha)|\leq\var g_\varphi\quad\text{and}\quad|(g_\varphi)_-(r_\alpha)-m(g_\varphi,I_\alpha)|\leq\var g_\varphi.\]
It follows that for every $\mathcal{O}\in\Sigma(\pi)$ we have
\begin{align*}
\big|\partial_{\pi}(\mathcal{M}(g_\varphi))_{\mathcal{O}}-\partial_{\pi}(g_\varphi)_{\mathcal{O}}\big|
&=\Big|\sum_{\pi_0(\alpha)\in\mathcal{O}}
\big(m(g_\varphi,I_\alpha)-(g_\varphi)_-(r_\alpha)\big)-\sum_{\pi_0(\alpha)-1\in\mathcal{O}}
\big(m(g_\varphi,I_\alpha)-(g_\varphi)_{+}(l_\alpha)\big)\Big|\\
&\quad\leq (\#\mathcal{A}^+_{\mathcal{O}}+\#\mathcal{A}^-_{\mathcal{O}})\var g_\varphi.
\end{align*}
Summing up these inequalities for all $\mathcal{O}\in\Sigma(\pi)$ we have\begin{equation}\label{eq:partbv}
\|\partial_\pi(\mathcal{M}(g_\varphi))-\partial_\pi(g_\varphi)\|\leq 2d\var g_\varphi
\end{equation}
which together with \eqref{eq:pardif} this completes the proof of \eqref{eq:parCphi}.

\medskip
\noindent By the definition of $\partial_{\pi}(g_\varphi)_{\mathcal{O}}$, we also have
\[|\partial_{\pi}(g_\varphi)_{\mathcal{O}}|\leq (\#\mathcal{A}^+_{\mathcal{O}}+\#\mathcal{A}^-_{\mathcal{O}})\|g_\varphi\|_{\sup}\quad\text{for every}\quad\mathcal{O}\in\Sigma(\pi).\]
This, together with \eqref{eq:Ophi}, gives
\[\|\partial_{\pi}(\varphi)\|\leq 2d\Big(\bl(\varphi)\frac{2}{\min_{\beta\in\mathcal{A}}|I_\beta|}+\|g_\varphi\|_{\sup}\Big).\]
As $\|\varphi\|_{\lv}=\bl(\varphi)+\var g_{\varphi}+\|g_\varphi\|_{\sup}$, this completes the proof of \eqref{eq:parphi}.
\end{proof}
%\begin{lemma}
%Suppose that $\varphi\in{\ol}(\sqcup_{\alpha\in \mathcal{A}}I_{\alpha})$ with $g_\varphi=0$.
%Then
%\begin{equation}\label{eq:minvIal}
%|m(\varphi,I_\alpha)-m(\varphi,I)|\leq\frac{5|I|}{\min_{\beta\in\mathcal A}|I_\beta|}\bl(\varphi)\text{ for }\alpha\in \mathcal A
%\end{equation}
%and
%\begin{equation}\label{eq:minvI}
%m(\varphi,I)=(1+\log|I|^{-1})\sum_{\alpha\in\mathcal A}(C^+_\alpha+C^-_\alpha).
%\end{equation}
%\end{lemma}
%\begin{proof}
%By assumption,
%\[\varphi=\sum_{\beta\in\mathcal A}(\varphi^+_\beta+\varphi^-_\beta),\]
%where
%\[\varphi^+_\beta(x)=-C^+_\beta\log|I|\{(x-l_\beta)/|I|\},\quad\varphi^-_\beta(x)=-C^-_\beta\log|I|\{(r_\beta-x)/|I|\}.\]
%Applying  \eqref{meanv1} to functions $\varphi^\pm_\beta$ and to the pair intervals $I_\alpha\subset I$, we obtain
%\[|m(\varphi^\pm_\beta,I_\alpha)-m(\varphi^\pm_\beta,I)|\leq |C^\pm_\beta|\frac{5|I|}{|I_\alpha|}.\]
%It follows that
%\begin{align*}
%|m(\varphi,I_\alpha)-m(\varphi,I)|\leq  5\bl(\varphi)\frac{|I|}{|I_\alpha|},
%\end{align*}
%which gives \eqref{eq:minvIal}. The formula \eqref{eq:minvI} is obtained by a direct computation.
%\end{proof}

\section{Renormalization of cocycles}\label{renormalization:sec}
The renormalization map on IETs given by Rauzy-Veech induction (or any of its accelerations) induce also a renormalization operator on cocycles over IETs defined in \S~\ref{sec:RVBS}.
%This is given by the operators given by taking \emph{special Birkhoff sums}, which we now define.

\subsection{Special Birkhoff sums}\label{SpecialBS}
Recall that for all $0\leq k<l$ the renormalization operator $S(k,l):L^1(I^{(k)})\to L^1(I^{(l)})$
is given by
\[S(k,l)\varphi(x)=\sum_{0\leq i<Q_{\beta}(k,l)}\varphi((T^{(k)})^ix)\text{ for }x\in I^{(l)}_\beta.\]
We write $S(k)\varphi$ for $ S(0,k)\varphi$ and we  use  the
convention that $S(k,k)\varphi:=\varphi$. Sums of this form are
usually called \emph{special Birkhoff sums}. Since Rokhlin towers representation allows to write  $I^{(k)}$ as
$$
I^{(k)}= \bigcup_{\beta\in\mathcal{A}} \bigcup_{i=0}^{Q_{\beta}(k,l)-1} (T^{(k)})^i  I^{(l)}_\beta ,
$$
where the intervals in the union are all pairwise disjoint, from the definition of special Birkhoff sums, one can see that for every $\varphi\in L^1(I^{(k)})$ we have
\begin{equation}\label{zachcal}
\int_{I^{(l)}}S(k,l)\varphi(x)\,dx=
\int_{I^{(k)}}\varphi(x)\,dx.
\end{equation}
Therefore we also have that
\begin{equation}\label{nase}
\| S(k,l)\varphi\|_{L^1(I^{(l)})}\leq
\|\varphi\|_{L^1(I^{(k)})}.
\end{equation}
 If $ g\in\bv(\sqcup_{\alpha\in
\mathcal{A}}I^{(k)}_{\alpha})$ then
\begin{equation}\label{navar}
\var S(k,l) g\leq \var g.
\end{equation}

The following Lemma, which was proved by the authors in \cite{Fr-Ul}, shows that \emph{constants} of logarithmic singularities, as a \emph{set}, is invariant under renormalization when logarithmic singularities are normalized suitably (i.e.~by the map $f(x)\mapsto f(\lambda \{x/ \lambda\}) $, where $\lambda$ is the length of the inducing interval).

\begin{lemma}[see \cite{Fr-Ul}]\label{comparingconstants}
For each $0\leq k\leq l$ and for each
$\varphi\in\ol(\sqcup_{\alpha\in \mathcal{A}}I^{(k)}_{\alpha})$
 of the form
\begin{equation}\label{eq:g=0}
\varphi(x)=-\sum_{\alpha\in\mathcal{A}}\left(C^+_\alpha
\log \left(|I^{(k)}|
\left\{\frac{x-l^{(k)}_\alpha}{|I^{(k)}|}\right\}\right)+C^-_\alpha
\log\left(|I^{(k)}|
\left\{\frac{r^{(k)}_\alpha-x}{|I^{(k)}|}\right\}\right)\right)
\end{equation}
there exists a permutation $\chi:\mathcal{A}\rightarrow
\mathcal{A}$ %is the permutation  in \S~\ref{sec:kernel} and
such that %we have that
\begin{align*}
S(k,l)\varphi(x)=&-\sum_{\alpha\in\mathcal{A}} C^+_\alpha
\log\big(|I^{(l)}|\{(x-l^{(l)}_\alpha)/|I^{(l)}|\}\big)   -\sum_{\alpha\in\mathcal{A}}
C^-_{\chi(\alpha)}\log\big(|I^{(l)}|\{(r^{(l)}_\alpha-x)/|I^{(l)}|\}\big)
+g_{S(k,l)\varphi}(x),
\end{align*}
where $g_{S(k,l)\varphi}\in\bv^1(\sqcup_{\alpha\in
\mathcal{A}}I^{(l)}_{\alpha})$.
\end{lemma}

\begin{remark}
In the general case, when $\varphi\in\ol(\sqcup_{\alpha\in \mathcal{A}}I^{(k)}_{\alpha})$ and $g_\varphi$ is non-trivial,
the map $\varphi-g_\varphi$ is of the form \eqref{eq:g=0}. It follows that
\begin{align*}
S(k,l)\varphi(x)=&S(k,l)(\varphi-g_{\varphi})(x)+S(k,l)(g_{\varphi})(x)\\
=&-\sum_{\alpha\in\mathcal{A}} C^+_\alpha
\log\big(|I^{(l)}|\{(x-l^{(l)}_\alpha)/|I^{(l)}|\}\big)   -\sum_{\alpha\in\mathcal{A}}
C^-_{\chi(\alpha)}\log\big(|I^{(l)}|\{(r^{(l)}_\alpha-x)/|I^{(l)}|\}\big)\\
&+g_{S(k,l)(\varphi-g_{\varphi})}(x)+S(k,l)(g_{\varphi})(x).
\end{align*}
As $g_{S(k,l)(\varphi-g_{\varphi})}$ and $S(k,l)(g_{\varphi})$ belong to $\ac(\sqcup_{\alpha\in \mathcal{A}}I^{(l)}_{\alpha})$,
we have
\begin{equation}\label{eq:gSk}
g_{S(k,l)\varphi}=g_{S(k,l)(\varphi-g_{\varphi})}+S(k,l)(g_{\varphi}).
\end{equation}
\end{remark}

\noindent  Recalling the definition of $\bl$ and $\as$ (see Definition~\ref{LVdef}) and of the various spaces of cocycles with logarithmic singularities (refer to \S~\ref{sec:cocycles}), we immediately have the following corollary:
\begin{corollary}[Invariance of $\bl$ and $\as$]\label{asblinvariance}
For every
$\varphi\in{\ol}^{\bv}(\sqcup_{\alpha\in \mathcal{A}}I^{(k)}_{\alpha})$
\begin{equation}\label{eq:blas}
\bl(S(k,l)\varphi)=\bl(\varphi)\quad\text{and}\quad\as(S(k,l)\varphi)=\as(\varphi).
\end{equation}
Therefore,  the operator $S(k,l)$ maps:
\begin{enumerate}[label=(\roman*)]
\item the space ${\ol}^{\bv}(\sqcup_{\alpha\in \mathcal{A}} I^{(k)}_{\alpha})$
 into the space
${\ol}^{\bv}(\sqcup_{\alpha\in \mathcal{A}} I^{(l)}_{\alpha})$;
\item the space ${\ol}(\sqcup_{\alpha\in \mathcal{A}} I^{(k)}_{\alpha})$ into the space
${\ol}(\sqcup_{\alpha\in \mathcal{A}} I^{(l)}_{\alpha})$;
\item the space ${\logs}^{\bv}(\sqcup_{\alpha\in \mathcal{A}} I^{(k)}_{\alpha})$
 into the space
${\logs}^{\bv}(\sqcup_{\alpha\in \mathcal{A}} I^{(l)}_{\alpha})$;
\item the space ${\logs}(\sqcup_{\alpha\in \mathcal{A}} I^{(k)}_{\alpha})$ into the space
${\logs}(\sqcup_{\alpha\in \mathcal{A}} I^{(l)}_{\alpha})$.
\end{enumerate}
\end{corollary}

The following result (Lemma~\ref{invariophi}) is a generalization of Lemma~3.2 in \cite{Fr-Ul}, which was proved
for cocycles with strongly symmetric logarithmic  singularities. Since the proof of the following lemma
runs in the same way, we skip it. The operator $\partial_\pi$ which appears in the statement was defined in \S~\ref{bd:ls}.

\begin{lemma}\label{invariophi}
For all $0\leq k\leq l$ and for every $\varphi \in \ol(\sqcup_{\alpha\in
\mathcal{A}} I^{(k)}_{\alpha})$ we have
\begin{equation}\label{eq:eqpar}
\|\partial_{\pi^{(l)}}( S(k,l) \varphi )\| =\|\partial_{\pi^{(k)}}(\varphi )\|.
\end{equation}
\end{lemma}

 %Let
%$\alpha^{(k)}_0:= (\pi^{(k)}_0)^{-1}(d)$ and $\alpha^{(k)}_1  :=
%(\pi^{(k)}_1)^{-1}(d) $.
%\begin{lemma}[\cite{Fr-Ul}]\label{comparingsingularities}
%For each $l\geq k \geq0$, for each $\alpha \in \mathcal{A}$, we have
%\begin{equation}\label{lcorrespondencekl}
%l^{(k)}_{\alpha} \in\{ (T^{(k)})^j l^{(l)}_\alpha , \, 0\leq j <
%Q_{\alpha} (k,l) \}.
%\end{equation}
%Moreover, if $\chi : \mathcal{A} \to \mathcal{A}$ is one of the
%permutations given by
%Lemma~\ref{comparingconstants},
%\begin{equation}\label{rcorrespondencekl}
 %r^{(k)}_{\chi(\alpha)}   \in\{  (\widehat{T^{(k)}})^j
%r^{(l)}_\alpha  , \, 0\leq j < Q_{\alpha} (k,l) \} \text{ if }\
%\alpha\neq \alpha^{(l)}_{\vep(\pi^{pl-1},\lambda^{pl-1})},
%\end{equation}
%while there exists ${\alpha}_\ast \in \mathcal{A} \setminus \{
%\alpha^{(l)}_{\vep(\pi^{pl-1},\lambda^{pl-1})}\}$ such that
%\begin{equation}
%\label{rcorrespondencekldod} r^{(k)}_{\alpha^{(k)}_0},
%r^{(k)}_{\alpha^{(k)}_1}  \in\{ (\widehat{T^{(k)}})^j
%r^{(l)}_{{\alpha}_\ast}  , \, 0\leq j < Q_{\alpha} (k,l) \}.
%\end{equation}
%Moreover, if $C^-_{\chi(\alpha)}\neq 0$ then $\alpha\neq
%\alpha^{(l)}_{\vep(\pi^{pl-1},\lambda^{pl-1})}$ and
%\eqref{rcorrespondencekl}  holds.
%\end{lemma}
%

\subsection{Cancellations for symmetric singularities.}\label{symmetric:sec}
The following property of cocycles with symmetric  logarithmic
singularities was proved by the second author in \cite{Ul:abs}
(see Proposition 4.1) and will play a crucial role to renormalize
cocycles with symmetric logarithmic singularities and in the proof
of ergodicity.

\smallskip
\noindent Let us denote by
$(x)^{+}$ the positive part of $x$, i.e.\ $(x)^{+}=x$ if $x\geq 0$
and $(x)^{+}=\infty$ if $x<0$, so that if $x<0$ then $1/(x)^+$ is zero. Using this notation, let us define, for every $\alpha \in \mathcal{A}$,
\begin{equation}\label{closestpts}
x_\alpha^{l} := \min_{0\leq i < Q_{\beta}(k)} (T^i x - l_\alpha)^+, \qquad
x_\alpha^{r} := \min_{0\leq i < Q_{\beta}(k)} (r_\alpha -T^i x)^+.
\end{equation}
Then $x_\alpha^{l}$ (resp.~$x_\alpha^{r}$) is the \emph{closest visit} to the singularity $l_\alpha$ from the right (resp.~to $r_\alpha$
 from the left) in the orbit segment $\{ T^i(x), \ 0\leq i <Q_{\beta}(k)\}$.

\begin{remark}[Closest visits comparison]\label{rk:closestpoints}
 By the proof of Proposition~3.2 in \cite{Fr-Ul}, for every $x\in I^{(k)}_\beta$ and any $\alpha\in\mathcal A$ we have
\begin{equation}\label{singularitiescomparisons}
\left| \frac{1}{x_\alpha^{l}} - \frac{1}{\left|I^{(k)}\right| \left\{\frac{x-l^{(k)}_\alpha}{\left|I^{(k)}\right|}\right\} }\right| \leq \frac{1}{\left|I_\alpha^{(k)}\right|},\
\left| \frac{1}{x_{\chi(\alpha)}^{r}} - \frac{1}{\left|I^{(k)}\right| \left\{ \frac{r^{(k)}_\alpha -x}{|I^{(k)}|}\right\}} \right| \leq \frac{1}{\left|I_\alpha^{(k)}\right|}.
\end{equation}
Thus, the closests visits defined above are comparable with the quantities expressed above in terms of $\{ \cdot \}$.
\end{remark}

The following Theorem (as the proof below indicates) follows from the results in \cite{Ul:abs}, combined with the acceleration defined in the \ref{UDC}:
\begin{theorem}[Cancellations for Symmetric Logarithmic Singularities]\label{cancellations_prop}
For almost every $(\pi,\lambda)\in\mathcal{G}\times\R^{\mathcal{A}}_{>0}$ there exists an accelerating sequence and a constant
$M=M_{(\pi,\lambda)}\geq 1$ such that $T_{(\pi,\lambda)}$ satisfies the \ref{UDC} (along the accelerating sequence) and for every
$\varphi\in\logs(\sqcup_{\alpha\in \mathcal{A}} I_{\alpha})$ with $g_\varphi'=0$, any $k\geq 1$ and  $x \in I^{(k)}_{\beta}$ we have
\begin{equation} \label{specialsum}
\tag{SUDC1} \left| (\varphi')^{(Q_{\beta}(k))} (x) - \sum_{\alpha \in \mathcal{A}}
\frac{C_\alpha^+}{x_\alpha^{l}} + \sum_{\alpha \in \mathcal{A}}
\frac{C_\alpha^-}{x_\alpha^{r}} \right| \leq  M \bl(\varphi) \frac{Q_{\beta}(k)}{|I|},
\end{equation} where
$x_\alpha^{l}$ and $x_\alpha^{r} $ are the closets visits defined in \eqref{closestpts}.

\smallskip
\noindent Moreover, for every $0\leq r<Q_{\beta}(k)$ and $x \in I^{(k)}_{\beta}$ we have
\begin{equation} \label{generalsum}
\tag{SUDC2} \left| (\varphi')^{(r)} (x)\right|\leq  \sum_{\alpha \in \mathcal{A}}
\frac{|C_\alpha^+|}{x_\alpha^{l}} + \sum_{\alpha \in \mathcal{A}}
\frac{|C_\alpha^-|}{x_\alpha^{r}} +  M \bl(\varphi) \frac{Q_{\beta}(k)}{|I|}.
\end{equation}
\end{theorem}
\begin{proof}
%{ I plan to change the proof a little; I would like to include the freedom of starting from an acceleration in the definition/proof of Theorem~\ref{thm:UDC} so we don't have to say 'redo the proof starting from $E_D$' which I don't like too much...}
By the proof of Propositions~4.1~and~4.2 in \cite{Ul:abs}, there exists a precompact subset $E_D\subset X(\mathcal{G})$
with positive measure such that $\widehat{A}_{E_D}$ and $\widehat{A}^{-1}_{E_D}$ are log-integrable and the accelerating
sequence defined by recurrence of $(\pi,\lambda,\tau)$ to $E_D$ is such that \eqref{specialsum} and \eqref{generalsum} hold
for every $k\geq 1$.

Then we repeat all steps of the proof of Theorem~\ref{thm:UDC} starting from the set $Y=E_D$. Since both \eqref{specialsum} and \eqref{generalsum}
also holds along a subsequence obtained taking further accelerations, this completes the proof.
\end{proof}

\begin{definition}[\customlabel{SUDC}{SUDC}]\label{def:SUDC}
We say that an IET $T$ satisfies the \emph{Symmetric Uniform Diophantine Condition}, or {SUDC} for short, if it satisfies the \ref{UDC} along an accelerating sequence $(n_k)_{k\geq 0}$ along which the cancellations \eqref{specialsum} and \eqref{generalsum} hold.
%Theorem~\ref{cancellations_prop} applies to $T$.
\end{definition}
\noindent Theorem~\ref{cancellations_prop} above thus shows that the \ref{SUDC} has full measure.

\begin{proposition}\label{lemlogreno}
Suppose that $T$ satisfies the \ref{SUDC}. For every
$\varphi\in\logs(\sqcup_{\alpha\in \mathcal{A}} I_{\alpha})$ with $g_\varphi=0$ and $k\geq 1$ we have $g_{S(k)\varphi}\in \bv^1(\sqcup_{\alpha\in \mathcal{A}}I^{(k)}_{\alpha})$ and
\begin{equation}\label{neq:wtphi}
\|g'_{S(k)\varphi}\|_{\sup}\leq
\frac{(M+1)\bl(\varphi)}{\min_{\beta\in\mathcal A}|I_\beta^{(k)}|}.
\end{equation}
\end{proposition}

\begin{proof}
The proof runs in the same way as the proof of Proposition~3.2 in \cite{Fr-Ul}, only replacing  Corollary~3.1 in \cite{Fr-Ul} with \eqref{specialsum}.

\smallskip
\noindent Let $\chi:\mathcal{A}\to\mathcal{A}$ be the permutation given by
Lemma~\ref{comparingconstants}. Then
\begin{equation}\label{firstdifference}
g'_{S(k)\varphi}(x) =S(k)\varphi'(x) -
\sum_{\alpha\in\mathcal{A}} \frac{C^+_\alpha}{|I^{(k)}| \big\{
\frac{x-l^{(k)}_\alpha}{|I^{(k)}|}\big\}} +
\sum_{\alpha\in\mathcal{A}} \frac{C^-_{\chi(\alpha)}}{ |I^{(k)}|
\big\{ \frac{r^{(k)}_\alpha -x}{ |I^{(k)} |}\big\}}.
\end{equation}

\noindent Notice that $S(k)\varphi' (x) =(\varphi')^{(Q_\beta(k))}(x) $ if $x\in I^{(k)}_\beta$. Thus, \eqref{firstdifference}, in view of \eqref{specialsum}  and Remark~\ref{rk:closestpoints}, and remarking that (since Rokhlin towers give a partition)
\[Q_\beta(k)\min_{\alpha\in\mathcal{A}} |I^{(k)}_\alpha|\leq Q_\beta(k) |I^{(k)}_\beta|\leq \sum_{\alpha \in\mathcal{A}} Q_\alpha(k) |I^{(k)}_\alpha| =|I| ,\]
 we get that, for every $x\in I^{(k)}_\beta$
\[|g'_{S(k)\varphi}(x)|\leq  M \bl(\varphi) \frac{Q_{\beta}(k)}{|I|}+\frac{\bl(\varphi)}{\min_{\alpha}|I^{(k)}_\alpha|}
\leq (M+1)\frac{\bl(\varphi)}{\min_{\alpha}|I^{(k)}_\alpha|}.\]
Taking the supremum over $x\in I^{(k)}$ concludes the proof.
\end{proof}

\begin{proposition}\label{lemlogreno1}
If $T$ satisfies the \ref{SUDC} then for
every $k\geq 1$ and for every
$\varphi\in{\logs}^{\bv}(\sqcup_{\alpha\in \mathcal{A}}
I_{\alpha})$ we have
\begin{equation}\label{nave}
\lv(S(k)\varphi)\leq 4M\frac{|I^{(k)}|}{\min_{\beta\in\mathcal A}|I_\beta^{(k)}|}\lv(\varphi)
\leq 4M\kappa\lv(\varphi).
\end{equation}
\end{proposition}

\begin{proof}
First suppose that $g_\varphi=0$.
By Proposition~\ref{lemlogreno}, we then have that $g_{S(k)\varphi}$ belongs to the space $\bv^1(\sqcup_{\alpha\in \mathcal{A}}
I^{(k)}_{\alpha})$  and
\[\var g_{S(k)\varphi} =\int_{I^{(k)}}
|g'_{S(k)\varphi}(x)|\,dx\ \leq  \|g'_{S(k)\varphi}\|_{\sup}|I^{(k)}|\leq
(M+1){\bl(\varphi)}\frac{|I^{(k)}|}{\min_{\beta\in\mathcal A}|I_\beta^{(k)}|}.\]

\noindent If $g_\varphi\neq 0$ then, by \eqref{navar}, we have
$\var(S(k) g_\varphi)\leq\var g_\varphi$. As $g_{\varphi-g_{\varphi}}=0$, by \eqref{eq:gSk}, it follows that
\begin{align*}\lv(S(k)\varphi)&=
\bl(S(k)\varphi)+\var g_{S(k)\varphi}=\bl(\varphi)+\var( g_{S(k)(\varphi-g_{\varphi})}+S(k)g_{\varphi})
\\&\leq \bl(\varphi)+\var( g_{S(k)(\varphi-g_{\varphi})})+\var(S(k)g_{\varphi}).
\end{align*}
Therefore, by Proposition~\ref{lemlogreno}, we get
\begin{align*}\lv(S(k)\varphi)&\leq \bl(\varphi)+
(M+1)\frac{|I^{(k)}|}{\min_{\beta\in\mathcal A}|I_\beta^{(k)}|}\bl(\varphi-g_\varphi)+\var(g_\varphi)
\leq 4M\frac{|I^{(k)}|}{\min_{\beta\in\mathcal A}|I_\beta^{(k)}|}\lv(\varphi).
\end{align*}
\end{proof}
\subsection{Non-symmetric case}
We now estimate Birkhoff sums for the derivative $\varphi'$ of a function $\varphi \in  \ol(\sqcup_{\alpha\in
\mathcal{A}} I^{(k)}_{\alpha})$ with \emph{asymmetric} logarithmic singularities. Birkhoff sums of this type of function over \emph{rotations} (which can be thought as IETs with $d=2$) were first estimated in the seminar work by Kocergin \cite{Ko:abs} (see also \cite{Ko:abs2}). When the base transformation is an IET, they were studied by the second author in \cite{Ul:mix} when there is a unique logarithmic singularity and by Ravotti in \cite{Rav} in the general case. A crucial estimate in all these works is provided by the following Remark, which was first used by Kocergin in \cite{Ko:abs}.

\begin{remark}[Inverses of an arithmetic progression]\label{rk:Kocergin}
If the points $(x_i)_{i=0}^N\subset [0,1]$ are such that, for some $\delta>0$,  $ |x_i-x_j|\geq \delta$ for every pair of $i\neq j$, then
$$
\sum_{i=0}^N \frac{1}{x_i} \leq \frac{1}{\min_{0\leq i\leq N} x_i} + \sum_{j=1}^N \frac{1}{j\, \delta}  \leq \frac{1}{\min_{0\leq i\leq N} x_i} + \frac{\log N+ 1}{\delta} .
$$
\end{remark}
\begin{lemma}\label{proposition:nonsym}
Suppose that  $T_{(\pi,\lambda)}$ satisfies the Keane condition. Then for every
$\varphi\in\ol(\sqcup_{\alpha\in \mathcal{A}} I_{\alpha})$ with $g_\varphi'=0$, any $k\geq 1$ and  $x \in I^{(k)}$ we have
\begin{equation} \label{asspecialsum}
\left| S(k)(\varphi')(x) - \sum_{\alpha \in \mathcal{A}}
\frac{C_\alpha^+}{x_\alpha^{l}} + \sum_{\alpha \in \mathcal{A}}
\frac{C_\alpha^-}{x_\alpha^{r}} \right| \leq  \bl(\varphi) \frac{1+\log \|Q(k)\|}{\min_{\alpha\in\mathcal A}|I^{(k)}_\alpha|}.
\end{equation}
\end{lemma}

\begin{proof}
Notice first that it is enough to prove \eqref{asspecialsum} in the special cases when
\begin{align*}
\varphi=\varphi_\alpha^+:=\log(|I|\{(x-l_{\alpha})/|I|\}) \text{ or }
\varphi=\varphi_\alpha^-:=\log(|I|\{(r_{\alpha}-x)/|I|\}).
\end{align*}
Indeed, taking the linear combination $\sum_{\alpha\in\mathcal A}C_\alpha^+\varphi_\alpha^+ +C_\alpha^-\varphi_\alpha^-$ then yields the general form of the result.
Since the reasoning is analogous for functions of the form  $\varphi_\alpha^+$  or $\varphi_\alpha^-$ we will only do the computations  for $\varphi_\alpha^+$.

For any $x\in I^{(k)}_\beta$ choose  $0\leq i_0 <Q_\beta(k)$ such that the iterate
$T^{i_0}x$ is the closest to $l_{\alpha}$ among all iterates $T^j x$ with $0\leq j<Q_\beta(k)$ belonging to the interval $(l_\alpha, |I|)$.
Then $x_\alpha^{l}= T^{i_0}(x)-l_\alpha$. Since all points in the orbit segment $\{ T^k x, \, 0\leq k<Q_\beta(k)\}$  belong to separate floors of a Rokhlin tower on which $T$ acts as an isometry on the floors, we also have that
$$\min \{ \vert T^i (x)- T^{j}(x)\vert, 0 \leq i\neq j <Q_\beta(k)\}\geq \min_{\alpha\in\mathcal{A}}|I_\alpha^{(k)}|.$$
Therefore, if we reorder the points in $\{ T^i x\, 0\leq i<Q_\beta(k)\} $ so that
$ l_\alpha <T^{i_0}x < T^{i_1}x< T^{i_2}x <\dots $, we have
$$ |I|\{(T^{i_j} (x)- l_\alpha)/|I|\}\geq \min_{\alpha\in\mathcal{A}}|I_\alpha^{(k)}| \, j \qquad  \text{for \ all} \ 1 \leq j<Q_\beta(k).$$
Thus, since by definition of special Birkhoff sum $S(k)\varphi'(x) = (\varphi')^{(Q_\beta(k))}(x)$ if $x\in I^{(k)}_\beta$,
\begin{align*}
\Big|(\varphi')^{(Q_\beta(k))}(x)-\frac{1}{x_\alpha^{l}}\Big|&\leq \sum_{0\leq i<Q_\beta(k), i\neq {i_0}}
\frac{1}{|I|\{(T^{i}x-l_{\alpha})/|I|\}}
\leq \sum_{1\leq j <Q_{\beta}(k)}\frac{1}{j \, \min_{\alpha\in \mathcal{A}}|I_\alpha^{(k)}|}\leq \frac{1+\log Q_\beta(k)}{\min_{\alpha\in \mathcal{A}}|I_\alpha^{(k)}|},
\end{align*}
were in the last inequality we have used the estimate given by Remark~\ref{rk:Kocergin}. This completes the proof.
\end{proof}

\begin{lemma}\label{lemma:boundLV}
Suppose that  $T_{(\pi,\lambda)}$ satisfies the Keane condition. Then for every
$\varphi\in{\ol}^{\bv}(\sqcup_{\alpha\in \mathcal{A}} I_{\alpha})$ and $k\geq 1$ we have
\begin{equation}\label{naveas}
\lv(S(k)\varphi)\leq \frac{|I^{(k)}|}{\min_{\alpha\in\mathcal{A}}|I^{(k)}_\alpha|}\lv(\varphi)(3+\log\|Q(k)\|).
\end{equation}
\end{lemma}

\begin{proof}
 First suppose that $g_{\varphi}=0$.  Then, by Lemma~\ref{comparingconstants} and in view of Remark~\ref{rk:closestpoints},
% (see \eqref{ass12}), if $x\in I^{(k)}_\beta$, by definition of special Birkhoff sums, % $g_{S(k)\varphi}(x)=S(k)\varphi(x)$ and therefore, by definition of special Birkhoff sum,
%$g'_{S(k)\varphi}(x)=S(k)\varphi'(x) = \varphi'^{(Q_\beta(k))}$.
%In view of Remark~\ref{rk:closestpoints} and the
we can apply the derivative estimates given by Lemma~\ref{proposition:nonsym} and get that
%(see in particular \eqref{asspecialsum}), we have thatwe have that,
for every $x\in I^{(k)}$,
% \eqref{singularitiescomparisons}, for every $x\in I^{(k)}$
%we have
\begin{align*}
|g'_{S(k)\varphi}(x)| &=\Big|S(k)\varphi'(x) -
\sum_{\alpha\in\mathcal{A}} \frac{C^+_\alpha}{|I^{(k)}| \big\{
({x-l^{(k)}_\alpha})/{|I^{(k)}|}\big\}} +
\sum_{\alpha\in\mathcal{A}} \frac{C^-_{\chi(\alpha)}}{ |I^{(k)}|
\big\{ ({r^{(k)}_\alpha -x})/{ |I^{(k)} |}\big\}}\Big|\\
&\leq\frac{\bl(\varphi)(2+\log\|Q(k)\|)}{\min_{\alpha\in\mathcal{A}}|I^{(k)}_\alpha|}. %\qquad  \text{for\ every}\ x\in I^{(k)}.
\end{align*}
It follows that
\[\var g_{S(k)\varphi} \leq  \|g'_{S(k)\varphi}\|_{\sup}|I^{(k)}|\leq
{\bl(\varphi)}(2+\log\|Q(k)\|)\frac{|I^{(k)}|}{\min_{\alpha\in\mathcal A}|I_\alpha^{(k)}|}.\]

If $g_\varphi\neq 0$ then, by \eqref{navar}, we have
$\var(S(k) g_\varphi)\leq\var g_\varphi$. As $g_{\varphi-g_{\varphi}}=0$, by \eqref{eq:gSk}, it follows that
\begin{align*}\lv(S(k)\varphi)&\leq \bl(\varphi)+\var( g_{S(k)(\varphi-g_{\varphi})})+\var(S(k)g_{\varphi})\\
&\leq \bl(\varphi)+
(2+\log\|Q(k)\|)\frac{|I^{(k)}|}{\min_{\alpha\in\mathcal A}|I_\alpha^{(k)}|}\bl(\varphi-g_\varphi)+\var(g_\varphi)\\
&\leq (3+\log\|Q(k)\|)\frac{|I^{(k)}|}{\min_{\alpha\in\mathcal A}|I_\alpha^{(k)}|}\lv(\varphi).
\end{align*}
\end{proof}
Since by Definition of the Diophantine condition \ref{UDC} (see \eqref{def:UDC-d} in Definition~\ref{def:balanced} and Definition~\ref{def:UDC}) the IETs obtained inducing on the subintervals $I^{(k)}$ are all $\kappa$-\emph{balanced}, i.e.~$|I^{(k)}|\leq \kappa \min_{\alpha\in\mathcal{A}}|I^{(k)}_\alpha|$, the conclusion of Lemma~\ref{lemma:boundLV} immediately give the following Corollary.

\begin{corollary}\label{corlv}
Let $T$ be an IET satisfying the \ref{UDC}
Then for all $0\leq k\leq l$ and for every function
$\varphi\in{\ol}^{\bv}(\sqcup_{\alpha\in \mathcal{A}} I^{(k)}_{\alpha})$
 we have
\begin{align}\label{nave1}
\lv(S(k,l)(\varphi))\leq \kappa(3+\log\|Q(k,l)\|)\lv(\varphi).
\end{align}
\end{corollary}

\section{Correction operators}\label{correction:sec}
This section contains the statement and the proof of the key  technical  result of the paper (Theorem~\ref{operatorcorrection} below), which we now motivate and then state.

\subsection{Correction operator for cocycles with  logarithmic singularities}\label{maincorrection:sec}
Let $\varphi$ be a function with logarithmic singularities and $T$ an IET satisfying the Keane condition. Let $S(k)\varphi$ be a sequence of special Birkhoff sums obtained by renormalization, see \S~\ref{SpecialBS}.
%{ corresponding to the acceleration? which one? add?}  ADD
Consider the sequence
\begin{equation}\label{normseq}
\|S(k)\varphi\|_{L^1(I^{(k)})}/|I^{(k)}|, \qquad k\in\mathbb{N},
\end{equation}
of $L^1$-norms, renormalized by $|I^{(k)}|$. Notice that if $S(k)\varphi$ were bounded, the sequence would simply be controlled by the sequence of sup norms $\|S(k)\varphi\|_{L^\infty(I^{(k)})}$, $k\in\mathbb{N}$.
 Typically, the sequence in \eqref{normseq} grows exponentially with an exponent related to the
Lyapunov exponents of the cocycle $A_Y$.

Our goal is to eliminate this growth, by \emph{correcting} the function $\varphi$, namely by subtracting
a piecewise constant function (constant on the continuity intervals of $T$). This piecewise constant function, which we call the \emph{correction}, can be defined for IETs which satisfy the \ref{UDC} and its values can be identified with a vector in
$H(\pi)$. The correction vector will be given by a \emph{correction operator} $\mathfrak{h} : {\ol}(\sqcup_{\alpha\in
\mathcal{A}} I_{\alpha}) \to H(\pi)$. We will call \emph{correcting operator} the operator $P:= I - \mathfrak{h}:  {\ol}(\sqcup_{\alpha\in
\mathcal{A}} I_{\alpha}) \to  {\ol}(\sqcup_{\alpha\in
\mathcal{A}} I_{\alpha})$ which \emph{performs} the correction, namely to $\varphi$ associates the \emph{corrected cocycle} $P(\varphi)= \varphi - \mathfrak{h} (\varphi)$ obtained subtracting the correction $\mathfrak{h}(\varphi)$.
 Under the assumption that $T$ satisfies the \ref{UDC}, for every $\varphi\in {\ol}(\sqcup_{\alpha\in
\mathcal{A}} I_{\alpha})$, the  correction $\mathfrak{h}(\varphi)$ will be such that
% Before correction the sequence $\|S(k)(\varphi)\|_{L^1(I^{(k)})}/|I^{(k)}|$
% for every we construct such that for
 the corrected function $P(\varphi)= \varphi-\mathfrak{h}(\varphi)$ produces a sequence
\begin{equation}\label{normseqcor}
\|S(k) \circ P (\varphi)\|_{L^1(I^{(k)})}/|I^{(k)}|, \qquad k\in\mathbb{N},
\end{equation}
which now has sub-exponential growth.
This will then be the starting point to show the existence of a full deviation spectrum for the $L^1$-norm (see \S~\ref{devspectrum:sec}).
Moreover, if additionally $T$ satisfies the \ref{SUDC} and $\varphi\in {\ol}(\sqcup_{\alpha\in
\mathcal{A}} I_{\alpha})$ satisfies a stronger symmetry condition, $\as(\varphi)=0$, then the sequence \eqref{normseqcor} is bounded along a subsequence,
and it will play a crucial role  in the proof of ergodicity (see \S~\ref{sec:ergskew}).

\subsubsection{The main result on correction of logarithmic cocycles}
The formal statement of the result that we are going to prove is the following.
\begin{theorem}[Existence of a correction operator]\label{operatorcorrection}
Assume that  $T=T_{(\pi,\lambda)}$ satisfies the \ref{UDC}.  There exists a bounded
linear operator $\mathfrak{h} : {\ol}(\sqcup_{\alpha\in
\mathcal{A}} I_{\alpha}) \to H(\pi)$ such that
for every  $\varphi\in{\ol}(\sqcup_{\alpha\in \mathcal{A}}
I_{\alpha})$ with $\mathfrak{h}(\varphi)=0$ we have
\begin{equation}\label{eq:thmcorr1}
\frac{\|S(k)\, {\varphi}\|_{L^1(I^{(k)})}}{|I^{(k)}|}\leq
C\Big(
C'_k(T)\|\varphi\|_{\lv}+\|Q_s(k)\|\frac{\|{\varphi}\|_{L^1(I)}}{|I|}\Big),
\end{equation}
where $C'_k(T)$ is the Diophantine series defined in Definition~\ref{def:series}.
\smallskip
Furthermore, if additionally $T$ satisfies the \ref{SUDC} and $\mathcal{AS}(\varphi)=0$ then
\begin{equation}\label{eq:thmcorr2}
\frac{\|S(k)\, {\varphi}\|_{L^1(I^{(k)})}}{|I^{(k)}|}\leq C\Big(
C_k(T) (\lv(\varphi)+\|\partial_{\pi^{(0)}}(\varphi)\|)+\|Q_s(k)\|\frac{\|{\varphi}\|_{L^1(I)}}{|I|}\Big),
\end{equation}
where $C_k(T)$ is the other Diophantine series defined in Definition~\ref{def:series}.
\end{theorem}
\noindent  Combining  Theorem~\ref{operatorcorrection} with the estimates on the Diophantine series given by Proposition~\ref{prop:estKC} (see in particular \eqref{eq:C'}),  we have the following corollary:
 \begin{corollary}[Subexponential growth of special Birkhoff sums of corrected cocycles]\label{cor:corrected}
 Given  $T$ and $\mathfrak{h}$ as in Theorem~\ref{operatorcorrection},
% For every $T=T_{\pi,\lambda}$ that satisfies the UDC an
 for every  $\varphi\in{\ol}(\sqcup_{\alpha\in \mathcal{A}}
I_{\alpha})$ with $\mathfrak{h}(\varphi)=0$,
%where   $\mathfrak{h} : {\ol}(\sqcup_{\alpha\in \mathcal{A}} I_{\alpha}) \to H(\pi)$ is one of the operators given by Theorem~\ref{operatorcorrection},
 we have
%\begin{equation}
%\label{eq:thmcorr1}
$$\frac{\|S(k)\varphi\|_{L^1(I^{(k)})}}{|I^{(k)}|}=O(e^{\tau k}).$$
%\end{equation}
 \end{corollary}
\noindent  Notice that, in virtue of the definition of the Diophantine series $C_k(T)$ and $C'_k(T)$, the control for the symmetric case given by \eqref{eq:thmcorr2} is finer than that given by \eqref{eq:thmcorr1}  since  $C'_k(T)$ has an additional term which is logarithmic in the matrix cocycle norms (which comes from the presence of $K'_l(T)$ instead of $K_l(T)$, see Definition~\ref{def:series}).
\begin{remark}\label{rk:correctionF} More precisely, we will show in the proof of Theorem~\ref{operatorcorrection}  that for any choice of a subspace $F\subset H(\pi)$ such that $F\oplus \Gamma_s(\pi)=H(\pi)$, where $\Gamma_s(\pi)$ is the stable space of $T=T_{(\pi,\lambda)}$, one can define a unique such operator $\mathfrak{h}= \mathfrak{h}_F$ such that $\mathfrak{h}_F(h)=h$ for any $h\in F$.
\end{remark}
\smallskip
The proof of Theorem~\ref{operatorcorrection} will take the rest of this section. We first of all comment on the difficulties which motivate the change of strategy in comparison to \cite{Ma-Mo-Yo} and \cite{Fr-Ul} and give an outline of the main steps.

\subsubsection{Difficulties and outline of the proof}\label{sec:outline}
%\subsection{Comments and outline of the proof}
The idea of \emph{correction} as well of the strategy for proving of Theorem~\ref{operatorcorrection} are inspired by the seminal work by Marmi-Moussa-Yoccoz on the cohomological equation in \cite{Ma-Mo-Yo} (see also \cite{Ma-Yo}). As we already anticipated in the introduction, though,  when considering functions with logarithmic singularities (or more in general BMO functions) and want to control the   $L^1$-norm (which is the only one that we can controlled for functions with logarithmic singularities, which are unbounded),  we need to modify substantially the original construction.
% it is impossible, therefore we focus on the $L^1$-norm, which
% causes some modifications compared to the original Marmi-Mousa-Yoccoz construction.
The construction presented here is a modification of the construction that we  introduced in \cite{Fr-Ul} to prove an analogous result for IETs of \emph{hyperbolic periodic type}. Working with almost every $T$, but  requires again some major changes in the basic steps of construction. We comment here on the differences while giving an outline of the steps in the proof of Theorem~\ref{operatorcorrection}.

%{\smallskip
%In this part instead of $\Delta P^{(k)}$ we should deal
%with $U^{(k)}M^{(k)}+\Delta P^{(k)}$ as the origin of
%the correction operator. The idea here is: let us try
%to correct by $M^{(k)}$, but this sequence is not equvariant,
%so we have to correct a correction $M^{(k)}$ by $\Delta P^{(k)}$.
%The corrected correction $U^{(k)}M^{(k)}+\Delta P^{(k)}$
%is already equvariant.
%How to write it transparently?}

\smallskip
First note that there is not an unique way to define a correction operator $\mathfrak{h} : {\ol}(\sqcup_{\alpha\in
\mathcal{A}} I_{\alpha}) \to H(\pi)$ with the desired properties (as in Theorem~\ref{operatorcorrection}),
since if we are given a function $\mathfrak{h}(\varphi)$ that satisfies the desired estimates (namely \eqref{eq:thmcorr1} and \eqref{eq:thmcorr2} in Theorem~\ref{operatorcorrection}) and add an element from the stable space $\Gamma_s$, we get a new function that still satisfies the same estimates. On the other hand, if we compose with the projection $U:\mathbb{R}^\mathcal{A}\to \mathbb{R}^\mathcal{A}/\Gamma_s $ to the quotient by the stable space,
the \emph{quotient operator} $$\mathfrak{h}^U := U \circ \mathfrak{h}:{\ol}(\sqcup_{\alpha\in
\mathcal{A}} I_{\alpha}) \to H(\pi)/\Gamma_s$$ is uniquely defined and is the operator we are going to construct.
%This, and the corresponding \emph{correcting operator} $P^{(0)}:= U\circ (I-\mathfrak{h})$ are therefore the operators we are going  to construct. { CHECK: is this correct?}

We will construct in fact a \emph{sequence} of \emph{correcting operators} with values in the quotient by the stable space, namely
 $$P^{(k)}:{\ol}(\sqcup_{\alpha\in
\mathcal{A}} I^{(k)}_{\alpha}) \to {\ol}(\sqcup_{\alpha\in
\mathcal{A}} I^{(k)}_{\alpha})/\Gamma^{(k)}_s, \qquad k\in \mathbb{N}$$ (notice that if $T$ satisfies the \ref{UDC}  the induced IET $T^{(k)}$
satisfies the \ref{UDC} for every $k\geq 1$). For $k=0$, the correcting operator $P^{(0)}$ will have the form $I- \mathfrak{h}^U$, where $ \mathfrak{h}^U$ is the sought  correction operator with values in the quotient.
We want the sequence of operators $P^{(k)}, k\in \N$, to be \emph{equivariant}  under the action of the renormalization, i.e.\ to commute with the operation of taking special  Birkhoff sums
%: for every $k,l\in \N$, $k<l$, we want to have $S(k,l)\circ P^{(k)}=P^{(l)}\circ S(k,l)$
 (see Lemma~\ref{commutes} for a precise statement).

%Each correction operator $P^{(k)}$ has the form of a projection $P^{(k)}= I - \Delta^{(k)}$ which sends $\varphi$ to $\varphi - \Delta^{(k)}\varphi$. Here we call $\Delta^{(k)}\varphi$ the \emph{correction} of $\varphi$.

\smallskip
\noindent The strategy to construct the sequence $P^{(k)}, k\in \N$  of equivariant correcting operators is the following:
\begin{enumerate}
\item As first approximation of the correction operators, consider, for $k\in\mathbb{N}$,  the mean value projections $\mathcal{M}^{(k)}:{\ol}(\sqcup_{\alpha\in
\mathcal{A}} I^{(k)}_{\alpha}) \to \Gamma^{(k)}$, as defined in \S~\ref{sec:meanprojection}, and the associated correcting operators $P^{(k)}_0:= I-\mathcal{M}^{(k)} $, $k\in\mathbb{N}$;  \vspace{.9mm}
%which removes  what we will call \emph{initial corrections}, the  sequence  of
%\end{enumerate}
\item The correcting operators $P^{(k)}_0$, $k\in \N$, are not equivariant and do not take values in the quotient. Let us hence modify them by subtracting a term $\Delta^{(k)}$ and composing with the projection $U^{(k)}$ to the quotient space  $\Gamma^{(k)}/\Gamma^{(k)}_s$, namely consider, for each $k\in\mathbb{N}$, a operator of the form $P^{(k)}:= U^{(k)}\circ P^{(k)}_0- \Delta^{(k)}$;\vspace{.9mm}
% so that $P^{(k)}= P^{(k)}_0- \Delta^{(k)}$ is equivariant.
\item
Following \cite{Ma-Mo-Yo}, one can see that  for $P^{(0)}$ defined as in $(2)$ to be equivariant, one needs to define $\Delta^{(0)}\varphi$ so that the modified correction operator $U\circ \mathcal{M}^{(0)}+\Delta^{(0)}$ is the limit (if it exists) of the sequence $U\circ Q(k)^{-1}\circ \mathcal{M}^{(k)}\circ S(k)(\varphi)$, which is obtained by '\emph{bringing back}' the correction of $\varphi\in {\ol}(\sqcup_{\alpha\in
\mathcal{A}} I^{(0)}_{\alpha})$ at time $k$, namely of the function $S(k)(\varphi)$,  to time $0$ by applying $Q(k)^{-1}$;\vspace{.9mm}
%\item[(4)]
\item Show that the sequence in $(3)$ converges, so that one can define the modification operator $\Delta^{(k)}$, then  the correcting operator $P^{(k)}= U^{(k)}\circ P^{(k)}_0- \Delta^{(k)}$
% the limit operator is  , i.e.~it
 has the required covariance and growths properties.\vspace{.9mm}
%\item finally, show that the acceleration given by the \ref{UDC} is sufficient to guarantee convergence.
\end{enumerate}
Thus, to obtain the desired correction operator one has to show that the sequence
\[U\circ Q(k)^{-1}\circ \mathcal{M}^{(k)}\circ S(k) (\varphi) \in H(\pi^{(0)})/\Gamma_s^{(0)}, \qquad k\in \mathbb{N}\]
 obtained in $(3)$ converges  for every $\varphi\in {\ol}(\sqcup_{\alpha\in
\mathcal{A}} I^{(0)}_{\alpha})$.
Notice that when $\mathcal{M}^{(k)}$ takes values in $H(\pi^{(k)})\subset\Gamma^{(k)}$ then $Q(k)^{-1}$ composed with the projection on $\Gamma/\Gamma_s$
contracts exponentially and this allows to prove the convergence.
In \cite{Ma-Mo-Yo} and \cite{Fr-Ul}, though, the mean value projection $\mathcal{M}^{(k)}$, obtained taking mean values of the function over every exchanged interval (see \eqref{eq:defMk} below) %. Since $\mathcal{M}^{(k)}$
 takes values also outside $H(\pi^{(k)})$. Therefore, the contraction argument does not apply. To circumvent  this problem, in \cite{Fr-Ul} we have used the the projection on $\Gamma/\Gamma_{cs}$, where $\Gamma_{cs}$ is the \emph{central stable} space. Unfortunately, though, this is not sufficient now, when we consider almost every IET.

\smallskip
One of the novelties in this part of the article in relations to the previous correction operators constructions is that we consider
% Instead, we take the new imperfect
\emph{initial corrections} $\mathcal{M}_H^{(k)}$
obtained by composing $\mathcal{M}^{(k)}$ with   the projection $p_{H(\pi^{(k)})}$ onto the space $H(\pi^{(k)})$ (see \S~\ref{sec:prelim}).  In view of the boundary operator estimate given by Lemma~\ref{bd:estimate} (see \S~\ref{sec:bd:estimate}), we  can control the displacement
between $\mathcal{M}_H^{(k)}$ and $\mathcal{M}^{(k)}$  in terms of the boundary operator $\partial_{\pi^{(k)}}$ (see \S~\ref{sec:prelim}, in particular the proof of Lemma~\ref{lemma:prelim}). It is
starting from this modified {preliminary correction} operators  in  step $(1)$ that allows to prove convergence and hence leads to a good definition of correction (and correcting) operators in the more general setting of this paper, but also requires
%of the new imperfect correction operator and the
proving a  series of new inequalities and adding some new technical steps to the construction.
The \ref{UDC} is devised exactly in order to guarantee convergence of this series.
In fact,  to show that the series that gives $\Delta^{(k)}$ (which is written in \eqref{defoperdel}) converges,
we will exploit  the exponential contraction  provided by the condition \eqref{def:UDC-a} and \eqref{eq:unstst}.
%of the last component $S_\flat(k,r+1))^{-1}$. In turn, we will exploit the exponential contraction is
%and the fact that $\mathcal{M}_H^{(k)}$ take values in $H(\pi^{(k)})$.

\smallskip
The final part of the proof is to show that any correction operator $\mathfrak{h}$ defined choosing a representative $\mathfrak{h}(\varphi)$ for the equivalence class $\mathfrak{h}^U (\varphi)$ in $H(\pi)/\Gamma^s$ is such that   $\|S(k)(\varphi-\mathfrak{h}(\varphi))\|_{L^1(I^{(k)})}/|I^{(k)}|$ has sub-exponential growth. This part follows quite closely the proof that we gave in \cite{Fr-Ul}, along the lines of \cite{Ma-Mo-Yo}.

%elated to this operation
%constitute the main novelity in this part of the article in relation to the previous articles.

%In the main results of this Section (Theorem~\ref{operatorcorrection}) we contract the (perfect) corrections $\Delta^{(k)}P$ and $\mathfrak{h}$
%according to .

%{ revise this part later, when the section has been split into subsections, adding references}
%The proof of Theorem~\ref{operatorcorrection} is split into steps, which follow the strategy described above: two lemmas and its final part (divided into two steps) at the end of the section. In Lemmas~\ref{thmcorre},~\ref{commutes}, we construct $\Delta^{(k)}P$ according to the strategy, estimate its norm and establish its intertwining with the renormalization operators.

\subsection{Preliminary corrections.}\label{sec:prelim}
To define initial corrections, let us consider the linear operators on ${\ol}^{\bv}(\sqcup_{\alpha\in \mathcal{A}}
I^{(k)}_{\alpha})$, $k\in\mathbb{N}$,  obtained by considering mean value-projections (which we defined in \S~\ref{sec:meanprojection})
\begin{equation}\label{eq:defMk}
\mathcal{M}^{(k)}:{\ol}^{\bv}(\sqcup_{\alpha\in \mathcal{A}}
I^{(k)}_{\alpha})\to\Gamma^{(k)},\quad
\mathcal{M}^{(k)}\varphi=\sum_{\alpha\in\mathcal{A}}m(\varphi,I^{(k)}_\alpha)\chi_{I^{(k)}_\alpha}.
\end{equation}
% which, by \eqref{eq:proj} and \eqref{eq:projnorm}, satisfies
%\begin{equation}\label{eq:proj1}
%\|h-p_{H(\pi^{(k)})}h\|\leq C_{\mathcal{G}}\|\partial_{\pi^{(k)}}h\|\text{ and }\|p_{H(\pi^{(k)})}h\|\leq \sqrt{d}\|h\|
%\end{equation}
%for every $k\geq 0$ and $h\in\Gamma^{(k)}$.
\subsubsection{Initial corrections}
The sequence of initial corrections that we want to use is given by composing these mean value-projections with the projection onto the space $H(\pi^{(k)})$. Recall that $p_{H(\pi^{(k)})}:\Gamma^{(k)}\to H(\pi^{(k)})$ is the orthogonal projection
on $H(\pi^{(k)})$.
\begin{definition}[Initial corrections]\label{def:P0}
Consider the operator
\[\mathcal{M}_H^{(k)}:{\ol}^{\bv}(\sqcup_{\alpha\in \mathcal{A}}
I^{(k)}_{\alpha})\to H(\pi^{(k)}), \qquad
%is defined to be
\mathcal{M}_H^{(k)}:=p_{H(\pi^{(k)})}\circ {\mathcal{M}}^{(k)}. \]
Set the corresponding initial approximation of the correction operator to be
\begin{align*}
P^{(k)}_0:\quad & {\ol}^{\bv}(\sqcup_{\alpha\in \mathcal{A}}
I^{(k)}_{\alpha})\to  {\ol}^{\bv}(\sqcup_{\alpha\i
\mathcal{A}} I^{(k)}_{\alpha})\\
 & \varphi\mapsto  P^{(k)}_0\varphi:=\varphi-\mathcal{M}_H^{(k)}\varphi.
\end{align*}
\end{definition}
\noindent The following properties of the initial corrections follow almost directly from the estimates  on mean average corrections that we proved in \S~\ref{sec:meanprojection} as preparatory work, combined with the control of the projection through the boundary operator (given by Lemma~\ref{bd:estimate}).
\begin{lemma}[Initial correction estimates]\label{lemma:prelim}
%the operators $\mathcal{M}_H^{(k)}$ and  $ P^{(k)}_0$ are such that
There exists a positive constant $C$ such that for every $k\in\mathbb{N}$, for every $\varphi\in \ol(\sqcup_{\alpha\in \mathcal{A}}
I^{(k)}_{\alpha})$,
\begin{align}
\|P_0^{(k)}\varphi\|_{L^1(I^{(k)})}&\leq
C |I^{(k)}|\big(\log\|Q(k)\|\as(\varphi)+\lv (\varphi)+\|\partial_{\pi^{(k)}}(\varphi)\|\big)\label{ineq:idph1}
\\
\|P_0^{(k)}\varphi\|_{L^1(I^{(k)})}&\leq
4dC|I^{(k)}|\log\big(2\kappa\|Q(k)\|\big)\|\varphi\|_{\lv}\label{ineq:idphnorm}
\\
\|\mathcal{M}_H^{(k)}\varphi\|&\leq  \frac{\kappa\sqrt{d}}{|I^{(k)}|}\|\varphi\|_{L^1(I^{(k)})}.\label{ineq:ph1}
\end{align}
\end{lemma}

\begin{proof}
To estimate $P_0^{(k)}$, we will compare $\mathcal{M}^{(k)}$ with ${\mathcal{M}_H^{(k)}}$, namely estimate
\begin{align}\label{MvsMH}
\|P_0^{(k)}\varphi\|_{L^1(I^{(k)})}&=\|\varphi-{\mathcal{M}_H^{(k)}}\varphi\|_{L^1(I^{(k)})}
  \leq \|\varphi-{\mathcal{M}}^{(k)}\varphi\|_{L^1(I^{(k)})}+\|{\mathcal{M}}^{(k)}\varphi-\mathcal{M}_H^{(k)}\varphi\|_{L^1(I^{(k)})}
\end{align}
Let us estimate separately the two terms in  \eqref{MvsMH}, namely the mean-value correcting operator and the difference of the mean value projections.

\smallskip
\noindent{\it Estimating the mean-value correcting operator.}
By the construction of the mean projection operator (see \eqref{eq:defMk} and the definition of $m$ in \S~\ref{prooflemlos}), we have
\begin{equation}\label{nace}
\|\mathcal{M}^{(k)}\varphi\|_{L^1(I^{(k)})}=\sum_{\alpha\in\mathcal{A}}|m(\varphi,I_\alpha^{(k)})||I_\alpha^{(k)}|
= \sum_{\alpha\in\mathcal{A}}|\int_{I_\alpha
^{(k)}}\varphi(x)\,dx|\leq \|\varphi\|_{L^1(I^{(k)})}
\end{equation}
and, by  \eqref{meanv2}, we can therefore estimate the first term in \eqref{MvsMH} by
\begin{equation}\label{nape}
\|\varphi-\mathcal{M}^{(k)}\varphi\|_{L^1(I^{(k)})}\leq 8 |I^{(k)}|\lv (\varphi).
\end{equation}

\smallskip
\noindent{\it Estimating the difference of the mean value projections.} To estimate the second term in \eqref{MvsMH}, we recall that $p_{H(\pi^{(k)})}$, by Lemma~\ref{bd:estimate}, % (see in particular \eqref{eq:proj} and \eqref{eq:projnorm}),
satisfies $\|h-p_{H(\pi^{(k)})}h\|\leq C_{\mathcal{G}}\|\partial_{\pi^{(k)}}h\|$
%$$
%\|h-p_{H(\pi^{(k)})}h\|\leq C_{\mathcal{G}}\|\partial_{\pi^{(k)}}h\|\quad \text{ and } \quad \|p_{H(\pi^{(k)})}h\|\leq \sqrt{d}\|h\|
%$$
for every $k\geq 0$ and $h\in\Gamma^{(k)}$. Thus,
\begin{align}
\label{ineq:idph}
\begin{split}
{\|\mathcal{M}^{(k)}\varphi-p_{H(\pi^{(k)})}\mathcal{M}^{(k)}\varphi\|_{L^1(I^{(k)})}}&\leq{|I^{(k)}|}\ \|\mathcal{M}^{(k)}\varphi-p_{H(\pi^{(k)})}\mathcal{M}^{(k)}\varphi\|
\leq C_{\mathcal{G}} {|I^{(k)}|} \ \|\partial_{\pi^{(k)}}\mathcal{M}^{(k)}\varphi\| .
\end{split}
\end{align}
\noindent Moreover, by Proposition~\ref{prop:partial},
\begin{equation}\label{eq:parck}
\|\partial_{\pi^{(k)}}\mathcal{M}^{(k)}\varphi\|\leq \|\partial_{\pi^{(k)}}(\varphi)\|+\Big(1+\log\frac{2\kappa}{|I^{(k)}|}\Big)\as(\varphi)+
14d\kappa\lv (\varphi).
\end{equation}

\smallskip
\noindent {\it Proof of \eqref{ineq:idph1}.} Going back to \eqref{MvsMH} and combining the two separate estimates just proved, namely
 \eqref{nape}, \eqref{ineq:idph} and \eqref{eq:parck}, it follows that
\begin{align*}
\|P_0^{(k)}\varphi\|_{L^1(I^{(k)})}
\leq |I^{(k)}|\Big(C_{\mathcal{G}}\|\partial_{\pi^{(k)}}(\varphi)\|+C_{\mathcal{G}}\Big(1+\log\frac{2\kappa}{|I^{(k)}|}\Big)\as(\varphi)+
(14d\kappa C_{\mathcal{G}}+8)\lv (\varphi)\Big).
\end{align*}
As, by \eqref{neq:dli}, we have $|I^{(k)}|^{-1}\leq \|Q(k)\|$, so we get $1+\log(2\kappa/|I^{(k)}|)=O(\|Q(k)\|)$
which yields \eqref{ineq:idph1}.

\smallskip
\noindent {\it Proof of \eqref{ineq:idphnorm}.}
  Recall now that, by  the estimate \eqref{eq:parphi} of $\|\partial_{\pi^{(k)}}(\varphi)\|$ and balance, we have that $\|\partial_{\pi^{(k)}}(\varphi)\|\leq 2d\log(2\kappa/|I^{(k)}|)\|\varphi\|_{\lv}\leq 2d\log(2\kappa d\|Q(k)\|)\|\varphi\|_{\lv}$. Thus,  as $\mathcal{AS}(\varphi)\leq \lv(\varphi)\leq \|\varphi\|_{\lv}$,
it follows from \eqref{ineq:idph1} that %{ CHECK that I didn't forget some pieces... it was a complicated cut and paste... maybe add some intermediate estimates in between? }
\begin{align*}
\|P_0^{(k)}\varphi\|_{L^1(I^{(k)})}&\leq C |I^{(k)}|\left(\log\|Q(k)\|+1+2d\log(2\kappa \|Q(k)\|)\right)\|\varphi\|_{\lv}\\
& \leq 4dC|I^{(k)}|\log(2\kappa\|Q(k)\|)\|\varphi\|_{\lv}.
\end{align*}
This proves also \eqref{ineq:idphnorm}.

%\smallskip MOVE
%Since furthermore
%\begin{equation}\label{eq:compnorms}
%\frac{|I^{(k)}|\|h\|}{\kappa}\leq\min_{\beta\in\mathcal A}|I^{(k)}_\beta|\|h\|\leq\|h\|_{L^1(I^{(k)})}\leq|I^{(k)}|\|h\|\text{ for %every }h\in \Gamma^{(k)},%
%\end{equation}

\smallskip
\noindent {\it Proof of \eqref{ineq:ph1}.}
Finally, to prove \eqref{ineq:ph1}, let us apply once more  Lemma~\ref{bd:estimate},
which also gives that, for every $k\geq 0$ and $h\in\Gamma^{(k)}$,  $\|p_{H(\pi^{(k)})}h\|\leq \sqrt{d}\|h\|$.
Using this  combined with   \eqref{nace}, we get
\begin{align*}
\|\mathcal{M}_H^{(k)}\varphi\|=\|p_{H(\pi^{(k)})}\mathcal{M}^{(k)}\varphi\|\leq \sqrt{d}\|\mathcal{M}^{(k)}\varphi\|
\leq\frac{\kappa\sqrt{d}}{|I^{(k)}|}\|\mathcal{M}^{(k)}\varphi\|_{L^1(I^{(k)})}\leq
\frac{\kappa\sqrt{d}}{|I^{(k)}|}\|\varphi\|_{L^1(I^{(k)})}
\end{align*}
which proves also \eqref{ineq:ph1} and concludes the proof.
\end{proof}

\subsubsection{The series bringing back the corrections}
We can now build the modification $\Delta^{(k)}$ as a series (see \eqref{ciagdop} below), obtained by quotienting and pulling back the preliminary corrections defined in the previous section.

 Consider, for $k\in\mathbb{N}$, the projections $U^{(k)}$ on the quotient by the stable space, namely
\[U^{(k)}:{\ol}^{\bv}(\sqcup_{\alpha\in \mathcal{A}} I^{(k)}_{\alpha})\to
{\ol}^{\bv}(\sqcup_{\alpha\in \mathcal{A}}
I^{(k)}_{\alpha})/\Gamma^{(k)}_{s}.\]
Since $S(k,l)\Gamma^{(k)}_{s}=\Gamma^{(l)}_{s}$ and $S(k,l):
\Gamma^{(k)}\to\Gamma^{(l)}$ is invertible,
the quotient linear transformation
\[S_\flat(k,l):{\ol}^{\bv}(\sqcup_{\alpha\in \mathcal{A}} I^{(k)}_{\alpha})/\Gamma^{(k)}_{s}
\to{\ol}^{\bv}(\sqcup_{\alpha\in \mathcal{A}}
I^{(l)}_{\alpha})/\Gamma^{(l)}_{s}\] is well defined and
$S_\flat(k,l):\Gamma^{(k)}/\Gamma^{(k)}_{s}\to\Gamma^{(l)}/\Gamma^{(l)}_{s}$
is invertible. Moreover,
\begin{equation}\label{splatanies}
S_\flat(k,l)\circ U^{(k)}\varphi=U^{(l)}\circ S(k,l)\varphi\text{
for }\varphi\in{\ol}^{\bv}(\sqcup_{\alpha\in \mathcal{A}}
I^{(k)}_{\alpha}).
\end{equation}
 The following Lemma shows that our Diophantine Condition guarantees the convergence of the series \eqref{ciagdop} obtained bringing back the corrections and hence it can be used to define a modification operator $\Delta^{(k)}$. Furthermore, it provides estimates that show that the modification operator is bounded.

\begin{lemma}[Convergence of the modification series]\label{thmcorre}
Suppose that $T$ satisfies the \ref{UDC}.
For every function $\varphi\in {\ol}(\sqcup_{\alpha\in
\mathcal{A}} I^{(k)}_{\alpha})$, the following limit
\begin{equation}\label{ciagdop}
\Delta^{(k)}\varphi = \lim_{l\rightarrow \infty} U^{(k)}\circ
S(k,l)^{-1}\circ \left(S(k,l)\circ P^{(k)}_0-P_0^{(l)}\circ
S(k,l)\right)\varphi
\end{equation}
exists in
$H(\pi^{(k)})/\Gamma_{s}^{(k)}$ and
\begin{equation}\label{szacdelty}
\|\Delta^{(k)}\varphi\|\leq C K'_k  (\lv (\varphi)+\|\partial_{\pi^{(k)}}(\varphi)\|).
\end{equation}
Moreover,
\begin{equation}\label{szacdelty0}
\|\Delta^{(k)}(S(k)\varphi)\|\leq C K'_k  (\lv (\varphi)+\|\partial_{\pi}(\varphi)\|)\text{ for every }\varphi\in {\ol}(\sqcup_{\alpha\in
\mathcal{A}} I_{\alpha}).
\end{equation}
If additionally $T$ satisfies the \ref{SUDC} and $\varphi\in {\ol}(\sqcup_{\alpha\in
\mathcal{A}} I_{\alpha})$ with $\as(\varphi)=0$ then for every $k\geq 1$ we have
\begin{equation}\label{eq:szacdeltysym}
\|\Delta^{(k)}(S(k)\varphi)\|\leq C K_k  (\lv (\varphi)+\|\partial_{\pi}(\varphi)\|).
\end{equation}
\end{lemma}
\noindent Let us first show that the Lemma implies that $\Delta^{(k)}$ is bounded.

\begin{corollary}[Boundedness of the modification]
For every $k\geq \mathbb{N}$,  the operator $\Delta^{(k)}:{\ol}(\sqcup_{\alpha\in \mathcal{A}} I^{(k)}_{\alpha})\to H(\pi^{(k)})/\Gamma^{(k)}_{s}$ defined by \eqref{ciagdop}  is bounded.
\end{corollary}
\begin{proof}
In view of \eqref{szacdelty} and \eqref{eq:parphi}, for every $\varphi\in{\ol}(\sqcup_{\alpha\in \mathcal{A}}I^{(k)}_{\alpha})$ we have
\begin{equation}\label{eq:delnorm}
\|\Delta^{(k)}\varphi\|\leq K'_k(1+2d\log( 2\kappa d\|Q(k)\|)) \|\varphi\|_\lv.
\end{equation}
This shows that $\Delta^{(k)}:{\ol}(\sqcup_{\alpha\in \mathcal{A}} I^{(k)}_{\alpha})\to H(\pi^{(k)})/\Gamma^{(k)}_{s}$ is bounded.
\end{proof}

\noindent The rest of this section is devoted to the proof of Lemma~\ref{thmcorre}.
\begin{proof}[Proof of Lemma~\ref{thmcorre}] Exploiting the telescopic nature of the series,
calculations similar to those in \cite{Fr-Ul} show that
\begin{align*}U^{(k)}&\circ
S(k,l)^{-1}\circ \left(S(k,l)\circ P^{(k)}_0-P_0^{(l)}\circ
S(k,l)\right)\\&=\sum_{k\leq r< l}(S_\flat(k,r+1))^{-1}\circ U^{(r+1)}\circ {\mathcal{M}_H^{(r+1)}}\circ
S(r,r+1)\circ P_0^{(r)}\circ S(k,r).
\end{align*}
It follows that we need to prove the convergence of the series
\begin{equation}\label{defoperdel}
\sum_{r\geq k}(S_\flat(k,r+1))^{-1}\circ U^{(r+1)}\circ \mathcal{M}_H^{(r+1)}\circ
S(r,r+1)\circ P_0^{(r)}\circ S(k,r)\varphi
\end{equation}
in $H(\pi^{(k)})/\Gamma^{(k)}_s$.

\smallskip

\noindent {\it Convergence of the series and the estimate \eqref{szacdelty}.} For any $ r\geq k$,  using \eqref{ineq:ph1}, \eqref{nase}, \eqref{ineq:idph1},
%\eqref{neq:dli}, \eqref{nave1},
we obtain
\begin{align*}
\|&\mathcal{M}_H^{(r+1)}  \circ S(r,r+1)\circ P_0^{(r)}\circ
S(k,r)\varphi\|\\
& \leq\frac{2\kappa\sqrt{d}}{|I^{(r+1)}|}\|S(r,r+1)\circ P_0^{(r)}\circ S(k,r)\varphi\|_{L^1(I^{(r+1)})}\\
&
\leq\frac{2\kappa\sqrt{d}}{|I^{(r+1)}|}\|P_0^{(r)}\circ S(k,r)\varphi\|_{L^1(I^{(r)})}\\ & \leq
C\frac{|I^{(r)}|}{|I^{(r+1)}|} \big(\as(S(k,r)\varphi)\log\|Q(r)\|+\lv (S(k,r)\varphi)+\|\partial_{\pi^{(r)}}(S(k,r)\varphi)\|\big).
\end{align*}
By the invariance of $\as $, $\lv$ and the boundary operator (see \eqref{nave1}, \eqref{eq:blas} in Corollary~\ref{asblinvariance}, \eqref{eq:eqpar}), \eqref{neq:dli} and \eqref{eq:asbl}   consecutively, we have
\begin{align*}
%\|\mathcal{M}_H^{(r+1)}\circ S(r,r+1)\circ P_0^{(r)}\circ
%S(k,r)\varphi\|
%\\
\|& \mathcal{M}_H^{(r+1)}\circ S(r,r+1)\circ P_0^{(r)}\circ
S(k,r)\varphi\|
\\
%& \leq\frac{2\kappa\sqrt{d}}{|I^{(r+1)}|}\|S(r,r+1)\circ P_0^{(r)}\circ S(k,r)\varphi\|_{L^1(I^{(r+1)})}\\
%&
%\leq\frac{2\kappa\sqrt{d}}{|I^{(r+1)}|}\|P_0^{(r)}\circ S(k,r)\varphi\|_{L^1(I^{(r)})}\\
& \leq
C\frac{|I^{(r)}|}{|I^{(r+1)}|} (\lv (S(k,r)\varphi)+\|\partial_{\pi^{(r)}}(S(k,r)\varphi)\|+\as(S(k,r)\varphi)\log\|Q(r)\|)
\\
&\leq C
\frac{|I^{(r)}|}{|I^{(r+1)}|}(\kappa(3+\log\|Q(k,r)\|)\lv(\varphi)+\|\partial_{\pi^{(k)}}(\varphi)\|+\as(\varphi)\log\|Q(r)\|)
\\
&\leq C'\|Z(r+1)\|(\lv(\varphi)+\|\partial_{\pi^{(k)}}(\varphi)\|)\log\|Q(r)\|.
\end{align*}
\noindent In view of \eqref{eq:unstst}, for $0\leq k< l$ and $h\in H(\pi^{(l)})$
we have
\begin{equation}\label{szacowanieniest}
\|(S_\flat(k,l))^{-1}\circ U^{(l)}(h)\|\leq
\|Q_s(k,l)\|\|U^{(l)}(h)\|\leq \|Q_s(k,l)\|\|h\|.
\end{equation}
\noindent Since %$\|U^{(r+1)}\|=1$ and
${\mathcal{M}_H^{(r+1)}}\circ
S(r,r+1)\circ P_0^{(r)}\circ S(k,r)\varphi\in
H(\pi^{(r+1)})$,
by \eqref{szacowanieniest},
the norm of the $r$-th element of the series  \eqref{defoperdel}
is bounded from above by
\[C'\|Q_s(k,r+1)\|
\|Z(r+1)\|(\lv(\varphi)+\|\partial_{\pi^{(k)}}(\varphi)\|)\log\|Q(r)\|.\]
%Fix any $0<\tau<\theta$. By condition $(a)$ in Definition~\ref{def:typiet} we have
%\[\|Q(r,r+1)\|\leq C_\tau \|Q(0,r)\|^\tau.\]
%The final estimate is
%\[C'C_\tau M\|Q(0,k)\|\|Q(0,r+1)\|^{-(\sigma-\tau)}
%\|\lv(\varphi).\]
%Since the sequence $(\|Q(0,k)\|)_{k\geq 0}$ grows exponentially,
%the series $\sum_{k\geq 1}\|Q(0,k)\|^{-(\sigma-\tau)}$ converges and hence
%the series  \eqref{defoperdel} converges.
Since $T$ satisfies the \ref{UDC}, by Proposition~\ref{prop:estKC}, the series
\[\sum_{r\geq k}\|Q_s(k,r+1)\|\|Z(r+1)\|\log\|Q(r)\|\]
is convergent and its sum is $K'_k$. As $\Delta^{(k)}\varphi$ is the sum of the series \eqref{defoperdel},
it follows that the operator $\Delta^{(k)}$ is well defined and \eqref{szacdelty} holds.

\medskip
\noindent {\it The estimates \eqref{szacdelty0}.}
If $\varphi\in {\ol}(\sqcup_{\alpha\in
\mathcal{A}} I_{\alpha})$ then we can repeat the above arguments for $S(k) \varphi\in {\ol}(\sqcup_{\alpha\in
\mathcal{A}} I^{(k)}_{\alpha})$ instead of $\varphi$. As
\begin{gather*}\lv (S(k,r)(S(k)\varphi))\leq C\log\|Q(r)\|\lv (\varphi),\\
 \|\partial_{\pi^{(r)}}(S(k,r)(S(k)\varphi))\|=\|\partial_{\pi}(\varphi)\|,\
\as(S(k,r)(S(k)\varphi))=\as(\varphi),
\end{gather*}
now the norm of the $r$-th element of the series  \eqref{defoperdel} where  $\varphi$ is replaced by $S(k)\varphi$,
is bounded from above by
\[C'\|Q_s(k,r+1)\|
\|Z(r+1)\|(\lv(\varphi)+\|\partial_{\pi}(\varphi)\|)\log\|Q(r)\|.\]
This gives also \eqref{szacdelty0}.

\medskip

\noindent {\it Symmetric singularities estimates.}  Now suppose that $T$ satisfies the \ref{SUDC} and $\varphi\in {\ol}(\sqcup_{\alpha\in
\mathcal{A}} I_{\alpha})$ with $\as(\varphi)=0$. Using \eqref{nave} and reasoning similar to the above
we obtain
\begin{align*}
\|P_0^{(r)}\circ S(k,r)(S(k)\varphi)\|_{L^1(I^{(r)})}&= \|P_0^{(r)} (S(r)\varphi)\|_{L^1(I^{(r)})} \leq
C|I^{(r)}| (\lv (S(r)\varphi)+\|\partial_{\pi^{(r)}}(S(r)\varphi)\|)
\\
& \leq
C'\|Z(r+1)\||I^{(r+1)}| (\lv (\varphi)+\|\partial_{\pi}(\varphi)\|).
\end{align*}
Thus
\begin{gather*}
\|(S_\flat(k,r+1))^{-1}\circ U^{(r+1)}\circ {\mathcal{M}_H^{(r+1)}}\circ
S(r,r+1)\circ P_0^{(r)}\circ S(k,r)(S(k)\varphi)\|\\
\leq  C\|Q_s(k,r+1)\|
\|Z(r+1)\|(\lv(\varphi)+\|\partial_{\pi}(\varphi)\|).
\end{gather*}
This gives \eqref{eq:szacdeltysym}.
 \end{proof}

\subsubsection{The equivariant correction operators}
Consider now the operator
\[P^{(k)}:{\ol}(\sqcup_{\alpha\in
\mathcal{A}}
I^{(k)}_{\alpha})\to{\ol}(\sqcup_{\alpha\in
\mathcal{A}} I^{(k)}_{\alpha})/\Gamma^{(k)}_{s}\]
given by $P^{(k)}=U^{(k)}\circ P_0^{(k)}-\Delta^{(k)}$. As the operators $U^{(k)}$ and $P_0^{(k)}$ (see \eqref{ineq:idphnorm}) are bounded
linear operators,  $P^{(k)}$ is also linear and bounded when ${\ol}(\sqcup_{\alpha\in
\mathcal{A}} I^{(k)}_{\alpha})/\Gamma^{(k)}_{s}$ is equipped with the $L^1(I^{(k)})/\Gamma^{(k)}_{s}$ norm.
We will now show that this modified correcting operator satisfies the sough equivariance property, i.e.~\emph{commutes} with the operation of considering special Birkhoff sums.
%\begin{remark}\label{Pzero}
%Note that if $\varphi\in\Gamma^{(k)}_0$ then
%$P_0^{(l)}(S(k,l)\varphi)=0$ for all $l\geq k$, hence $\Delta^{(k)}\varphi=0$ and $P^{(k)}\varphi=0$.
%\end{remark}
%\noindent The correction $\Delta^{(k)}$ is defined so that
%$P^{(k)}$ has the crucial property of commuting with the special
%Birkhoff sums operators, as shown by the next Lemma.

\begin{lemma}[Equivariance]\label{commutes}
Suppose that $T$ satisfies the \ref{UDC}. For all $0\leq k\leq l$  we have
\begin{equation}
\label{przem}
S_\flat(k,l)\circ P^{(k)}=P^{(l)}\circ S(k,l).
\end{equation}
%and
%\begin{equation}\label{szak}
%\|P^{(k)}\varphi\|_{L^1(I^{(k)})/\Gamma^{(k)}_{s}}\leq 2 |I^{(k)}|
%\big(K_k (\lv (\varphi)+\|\partial_{\pi^{(k)}}(\varphi)\|)+K'_k\as(\varphi)\big).
%\end{equation}
Moreover, for every $\varphi\in{\ol}(\sqcup_{\alpha\in \mathcal{A}}
I_{\alpha})$ we have
\begin{equation}\label{eq:inSk}
\frac{1}{|I^{(k)}|}\|P^{(k)}(S(k)\varphi)\|_{L^1(I^{(k)})/\Gamma^{(k)}_{s}}\leq \Theta_k(\varphi):=CK'_k\|\varphi\|_{\lv}.
\end{equation}

\noindent If additionally $T$ satisfies the \ref{SUDC} and $\varphi\in {\ol}(\sqcup_{\alpha\in
\mathcal{A}} I_{\alpha})$ with $\as(\varphi)=0$ then \eqref{eq:inSk} holds with
\[\Theta_k(\varphi):=
CK_k(\lv(\varphi)+\|\partial_{\pi}(\varphi)\|).\]
\end{lemma}
\begin{proof}
The condition \eqref{przem} is a direct consequence of the definition of
$P^{(k)}$. Its proof run along similar lines as the proof of the first part of Lemma~4.2 in \cite{Fr-Ul}.

\smallskip
\noindent In view of $\|{U^{(k)}}\|= 1$,
\eqref{ineq:idph1}, \eqref{szacdelty0},  \eqref{naveas} and \eqref{eq:eqpar} we get
\begin{align*}
\|P^{(k)}(S(k)\varphi)\|_{L^1(I^{(k)})/\Gamma^{(k)}_{s}}
&\leq\|P_0^{(k)}(S(k)\varphi)\|_{L^1(I^{(k)})}+|I^{(k)}|\|\Delta^{(k)}(S(k)\varphi)\| \leq  C |I^{(k)}|
 K'_k (\lv (\varphi)+\|\partial_{\pi}(\varphi)\|).
\end{align*}
Moreover, using \eqref{ineq:idphnorm} and \eqref{eq:delnorm} instead of \eqref{ineq:idph1} and \eqref{szacdelty0}, we also have
\begin{align*}
\|P^{(k)}(S(k)\varphi)\|_{L^1(I^{(k)})/\Gamma^{(k)}_{s}}\leq C|I^{(k)}|K'_k\|\varphi\|_{\lv},
\end{align*}
which give \eqref{eq:inSk}.

\medskip

\noindent {\it Symmetric singularities case.} Suppose that $T$ satisfies the \ref{SUDC} and $\varphi\in {\ol}(\sqcup_{\alpha\in
\mathcal{A}} I_{\alpha})$ with $\as(\varphi)=0$. Then, using \eqref{eq:szacdeltysym} and \eqref{nave} instead of \eqref{szacdelty} and \eqref{naveas},
we get \eqref{eq:inSk} with $\Theta_k(\varphi)=
CK_k(\lv(\varphi)+\|\partial_{\pi}(\varphi)\|)$.
\end{proof}

\subsection{Proof of Theorem~\ref{operatorcorrection}}
Now that we have build the correcting operator $P^{(0)}$ with values in the space  ${\ol}(\sqcup_{\alpha\in
\mathcal{A}} I^{(0)}_{\alpha})/\Gamma^{(0)}_{s}$ and the desired equivariance properties (see Lemma~\ref{commutes}, we want to
check that any choice of representative for the equivalence class $P^{(0)}\varphi$ satisfies the desired growth estimates and then
to lift $P^{(0)}$ to an operator $I-\mathfrak{h}$ with values in ${\ol}(\sqcup_{\alpha\in
\mathcal{A}} I^{(k)}_{\alpha})$. We first prove a  Lemma that shows that any choice of  representative of the equivalence class $P^{(0)}(\varphi)$ satisfies the desired estimates hold (see Lemma~\ref{BSestimates} and in particular the estimates in \ref{l68ii}) and  then use it to show that the correction is uniquely defined (see Corollary~\ref{cor:uniqueness}). The proof of Theorem~\ref{operatorcorrection} then follows easily from this Lemma~\ref{BSestimates} and Corollary~\ref{cor:uniqueness} and is given at the end of the section.

\smallskip
 Recall that we defined the equivariant correction operator $P^{(0)}$ by setting $P^{(0)} = U^{(0)}\circ P_0^{(0)} -\Delta^{(0)}$.
We say that a map $\widehat{\varphi}\in {\ol}(\sqcup_{\alpha\in
\mathcal{A}} I^{(0)}_{\alpha})$ is  a \emph{correction}  of $\varphi$ if it is a \emph{representative of the corrected equivalence class} $P^{(0)}\varphi$, i.e.\ $U^{(0)}(\widehat{\varphi})=P^{(0)}(\varphi)$.
%In this case, indeed, $\varphi $ belongs to the equivalence class $ P^{(0)}\varphi$.
 With this in mind, the following Lemma shows that any correction of $\varphi$
 %representative of the equivalence class of $P^{(0)}(\varphi)$ %when $\varphi$ is corrected
 satisfies the desired estimates on the growth of Birkhoff sums. The constants  $C_k$ and $C'_k$ which appear in the estimates of Birkhoff sums of corrected functions (see part \ref{l68ii} of the Lemma below) are given by the  \emph{Diophantine series} $C_k(T)$ and $C'_k(T)$
which we defined for any $k\in\mathbb{N}$ in \S~\ref{def:series}  and showed that they converge and hence are well defined under the assumption that $T$ satisfies the \ref{UDC} or \ref{SUDC}.
%These two series now appear as constants.

\begin{lemma}[Birkhoff sums estimates for corrected functions]\label{BSestimates}
Suppose that $T$ satisfies the \ref{UDC}.
Assume that $\varphi, \widehat{\varphi}\in{\ol}(\sqcup_{\alpha\in
\mathcal{A}} I^{(0)}_{\alpha})$ and that
$U^{(0)}\widehat{\varphi}= P^{(0)}\varphi$.  Then:
\begin{enumerate}[label=(\roman*)]
\item \label{l68i} $\widehat{\varphi}-\varphi\in H(\pi^{(0)})$.
\item \label{l68ii}  For any  $k\geq 1$ we have
\begin{equation}\label{thmcorrecgener2}
\frac{\|S(k)(\widehat{\varphi})\|_{L^1(I^{(k)})}}{|I^{(k)}|}\leq C\Big(
C'_k\|\varphi\|_{\lv}+C''_k\frac{\|{\widehat{\varphi}}\|_{L^1(I^{(0)})}}{|I^{(0)}|}\Big)
\end{equation}
and if $T$ satisfies the \ref{SUDC} and $\as(\varphi)=0$ then
\begin{equation}\label{thmcorrecgener3}
\frac{\|S(k)(\widehat{\varphi})\|_{L^1(I^{(k)})}}{|I^{(k)}|}\leq C\Big(
C_k \big(\lv(\varphi)+\|\partial_{\pi^{(0)}}(\varphi)\|\big)+C''_k\frac{\|\widehat{\varphi}\|_{L^1(I^{(0)})}}{|I^{(0)}|}\Big)
\end{equation}
with $C_k:=C_k(T)$, $C'_k:=C'_k(T)$ (refer to \S~\ref{def:series} for the definition of  the Diophantine series $C_k$ and $C'_k$) and $C''_k:=\|Q_s(k)\|$.
\end{enumerate}
\end{lemma}
\noindent The Lemma shows that every correction of $\varphi$ is of the form $\varphi-h$ with $h\in H(\pi^{(0)})$. Let us first show that the Lemma also implies that the correction $h$ is uniquely defined, once we fix a complement to $\Gamma_s^{(0)}$ in $H(\pi^{(0)})$.
\begin{corollary}[Uniqueness of the correction]\label{cor:uniqueness}
Fix a subspace $F\subset H(\pi^{(0)})$ such that $F\oplus \Gamma_s^{(0)}=H(\pi^{(0)})$. Suppose that $h_1,h_2\in F$ are two vectors such
that

\[U^{(0)}(\varphi-h_1)=U^{(0)}(\varphi-h_2)=P^{(0)}\varphi.\]
Then $h_1=h_2$.
\end{corollary}
\begin{proof} In view of \eqref{thmcorrecgener2} of Lemma~\ref{BSestimates} combined with \eqref{eq:C'}, we have
\[\limsup_{k\to+\infty}\frac{\log\frac{\|S(k)(\varphi-h_i)\|_{L^1(I^{(k)})}}{|I^{(k)}|}}{k}\leq 0 \text{ for }i=1,2.\]
Thus
\[\limsup_{k\to+\infty}\frac{\log\|Q(k)(h_1-h_2)\|}{k}\leq 0.\]
As $h_1-h_2\in H(\pi^{(0)})$, by the condition \eqref{def:O} in Definition~\ref{def:UDC}, it follows that $h_1-h_2\in\Gamma^{(0)}_{s}$.
Since $h_1-h_2\in F$ and $\Gamma^{(0)}_{s} \cap F = \{ {0}\}$, we have $h_1=h_2$.
\end{proof}

\smallskip
\noindent Let us now prove the Lemma.
\begin{proof}[Proof of Lemma~\ref{BSestimates}] Since by definition of the operators
\[ U^{(0)}\widehat{\varphi}=
P^{(0)}\varphi= U^{(0)}\circ P_0^{(0)}\varphi-\Delta^{(0)}\varphi=
U^{(0)}\varphi-U^{(0)}\circ \mathcal{M}_H^{(0)}\varphi-\Delta^{(0)}\varphi,\]
we have $$U^{(0)}(\varphi-\widehat{\varphi})= U^{(0)}\circ \mathcal{M}_H^{(0)}\varphi+\Delta^{(0)}\varphi \in  H(\pi^{(0)})/\Gamma_s^{(0)}.$$
Therefore
\begin{equation}\label{eq:difphi}
\varphi-\widehat{\varphi}\in  H(\pi^{(0)})+\Gamma_s^{(0)}\subset  H(\pi^{(0)}).
\end{equation}
In view of \eqref{splatanies} and \eqref{przem},
\[U^{(k)}\circ
S(k)\widehat{\varphi}=S_\flat(k)\circ U^{(0)}\widehat{\varphi}
=S_\flat(k)\circ P^{(0)}\varphi=P^{(k)}\circ S(k)\varphi.\] Therefore,
from \eqref{eq:inSk}, we have
\[
\|U^{(k)}\circ
S(k)\widehat{\varphi}\|_{L^1(I^{(k)})/\Gamma^{(k)}_{s}}
=\|P^{(k)}(S(k)\varphi)\|_{L^1(I^{(k)})/\Gamma^{(k)}_{s}}
|I^{(k)}|\leq \Theta_k(\varphi).
\]
It follows from the definition of $\| \cdot
\|_{L^1(I^{(k)})/\Gamma^{(k)}_{s}}$ on the quotient space that
for every $k\geq 0$ there exists
$\varphi_k\in{\ol}(\sqcup_{\alpha\in \mathcal{A}}
I^{(k)}_{\alpha})$ and $s_k\in\Gamma^{(k)}_{s}$ such that
\begin{equation}\label{szcsk1}
S(k)\widehat{\varphi}=\varphi_k+s_k\text{ and
}\frac{\|\varphi_k\|_{L^1(I^{(k)})}}{|I^{(k)}|}\leq \Theta_k(\varphi).
\end{equation}
Next note that
\begin{equation}\label{roznren}
\varphi_{k+1}+s_{k+1}=S(k+1)\widehat{\varphi}=S(k,k+1)S(k)\widehat{\varphi}=S(k,k+1)\varphi_k+Q(k,k+1)s_k,
\end{equation}
so setting $\Delta s_{k+1}=s_{k+1}-Z(k+1)s_k$ ($\Delta s_0=s_0$) we
have
\[\Delta s_{k+1}=-\varphi_{k+1}+S(k,k+1)\varphi_k.\]
Moreover, by \eqref{nase}, for $k\geq 0$,
\begin{align*}
\|\Delta
 s_{k+1}   \|&\leq \frac{\kappa}{|I^{(k+1)}|}\|\Delta s_{k+1}   \|_{L^1(I^{(k+1)})}
 =\frac{\kappa}{|I^{(k+1)}|}\|\varphi_{k+1} - S(k,k+1)\varphi_{k}\|_{L^1(I^{(k+1)})}\\
&\leq
\frac{\kappa}{|I^{(k+1)}|}\big(\|\varphi_{k+1}\|_{L^1(I^{(k+1)})}+\|S(k,k+1)\varphi_{k}\|_{L^1(I^{(k+1)})}\big)\\
&\leq \kappa\Big(\frac{\|\varphi_{k+1}\|_{L^1(I^{(k+1)})}}{|I^{(k+1)}|}+\frac{|I^{(k)}|}{|I^{(k+1)}|}\frac{\|\varphi_{k}\|_{L^1(I^{(k)})}}{|I^{(k)}|}\Big).
\end{align*}
Next, by  \eqref{szcsk1} and \eqref{neq:dli}, it follows that
\[ \|\Delta s_{k+1}\|\leq \kappa\big(\|Z(k+1)\|\Theta_k(\varphi)+\Theta_{k+1}(\varphi)\big)\text{ for }k\geq 0\]
and
\begin{align*}
\|\Delta s_{0}\|\leq \kappa\frac{\|s_{0}\|_{L^1(I^{(0)})}}{|I^{(0)}|}
=\kappa\frac{\|\widehat{\varphi}-\varphi_0\|_{L^1(I^{(0)})}}{|I^{(0)}|}\leq
\kappa\frac{\|\widehat{\varphi}\|_{L^1(I^{(0)})}}{|I^{(0)}|}
+
 \kappa\Theta_0(\varphi).
\end{align*}
Since $s_k=\sum_{0\leq l\leq k}Q(l,k)\Delta s_{l}$ and $\Delta
s_l\in\Gamma^{(l)}_{s}$, setting $ \Theta_{-1}:=0$, we have
\begin{align*}
\|s_k\|&\leq\sum_{0\leq l\leq k}\|Q(l,k)\Delta s_{l}\|\leq\sum_{0\leq l\leq k}\|Q_s(l,k)\|\|\Delta s_{l}\|\\
&\leq \kappa  \sum_{0\leq l\leq k}\|Q_s(l,k)\| (\Theta_{l}(\varphi)+\|Z(l)\|\Theta_{l-1}(\varphi))+\kappa\|Q_s(k)\|\frac{\|\widehat{\varphi}\|_{L^1(I^{(0)})}}{|I^{(0)}|}.
\end{align*}
In view of
\eqref{szcsk1} and taking $\Theta_k(\varphi)=CK'_k\|\varphi\|_{\lv}$, it follows that for $k\geq 1$,
\begin{align*}
\frac{\|S(k)\widehat{\varphi}\|_{L^1(I^{(k)})}}{|I^{(k)}|}
\leq\frac{\|{\varphi}_k\|_{L^1(I^{(k)})}}{|I^{(k)}|}+\|s_k\|\leq d\kappa C
\Big(C'_k\|\varphi\|_{\lv}+C''_k\frac{\|{\varphi}\|_{L^1(I^{(0)})}}{|I^{(0)}|}\Big).
\end{align*}
If $T$ satisfies the \ref{SUDC} and $\as(\varphi)=0$ then the same argument applied to $\Theta_k(\varphi)=
CK_k(\lv(\varphi)+\|\partial_{\pi^{(0)}}(\varphi)\|)$ shows also
\eqref{thmcorrecgener3}.
\end{proof}

\noindent We have now all the elements to conclude the proof of Theorem~\ref{operatorcorrection}.
\begin{proof}[Proof of Theorem~\ref{operatorcorrection}]
Fix a subspace $F\subset H(\pi^{(0)})$ such that $F\oplus \Gamma_s^{(0)}=H(\pi^{(0)})$.  Choose any $\widehat{\varphi}\in{\ol}(\sqcup_{\alpha\in \mathcal{A}}
I_{\alpha})$ with $U^{(0)}(\widehat{\varphi})=P^{(0)}\varphi$. By \ref{l68i} of Lemma~\ref{BSestimates},
$\widehat{\varphi}-\varphi\in H(\pi^{(0)})$. Therefore,
there exist $h\in F$ and
$h'\in\Gamma^{(0)}_{s}$ such that
$\varphi-h=\widehat{\varphi}+h'$. As $U^{(0)}(\widehat{\varphi})=
P^{(0)}\varphi$, it follows that
\[U^{(0)}(\varphi-h)=U^{(0)}(\widehat{\varphi}+h')=U^{(0)}(\widehat{\varphi})=P^{(0)}\varphi.\]
By Corollary~\ref{cor:uniqueness},
%We first show that
for every
$\varphi\in{\ol}(\sqcup_{\alpha\in \mathcal{A}}
I_{\alpha})$ there exists  a unique $h=h(\varphi)\in F$ such that $U^{(0)}(\varphi-h)=P^{(0)}\varphi$.
%$\varphi\in{\ol}(\sqcup_{\alpha\in \mathcal{A}}
%I_{\alpha})$ there exists  a unique $h=h(\varphi)\in F$ such that $U^{(0)}(\varphi-h)=P^{(0)}\varphi$.
 Thus, there
exists a unique linear operator
$\mathfrak{h}:{\ol}(\sqcup_{\alpha\in \mathcal{A}}
I_{\alpha})\to F$ (the
{\em correction operator}) such that
\begin{equation}\label{def:oph}
U^{(0)}(\varphi-\mathfrak{h}(\varphi))= P^{(0)}(\varphi).
\end{equation}
%Note that, by Remark~\ref{Pzero}, ${P}^{(0)}(h)=0$ for each
%$h\in\Gamma^{(0)}_0$, so
%\begin{equation}\label{opernagama}
%{\mathfrak{h}}(h)=h\;\text{ if }\;h\in\Gamma^{(0)}_u\cap
%\Gamma^{(0)}_0\;\;\;\text{ and }\;\;\;{\mathfrak{h}}(h)=0\;\text{ if
%}\;h\in\Gamma^{(0)}_{cs}.\end{equation}
%In particular, the image of $\mathfrak{h}$ is $\Gamma^{(0)}_u\cap
%\Gamma^{(0)}_0$ which has dimension $g-1$.

\noindent As the operator
$P^{(0)}:{\ol}(\sqcup_{\alpha\in \mathcal{A}}
I_{\alpha})\to  {\ol}(\sqcup_{\alpha\in \mathcal{A}}
I_{\alpha})/\Gamma^{(0)}_{s}$
is bounded, by the
closed graph theorem, the operator ${\mathfrak{h}}$ is also
bounded. Indeed, if $\varphi_n\to \varphi$ in ${{\ol}}$
and ${\mathfrak{h}}(\varphi_n)\to h$ in
$F$ then have both
\begin{align*}
& {P}^{(0)}\varphi_n\to
{P}^{(0)}\varphi=U^{(0)}(\varphi-{\mathfrak{h}}(\varphi)),\\
& {P}^{(0)}\varphi_n = U^{(0)}(\varphi_n-{\mathfrak{h}}(\varphi_n)) \to U^{(0)}(\varphi-h).
\end{align*}
It follows that ${\mathfrak{h}}(\varphi)-h \in
F$ and at the same time
${\mathfrak{h}}(\varphi)-h\in\Gamma^{(0)}_{s}$, so
$h={\mathfrak{h}}(\varphi)$. Since the  vector norm and the
$L^1$-norm are equivalent on $\Gamma^{(0)}$,
we get that the operator is bounded.

\smallskip
\noindent Suppose now that
$\mathfrak{h}(\varphi)=0$. Then
\[
U^{(0)}(\varphi)=U^{(0)}(\varphi-\mathfrak{h}(\varphi))=P^{(0)}(\varphi).
\]
Therefore, \eqref{eq:thmcorr1} and \eqref{eq:thmcorr2}  follow directly from
\eqref{thmcorrecgener2} and \eqref{thmcorrecgener3} of the part \ref{l68ii} of Lemma~\ref{BSestimates} respectively. This concludes the proof and proves as well the statement of Remark~\ref{rk:correctionF}.
 \end{proof}

\section{Deviations of Birkhoff sums and integrals}\label{sec:integrals}
In this section we prove the main results on the deviation spectrum of locally Hamiltonian flows, by first reducing the study of integrals along a locally Hamiltonian flow to the study of Birkhoff sums (see \S~\ref{sec:reductionintegrals}), then exploiting the correction operator built in \S~\ref{correction:sec} to build (in the spirit of Bufetov functionals and Bufetov work \cite{Bu}) the cocycles  which correspond to pure power behaviour, see \S~\ref{devspectrum:sec}.

\subsection{Estimates of Birkhoff integrals through Birkhoff sums}\label{sec:reductionintegrals}%\label{sec:estimates}
In this section we provide effective estimate for the growth of Birkhoff integrals (Proposition~\ref{thm:ints}),
which can be applied when the roof function $g$ is unbounded. We first exploit the special flow representation of the flow as a suspension flow over an IET under a roof function with logarithmic singularities (refer to \S~\ref{sec:reductions})  to reduce to estimates of Birkhoff sums, see \S~\ref{sec:decompflow}. We then exploit a standard decomposition of Birkhoff sums in special Birkhoff sums, see \S~\ref{sec:BSdecomp}.  The estimates relies on the speed of decay of the tails of $g$.  This crucial new ingredient is explained in \S~\ref{sec:estimates}. The main result of this section is then the estimate given by Proposition~\ref{thm:ints} in \S~\ref{integralsestaimtes}.

\subsubsection{Reduction of  integrals along the flow to Birkhoff sums}\label{sec:decompflow}
Let $T:I\to I$ be an ergodic IET and  let $g:I\to\R_{>0}\cup\{+\infty\}$ be an integrable function such that $\underline{g}=\inf_{x\in I}g(x)>0$. Following \S~\ref{sec:specialflowdef}, we denote by
$T^g_\R: I^g\to I^g$ the special flow over $T$ under the roof $g$.
For every integrable function $f:I^g\to\R$ let $\varphi_f:I\to\R$ be given by $\varphi_f(x)=\int_0^{g(x)}f(x,r)\,dr$.
By Fubini's theorem, $\varphi_f$ is well defined for a.e.\ $x\in I$, is integrable and
\[\int_I\varphi_f(x)\,dx=\int_{I^g}f(x,r)\,dx\,dr.\]
For every $(x,r)\in I^g$ and $s>0$ denote by $n(x,r,s)\geq 0$ the number of times the orbit segment $\{T^g_t(x,r):t\in [0,s]\}$ crosses the interval $I$ (identified with $I\times\{0\}$), i.e.\  the unique non-negative integer number such that
\begin{equation}\label{def:n}
g^{(n(x,r,s))}(x)\leq r+s<g^{(n(x,r,s)+1)}(x).
\end{equation}
Then $0\leq n(x,r,s)\leq s/\underline{g}+1$.

\smallskip
For every $c\geq \underline{g}$, let $I_c\subset I$ be the level set defined by  $g(x)\leq c$ for every $x\in I_c$. Moreover, for every $s\geq 0$ let
\begin{equation}\label{A:def}
A_c^s:=\{(x,r)\in I^g:x\in I_c\}\setminus\{T^g_{-t}(x,0):x\in I\setminus I_c,\,0\leq t\leq s\}\subset I^g.
\end{equation}
\noindent The following elementary Lemma relates the Birkhoff integrals of $f$ for the flow $T^g_\R$ with the Birkhoff sums of $\varphi_f$ for the IET $T$.
\begin{lemma}\label{lem:int0s}
Suppose that  $f:I^g\to\R$ is bounded. For every  $s>0$ and $c\geq \underline{g}$
if $(x,r)\in A^s_c$ then
%\begin{equation}\label{eq:ia}
$T^ix\in I_c$ for all $0\leq i\leq n(x,r,s)$, and
\begin{equation}
\Big|\int_0^sf(T^g_t(x,r))\,dt\Big|
\leq |\varphi_f^{(n(x,r,s))}(x)|+2c\|f\|_{L^{\infty}}.
\end{equation}
\end{lemma}

\begin{proof}
For every $(x,r)\in A_c^s$ we decompose the orbit segment $\{T^g_t(x,r):t\in[0,s]\}$ into $n(x,r,s)+1$-pieces
using its meeting points with $I\times\{0\}\subset I^g$, i.e.\ along crossing times
\[0<t_1<\ldots<t_n<s,\text{ where }n:=n(x,r,s)\text{ and }t_i:=g^{(i)}(x)-r\text{ for }1\leq i\leq n.\]
Then $T^g_{t_i}(x,r)=(T^ix,0)$ for $0\leq i\leq n$, with $t_0:=-r$. As $(x,r)\in A_c^s$, it follows that
$g(T^ix)\leq c$ for $0\leq i\leq n$, which proves the first part of the Lemma. As $t_{i+1}-t_i=g^{(i+1)}(x)-g^{(i)}(x)=g(T^ix)$, according to the decomposition we obtain
\begin{align*}
\int_0^sf(T^g_t(x,r))\,dt
&=\int_0^{t_0}f(T^g_t(x,r))\,dt+\sum_{0\leq j< n}\int_{t_j}^{t_{j+1}}f(T_t^g(x,r))\,dt
+\int_{t_n}^{s}f(T_t^g(x,r))\,dt\\
&=\sum_{0\leq j< n}\int_{0}^{g(T^jx)}f(T^jx,t)\,dt
-\int_{0}^{r}f(x,t)\,dt+\int_{0}^{s-t_n}f(T^{n}x,t)\,dt\\
&=\varphi_f^{(n)}(x)-\int_{0}^{r}f(x,t)\,dt+\int_{0}^{s-t_n}f(T^{n}x,t)\,dt.
\end{align*}
Since $r<g(x)\leq c$ and $s-t_n<g(T^{n})\leq c$, we also have
\[\Big|\int_{0}^{r}f(x,t)\,dt\Big|\leq \int_{0}^{g(x)}|f(x,t)|\,dt\leq c\|f\|_{L^\infty}\]
and
\[
\Big|\int_{0}^{s-t_n}f(T^{n}x,t)\,dt\Big|\leq\int_{0}^{g(T^{n}x)}|f(T^{n}x,t)|\,dt\leq c\|f\|_{L^\infty}.
\]
Therefore
\[
\Big|\int_0^sf(T^g_t(x,r))\,dt\Big|
\leq \big|\varphi_f^{(n(x,r,s))}(x)\big|+2c\|f\|_{L^{\infty}}
\]
for every $(x,r)\in A_c^s$.
\end{proof}

\subsubsection{Decomposition of Birkhoff sums in special Birkhoff sums}\label{sec:BSdecomp}
In this subsection we estimate $\varphi^{(n)}(x)$ by decomposing the sum into special Birkhoff sums introduced by Zorich in \cite{Zo}.  Let $T:I\to I$ be an arbitrary IET satisfying Keane's condition.
For every $x\in I$ and $n\geq 0$ set
\[m(x,n)=m(x,n,T):=\max\big\{l\geq 0:\#\{0\leq k\leq n:T^kx\in I^{(l)}\}\geq 2\big\}.\]
\begin{proposition}[see \cite{Zo} or \cite{ViB}] \label{relmn}
For every $x\in I$ and $n>0$ we have
\[\min_{\alpha\in\mathcal{A}}Q_{\alpha}(m)\leq n\leq d\max_{\alpha\in\mathcal{A}}
Q_{\alpha}(m+1)=d\|Q(m+1)\|, \text{ where }m=m(x,n).\]
\end{proposition}
\noindent Since the sequence $\big(\min_{\alpha\in\mathcal{A}}Q_{\alpha}(m)\big)_{m\geq 0}$
 increases to the infinity \[m(n)=m(n,T):=\max\{m(x,n):x\in I\}\] is well defined.
If $T$ additionally satisfies the \ref{UDC} then, by \eqref{def:UDC-c} and \eqref{eq:minqk}, for every $\tau>0$
we have
\begin{equation}\label{eq:mnrel}
e^{\lambda_1 m(n)}\leq O(\|Q(m(n))\|)\leq O\Big(\min_{\alpha\in\mathcal{A}}Q_\alpha(m(n))^{1+\tau}\Big)= O(n^{1+\tau}).
\end{equation}

\begin{proposition}\label{twdevpre1}
For every $s> 0$ and $c\geq \underline{g}$
if $(x,r)\in A_c^s$ then
\begin{equation}\label{eq:twdevpre1}
\big|\varphi_f^{(n(x,r,s))}(x)\big|\leq 2\sum_{k=0}^{m(n(x,r,s))}\|Z(k+1)\|\|S(k)\varphi_f\|_{L^{\infty}( I^{(k)}(c))},
\end{equation}
with
\[I^{(k)}(c):=\bigcup_{\alpha\in\mathcal A}\{x\in I^{(k)}_\alpha: \forall_{0\leq j<Q_\alpha(k)}T^jx\in I_c\}.\]
\end{proposition}

\begin{proof}
Fix $s>0$ and $c>0$.
For each point $(x,r)\in I^g$ we will decompose the orbit segment
\[x,Tx,\ldots,T^{n-1}x\text{ with }n:=n(x,r,s)\] into segments. Let $m:=m(x,n)$, so $I^{(m)}$ is hit by the the orbit segment at least twice
and $I^{(m+1)}$ at most once.  For each $0\leq k\leq m$ let
\[n_k^{+}=\min\{j\geq 0:T^jx\in I^{(k)}\},\quad n_k^-=\min\{j\geq 1:T^{n-j}x\in I^{(k)}\}.\]
%Then for each $(x,r)\in I^g$ and $0\leq k\leq m$ we have
%\[0\leq n_k^{+}< \|Q(k)\|\text{ and }0< n_k^-\leq \|Q(k)\|.\]
For $0\leq k<m$ we also have
\[T^{n_{k+1}^{+}}x=(T^{(k)})^{b_k^{+}}T^{n_k^{+}}x\text{ and }
T^{n-n_{k+1}^-}x=(T^{(k)})^{-b_k^-}T^{n-n_k}x\] with
\begin{equation}\label{eq:b1}
0\leq b_k^{+},b_k^{-}<\|Z(k+1)\|.
\end{equation}
Moreover,
\begin{equation}\label{eq:b2}
(T^{(m)})^{b_m}T^{n_m^{+}}x=T^{n-n_m^-}x\text{ with }1\leq b_m\leq \|Z(m+1)\|.
\end{equation}
Here $T^{n_m^{+}}x$, $T^{n-n_m^-}x$ are the first and the last visit of the orbit segment in $I^{(m)}$.
Thus
\begin{align*}
\varphi_f^{(n)}(x)=&\sum_{k=0}^{m-1}\sum_{j=0}^{b_k^{+}-1}(S(k)\varphi_f)((T^{(k)})^{j}T^{n_k^{+}}x)+
\sum_{j=0}^{b_m-1}(S(m)\varphi_f)((T^{(m)})^{j}T^{n_m^{+}}x)\\&\quad+
\sum_{k=0}^{m-1}\sum_{j=0}^{b_k^--1}(S(k)\varphi_f)((T^{(k)})^{j}T^{n-n_{k+1}^-}x).
\end{align*}
If $(x,r)\in A_s^c$, then, by the first part of Lemma~\ref{lem:int0s}, $T^lx\in I_c$ for all $0\leq l\leq n$. Hence
\[(T^{(k)})^{j}T^{n_k^{+}}x,(T^{(k)})^{j}T^{n-n_{k+1}^-}x\in I^{(k)}(c).\]
In view of \eqref{eq:b1} and \eqref{eq:b2}, it follows that
\[|\varphi_f^{(n)}(x)|\leq 2\sum_{k=0}^{m}\|Z(k+1)\|\|S(k)\varphi_f\|_{L^{\infty}( I^{(k)}(c))},\]
which proves \eqref{eq:twdevpre1}.
\end{proof}

\subsubsection{Control of the tail behaviour}\label{sec:estimates}
Let $g:I\to\R_{>0}\cup\{+\infty\}$ be an integrable roof map with $\underline{g}=\min_{x\in I}g(x)>0$.
Suppose that for every $s\geq \underline{g}$ we have a subset $I_s\subset I$ such that $g(x)\leq s$ for $x\in I_s$.
Let us consider the map $\xi:[\underline{g},+\infty)\to\R_{\geq 0}$ given by
\begin{equation}\label{def:xi}
\xi(s):=Leb(I\setminus I_s).
\end{equation}
Denote by $F_g:\R_{\geq 0}\to\R_{\geq 0}$ the \emph{tail distribution function} of $g$, i.e.\
\[F_g(s):=Leb(\{x\in I:g(x)>s\})\text{ for }s\geq 0.\]
By definition,
\begin{equation}\label{eq:Fg}
\{x\in I:g(x)>s\}\subset I\setminus I_s\text{ and }F_g(s)\leq \xi(s)\text{ for } s\geq \underline{g}.
\end{equation}

\begin{lemma}\label{lem:xiXi}
Suppose that $\xi:[\underline{g},+\infty)\to\R_{\geq 0}$ is decreasing integrable and of class $C^1$ map
with $\lim_{s\to+\infty}s\xi(s)=0$. Let us consider $\Xi:[\underline{g},+\infty)\to\R_{\geq 0}$ be given by $\Xi(s)=\int_s^{+\infty}\xi(t)\,dt$ for $s\geq \underline{g}$. Then for every $s>0$ and $c\geq \underline{g}$ we have
\begin{equation}\label{ineq:A}
Leb(I^g\setminus A_c^s)\leq s\xi(c)+2c\xi(c)+\Xi(c).
\end{equation}
\end{lemma}

\begin{proof}
By the definition of $\Xi$ and \eqref{eq:Fg}, using integration by part we have
\[\int_{\{x\in I:g(x)\geq c\}}g(x)dx=-\int_c^{\infty}t\,dF_g(t)=cF_g(c)+\int_c^{\infty}F_g(t)\,dt\leq c\xi(c)+\Xi(c). \]
Therefore
\[\int_{I\setminus I_c}g(x)dx\leq \int_{\{x\in I:g(x)\geq c\}}g(x)dx+\int_{\{x\in I\setminus I_c:g(x)\leq c\}}g(x)dx\leq 2c\xi(c)+\Xi(c). \]
It follows that for every $c\geq \underline{g}$ and $s\geq 0$ we have
\begin{equation*}
Leb(I^g\setminus A_c^s)\leq \int_{I\setminus I_c}g(x)dx+s Leb(I\setminus I_c)=s\xi(c)+2c\xi(c)+\Xi(c),
\end{equation*}
which completes the proof
\end{proof}

\begin{remark}
Note that, by definition, $\Xi$ is a decreasing $C^2$-map and $\displaystyle\lim_{s\to+\infty}\Xi(s)=0$.
\end{remark}

\begin{remark}\label{rem:logg}
Suppose that the roof function $g\in {\ol}(\sqcup_{\alpha\in \mathcal{A}}
I_{\alpha})$. Then there exist two positive constants
$C,b>0$ such that for every $s\geq \underline{g}$ we have
\[g(x)\leq s\text{ for all }x\in\bigcup_{\alpha\in\mathcal A}[l_\alpha+Ce^{-bs},r_\alpha-Ce^{-bs}].\]
Let us define the following sets (corresponding tail level sets):
\[I_s:=\bigcup_{\alpha\in\mathcal A}[l_\alpha+Ce^{-bs},r_\alpha-Ce^{-bs}]\text{ for any }s\geq \underline{g}.\]
Then $\xi(s)=dCe^{-bs}$ and  $\Xi(s)=(dC/b)e^{-bs}$,
so they satisfy the assumptions of Lemma~\ref{lem:xiXi}. In view of Lemma~\ref{lem:xiXi}, taking $c(s)=\tfrac{a}{b}\log s$ for some $a>1$ we have
\begin{equation}\label{eq:lebc}
Leb(I^g\setminus A_{c(s)}^s)\leq sdCs^{-a}+2\frac{a}{b}(\log s)dCs^{-a}+\frac{dC}{b}s^{-a}=O(s^{-(a-1)}),
\end{equation}
so the measure of $I^g\setminus A_{c(s)}^s$ decays with the polynomial speed.
\end{remark}

\subsubsection{Estimates of integrals and tails}\label{integralsestaimtes}
We can now combine the results on the two previous subsections, i.e.~the reduction of integrals along the flow to Birkhoff sums (Lemma~\ref{lem:int0s}) and the decomposition of Birkhoff sums into special Birkhoff sums (Propostion~\ref{twdevpre1}), to get the following estimate of ergodic integrals in terms of special Birkhoff sums:

\begin{proposition}\label{thm:ints}
Let $\eta:\R_{\geq 0}\to[\underline{g},+\infty)$ be an increasing $C^1$-map.
Let $f:I^g\to\R$ be a measurable bounded map. Then, for every $s\in \R_{\geq 0}$, %there exists a set $E_s\subset I^g$ such that
\begin{equation}\label{ineq:A1}
Leb(I^g\setminus A^s_{\eta(s)})\leq s\ \xi(\eta(s))+2\eta(s)\xi(\eta(s))+\Xi(\eta(s))
\end{equation}
(where $\xi(\, \cdot \, )$ is defined by \eqref{def:xi} and $\Xi(\, \cdot\, )$ is given by Lemma~\ref{lem:xiXi}),
and for every $(x,r)\in A^s_{\eta(s)}$,
\begin{equation}\label{eq:ints2}
\Big|\int_{0}^sf( T^g_t(x,r))dt\Big|\leq  2\sum_{k\geq 0}
\|Z(k+1)\|\|S(k)\varphi_f\|_{L^{\infty}(I^{(k)}(\eta(s)))}+2\|f\|_{L^\infty}\eta^2(s)
\end{equation}
\end{proposition}
\begin{proof}
%For $\eta:\R_{\geq 0}\to[\underline{g},+\infty)$ as in the assumptions, let us define $E_s:= A^s_{\eta(s)}$, where  $A^s_{\eta(s)}$ denotes the level set defined in \eqref{A:def}. Then t
The result follows by combining Lemma~\ref{lem:int0s}, Proposition~\ref{twdevpre1} and Lemma~\ref{lem:xiXi} with $c=\eta(s)$.
\end{proof}

\subsection{Deviation spectrum  and asymptotic behaviour of ergodic integrals}\label{devspectrum:sec}
We present in
 this section the proof of { Theorem~\ref{mainthm:BSN} and  the first part of Theorem~\ref{mainthm:BS}, namely the existence of the asymptotic spectrum for ergodic integrals both in the minimal and non-minimal case}.   We first define, in \S~\ref{sec:cocycles1}, the cocycles that will govern the asymptotic behaviour of the ergodic integrals. { Notice that, since we are proving at the same time the existence of the expansions in Theorems~\ref{mainthm:BS} and~\ref{mainthm:BSN},  we will define cocycles $u_\sigma$ parametrized by $\sigma \in Fix(\psi_\R)\cap M'$ also when considering the restriction of a typical $\psi_\R\in \mathcal{U}_{\neg min}$ to a minimal component $M'$ (even if these do not appear explicitely in the statement of  Theorem~\ref{mainthm:BSN}, where they are absorbed in $err(f, T, \cdot)$)}.  We then estimate the \emph{error term} and shows that it exhibit  subpolynomial deviations, see \S~\ref{sec:subpol} and  then prove in \S~\ref{sec:puredeviation} that the cocycles that we build have the desired \emph{pure power} behaviour, i.e.\ each has oscillations of the order of $T^{\nu_i}$ where $\nu_i$ is one of the $g$ distinct exponents in the power spectrum. Finally, in \S~\ref{sec:proofmain} we conclude the proof.

\subsubsection{Definition of the distributions and the cocycles.}\label{sec:cocycles1}
Assume that $T=T_{(\pi,\lambda)}$ satisfies the \ref{UDC}. Then, in view of the Oseledets genericity property \eqref{def:O} of the \ref{UDC} condition (refer to Definition~\ref{def:UDC}) there exists  vectors
 $h_1,\ldots,h_g\in H(\pi^{(0)})$  such that
 \begin{equation}\label{eq:asshi}
\lim_{k\to+\infty}\frac{1}{k}\|Q(k)h_i\|=\lambda_i\text{ for }1\leq i\leq g.
\end{equation}
and furthermore $\operatorname{span}\{h_1,\ldots,h_g\}\oplus \Gamma^{(0)}_s=H(\pi^{(0)})$.
We will now use these vectors $h_i$ to define the distributions and the cocycles which appear in the asymptotic expansion.

\smallskip
\noindent{\it The distributions.} By Theorem~\ref{operatorcorrection} (in view of  Remark~\ref{rk:correctionF}) and Corollary~\ref{cor:corrected}
applied to $F:=\operatorname{span}\{h_1,\ldots,h_g\}$,
%$F$ to the subspace $F:=\operatorname{span}\{h_1,\ldots,h_g\}$),
there exists a bounded operator
$\mathfrak{h}:{\ol}(\sqcup_{\alpha\in \mathcal{A}} I_{\alpha})\to F$,
such that $\mathfrak{h}(h)=h$ for every $h\in F$ and  for every $\tau>0$ if $\varphi\in {\ol}(\sqcup_{\alpha\in \mathcal{A}} I_{\alpha})$ and $\mathfrak{h}(\varphi)=0$ then
\begin{equation}\label{eq:asszero}
\frac{\|S(k)\varphi\|_{L^1(I^{(k)})}}{|I^{(k)}|}=O(e^{\tau k}).
\end{equation}
\noindent Let $d_i:{\ol}(\sqcup_{\alpha\in \mathcal{A}} I_{\alpha})\to\R$, $i=1,\ldots,g$ be bounded operators  such that
\begin{equation}\label{diexp}\mathfrak{h}(\varphi)=\sum_{i=1}^g d_i(\varphi)h_i\text{ for every }\varphi\in {\ol}(\sqcup_{\alpha\in \mathcal{A}} I_{\alpha}).\end{equation}
We can then define bounded operators $D_i:C^{2+\epsilon}(M)\to\R$, for $i=1,\ldots,g$, by  using the map  $f\mapsto \varphi_f$
(see Proposition~\ref{prop:oper} for its basic properties)
 which associates to an observable $f : M \to \R$ the cocycle which arise in the skew-product representation of the Poincar{\'e} map described in \S~\ref{sec:reductionsp} and
setting
\begin{equation}\label{defdi} D_i(f):=d_i(\varphi_f), \qquad 1\leq i\leq g.
\end{equation}
We will prove in \S~\ref{sec:proofmain} that these are the distributions which enter in the asymptotic expansion.

\smallskip
\noindent{\it The power growth cocycles.} To construct the cocycles we exploit the following Lemma, proved in \cite{Co-Fr}.
\begin{lemma}[Lemma~7.4 in \cite{Co-Fr}]\label{lemma:findf}
For every $h\in  H(\pi)$ there exists a $\mathcal{C}^\infty$-function $f : M \to \R$, which vanishes on a neighborhood %of any fixed point
$\mathrm{Fix}(\psi_\R)$, such that $\varphi_f = h$.
\end{lemma}
\noindent Let $f_i\in C^{\infty}(M)$ be the observable such that $\varphi_{f_i}=h_i$, given by Lemma~\ref{lemma:findf} applied to $h=h_i$.
Let us now define
\[
u_i(T,x):=\int_0^Tf_i(\psi_s(x))\,ds, \qquad \text{for}\ 1\leq i\leq g.
\]
%As $\mathfrak{h}(\varphi_{f_i})=\mathfrak{h}(h_i)=h_i$, by the definition of $d_j$ and $D_j$, we have
%\begin{equation}\label{eq:Dij}
%D_j(f_i)=d_j(\varphi_{f_i})=\delta_{ij}.
%\end{equation}

\smallskip
\noindent{\it The singular cocycles.}
For every $\sigma\in\mathrm{Fix}(\psi_\R)\cap M'$, to define $u_\sigma$,  let $\bar{\xi}_\sigma:M\to\R$ be any $C^\infty$-map which is equal to $1$ on an open
neighbourhood of $\sigma$ and equal to zero on an open
neighbourhood of all other fixed points. Let $\xi_\sigma:M\to\R$ be a $C^\infty$-map given by
\[\xi_\sigma:=\bar{\xi}_\sigma-\sum_{i=1}^g D_i(\bar{\xi}_\sigma)f_i.\]
Then, { since each $f_i$ given by Lemma~\ref{lemma:findf}  vanishes on a neighbourhood of $Fix(\varphi_\R)$ (see Lemma~\ref{lemma:findf}),} $\xi_\sigma$ is also equal to $1$ on an open
neighbourhood of $\sigma$ and equal to zero on an open
neighbourhood of all other fixed points. Moreover, {
by linearity of the operator $\mathfrak{h}$, the definition \eqref{defdi} of $D_i$ and \eqref{diexp}}, %for every $1\leq j\leq g$ we have
\begin{equation}\label{eq:hxi}
\mathfrak{h}(\varphi_{\xi_\sigma})=\mathfrak{h}(\varphi_{\bar{\xi}_\sigma})-\sum_{i=1}^g D_i(\bar{\xi}_\sigma)\mathfrak{h}(\varphi_{f_i})
=\mathfrak{h}(\varphi_{\bar{\xi}_\sigma})-\sum_{i=1}^g d_i(\varphi_{\bar{\xi}_\sigma})h_i=0.
\end{equation}
Finally, the cocycle $u_\sigma:\R\times M\to\R$ is defined by
\[u_\sigma(T,x):=\int_0^T\xi_\sigma(\psi_s(x))\,ds.\]

We will show in \S~\ref{sec:puredeviation} that each $u_i$, in view of \eqref{eq:asshi}, displays the desired deviation behaviour and in \S~\ref{sec:proofmain} that they are indeed the desired asymptotic cocycles. We first estimate the error term though.

\subsubsection{Subpolynomial deviation case}\label{sec:subpol}
The following Proposition provides subpolynomial estimates for the growth of \emph{corrected} ergodic integrals  (in light of Corollary~\ref{cor:corrected}) and will be used in \S~\ref{sec:proofmain} to control the \emph{error term} in the asymptotic expansion.
\begin{proposition}[Subpolynomial deviation]\label{thm:specflowzeroexp}
Suppose that the IET $T:I\to I$ satisfies the \ref{UDC}.
Assume that $g,\varphi_f\in {\ol}(\sqcup_{\alpha\in \mathcal{A}}
I_{\alpha})$ and
\[\frac{1}{|I^{(k)}|}\|S(k)\varphi_f\|_{L^1(I^{(k)})}=O(e^{\tau k})\text{ for every }\tau >0.\]
Then
for a.e.\ $(x,r)\in I^g$ we have
\begin{equation}\label{eq:Birk0}
\limsup_{s\to+\infty}\frac{\log|\int_{0}^sf(T^g_t(x,r))dt|}{\log s}\leq 0.
\end{equation}
Moreover, for every $p\geq 1$ we have
\begin{equation}\label{eq:BirkL1}
\limsup_{s\to+\infty}\frac{\log\|\int_{0}^sf\circ T^g_tdt\|_{L^p(I^g)}}{\log s}\leq 0.
\end{equation}
\end{proposition}
\begin{proof}
As we have already seen in Remark~\ref{rem:logg}, there exist two positive constants
$C,b>0$ such that for every $s\geq \underline{g}$ we have
\[g(x)\leq s\text{ for all }x\in I_s:=\bigcup_{\alpha\in\mathcal A}[l_\alpha+Ce^{-bs},r_\alpha-Ce^{-bs}].\]
Then
\[\xi(s)=Leb(I\setminus I_s)=dCe^{-bs}\quad\text{and}\quad\Xi(s)=(dC/b)e^{-bs}.\]

\noindent Take any $a>2$ and set $\eta(s)=\frac{a}{b}\log s$. By the description of $C,b>0$, we have $[0,Ce^{-b\eta(s)}]\subset I\setminus I_{\eta(s)}$.
Hence, if $|I^{(k)}|\leq Ce^{-b\eta(s)}=C/s^{a}$ then $I^{(k)}(\eta(s))=\emptyset$.
%Since \[\|Q(k)\||I^{(k)}|\leq\kappa\sum_{\alpha\in\mathcal A}Q_\alpha(k)|I_\alpha^{(k)}|=\kappa\]
 By condition \eqref{def:UDC-c} and \eqref{eq:lg}, it follows that
\[I^{(k)}(\eta(s))\neq\emptyset\Rightarrow |I^{(k)}|>C/s^{a}\Rightarrow
\|Q(k)\|<\kappa s^{a}/C \Rightarrow k\leq \frac{a}{\lambda_1 }\log (C's).\]
Moreover, if $x\in I^{(k)}(\eta(s))\cap I^{(k)}_\alpha$ then
\[x\in[l^{(k)}_\alpha+Ce^{-b\eta(s)},r^{(k)}_\alpha-Ce^{-b\eta(s)}]=[l^{(k)}_\alpha+C/s^a,r^{(k)}_\alpha-C/s^a].\]
In view of \eqref{eq:valuephi1}, \eqref{nave1} and \eqref{def:UDC-c}, it follows that for every $x\in I^{(k)}(\eta(s))$,
\begin{align*}
|(S(k)\varphi)(x)|&\leq 2\kappa\frac{\|S(k)\varphi\|_{L^1(I^{(k)})}}{|I^{(k)}|}+\lv(S(k)\varphi)(1+\log(|I^{(k)}|s^{a}/C))\\
& = O(e^{\tau k})+O(\log s\log\|Q(k)\|)= O(e^{\tau k})+O(k\log s).
\end{align*}
Therefore, by \eqref{eq:ints2}, for every $(x,r)\in A^s_{\eta(s)}$ we have
\begin{align*}
\Big|\int_{0}^sf(T^g_t(x,r))dt\Big|&\leq O(\log^2 s)+O\Big(\sum_{0\leq k\leq \frac{a}{\lambda_1 }\log (C's)}\|Z(k+1)\|e^{\tau k}\Big)
+O\Big(\log s\sum_{0\leq k\leq \frac{a}{\lambda_1 }\log (C's)}\|Z(k+1)\|k\Big)\\
&\leq O(\log^2 s)+O(s^{2a\tau/\lambda_1 })+O(s^{a\tau/\lambda_1 }\log^2 s )=O(s^{2a\tau/\lambda_1 }).
\end{align*}
Moreover, by \eqref{eq:lebc}, we have $Leb(I^g\setminus A_{\eta(s)}^s)=O(1/s^{a-1})$
with
$a-1>1$. Therefore, for every $\tau>0$ and $a>2$ there exists $C_{\tau,a}>0$ such that
for every $s>0$ we have
\begin{equation}\label{eq:maxineq}
Leb\Big\{(x,r)\in I^g:\Big|\int_{0}^sf(T^g_t(x,r))dt\Big|>C_{\tau,a} s^{2\tau a/\lambda_1 }\Big\}\leq Leb(I^g\setminus A_{\eta(s)}^s)< \frac{C_{\tau,a}}{s^{a-1}}.
\end{equation}
It follows that for a.e.\ $(x,r)\in I^g$ we have
\[\limsup_{s\to+\infty}\frac{\log|\int_{0}^sf(T^g_t(x,r))dt|}{\log s}\leq 2\tau a/\lambda_1.\]
This gives \eqref{eq:Birk0}.

\smallskip
\noindent Finally, the inequality \eqref{eq:BirkL1} follows also directly from \eqref{eq:maxineq}. Indeed, if $a\geq {p+1}$, then
\begin{align*}
\Big\|\int_{0}^sf\circ T^g_tdt\Big\|^p_{L^p(I^g)}&\leq \int_{A_{\eta(s)}^s}\Big|\int_{0}^sf\circ T^g_t(x,r)dt\Big|^p\,dx\,dr
+Leb(I^g\setminus A_{\eta(s)}^s)s^p\|f\|_{L^\infty}^p\\
&=O(s^{2pa\tau/\lambda_1 })+O(s^{p+1-a})=O(s^{2pa\tau/\lambda_1 }).
\end{align*}
\end{proof}

\begin{corollary}\label{cor:zeroint}
Suppose that $T$ is an IET satisfying the \ref{UDC} and $\varphi\in{\ol}(\sqcup_{\alpha\in \mathcal{A}}
I_{\alpha})$. If $\mathfrak{h}(\varphi)=0$ then $\int_I\varphi(x)\,dx=0$.
\end{corollary}

\begin{proof}
Let us consider any roof function $g:I\to\R_{>0}$ such that $\varphi\in{\ol}(\sqcup_{\alpha\in \mathcal{A}}
I_{\alpha})$ and $|\varphi(x)|\leq g(x)$ for $x\in I$. Let $f:I^g\to\R$ be given by
$f(x,r)=\varphi(x)/g(x)$ for $(x,r)\in I^g$. Then $f$ is bounded and $\varphi_f=\varphi$.
In view of Theorem~\ref{thm:specflowzeroexp} and the ergodicity of $T$, for every $0<\tau<1$,
for a.e.\ $x\in I$ and a.e.\ $r\in[0,\underline{g}]$ we have
\[g^{(n)}(x)=O(n)\ \text{ and }\ \int_0^{g^{(n)}(x)}f(T^g_t(x,r))\,dt=O((g^{(n)}(x))^{\tau}).\]
As
\begin{align*}\Big|\varphi^{(n)}(x)-\int_0^{g^{(n)}(x)}f(T^g_t(x,r))\,dt\Big|
&=\Big|\int_0^{g^{(n)}(x)}f(T^g_t(x,0))\,dt-\int_0^{g^{(n)}(x)}f(T^g_t(x,r))\,dt\Big|\\
&\leq \Big|\int_0^{r}f(T^g_t(x,0))\,dt\Big|+\Big|\int_0^{r}f(T^g_t(T^nx,0))\,dt\Big|\leq 2\underline{g}\|f\|_{\sup},
\end{align*}
it follows that $\varphi^{(n)}(x)=O(n^\tau)$. On the other hand, for a.e.\ $x\in I$ we have $\varphi^{(n)}(x)/n\to \int_I\varphi(x)\,dx$.
This gives $\int_I\varphi(x)\,dx=0$.
\end{proof}

\subsubsection{Pure power deviation case}\label{sec:puredeviation}
We consider first a function $f$ such that $\varphi_f= h$, where $h$ has exponential growth rate $\lambda$.
\begin{proposition}[Pure deviation]\label{thm:specflownonzero}
Suppose that the IET $T:I\to I$ satisfies the \ref{UDC}.
Assume that the roof function  $g\in {\ol}(\sqcup_{\alpha\in \mathcal{A}}
I_{\alpha})$ and $f:I^g\to \R$ is a bounded function such that there exists $K>0$ for which $f(x,r)=0$ for $r\geq K$ { and $\varphi_f\in L^{\infty}(I)$.
Suppose that for some $\lambda\geq 0$ we have
\[\limsup_{k\to+\infty}\frac{\log\|S(k)(\varphi_f)\|_{L^{\infty}(I^{(k)})}}{k}\leq\lambda.\]
Then
\begin{equation}\label{eq:Birkh0}
\limsup_{s\to+\infty}\frac{\log\|\int_{0}^sf\circ T^g_t\,dt\|_{L^\infty}}{\log s}\leq \frac{\lambda}{\lambda_1}.
\end{equation}
If additionally $\varphi_f= h=(h_\alpha)_{\alpha\in\mathcal{A}}\in H(\pi)$, $\lambda>0$ and
\[\lim_{k\to+\infty}\frac{\log\|Q(k)h\|}{k}=\lambda,\]
then
\begin{equation}\label{eq:Birkh1}
\limsup_{s\to+\infty}\frac{\log\|\int_{0}^sf\circ T^g_t\,dt\|_{L^{\infty}}}{\log s}= \frac{\lambda}{\lambda_1}.
\end{equation}
}
\end{proposition}
\begin{proof}
Let us consider the trimmed roof function $g_K:I\to[0,K]$, $g_K(x)=\min\{g(x),K\}$.
Taking $\eta=K$ and $I_{\eta(s)}=I_K=I$ we have $A^s_{\eta(s)}=I^{g_K}$.
Note that, by assumption, the map $\varphi_f$ does not change after passing to the trimmed roof function.
 In view of \eqref{eq:ints2},
for every regular point $(x,r)\in I^{g_K}$ we have
\begin{equation*}
\Big|\int_{0}^sf( T^{g_K}_t(x,r))dt\Big|\leq  2\sum_{k=0}^{m(n_K(x,r,s))}
\|Z(k+1)\|\|S(k)\varphi_f\|_{L^{\infty}(I^{(k)})}+2\|f\|_{L^\infty}K^2,
\end{equation*}
where $n_K(x,r,s)$ is defined by \eqref{def:n} for the roof $g_K$. Then
\[0\leq n_K(x,r,s)\leq n(x,r,s)\leq s/\underline{g}+1.\]
By assumption, for every $\tau>0$ we have
\[\|S(k)\varphi_f\|_{L^{\infty}(I^{(k)})}=O(e^{(\lambda+\tau)k}).\]
Moreover, by \eqref{eq:mnrel},
\[e^{\lambda_1m(n_K(x,r,s))}=O(n_K(x,r,s)^{1+\tau})=O(s^{1+\tau}).\]
Therefore, by \eqref{eq:zk1}, it follows that
\begin{align*}
\Big|\int_{0}^sf( T^{g_K}_t(x,r))dt\Big|&\leq  O\Big(\sum_{k=0}^{m(n_K(x,r,s))}
\|Z(k+1)\|\|S(k)\varphi_f\|_{L^{\infty}(I^{(k)})}+\|f\|_{L^\infty}K^2\Big)\\
&=
O\Big(\sum_{k=0}^{m(n_K(x,r,s))}
e^{(\lambda+2\tau)k}+\|f\|_{L^\infty}K^2\Big)\\
&=O\big(e^{(\lambda+2\tau)m(n_K(x,r,s))}+\|f\|_{L^\infty}K^2\big)\\
&=O(s^{(\lambda+2\tau)(1+\tau)}).
\end{align*}
By assumption and the definition of $g_K$, for every regular $(x,r)\in I^g$ and $s>0$ there exists $0\leq s'=s'(x,r,s)\leq s$ such that
\[\Big|\int_{0}^sf( T^{g}_t(x,r))dt\Big|=\Big|\int_{0}^{s'}f(T^{g_K}_t(x,r))dt\Big|
=O\big({(s')}^{(\lambda+2\tau)(1+\tau)}\big)=O(s^{(\lambda+2\tau)(1+\tau)}).\]
This gives \eqref{eq:Birkh0} and proves one inequality (namely the upper bound) in \eqref{eq:Birkh1}.

\smallskip
To prove the inverse inequality and therefore \eqref{eq:Birkh1}, note that for every $x\in I^{(k)}_\alpha$ we have
\begin{align*}\int_0^{S(k)g(x)}f(T^g_t(x,0))\,dt&=\int_0^{g^{(Q_\alpha(k))}(x)}f(T^g_t(x,0))\,dt=\varphi_f^{(Q_\alpha(k))}(x)
=S(k)\varphi_f(x)=(Q(k)h)_\alpha.
\end{align*}
Moreover, by assumption, for every $\tau>0$ there exists $c>0$ such that for every $k\geq 0$ we have
\[\sum_{\alpha\in\mathcal A}|(Q(k)h)_\alpha|=\|Q(k)h\|\geq c e^{\lambda(1-\tau)k}.\]
As $g$ is positive, by \eqref{def:UDC-d}, \eqref{neq:dli} and \eqref{def:UDC-c}, we have
\begin{align*}
m(S(k)g,I^{(k)}_\alpha)&\leq \frac{|I^{(k)}|}{|I^{(k)}_\alpha|}m(S(k)g,I^{(k)})\leq\frac{\kappa m(g,I)}{|I^{(k)}|}
\leq\kappa m(g,I)\|Q(k)\|\leq \kappa m(g,I)Ce^{\lambda_1(1+\tau)k}.
\end{align*}
For every $k\geq 0$ choose $\alpha\in\mathcal A$ such that $|(Q(k)h)_\alpha|=\tfrac{1}{d}\|Q(k)h\|$ and then we take any
$x^{(k)}\in I^{(k)}_\alpha$ such that $s_k:=S(k)g(x^{(k)})\leq \kappa m(g,I)Ce^{\lambda_1(1+\tau)k}$. Then
\begin{align*}
\Big\|\int_0^{s_k}f\circ T^g_t\,dt\Big\|_{L^\infty}&\geq \Big|\int_0^{S(k)g(x^{(k)})}f(T^g_t(x^{(k)},0))\,dt\Big|=|(Q(k)h)_\alpha|\\
&=\frac{1}{d}\|Q(k)h\|\geq \frac{c}{d} e^{\lambda(1-\tau)k}
\geq\frac{c}{d(\kappa m(g,I)C)^{\frac{\lambda}{\lambda_1}\frac{1-\tau}{1+\tau}}} (s_k)^{\frac{\lambda}{\lambda_1}\frac{1-\tau}{1+\tau}k}.
\end{align*}
It follows that for every $\tau>0$ we have
\[\limsup_{s\to+\infty}\frac{\log\|\int_{0}^sf\circ T^g_t\,dt\|_{L^{\infty}}}{\log s}\geq \frac{\lambda}{\lambda_1}\frac{1-\tau}{1+\tau},\]
which gives \eqref{eq:Birkh1}.
\end{proof}

{ To have uniform control over the asymptotics of the error growth, we also need the following Corollary.
\begin{corollary}\label{cor:passac}
Let  $T:I\to I$ is an IET satisfying the \ref{UDC} and $g\in {\ol}(\sqcup_{\alpha\in \mathcal{A}}
I_{\alpha})$ be a roof function.  Suppose that $f:I^g\to \R$ is a bounded function such that  $\varphi_f, \varphi_{|f|}\in L^{\infty}(I)$.
Then for every $\lambda\geq 0$,
\begin{equation}\label{eq:implexp}
\limsup_{k\to+\infty}\frac{\log\|S(k)(\varphi_f)\|_{L^{\infty}(I^{(k)})}}{k}\leq\lambda\
\Longrightarrow\
\limsup_{s\to+\infty}\frac{\log\|\int_{0}^sf\circ T^g_t\,dt\|_{L^\infty}}{\log s}\leq \frac{\lambda}{\lambda_1}.
\end{equation}
\end{corollary}

\begin{proof}
For any $K>0$ let us consider the bounded map $f_K:I^g\to\R$ given by
\[f_K(x,r)=\left\{
\begin{array}{ccl}
f(x,r)&\text{if}& g(x)\leq K\\
\varphi_f(x)/K&\text{if}& g(x)> K\text{ and } r\leq g(x)\\
0&\text{if}& g(x)> K\text{ and } r> g(x).
\end{array}
\right.\]
Then $f_K$ satisfies the assumptions of the first part of Proposition~\ref{thm:specflownonzero} and $\varphi_{f_K}=\varphi_f$. Hence
 \begin{equation}\label{eq:expfK}
 \limsup_{s\to+\infty}\frac{\log\|\int_{0}^sf_K\circ T^g_t\,dt\|_{L^\infty}}{\log s}\leq \frac{\lambda}{\lambda_1}.
 \end{equation}
Note that for every $x\in I$ in the interior of exchanged intervals and any pair $0\leq r_1<r_2\leq g(x)$ we have
\begin{gather*}|\int_{r_1}^{r_2}f(x,r)\,dr|\leq \int_{r_1}^{r_2}|f(x,r)|\,dr\leq \int_{0}^{g(x)}|f(x,r)|\,dr=\varphi_{|f|}(x)\leq\|\varphi_{|f|}\|_{\sup},\\
|\int_{r_1}^{r_2}f_K(x,r)\,dr|\leq \varphi_{|f_K|}(x)| \leq  \varphi_{|f|}(x)\leq\|\varphi_{|f|}\|_{\sup}.
\end{gather*}
As
\[\int_{0}^{g(x)}f_K(x,r)\,dr=\varphi_{f_K}(x)=\varphi_f(x)= \int_{0}^{g(x)}f(x,r)\,dr,\]
it follows that for every regular point $(x,r)\in I^g$ and any $s>0$ we have
\[\Big|\int_{0}^{s}f(T^g_t(x,r))\,dt-\int_{0}^{s}f_K(T^g_t(x,r))\,dt\Big|\leq 4 \|\varphi_{|f|}\|_{\sup}.\]
Together with \eqref{eq:expfK} this yields \eqref{eq:implexp}.
\end{proof}
}
\subsubsection{Power deviation spectrum}\label{sec:proofmain}
Combining the results in the two previous subsections, we can now prove the full deviation spectrum result { stated in Theorem~\ref{mainthm:BSN} as well as the existence of the asymptotic expansion in Theorem~\ref{mainthm:BS}.
}

\begin{proof}[{Proof of Theorem~\ref{mainthm:BSN}  and of the first part of Theorem~\ref{mainthm:BS}}]
Let $D_i, 1\leq i\leq g$, $u_i, 1\leq i\leq g$ { and $u_\sigma$, $\sigma\in\mathrm{Fix}(\psi_\R)\cap M'$}, be respectively the distributions and the cocycles defined in \S~\ref{sec:cocycles1}.  One can see that, for each $1\leq i\leq g$, $u_i$ displays the desired power behaviour \eqref{eq:explambdai}, by the pure deviation Theorem~\ref{thm:specflownonzero} proved in \S~\ref{sec:puredeviation}, which can be applied to $f=f_i$ since  by construction  $\varphi_{f_i}=h_i$ and  $h_i$ has exponential growth rate $\lambda_i$, see \eqref{eq:asshi}.

\smallskip
\noindent {\it  The error term function.}
Let us consider $f_e\in C^{2+\epsilon}(M)$ given by
\[f_e:= f-\sum_{i=1}^gD_i(f)f_i.\]
By the definition of $f_i$, $i=1,\ldots,g$,
\begin{equation}\label{eq:sumfe}
f_e(\sigma)=f(\sigma)\text{ for every }\sigma\in\mathrm{Fix}(\psi_\R)\cap M'.
\end{equation}
Then we set
\[err(f,T,x):=\int_0^Tf_e(\psi_s(x))\,ds.\]
Let $\varphi_{f_e}$ be the cocycle associated to $f_e$ (refer to \S~\ref{sec:reductionsp}).
We can then check that $\mathfrak{h}(\varphi_{f_e})=0$, since
\begin{align}\label{eq:hfe}
\begin{split}
\mathfrak{h}(\varphi_{f_e})=\mathfrak{h}(\varphi_{f})-\sum_{i=1}^gD_i(f)\mathfrak{h}(\varphi_{f_i})
=\mathfrak{h}(\varphi_{f})-\sum_{i=1}^gD_i(f)\mathfrak{h}(h_i)
=\mathfrak{h}(\varphi_{f})-\sum_{i=1}^gd_i(\varphi_{f})h_i=0.
\end{split}
\end{align}

We now show that for every non-zero $\xi\in C^{2+\epsilon}(M)$ such that $\mathfrak{h}(\varphi_\xi)=0$
we have
\begin{gather}
\label{eq:expxi}
\limsup_{T\to+\infty}\frac{\log|\int_0^T\xi(\psi_t(x))\,dt|}{\log T}=0\text{ for a.e. }x\in M', \quad \limsup_{T\to+\infty}\frac{\log\|\int_0^T\xi\circ\psi_t\,dt\|_{L^{p}(M')}}{\log T}=0.
\end{gather}
As $\mathfrak{h}(\varphi_\xi)=0$, in view of Corollary~\ref{cor:corrected},
we can apply the subpolynomial deviation  Theorem~\ref{thm:specflowzeroexp} to $f=\xi$ and prove both inequalities $\leq$ in \eqref{eq:expxi}.

%\eqref{eq:expzero1} and \eqref{eq:expzero2}.

\medskip

\noindent {\it Almost everywhere error estimates.} Suppose now that the left equality in \eqref{eq:expxi} does not hold. Then there exists a subset $B\subset M'$ with positive area such that
\[\lim_{T\to+\infty}\int_0^T\xi(\psi_t(x))\,dt=0\text{ for all }x\in B.\]
By the ergodicity of the flow, for $\mu$-a.e.\ $x\in M'$, the limit
\[\zeta(x)=\lim_{T\to+\infty}\int_0^T\xi(\psi_t(x))\,dt\text{ exists.}\]
Then $\zeta:M'\to\R$ is a measurable map such that $\zeta(x)=0$ for any $x\in B$ and
\begin{equation}\label{eq:zeta}
\zeta(x)-\zeta(\psi_sx)=\int_0^s\xi(\psi_t(x))\,dt\text{ for every }s>0\text{ and a.e. }x\in M'.
\end{equation}
Note that $\zeta\equiv 0$. Indeed, by definition and \eqref{eq:zeta}, for a.e.\ $x\in M'$ we have $\lim_{s\to+\infty}\zeta(\psi_sx)=0$.
Since $\psi_\R$ is ergodic, this gives $\zeta\equiv 0$. Therefore,
\begin{equation*}
\frac{1}{s}\int_0^s\xi(\psi_t(x))\,dt=\frac{1}{s}(\zeta(x)-\zeta(\psi_sx))=0\text{ for every }s>0\text{ and a.e. }x\in M'.
\end{equation*}
As $\xi$ is continuous, it follows that for a.e.\ $x\in M'$ we have
\[\xi(x)=\lim_{s\to 0}\frac{1}{s}\int_0^s\xi(\psi_t(x))\,dt=0,\]
contrary to the assumption  $\xi$ is non-zero.

\medskip

\noindent {\it Error estimates in $L^p$ norm.} Suppose now
 that the right equality in \eqref{eq:expxi} does not hold. Then
\[\lim_{T\to+\infty}\int_0^T\xi\circ\psi_t\,dt=0\text{ in }L^p.\]
Hence, for every $s>0$ we have
\[\int_0^s\xi\circ\psi_t\,dt=\lim_{T\to+\infty}\int_0^{T+s}\xi\circ\psi_t\,dt-\lim_{T\to+\infty}\int_0^T\xi\circ\psi_t\circ\psi_s\,dt=0\]
in $L^p$. It follows that $\frac{1}{s}\int_0^s\xi(\psi_t(x))\,dt=0$ for every $s>0$ and a.e.\ $x\in M'$. The final contradiction argument is the same as above. This completes the proof of \eqref{eq:expxi}.

\medskip

\noindent In view of \eqref{eq:hxi} and \eqref{eq:hfe}, $\mathfrak{h}(\varphi_{\xi_\sigma})=0$ and $\mathfrak{h}(\varphi_{f_e})=0$, so we can apply
\eqref{eq:expxi} to $\xi=\xi_\sigma$ and $\xi=f_e$. {This yields \eqref{eq:expusigma} in in  Theorem~\ref{mainthm:BS} as well as \eqref{eq:expzero1} and \eqref{eq:expzero2} in  Theorem~\ref{mainthm:BSN}.}

\medskip

\noindent \emph{Uniform estimates of $err_b$.}
Let us consider $f_{eb}\in C^{2+\epsilon}(M)$ given by
\begin{equation}\label{deffeb} f_{eb}=f_e-\sum_{\sigma\in\mathrm{Fix}(\psi_\R)\cap M'}f(\sigma)\xi_\sigma.\end{equation}
Then
\begin{align*}\int_0^T f_{eb}(\psi_t x)\,dt&=\int_0^T f_e(\psi_t x)\,dt-\sum_{\sigma\in\mathrm{Fix}(\psi_\R)\cap M'}f(\sigma)\int_0^T \xi_\sigma(\psi_t x)\,dt\\
&=err(f,T,x)-\sum_{\sigma\in\mathrm{Fix}(\psi_\R)\cap M'}f(\sigma)u_\sigma(T,x)=err_b(f,T,x).
\end{align*}
Since $f_e(\sigma)=f(\sigma)$ and $\xi_\sigma(\sigma')=\delta_{\sigma,\sigma'}$, for every $\sigma\in\mathrm{Fix}(\psi_\R)\cap M'$ we have
\begin{equation}\label{eq:ebzero}
f_{eb}(\sigma)=f_e(\sigma)-\sum_{\sigma'\in\mathrm{Fix}(\psi_\R)\cap M'}f(\sigma')\xi_{\sigma'}(\sigma)=0.
\end{equation}
In view of Proposition~\ref{prop:oper}, $\varphi_{f_{be}}\in {\ac}(\sqcup_{\alpha\in \mathcal{A}} I_{\alpha})$. As $\mathfrak{h}(\varphi_{\xi_\sigma})=0$ and $\mathfrak{h}(\varphi_{f_e})=0$, we also have
\begin{equation}\label{eq:heb}
\mathfrak{h}(\varphi_{f_{eb}})=\mathfrak{h}(\varphi_{f_e})-\sum_{\sigma\in\mathrm{Fix}(\psi_\R)\cap M'}f(\sigma)\mathfrak{h}(\varphi_{\xi_\sigma})=0.
\end{equation}
In view of \eqref{eq:supL1bv}, {the property \eqref{def:UDC-d} of the \ref{UDC}}, \eqref{navar} and Corollary~\ref{cor:corrected}, it follows that for every $\tau>0$ we have
\begin{align*}
\|S(k)\varphi_{f_{eb}}\|_{\sup}&\leq \frac{|I^{(k)}|}{\min_{\alpha\in\mathcal{A}}|I_\alpha^{(k)}|}\frac{\|S(k)\varphi_{f_{eb}}\|_{L^1(I^{(k)})}}{|I^{(k)}|}+\var(S(k)\varphi_{f_{eb}})\\
&\leq \kappa\frac{\|S(k)\varphi_{f_{eb}}\|_{L^1(I^{(k)})}}{|I^{(k)}|}+\var(\varphi_{f_{eb}})=O(e^{\tau k}).
\end{align*}
By \eqref{eq:ebzero} and { Proposition~\ref{prop:oper} (see in particular property $(i)$), the map $\varphi_{|f|}:I\to\R$ is bounded. In view of Corollary~\ref{cor:passac}, this gives \eqref{eq:experrb}.}

{ This completes the proof of Theorem~\ref{mainthm:BSN} as well as the proof of the first part of Theorem~\ref{mainthm:BS}.}
\end{proof}
\noindent
The second part of the statement of  Theorem~\ref{mainthm:BS}, namely the equidistribution statement for the error term (which is a consequence of ergodicity) { and the uniform estimates on $err_b$}  will be proved at the end, in \S~\ref{sec:proofBSerror}.

\section{Ergodicity of extensions}\label{sec:ergskew}
The goal of this section is to prove  Main Theorem~\ref{main1} and complete the proof of  Main Theorem~\ref{mainthm:BS}. In view of the reduction explained in \S~\ref{sec:reductionsp} and the equivalence between ergodicity
of the extension $\Phi^f_\R$ on $M\times \R$ and of the skew product $T_{\varphi_f}$ on $ I\times \R$ obtained via a Poincar{\'e} first return, we treat first the case of skew products of this form. The main result on ergodicity of skew products is Theorem~\ref{thm:erg} stated in \S~\ref{sec:ergsp} below.
In \S~\ref{sec:ergodicitycriterium} we state the ergodicity criterium that we will use to prove it (see Proposition~\ref{prop:ergodicxi}). Theorem~\ref{thm:erg}  is then proved in \S~\ref{sec:ergodicitysp}. Finally, in \S~\ref{sec:reducibility} we prove  Main Theorem~\ref{main1}, by combining the ergodicity result for skew products with a discussion on reducibility.
%{ We should first of all state (either here or in the intro) a flow version of the ergodicity results... The assumptions that the flow flow $(\psi_t)_{t\in\R}$ on $M(\eta)=M$ (better to say explicitly that there is a unique minimal component and the flow is minimal...) are assumptions for the flow that is never mentioned in the section.. and in any case I would not put them in the title;  We can then explain that we reduce it to a statement on skew products of independent interest and state the result below... }\label{sec:erg}

%\begin{theorem}\label{thm:noncob}
%Suppose that $T:I\to I$ satisfies the \text{UDC}. Let $\varphi\in {\ol}(\sqcup_{\alpha\in \mathcal{A}} I_{\alpha})$ be a map such that $\mathcal{L}(\varphi)>0$.
%Then the cocyle $\varphi:I\to\R$ is not a coboudary.
%\end{theorem}

\subsection{Ergodicity of skew products over IETs with logarithmic singularities}\label{sec:ergsp}
We state in this section the ergodicity result for skew-products over IETs with cocycles with logarithmic singularities. We also show that the ergodicity result for locally Hamiltonian flows (Main Theorem~\ref{main1}) can be reduced to it.

\begin{theorem}[Ergodicity of skew-products with log-singularities over IETs]\label{thm:erg}
Suppose that $T:I\to I$ satisfies the \ref{SUDC}. Let $\varphi\in {\ol}(\sqcup_{\alpha\in \mathcal{A}} I_{\alpha})$ be a cocycle with logarithmic singularities of geometric type  so that
\[
\mathcal{L}(\varphi)>0,\qquad \mathcal{AS}(\varphi)=0,\qquad g'_{\varphi}\in {\ol}(\sqcup_{\alpha\in \mathcal{A}} I_{\alpha}).
\]
Assume furthermore that
%the boundary $\partial_{\pi}(\varphi)=0$ and that
$\varphi$ is \emph{corrected}, namely $\mathfrak{h}(\varphi)=0$.
Then the skew product $T_{\varphi}$ on $I\times \R$ is ergodic.
\end{theorem}
The proof of the Theorem will take most of the section, from \S~\ref{thm:erg} to the end.
%We first of all show that Theorem~\ref{main1} on ergodicity of locally Hamiltonian flows can be proved assuming Theorem~\ref{thm:erg}.
We first state the ergodicity criterium which will be exploited (see \S~\ref{sec:ergodicitycriterium})  and proceed with the proof, which will take \S~\ref{sec:ergodicitysp}.
\subsubsection{An ergodicity criterium}\label{sec:ergodicitycriterium}
We now formulate a quite classical criterium (Proposition~\ref{prop:ergodicxi}) for ergodicity of a special flow. It shows that one can deduce the existence of \emph{essential values} (a classical tool to prove ergodicity, see e.g.~\cite{Aa, Sch}) to the presence of \emph{rigidity sets} for the base transformation on which Birkhoff sums (up to the time which gives rigidity) are \emph{tight}. The criterium was in particular used (and proved) in  \cite{Fr-Ul}.  For simplicity in this section we constantly assume that $|I|=1$.

\smallskip
We first give the definition of \emph{rigidity sequence} for IETs (which are the base transformations in the special flow).
\begin{definition}[Rigidity sequences for an IET]
Let $T:I\to I$ be an IET. Let $(\Xi_n)_{n\geq 1}$ be a sequence of towers of intervals of the form $\Xi_n=\{T^i J_n:0\leq i<p_n\}$.
We say that $(\Xi_n)_{n\geq 1}$ is a \emph{rigid sequence of towers} if there exists a strictly increasing sequence $(q_n)_{n\geq 1}$,
called the \emph{rigidity sequence}, and $\delta>0$ such that
\begin{gather*}
Leb(\Xi_n)\geq \delta\text{ and }\sup_{x\in\Xi_n}|T^{q_n}x-x|\to 0.
\end{gather*}
\end{definition}

The following Proposition is the ergodicity criterium that we will exploit. It was proved in  \cite{Fr-Ul} (using Proposition~2.3 and the end of the proof of Proposition 5.2 in \cite{Fr-Ul}). %{ was it proved there?}

\begin{proposition}[Ergodicity criterium, see \cite{Fr-Ul}]\label{prop:ergodicxi}
Assume that $T:I\to I$ is an ergodic IET and $\varphi:I\to\R$ a measurable map.
Suppose that $(\Xi_n)_{n\geq 1}$ is a rigid sequence of tower and $(q_n)_{n\geq 1}$ its rigidity sequence.
If for all $|s|\geq s_0$ we have
\begin{equation}\label{eq:assumerg}
\int_{\Xi_n}|\varphi^{(q_n)}(x)|\,dx=O(1)\text{ and }\int_{\Xi_n}e^{2\pi s\varphi^{(q_n)}(x)}\,dx=\frac{2}{3}Leb(\Xi_n)+O(|s|^{-1}),
\end{equation}
then the skew product  $T_{\varphi}$ on $I\times \R$ is ergodic.
\end{proposition}

\begin{remark}\label{rem:rigtower}
Suppose that $T$ satisfies the \ref{SUDC}. For every $k\geq 1$ let $J^{(k)}\subset I^{(k)}_{\alpha_k}$ be a sequence of intervals such that $\liminf|J^{(k)}|/|I^{(k)}_{\alpha_k}|>0$. Then $\Xi_k=\{T^iJ^{(k)}:0\leq i<p_k\}$ establishes a rigid sequence of towers with the rigidity sequence given by $q_k:=Q_{\alpha_k}(k)$.
It follows directly from \eqref{def:UDC-d} and \eqref{def:UDC-g}. Since $T^{q_k}J^{(k)}\subset I^{(k)}$, for every $0\leq l\leq q_k$ we have that $T^{l}\Xi_k=\{T^{l+i}J^{(k)}:0\leq i<p_k\}$
is also a tower of intervals.
\end{remark}

Specializing the ergodicity criterion to our setting, we have the following Proposition, that shows that to prove ergodicity (and Theorem \ref{thm:erg}) it is sufficient to verify the assumptions in the statement:
\begin{proposition}\label{thm:erg1}
Suppose that $T$ satisfies the \ref{SUDC} and let $(\Xi_k)_{k\geq 1}$ and $(q_k)_{k\geq 1}$ be a sequence of rigid towers and its rigidity sequence as in Remark~\ref{rem:rigtower}.
Let $\varphi\in{\ol}(\sqcup_{\alpha\in \mathcal{A}} I_{\alpha})$ be a map such that $g'_{\varphi}\in{\bv}(\sqcup_{\alpha\in \mathcal{A}} I_{\alpha})$.
We additionally assume that there exists $c>0$ such that
\begin{enumerate}[label=(\roman*)]
\item \label{p84i} the sequence $\frac{1}{|I^{(k)}|}\|S(k)\varphi\|_{L^1(I^{(k)})}$ is bounded;
\item \label{p84ii} $\operatorname{dist}\big(\bigcup_{i=0}^{q_k}T^{i}\Xi_k,End(T)\big)\geq c/q_k$;
\item \label{p84iii} for every $0\leq j<p_k$ there exists an interval $J^{(k)}_j\subset T^jJ^{(k)}$ such that
$|J^{(k)}_j|\geq |J^{(k)}|/3$ and
$|(\varphi')^{(q_k)}(x)|\geq cq_k$ for all $x\in J^{(k)}_j$.
\end{enumerate}
Then the skew product $T_{\varphi}$ on $I\times \R$ is ergodic.
\end{proposition}

The proof is a variation on arguments from \cite{Fr-Ul}.  We present it for completeness in the Appendix~\ref{app:ergodicitycriterium}.

\begin{remark}\label{rem:subseqerg}
One can see from the proof presented in Appendix~\ref{app:ergodicitycriterium} that the same conclusion about the ergodicity of the skew product can be deduced under the assumption that
conditions \ref{p84i}, \ref{p84ii} and \ref{p84iii} hold along any subsequence.
\end{remark}

\subsubsection{Proof of ergodicity of skew products}\label{sec:ergodicitysp}
We will now prove Theorem~\ref{thm:erg} by showing that the assumptions of the criterion for ergodicity of skew products with logarithmic singularities over IETs (namely Proposition~\ref{thm:erg1}) hold.

\smallskip
For every $\varphi\in{\ol}(\sqcup_{\alpha\in \mathcal{A}}
I_{\alpha})$ such that $\mathcal{AS}(\varphi)=0$ and $\bl(\varphi)>0$ (using the definitions introduced in \S~\ref{sec:cocycles}) we want to  construct
a sequence of rigid towers as in Remark~\ref{rem:rigtower} for which the condition \ref{p84ii} and \ref{p84iii} in Proposition~\ref{thm:erg1} hold.
In view of \eqref{def:UDC-g} and  \cite[Lemma~5.1]{Fr-Ul}, we have the following result:

\begin{lemma}
Let $\varphi\in{\ol}(\sqcup_{\alpha\in \mathcal{A}}
I_{\alpha})$ be such that $\mathcal{AS}(\varphi)=0$ and $\bl(\varphi)> 0$.
There exists a sequence  $(\alpha_k)_{k\geq 1}$ in $\mathcal A$ and a sequence on natural number $(j_k)_{k\geq 1}$ such that
$p_k\leq j_k<Q_{\alpha_k}(k)$ and at least one of the following cases hold:
\begin{itemize}
\item[(L):] $C_{\alpha_0}^+\neq 0$ and $T^{j_k}l^{(k)}_{\alpha_k}=l_{\alpha_0}$ or
\item[(R):] $C_{\alpha_0}^-\neq 0$ and $\widehat{T}^{j_k}r^{(k)}_{\alpha_k}=r_{\alpha_0}$.
\end{itemize}
Moreover, the closures of the intervals $T^jI^{(k)}_{\alpha_k}$ for $Q_{\alpha_k}(k)\leq j< Q_{\alpha_k}(k)+p_k$
do not intersect $End(T)$.
\end{lemma}

\begin{definition}\label{def:cbar}
For any $0\leq \bar{c}<1/2$ we  define the base $J^{(k)}\subset I^{(k)}_{\alpha_k}=[a,b)$ of the  tower $\Xi_k$ as follows:
\begin{align*}
J^{(k)}=\left(a+\frac{\bar{c}}{2}\lambda^{(k)}_{\alpha_k},a+\bar{c}\lambda^{(k)}_{\alpha_k}\right)&\text{ in case }(L);\\
J^{(k)}=\left(a-\bar{c}\lambda^{(k)}_{\alpha_k},b-\frac{\bar{c}}{2}\lambda^{(k)}_{\alpha_k}\right)&\text{ in case }(R).
\end{align*}
\end{definition}

\begin{lemma}\label{lem:''}
Let $\varphi\in{\ol}(\sqcup_{\alpha\in \mathcal{A}}
I_{\alpha})$ be such that $\mathcal{AS}(\varphi)=0$ and $\bl(\varphi)>0 $ and $g'_{\varphi},g''_{\varphi} \in{\ac}(\sqcup_{\alpha\in \mathcal{A}}
I_{\alpha})$. Let $(\Xi_k)_{k\geq 1}$ be a sequence of rigid towers defined in Definition~\ref{def:cbar} with
\begin{equation}\label{eq:defcbar}
\bar{c}=\sqrt{\frac{|C^{\pm}_{\alpha_0}|}{2(6\bl(\varphi)+\var g'_{\varphi})}}.
\end{equation}
Then
\[|(\varphi'')^{(q_k)}(x)|\geq \frac{\underline{c}}{(\lambda^{(k)}_{\alpha_k})^2}\text{ for all }x\in \Xi_k,\]
where $\underline{c}=6\bl(\varphi)+\var g'_{\varphi}$.
\end{lemma}

\begin{proof}
We present the proof only in the case $(L)$. The other case is similar.

Suppose that $x\in T^l J^{(k)}$ for some $0\leq l<p_k$. By assumption,
\[\frac{|C_{\alpha_0}^+|}{\{T^{j_k-l}x-l_{\alpha_0}\}^2}\geq \frac{|C_{\alpha_0}^+|}{\bar{c}^2(\lambda^{(k)}_{\alpha_k})^2},\]
the elements of the orbit $T^jx$ for $0\leq j<p_k$ are distant from each other at least $\lambda^{(k)}_{\alpha_k}$ and for  $j\neq j_k-l$ we have   $\{T^{j}x-l_{\alpha_0}\}\geq \lambda^{(k)}_{\alpha_k}$. It follows that
\[\sum_{0\leq j<p_k,\ j\neq j_k-l}\frac{|C_{\alpha_0}^+|}{\{T^{j}x-l_{\alpha_0}\}^2}\leq \frac{|C_{\alpha_0}^+|}{(\lambda^{(k)}_{\alpha_k})^2}\sum_{j=1}^{p_k}\frac{1}{j^2}\leq \frac{\pi^2}{6}\frac{|C_{\alpha_0}^+|}{(\lambda^{(k)}_{\alpha_k})^2}\leq 2\frac{|C_{\alpha_0}^+|}{(\lambda^{(k)}_{\alpha_k})^2}.\]
Since for every $0\leq j<p_k$ we have $\{T^{j}x-l_{\alpha}\}\geq \lambda^{(k)}_{\alpha_k}$ for $\alpha\neq \alpha_k$ and
$\{r_{\alpha}-T^{j}x\}\geq \lambda^{(k)}_{\alpha_k}/2$ for all $\alpha\in\mathcal{A}$, the same arguments show that for all $\alpha\in\mathcal{A}$
we have
\[\sum_{0\leq j<p_k}\frac{|C_{\alpha}^-|}{\{r_\alpha-T^{j}x\}^2}\leq 6\frac{|C_{\alpha}^-|}{(\lambda^{(k)}_{\alpha_k})^2}
\text{ and }
\sum_{0\leq j<p_k}\frac{|C_{\alpha}^+|}{\{T^{j}x-l_{\alpha}\}^2}\leq 2\frac{|C_{\alpha}^+|}{(\lambda^{(k)}_{\alpha_k})^2}
\text{ if }\alpha\neq \alpha_k.\]
Moreover, for every $x\in I$ we have
\[|(g''_{\varphi})^{(q_k)}(x)|\leq q_k\|g''_{\varphi}\|_{\sup}\leq \frac{\var g'_{\varphi}}{(\lambda^{(k)}_{\alpha_k})^2}.\]
As
\[\varphi''(x)=\sum_{\alpha\in\mathcal{A}}\frac{C_\alpha^+}{\{x-l_\alpha\}^2}+\sum_{\alpha\in\mathcal{A}}\frac{C_\alpha^-}{\{r_\alpha-x\}^2}+g''_{\varphi}(x),\]
it follows that for every $x\in\Xi_k$ we have
\begin{align*}
|(\varphi'')^{(q_k)}(x)|&\geq \frac{|C_{\alpha_0}^+|}{\bar{c}^2(\lambda^{(k)}_{\alpha_k})^2}
-\frac{6\bl(\varphi)}{(\lambda^{(k)}_{\alpha_k})^2}-\frac{\var g'_{\varphi}}{(\lambda^{(k)}_{\alpha_k})^2}
=\Big(\frac{|C_{\alpha_0}^+|}{\bar{c}^2}-6\bl(\varphi)-\var g'_{\varphi}\Big)\frac{1}{(\lambda^{(k)}_{\alpha_k})^2}
=\frac{\underline{c}}{(\lambda^{(k)}_{\alpha_k})^2}.
\end{align*}
\end{proof}
\noindent The following elementary lemma will help us to choose the subintervals $J^{(k)}_l\subset T^lJ^{(k)}$ satisfying
condition \ref{p84iii} in Proposition~\ref{thm:erg1}.

\begin{lemma}\label{lem:elem}
Let $f:I\to\R$ be a $C^1$ map defined on a closed interval $I$ and such that $|f'(x)|\geq c>0$ for all $x\in I$.
Then there exists a closed subinterval $J\subset I$ such that $|J|\geq |I|/3$ and $|f(x)|\geq c|I|/6$ for all $x\in I$.
\end{lemma}

\begin{proof}[Proof of Theorem~\ref{thm:erg}]
Recall that $\varphi\in {\ol}(\sqcup_{\alpha\in \mathcal{A}} I_{\alpha})$ is  a cocycle such that
\[
\mathcal{L}(\varphi)>0,\quad \mathcal{AS}(\varphi)=0,\quad \mathfrak{h}(\varphi)=0\quad \text{and}\quad g'_{\varphi}\in {\ol}(\sqcup_{\alpha\in \mathcal{A}} I_{\alpha}).
\]
By definition, $\varphi=\varphi_0+g_\varphi$, where
\[{\varphi}_0(x)=-\sum_{\alpha\in\mathcal A}C_\alpha^+{\log}\{x-l_\alpha\}-\sum_{\alpha\in\mathcal A}C_\alpha^-{\log}\{r_\alpha-x\}\]
and $g_{\varphi}\in {\ac}(\sqcup_{\alpha\in \mathcal{A}} I_{\alpha})$ with $g'_{\varphi}\in {\ol}(\sqcup_{\alpha\in \mathcal{A}} I_{\alpha})$.
By Proposition~\ref{theorem:coblog}, $g_{\varphi}$ is cohomologous to a piecewise linear map
$\psi\in\ac(\sqcup_{\alpha\in \mathcal{A}} I_{\alpha})$ with $\mathfrak{h}(\psi)=\mathfrak{h}(g_{\varphi})$.
It follows that $\varphi$ is cohomologous to $\bar{\varphi}:=\varphi_0+\psi$. Then $\bar{\varphi}\in {\ol}(\sqcup_{\alpha\in \mathcal{A}} I_{\alpha})$ is such that $\mathcal{L}(\bar{\varphi})=\mathcal{L}({\varphi}_0)=\mathcal{L}(\varphi)>0$, $\mathcal{AS}(\bar{\varphi})=\mathcal{AS}({\varphi}_0)=\mathcal{AS}(\varphi)=0$,
\[\mathfrak{h}(\bar{\varphi})=\mathfrak{h}({\varphi}_0)+\mathfrak{h}(\psi)=\mathfrak{h}({\varphi}_0)+\mathfrak{h}(g_{\varphi})
=\mathfrak{h}({\varphi})=0\]
and $g_{\bar{\varphi}}=\psi$, so $g'_{\bar{\varphi}},g''_{\bar{\varphi}}\in \ac(\sqcup_{\alpha\in \mathcal{A}} I_{\alpha})$.
As $\varphi$ is cohomologous to $\bar{\varphi}$, the skew products $T_{\varphi}$ and $T_{\bar{\varphi}}$ are isomorphic, so it is sufficient to show
the ergodicity of $T_{\bar{\varphi}}$.

\medskip

Let $(\Xi_k)_{k\geq 1}$ be a sequence of rigid towers defined in Definition~\ref{def:cbar} with $\bar{c}$ given by \eqref{eq:defcbar}
for that function $\bar{\varphi}$. In view of \eqref{def:UDC-g} and \eqref{eq:qlambda}, this sequence satisfies
\[\operatorname{dist}\Big(\bigcup_{i=0}^{q_k}T^{i}\Xi_k,End(T)\Big)\geq \frac{1}{2}\bar{c}\lambda^{(k)}_{\alpha_k}\geq \frac{\delta\bar{c}}{2\kappa q_k},\]
so condition \ref{p84ii} in Proposition~\ref{thm:erg1} holds. Moreover, by Lemma~\ref{lem:''},
\[|(\bar{\varphi}'')^{(q_n)}(x)|\geq \frac{\underline{c}}{(\lambda^{(k)}_{\alpha_k})^2}\text{ for every }x\in T^lJ^{(k)},\ 0\leq l<p_k.\]
Since $|T^lJ^{(k)}|=|J^{(k)}|=\frac{\bar{c}}{2}\lambda^{(k)}_{\alpha_k}$, by Lemma~\ref{lem:elem}, for every $0\leq l<p_k$ there exists
an interval $J^{(k)}_l\subset T^lJ^{(k)}$ such that
\[|J^{(k)}_l|\geq |J^{(k)}|/3\text{ and }|(\bar{\varphi}')^{(q_n)}(x)|\geq \frac{\underline{c}}{(\lambda^{(k)}_{\alpha_k})^2}\frac{\bar{c}}{12}\lambda^{(k)}_{\alpha_k}\geq \frac{\underline{c}\bar{c}}{12}q_k\text{ for every }x\in J^{(k)}_l,\]
so condition \ref{p84iii} in Proposition~\ref{thm:erg1} holds.

Since $\mathcal{AS}(\bar{\varphi})=0$ and $\mathfrak{h}(\bar{\varphi})=0$, by Theorem~\ref{operatorcorrection} and \eqref{eq:C} in Proposition~\ref{prop:estKC},
\begin{equation*}
\frac{\|S(r_n)\bar{\varphi}\|_{L^1(I^{(r_n)})}}{|I^{(r_n)}|}\text{ is bounded,}
\end{equation*}
so condition \ref{p84i} in Proposition~\ref{thm:erg1} holds along a subsequence. In view of Proposition~\ref{thm:erg1} together with Remark~\ref{rem:subseqerg},
this gives the ergodicity of $T_{\bar{\varphi}}$, and hence the ergodicity of $T_{{\varphi}}$.
\end{proof}

\subsection{Reducibility and final arguments}\label{sec:reducibility}
The main goal of this section is to prove Main Theorem~\ref{main1}, in particular the dichotomy between ergodicity and reducibility for typical extensions with observables in a suitable subspace of smooth functions.  We also deduce from Main Theorem~\ref{main1} the second and final part of Main Theorem~\ref{mainthm:BS}. We first need to state an auxiliary result that we call \emph{cohomological reduction}.

\subsubsection{Cohomological reduction}\label{sec:cohomreduction}
The following result allows to reduce the study of cocycles whose derivatives have logarithmic singularities (up to coboundaries and hence cohomological equivalence) to piecewise linear cocycles (whose derivative is piecewise-constant). An analogous result was proved also in \cite{Fr-Ul}, but only in the special measure zero class of self-similar IETs considered there. %, but ... { add ref to our paper and comment on difficulties this time...}
\begin{theorem}\label{theorem:coblog}
Assume that  $T$ satisfies the \ref{UDC}. Then every
$\varphi\in\ac(\sqcup_{\alpha\in \mathcal{A}} I_{\alpha})$ with %$\partial_{\pi}(\varphi)=0$ and
$\varphi'\in{\ol}(\sqcup_{\alpha\in \mathcal{A}}
I_{\alpha})$
% and $\mathcal{AS}(\varphi')=0$
is cohomologous (via a bounded transfer function)
to a piecewise linear cocycle $\psi\in\ac(\sqcup_{\alpha\in \mathcal{A}} I_{\alpha})$ with $\mathfrak{h}(\psi)=\mathfrak{h}(\varphi)$,
$\partial_{\pi}(\psi)=\partial_{\pi}(\varphi)$ and $\|S(k)(\varphi-\psi)\|_{\sup}$ tends to $0$ exponentially.
\end{theorem}
\noindent The proof of the theorem, which generalizes the proof in  \cite{Fr-Ul} to full measure, is included in the Appendix~\ref{app:cohomreduction}.
In view of Theorem~A in \cite{Ma-Mo-Yo}, if $T$ is a Roth-type IET and $\varphi\in\ac(\sqcup_{\alpha\in \mathcal{A}} I_{\alpha})$ is such that $s(\varphi)=0$ and
$\varphi'\in\bv(\sqcup_{\alpha\in \mathcal{A}} I_{\alpha})$, then $\varphi$ is cohomologous (via a bounded transfer function) to a piecewise constant map $h$ and $\|S(k)(\varphi-h)\|_{\sup}$ tends to $0$ exponentially.

\smallskip
Assume additionally that $\varphi$ in Theorem~\ref{theorem:coblog} satisfies $s(\varphi)=0$. Then
\[s(\psi)=\sum_{\mathcal{O}\in\Sigma(\pi)}(\partial_\pi(\psi))_{\mathcal{O}}=
\sum_{\mathcal{O}\in\Sigma(\pi)}(\partial_\pi(\varphi))_{\mathcal{O}}=s(\varphi)=0.\]
So, by Theorem~A in \cite{Ma-Mo-Yo}, it follows that $\psi$ is cohomologous to piecewise constant map. As the \ref{UDC} implies Roth-type (see Remark~\ref{rem:QQ}), this gives the following important corollary, which gives a generalization of Theorem~A in \cite{Ma-Mo-Yo}.

\begin{corollary}\label{cor:cohom}
Assume that $T$  satisfies the \ref{UDC}. Then every
$\varphi\in\ac(\sqcup_{\alpha\in \mathcal{A}} I_{\alpha})$ with $s(\varphi)=0$ and
$\varphi'\in{\ol}(\sqcup_{\alpha\in \mathcal{A}}
I_{\alpha})$
%and $\mathcal{AS}(\varphi')=0$
is cohomologous (via a bounded transfer function) to a piecewise constant map
$h$ and $\|S(k)(\varphi-h)\|_{\sup}$ tends to $0$ exponentially.
\end{corollary}
\noindent The importance of this result is that in view of Proposition~\ref{prop:oper} it applies to solve cohomological equations
for a.e.\ $\psi_\R\in\mathscr{U}_{\min}$ and for functions $f\in C^{2+\epsilon}(M)$ vanishing on $\mathrm{Fix}(\psi_\R)$. Recall that
Theorem~A in \cite{Ma-Mo-Yo} applies only when $f$ vanishes on an open neighborhood of $\mathrm{Fix}(\psi_\R)$.

Classical
Gottschalk-Hedlund type arguments, first applied in the context of IETs in \cite[\S 3.4]{Ma-Mo-Yo}, show the following.
\begin{lemma}[\cite{Ma-Mo-Yo}]\label{rmk:GH}
Suppose that $T:I\to I$ is a minimal IET and $\varphi\in\ac(\sqcup_{\alpha\in \mathcal{A}} I_{\alpha})$.
The following conditions are equivalent:
\begin{enumerate}[label=(\roman*)]
\item \label{rmcobi} the sequence $\|\varphi^{(n)}\|_{\sup}$, $n\in\N$, is bounded;
\item \label{rmcobii} $\varphi=g-g\circ T$, where $g:I\to \R$ is bounded;
\item \label{rmcobiii} $\varphi=g-g\circ T$, where $g:I\to \R$ is bounded and has at most countably many discontinuities.
\end{enumerate}
\end{lemma}
\begin{proof}
The implications follow from the classical   Gottschalk-Hedlund theorem, that can be applied to IETs by extending them to a homeomorphism to a Cantor space, see  \cite[\S 3.4]{Ma-Mo-Yo}.
The only non-classical implication, \ref{rmcobiii}$\Rightarrow$\ref{rmcobi} is also proved in \cite[\S 3.4]{Ma-Mo-Yo}, where the authors show that the transfer map $g$ exists and is the composition of a continuous map
and a monotonic map, so is bounded and has at most countably many discontinuities.
\end{proof}

\subsubsection{Reduction to coboundaries.}\label{sec:coboundedness}
Now that we reduced to the study of cocycles which are piecewise absolutely continuous (i.e.~to $\varphi\in\ac(\sqcup_{\alpha\in \mathcal{A}} I_{\alpha})$), we can prove reducibility exploiting the following result on coboundaries.
\begin{proposition}\label{thm:cobo}
Assume that $T$ satisfies the \ref{UDC}. Then every
$\varphi\in\ac(\sqcup_{\alpha\in \mathcal{A}} I_{\alpha})$ with $\partial_{\pi}(\varphi)=0$, $\mathfrak{h}(\varphi)=0$ and
$\varphi'\in{\ol}(\sqcup_{\alpha\in \mathcal{A}}
I_{\alpha})$
%and $\mathcal{AS}(\varphi')=0$
 is a coboundary with a bounded transfer map having at most countably many discontinuities.
\end{proposition}
\begin{proof}
By Corollary~\ref{cor:cohom}, there exists $h\in\Gamma$ and a bounded map $g:I\to\R$
 such that $\varphi=h+g-g\circ T$. Moreover,
$\|S(k)(\varphi-h)\|_{\sup}$ tends to $0$ exponentially. As
\[\|\partial_{\pi}(\varphi-h)\|=\|\partial_{\pi^{(k)}}(S(k)(\varphi-h))\|\leq 2d\|S(k)(\varphi-h)\|_{\sup},\]
it follows that $\partial_{\pi}(h)=\partial_{\pi}(\varphi)=0$, so $h\in H(\pi)$. Moreover, as
\[\frac{\|S(k)(\varphi-h)\|_{L^1(I^{(k)})}}{|I^{(k)}|}\leq \|S(k)(\varphi-h)\|_{\sup}\to 0,\]
by the definition of the operator $\mathfrak{h}$ and Corollary~\ref{cor:corrected}, we have $\mathfrak{h}(\varphi-h)=0$.
It follows that $\mathfrak{h}(h)=\mathfrak{h}(\varphi)=0$, so $h\in \Gamma_s$. Therefore, $h$ is also a coboundary with a bounded transfer map. As the sum of two coboundaries, $\varphi=h+g-g\circ T$
is also a coboundary with a bounded transfer map. Finally, in view of Lemma~\ref{rmk:GH}, the transfer map has  at most countably many discontinuities.
\end{proof}

\subsubsection{Proof of the dichotomy for extensions.}\label{sec:proofdichotomyext}
We have now all ingredients needed for the proof of the dichotomy in  Main Theorem~\ref{main1}.
%, which will follow by combining the reduction to skew products described by Propositions~\ref{prop:redtoskew}, \ref{prop:oper}  and \ref{prop:reduc}; the criterion for the ergodicity of skew product given by Theorem~\ref{thm:erg} and  the above criterion for the  coboundedness of piecewise absolutely cocycles.
\begin{proof}[Proof of Main Theorem~\ref{main1}]
Let us say that a locally Hamiltonian flows $\psi_\R$ satisfies the \ref{SUDC}  condition %  and Roth-type condition
%Suppose that $T=T_{(\pi,\lambda)}$ for some $(\pi,\lambda)\in \mathcal{S}^0_{\mathcal A}\times \R^{\mathcal{A}}_{>0}$ with $\#\mathcal A=d$. Then $\eta\mapsto (\lambda_\alpha)_{\alpha\in\mathcal A}$ is well defined in a neighbourhood of $\eta$ in $U_{s_i,s_b}$ and
%coincides with the period map $\Theta$, with the appropriate  choice of basis in $H_1(M,\mathrm{Fix}(\psi_\R), \Z)$. We say that a $1$-form $\eta \in U_{s_i,s_b}$
%satisfies UDC (\text{SDUC} resp.) if the corresponding IET satisfies UDC (\text{SDUC} resp.) from Definition~\ref{def:UDC}. In view of Theorem~\ref{thm:UDC},
%almost every $\eta \in U_{s_i,s_b}$ satisfies SUDC so also UDC.
%TO FINISH
if and only if $\psi_\R$ has a section $I\subset M$ such that the corresponding IET  $T$ satisfies the condition \ref{SUDC}. Then, since the \ref{SUDC} has full measure by Theorems~\ref{thm:UDC}~and~\ref{cancellations_prop}, one can show
 by definition of the measure class on $\mathscr{U}_{min}$ (see for example \cite{Ul:abs})
%Let us consider in $\mathscr{U}_{min}$ a sub
that   the set of locally Hamiltonian flows satisfying the condition \ref{SUDC}
%Theorem~B in \cite{Ma-Mo-Yo},
 has full measure in $\mathscr{U}_{min}$ (in the sense of \S~\ref{sec:ergodicity}).

\medskip
In view of Propositions~\ref{prop:redtoskew} and \ref{prop:reduc}, we equivalently need to prove the dichotomy between ergodicity and reducibility for the skew product map $T_{\varphi_f}$. Furthermore, we know from Proposition~\ref{prop:oper} that the cocycle $\varphi_f$ is such that
$\varphi_f\in\ol(\sqcup_{\alpha\in \mathcal{A}} I_{\alpha})$, $\partial_{\pi}(\varphi_f)=0$,
$g'_{\varphi_f}\in{\ol}(\sqcup_{\alpha\in \mathcal{A}} I_{\alpha})$ and  $\mathcal{AS}(\varphi_f)=0$.
\medskip

\noindent
\emph{Definition of the subspace $K$.}
Let us consider the linear operator $\mathfrak{H}:C^{2+\epsilon}(M)\to F\simeq \R^g$
given by $\mathfrak{H}(f)=\mathfrak{h}(\varphi_f)$. As the composition of two bounded operators, it is also bounded.
Let $K:=\ker \mathfrak{H}\subset C^{2+\epsilon}(M)$. Then $K$ is a closed subspace of codimension $g$ (the genus of $M$).

\medskip

\noindent
\emph{Ergodicity.} Suppose that $f\in K$ and $\sum_{\sigma\in \mathrm{Fix}(\psi_\R)}|f(\sigma)|>0$. By Proposition~\ref{prop:oper},
 $\varphi\in\ol(\sqcup_{\alpha\in \mathcal{A}} I_{\alpha})$, $\mathcal{L}(\varphi_f)>0$,
$g'_{\varphi_f}\in{\ol}(\sqcup_{\alpha\in \mathcal{A}}
I_{\alpha})$ and  $\mathcal{AS}(\varphi_f)=0$. As $f\in K$, we additionally have $\mathfrak{h}(\varphi_f)=0$. In view of Theorem~\ref{thm:erg},
this gives the ergodicity of the skew product $T_{\varphi_f}$. By Proposition~\ref{prop:redtoskew}, we have the ergodicity of the extended flow $\Phi^f_\R$.

\medskip

\noindent
\emph{Reducibility.} Suppose that $f\in K$ and $\sum_{\sigma\in \mathrm{Fix}(\psi_\R)}|f(\sigma)|=0$. By Proposition~\ref{prop:oper},
 $\varphi\in\ac(\sqcup_{\alpha\in \mathcal{A}} I_{\alpha})$ with  $\partial_{\pi}(\varphi_f)=0$ and
$\varphi_f'\in{\ol}(\sqcup_{\alpha\in \mathcal{A}}
I_{\alpha})$. As $f\in K$, we additionally have $\mathfrak{h}(\varphi_f)=0$. In view of Theorem~\ref{thm:cobo},
$\varphi_f$ is a coboundary with a bounded transfer map having at most countably many discontinuities. By Proposition~\ref{prop:reduc}, this gives the reducibility of the extended flow $\Phi^f_\R$.
\end{proof}

\subsubsection{Equidistribution of the error in the symmetric case.} \label{sec:proofBSerror}
We can now conclude also the proof of { Theorem~\ref{mainthm:BS}, by proving that in this case $err_b$ is uniformely bounded and  deducing from ergodicity the equidistribution statement for the singular cocycles as well as the error term. }
\begin{proof}[Proof of the second part of Main Theorem~\ref{mainthm:BS}]
 Suppose that $\psi_\R\in \mathcal{U}_{min}$ is minimal and satisfies the \ref{SUDC}.  Let $f:M\to \R$ be any $C^{2+\epsilon}$-observable.

\medskip
{ \noindent{\it Boundedness of the error.}
Let  $f_{eb}:M\to\R$ be the map defined in \eqref{deffeb}. By construction (see  \eqref{eq:ebzero} and \eqref{eq:heb}), $f_{eb}$
is a  $C^{2+\epsilon}$-map such that $f_{eb}(\sigma)=0$ for all $\sigma\in\mathrm{Fix}(\psi_\R)$ and $\mathfrak{h}(\varphi_{f_{eb}})=0$.
By Proposition~\ref{prop:oper}, we know furthermore} that $\varphi_{f_{eb}}\in {\ac}(\sqcup_{\alpha\in \mathcal{A}} I_{\alpha})$, $\varphi'_{f_{eb}}\in {\ol}(\sqcup_{\alpha\in \mathcal{A}} I_{\alpha})$ and $\partial_\pi(\varphi_{f_{eb}})=0$. In view of Proposition~\ref{thm:cobo},
$\varphi_{f_{eb}}$ is a coboundary with a bounded transfer map having at most countably many discontinuities. By Proposition~\ref{prop:reduc}, this gives the reducibility of the extended flow $\Phi^{f_{eb}}_\R$, so there exists a continuous map $u:M\to\R$ such that $\int_0^t f_{eb}(\psi_s x)\,ds=u(x)-u(\psi_t x)$. It follows that for every regular $x\in M$ and $t>0$ we have
\[|err_b(f,t,x)|=\Big|\int_0^t f_{eb}(\psi_s x)\,ds\Big|\leq 2\|u\|_{\sup},\]
which completes the proof.

\medskip {
\noindent{\it Equidistribution of the singular cocycles and the error term.} {Assume now in addition that $f\in C^{2+\epsilon}(M)$ is  not identically zero on  $\mathrm{Fix}(\psi_\R)$.} We will prove at the same time $err(f,t,x)=\int_0^t f_e(\psi_\tau x)d\tau$ and $u_\sigma(t,x)=\int_0^t \xi_\sigma(\psi_\tau x)d\tau$ are equidistributed on $\R $, in the sense of \eqref{equid:usigma}.}

\smallskip
{ Let  $\xi$ be respectively $\xi=f_e$ or $\xi=\xi_\sigma$. We want to show that the assumptions of  Theorem~\ref{thm:erg} hold for $\varphi_{\xi}$ so that we can deduce that  the skew product $T_{\varphi_{\xi}}$ on $I\times\R$ is ergodic.
In both cases, by Proposition~\ref{prop:oper}, $\varphi_{\xi},g'_{\varphi_{\xi}}\in {\ol}(\sqcup_{\alpha\in \mathcal{A}} I_{\alpha})$ and (by property $(ii)$, since $\varphi_\R\in \mathcal{U}_{min}$)
$\mathcal{AS}(\varphi_{\xi})=0$,  $\partial_\pi(\varphi_{\xi})=0$. We claim furthermore that we also have that $\sum_{\sigma\in\mathrm{Fix}(\psi_\R)}|\xi(\sigma)|>0$ and therefore, also by Proposition~\ref{prop:oper},  $\mathcal{L}(\varphi_{\xi})>0$. To see this for $\xi=f_e$, recall that  in  view of \eqref{eq:sumfe}, $\sum_{\sigma\in\mathrm{Fix}(\psi_\R)}|f_e(\sigma)|=\sum_{\sigma\in\mathrm{Fix}(\psi_\R)}|f(\sigma)|>0$.
Furthermore, in view of \eqref{eq:hfe}, $\mathfrak{h}(\varphi_{f_e})=0.$
For $\xi=\xi_\sigma$, on the other hand, recall that, by the definition of $\xi_\sigma$ and \eqref{eq:hxi}, for every $\sigma\in\mathrm{Fix}(\psi_\R)$ we have $\xi_\sigma(\sigma)=1$
and $\mathfrak{h}(\xi_\sigma)=0$.

Thus, since $T$ satisfies the \ref{SUDC}, all assumptions of Theorem~\ref{thm:erg} hold and we conclude that} the skew product $T_{\varphi_{\xi}}$ on $I\times\R$ is ergodic. It follows that also the skew product flow $(\Phi_t^{\xi})_{t\in\R}$
on $M\times\R$ given by
\[\Phi_t^{\xi}(x,r)=\Big(\psi_tx,r+\int_0^t\xi(\psi_\tau x)d\tau\Big)\]
is ergodic.
%The realationship between both skew product system is presented at the begining of Section~\ref{sec:erg}.
We now apply the ratio ergodic theorem to the ergodic flow $(\Phi_t^{f_e})_{t\in\R}$ and to the characteristic functions of the sets $I\times J_1$ and $I\times J_2$.
Then for a.e.\ $(x,r)\in I\times\R$ for any pair of finite intervals $J_1,J_2\subset \R$ we have
\begin{align*}
\frac{Leb\{t\in[0,T]:\int_0^t\xi(\psi_\tau x)d\tau\in J_1\}}{Leb\{t\in[0,T]:\int_0^t\xi(\psi_\tau x)d\tau\in J_2\}}=\frac{\int_0^T \chi_{I\times (J_1+r)}(\Phi_t^{\xi}(x,r))\,dt}{\int_0^T \chi_{I\times (J_2+r)}(\Phi_t^{\xi}(x,r))\,dt}\to \frac{|J_1+r|}{|J_2+r|}=\frac{|J_1|}{|J_2|}.
\end{align*}
As $err(f,t,x)=\int_0^t f_e(\psi_\tau x)d\tau$ and $u_\sigma(t,x)=\int_0^t \xi_\sigma(\psi_\tau x)d\tau$, this gives the equidistribution
of cocycles $t\mapsto err(f,t,x)$ and $t\mapsto u_\sigma(t,x)$ for a.e.\ $x\in M$.
\end{proof}

\appendix

\section{}\label{sec:appendix}
In this Appendix we present the proofs of two auxiliary results, the ergodicity criterium (Proposition~\ref{thm:erg1}) in \S~\ref{app:ergodicitycriterium} and the cohomological reduction   to piecewise linear cocycles (Theorem~\ref{theorem:coblog}) in \S~\ref{app:cohomreduction}.

\subsection{Ergodicity criterium}\label{app:ergodicitycriterium}
In this Appendix we prove the ergodicity criterium stated as Proposition~\ref{thm:erg1}. The proof repeats arguments from   the proof of Propositions~5.1,~5.2  in \cite{Fr-Ul} and is included for convenience.

\begin{proof}[Proof of Proposition \ref{thm:erg1}]
For simplicity assume that $|I|=1$. First we show that there exists $C>0$ such that
\begin{equation}\label{eq:tm}
|\varphi^{(q_k)}(x)-\varphi^{(q_k)}(T^mx)|\leq C\text{ for all }0\leq m<p_k,\ x\in J^{(k)}.
\end{equation}
Note that
\[|\varphi^{(q_k)}(x)-\varphi^{(q_k)}(T^mx)|=|\varphi^{(m)}(x)-\varphi^{(m)}(T^{q_k}x)|\leq\Big|\int_{x}^{T^{q_k}x}|(\varphi')^{(m)}(y)|\,dy\Big|.\]
Assume that $g_{\varphi}=0$. In view of \eqref{generalsum} in Proposition~\ref{cancellations_prop}, for every $y\in I^{(k)}$ we have
\[|(\varphi')^{(m)}(y)|\leq \sum_{\alpha\in\mathcal A}\Big(\frac{|C_\alpha^+|}{\min_{0\leq i<m}|T^iy-l_\alpha|}+\frac{|C_\alpha^-|}{\min_{0\leq i<m}|T^iy-r_\alpha|}\Big)+M\bl(\varphi)\|Q(k)\|.\]
As $x\in J^{(k)}$, by assumption~\ref{p84ii},  there exists $c>0$ such that
\[|T^ix-l_\alpha|\geq c/q_k,\ |T^ix-r_\alpha|\geq c/q_k,\ |T^i(T^{q_k}x)-l_\alpha|\geq c/q_k,\ |T^i(T^{q_k}x)-r_\alpha|\geq c/q_k\]
for all $\alpha\in \mathcal A$ and $0\leq i<p_k$. As $x,T^{q_k}x\in I^{(k)}$, it follows that
\[|T^iy-l_\alpha|\geq c/q_k,\ |T^iy-r_\alpha|\geq c/q_k \text{ for all }y\in[x,T^{q_k}x].\]
In view of \eqref{eq:lg}, this gives
\begin{align*}
\Big|\int_{x}^{T^{q_k}x}|(\varphi')^{(m)}(y)|\,dy\Big|&
\leq |x-T^{q_k}x|\bl(\varphi)(q_k/c+M\|Q(k)\|)\\&\leq |I^{(k)}|\|Q(k)\|(M+1/c)\bl(\varphi)\leq \kappa(M+1/c)\bl(\varphi).
\end{align*}
Suppose that $g_{\varphi}\neq 0$. As $x,T^{q_k}x\in I^{(k)}$, we have that $\{T^i[x,T^{q_k}x]:0\leq i<m\}$ is a tower of intervals.
Hence
\[|g_{\varphi}^{(m)}(x)-g_{\varphi}^{(m)}(T^{q_k}x)|\leq \sum_{0\leq i<m}|g_{\varphi}(T^{i}x)|-g_{\varphi}(T^i(T^{q_k}x))|\leq \var g_{\varphi}.\]
 This gives \eqref{eq:tm}. Therefore, for every $0\leq i<p_k$ we have
\[\int_{T^iJ^{(k)}}|\varphi^{(q_k)}(x)|\,dx\leq \int_{J^{(k)}}|\varphi^{(q_k)}(x)|\,dx+|J^{(k)}|C=\int_{J^{(k)}}|S(k)\varphi(x)|\,dx+|J^{(k)}|C.\]
Hence
\[\int_{\Xi_k}|\varphi^{(q_k)}(x)|\,dx\leq p_k\int_{I^{(k)}}|\varphi^{(q_k)}(x)|\,dx+p_k|J^{(k)}|C=\frac{1}{|I^{(k)}|}\int_{I^{(k)}}|S(k)\varphi(x)|\,dx+C.\]
In view of assumption~\ref{p84i}, this gives the left condition in \eqref{eq:assumerg}.

\medskip

For every $0\leq l<p_k$ let $[a_l,b_l]=J^{(k)}_l$. Repeating some integration by parts arguments from the proof of Proposition~5.2 in \cite{Fr-Ul},
we have
\[\Big|\int_{J^{(k)}_l}e^{2\pi s\varphi^{(q_k)}(x)}\,dx\Big|\leq \frac{1}{|s|}\Big(\frac{2}{\min_{x\in[a_l,b_l]}|(\varphi')^{(q_k)}(x)|}+\var|_{[a_l,b_l]}\frac{1}{(\varphi')^{(q_k)}}\Big).\]
In view of \ref{p84iii}, it follows that
\begin{align*}\Big|\int_{J^{(k)}_l}e^{2\pi s\varphi^{(q_k)}(x)}\,dx\Big|&\leq
\frac{1}{|s|}\Big(\frac{2}{cq_k}+\frac{1}{c^2q_k^2}\var|_{[a_l,b_l]}(\varphi')^{(q_k)}\Big)
\leq
\frac{1}{|s|}\Big(\frac{2}{cq_k}+\frac{1}{c^2q_k^2}\sum_{0\leq i<q_k}\var|_{T^i[a_l,b_l]}\varphi'\Big)
.\end{align*}
By \ref{p84ii}, $\{T^i[a_l,b_l]:0\leq i<q_k\}$ is a tower of intervals and each level interval $T^i[a_l,b_l]$  is distant from the set $End(T)$ by at least
$c/q_k$. Recall that
\[\varphi'(x)=-\sum_{\alpha\in\mathcal{A}}\frac{C_\alpha^+}{\{x-l_\alpha\}}+\sum_{\alpha\in\mathcal{A}}\frac{C_\alpha^-}{\{r_\alpha-x\}}+g'_{\varphi}(x).\]
Moreover,
\begin{gather*}
\sum_{0\leq i<q_k}\var|_{T^i[a_l,b_l]}\frac{1}{\{x-l_\alpha\}}=\var_{[c/q_k,1]}\frac{1}{x}\leq \frac{q_k}{c},\\
\sum_{0\leq i<q_k}\var|_{T^i[a_l,b_l]}\frac{1}{\{r_\alpha-x\}}=\var_{[0,1-c/q_k,1]}\frac{1}{1-x}\leq \frac{q_k}{c}
\end{gather*}
and
\[\sum_{0\leq i<q_k}\var|_{T^i[a_l,b_l]}g'_{\varphi}\leq \var g'_{\varphi}.\]
It follows that for every $0\leq l<p_k$,
\[\Big|\int_{J^{(k)}_l}e^{2\pi s\varphi^{(q_k)}(x)}\,dx\Big|\leq
\frac{1}{|s|}\Big(\frac{2}{cq_k}+\frac{1}{c^2q_k^2}\Big(\bl(\varphi)\frac{q_k}{c}+\var(g'_{\varphi})\Big)\Big).\]
As
\[Leb\Big(\Xi_k\setminus \bigcup_{0\leq l<p_k}J^{(k)}_l\Big)=\sum_{0\leq l<p_k}Leb(T^lJ^{(k)}\setminus J^{(k)}_l)
\leq \frac{2}{3}\sum_{0\leq l<p_k}|T^lJ^{(k)}|=\frac{2}{3}Leb(\Xi_k),\]
this yields
\[\Big|\int_{\Xi_k}e^{2\pi s\varphi^{(q_k)}(x)}\,dx\Big|\leq \frac{2}{3}Leb(\Xi_k)+
\frac{1}{|s|}\Big(\frac{2}{c}+\frac{1}{c^2}\Big(\frac{\bl(\varphi)}{c}+\var(g'_{\varphi})\Big)\Big),\]
which gives the right condition in \eqref{eq:assumerg}. By Proposition~\ref{prop:ergodicxi}, we have the ergodicity of $T_\varphi$.
\end{proof}

\subsection{Cohomological reduction}\label{app:cohomreduction}
%{ Do we still like this name? to add some words of motivation... }
In this Appendix we present the proof of the cohomological reduction stated as Theorem~\ref{theorem:coblog}. We will assume throughout that $T$ satisfies the \ref{UDC}. For simplicity will also assume that $|I|=1$. Let us denote by
$$\ac^{\mathfrak{h}}(\sqcup_{\alpha\in \mathcal{A}} I_{\alpha}):=\{ \varphi\in\ac(\sqcup_{\alpha\in \mathcal{A}}
I_{\alpha}), \ \text{such\ that}\
\varphi'\in{\ol}(\sqcup_{\alpha\in \mathcal{A}}
I_{\alpha}) \
%, $\mathcal{AS}(\varphi')=0$
\mathrm{and}\  \mathfrak{h}(\varphi')=0 \}. $$

\noindent {\it Outline of the proof.}
We will show first of all  that every $\varphi\in\ac(\sqcup_{\alpha\in \mathcal{A}}
I_{\alpha})$  with $\varphi'\in{\ol}(\sqcup_{\alpha\in \mathcal{A}}
I_{\alpha})$ can be modified by a piecewise linear map such that its modification is in $\ac^{\mathfrak{h}}(\sqcup_{\alpha\in \mathcal{A}} I_{\alpha})$, by showing that one can  subtract a map whose derivative is $\mathfrak{h}(\varphi')$ (see Steps 1 and 2 of the proof of Theorem~\ref{theorem:coblog} below).
The next step of the proof is to apply the correction by a piecewise constant function described in \S~\ref{correction:sec} (see Step 3 of the proof of Theorem~\ref{theorem:coblog}). We then show that, after this further correction, the resulting map $\widetilde{\varphi}$ is a coboundary. We will show more precisely that $\|S(k)\widetilde{\varphi}\|_{\sup}$ decays exponentially (see Theorem~\ref{thmcorrecverexp}).
%The main tool here is an exponential decay of $\|S(k)\widetilde{\varphi}\|_{\sup}$.
Then standard arguments based on decompositions of Birkhoff sums
(see \S~\ref{sec:BSdecomp}) and the Gottschalk-Hedlund theorem  yield that  $\widetilde{\varphi}$ is a coboundary (see Step 4 of the proof of Theorem~\ref{theorem:coblog}).

The proof of Theorem~\ref{thmcorrecverexp} (namely of exponential decay of $\|S(k)\widetilde{\varphi}\|_{\sup}$) is similar to the proof of
Theorem~\ref{operatorcorrection} in \S~\ref{correction:sec}, or more precisely to the the proof of sub-exponential growth of $\|S(k)\widetilde{\varphi}\|_{L^1(I^{(k)})}/|I^{(k)}|$ (see in particular
\eqref{eq:thmcorr2} in the statement of Theorem~\ref{operatorcorrection}).
  One of the key arguments in this proof was showing that
$\lv(S(k)\varphi)$ was bounded (or had sub-exponential growth in the non-symmetric case). Here,  we will have a stronger input, namely the exponential decay of $\lv(S(k)\varphi)$: indeed,
for every $\varphi\in
\ac^{\mathfrak{h}}(\sqcup_{\alpha\in \mathcal{A}} I_{\alpha})$,
since $\varphi$ is piecewise absolutely continuous, we have that $\lv(S(k)\varphi)=\var(S(k)\varphi)$ and therefore,
 in view of
Theorem~\ref{operatorcorrection} (applied to $\varphi'$) and the control of the $L^1$ norm via $\|\cdot \|_\lv$ given by \eqref{eq:norl1lv}, for every
$k\geq 1$,
\begin{align}\label{znowuvar}
\begin{split}
\lv(S(k)\varphi)= \var&(S(k)\varphi)= \|S(k)({\varphi'})\|_{L^1(I^{(k)})}\leq C|I^{(k)}|C'_k(T)\|\varphi'\|_\lv.
\end{split}
\end{align}
%thanks to \eqref{znowuvar},
Since $|I^{(k)}|$ decays exponentially, this shows that $\lv(S(k)\varphi)$ decays exponentially.
Exploiting this exponential decay, analyzing its effect on all inequalities used in \S~\ref{correction:sec}, we will prove
the exponential decay of $\|S(k)\widetilde{\varphi}\|_{\sup}$. Differently than in \S~\ref{correction:sec}, though, instead of the $L^1$-norm, we have now to always use the $\sup$-norm. This requires a detailed and patient analysis of all steps used in \S~\ref{correction:sec} in this new context, which is performed for example in the proofs of Lemmas~\ref{lemma:tildep} and~\ref{thmcorrecver} below.

\smallskip
We begin by stating  and proving the following exponential decay result.
\begin{theorem}\label{thmcorrecverexp}
Assume that  $T$ satisfies the \ref{UDC}. Suppose that $\varphi\in
\ac^{\mathfrak{h}}(\sqcup_{\alpha\in \mathcal{A}} I_{\alpha})$,  $\partial_{\pi}(\varphi)=0$ and $\mathfrak{h}(\varphi)=0$. Then
$$\|S(k){\varphi}\|_{\sup} =O(e^{-\lambda k}).$$
\end{theorem}
\noindent The proof of Theorem~\ref{thmcorrecverexp} will follow from combining the following three Lemmas (Lemma~\ref{lemma:tildep},  Lemma~\ref{thmcorrecver} and Lemma~\ref{lem:vk}). The first is an improved estimate of the growth of the image $P^{(k)}\varphi $ of the  correcting operators $P^{(k)}$ (introduced in \S~\ref{correction:sec}) when $\varphi\in \ac^{\mathfrak{h}}(\sqcup_{\alpha\in \mathcal{A}} I^{(0)}_{\alpha})$ and $\partial_{\pi^{(0)}}(\varphi)=0$.

\begin{lemma}\label{lemma:tildep}%szapk1
 The correcting operator $P^{(k)}:
\ac(\sqcup_{\alpha\in \mathcal{A}} I^{(k)}_{\alpha})\to \ac(\sqcup_{\alpha\in \mathcal{A}} I^{(k)}_{\alpha})/\Gamma^{(k)}_{s}$
 is such that, for every $\varphi\in \ac^{\mathfrak{h}}(\sqcup_{\alpha\in \mathcal{A}} I^{(0)}_{\alpha})$ with $\partial_{\pi^{(0)}}(\varphi)=0$,
\begin{equation}\label{szapk1}
\|P^{(k)}(S(k)\varphi)\|_{\sup/\Gamma^{(k)}_s}\leq C  \|\varphi'\|_\lv \, W_k, \quad \text{where} \quad W_k:=\sum_{r\geq k} \|Q_s(k,r+1)\|\|Z(r+1)\||I^{(r)}|C'_{r}(T).
\end{equation}
% where
%\begin{align*}
%W_k&=\sum_{r\geq k} \|Q_s(k,r+1)\|\|Z(r+1)\||I^{(r)}|C'_{r}(T).
%,\\ \hat{W}_k&=\sum_{r\geq k} \|Q_s(k,r+1)\|\|Z(r+1)\|\|Q_s(r)\||I^{(r)}|.
%\end{align*}
\end{lemma}

\begin{proof} Recall that ${P}^{(k)}$ is  given by ${P}^{(k)}={U}^{(k)}\circ P_0^{(k)}-\Delta^{(k)}$.
Let us first give a preliminary estimate for the modifying operator $\Delta^{(k)}:
\ac(\sqcup_{\alpha\in \mathcal{A}} I^{(k)}_{\alpha})\to H(\pi^{(k)})/\Gamma^{(k)}_{s}$ starting from the definition of  $\Delta^{(k)}$ as the series given by \eqref{defoperdel}.
Let $\varphi\in \ac^{\mathfrak{h}}(\sqcup_{\alpha\in \mathcal{A}} I^{(0)}_{\alpha})$ with $\partial_{\pi^{(0)}}(\varphi)=0$.

% In view of
%Theorem~\ref{operatorcorrection} and \eqref{eq:norl1lv}, for every $\varphi\in
%\ac^{\mathfrak{h}}(\sqcup_{\alpha\in \mathcal{A}} I_{\alpha})$ and
%$k\geq 1$,
%\begin{align}\label{znowuvar}
%\begin{split}
%\var&(S(k)\varphi)= \|S(k)({\varphi'})\|_{L^1(I^{(k)})}\leq C|I^{(k)}|C'_k(T)\|\varphi'\|_\lv.
%%%\\
%%%&\leq C|I^{(k)}|\big(
%%C_k(T) (\lv(\varphi')+\|\partial_{\pi}(\varphi')\|)+\|Q_s(k)\||I^{(k)}|\|\varphi'\|_{L^1(I)}\big)\\
%%&\leq C|I^{(k)}|C_k(T)\|\varphi'\|_\lv^\partial,
%\end{split}
%\end{align}
%%where $\|\varphi'\|_\lv^\partial=\lv(\varphi')+\|\varphi'\|_{L^1(I)}+\|\partial_{\pi}(\varphi')\|$.

\medskip
\noindent {\it Step 1. Estimates of $\Delta^{(k)}\varphi$.} To estimate the series \eqref{defoperdel} (with $S(k)(\varphi)$ instead than $\varphi$), for each fixed $r\geq k$ we need to estimate
$$(S_\flat(k,r+1))^{-1}\circ U^{(r+1)}\circ \mathcal{M}_H^{(r+1)}\circ
S(r,r+1)\circ P_0^{(r)}\circ S(r)(\varphi).
$$
Let us start from right to left, by estimating first the action of  $P_0^{(r)}$ on $S(r)(\varphi)$, then that one of $\mathcal{M}_H^{(r+1)}\circ
S(r,r+1)$ and finally  applying and estimating $(S_\flat(k,r+1))^{-1}\circ U^{(r+1)}$.

\smallskip
 \noindent {\it Step 1 (i). The action of $P_0^{(r)}$.} Recall that (by Definition~\ref{def:P0}) $P_0^{(r)} = I -p_{H(\pi^{(r)})}\circ {\mathcal{M}}^{(r)}$, where $I$ is the identity operator and $\mathcal{M}^{(r)}$ the preliminary correction given by subtracting the mean in each $I^{(r)}_{\alpha}$.
Since $S(r)(\varphi)\in\ac(\sqcup_{\alpha\in \mathcal{A}} I^{(r)}_{\alpha})$ is piecewise absolutely continuous,
\[\|S(r)(\varphi)-\mathcal{M}^{(r)}(S(r)(\varphi))\|_{\sup}\leq \var(S(r)(\varphi)).\]
Using first this estimate together with the control of the projection by the boundary operator $\|h-p_{H(\pi^{(r)})}h\|\leq C_{\mathcal{G}}\|\partial_\pi^{(r)} h\| $ given by Lemma~\ref{bd:estimate} (see in particular \eqref{eq:proj}),
then the comparision between $\partial_{\pi^{(r)} }$ and  $\partial_{\pi^{(r)} }\circ \mathcal{M}^{(r)}$ given  by
\eqref{eq:partbv} and finally the estimate \eqref{znowuvar} of the variance together with the invariance of the boundary \eqref{eq:eqpar} and the assumption that  $\partial_{\pi^{(r)}}(S(r)\varphi)=\partial_{\pi^{(0)}}(\varphi)=0$,
 we get the following chain of inequalities:
 %combined with \eqref{eq:ckl} (i.e.\ $C_{r-k}(T^{(k)})\leq C_r(T)=C_r$),  we have
\begin{align}\label{szaflat}
\begin{split}
\|P_0^{(r)}\circ S(r)(\varphi)\|_{\sup}&\leq \|S(r)(\varphi)-\mathcal{M}^{(r)}(S(r)(\varphi))\|_{\sup}+\|\mathcal{M}^{(r)}(S(r)(\varphi))-p_{H(\pi^{(r)})}\mathcal{M}^{(r)}(S(r)(\varphi))\|_{\sup}\\
&\leq\var(S(r)(\varphi))+C_{\mathcal{G}}\|\partial_{\pi^{(r)}}\mathcal{M}^{(r)}(S(r)(\varphi))\|\\
&\leq(1+2dC_{\mathcal{G}})\var(S(r)(\varphi))+C_{\mathcal{G}}\|\partial_{\pi^{(r)}}(S(r)\varphi)\|\\
&\leq C'|I^{(r)}|C'_r(T)\|\varphi'\|_\lv.
\end{split}
\end{align}

\smallskip \noindent {\it Step 1(ii). The action of
$\mathcal{M}_H^{(r+1)}\circ
S(r,r+1)$.}
In view of the initial correction estimates of Lemma~\ref{lemma:prelim} (in particular \eqref{ineq:ph1}), the $L^1$-norm of special  Birkhoff sums estimate \eqref{nase} and the  interval lenght control in terms of cocycle matrix norms given by \eqref{neq:dli}, for every $\phi\in\ac(\sqcup_{\alpha\in \mathcal{A}} I^{(r)}_{\alpha})$,
\begin{align*}
\|&{\mathcal{M}_H^{(r+1)}}\circ S(r,r+1)\phi\|\leq \frac{\kappa\sqrt{d}}{|I^{(r+1)}|}\| S(r,r+1)\phi\|_{L^1(I^{(r+1)})}\\
&\leq \frac{\kappa\sqrt{d}}{|I^{(r+1)}|}\|\phi\|_{L^1(I^{(r)})}\leq \kappa\sqrt{d}\frac{|I^{(r)}|}{|I^{(r+1)}|}\|\phi\|_{\sup}\leq \kappa \sqrt{d}\|Z(r+1)\|\|\phi\|_{\sup}.
\end{align*}

\smallskip \noindent {\it Step 1(iii). The action of  $(S_\flat(k,r+1))^{-1}\circ U^{(r+1)}$}.
Since $\|{U}^{(r+1)}\|= 1$, by \eqref{szaflat}, this gives (applied to $\phi=P_0^{(r)}\circ S(r)(\varphi)$)
\begin{align*}
\|&(S_{\flat}(k,r+1))^{-1}\circ {U}^{(r+1)}\circ
{\mathcal{M}_H^{(r+1)}}\circ S(r,r+1)\circ P_0^{(r)}\circ S(k,r)(S(k)\varphi)\|\\
&\leq
\kappa \sqrt{d} C'\|Q_s(k,r+1)\|\|Z(r+1)\||I^{(r)}|C'_r(T)\|\varphi'\|_\lv.
\end{align*}
As $\Delta^{(k)}(S(k)\varphi)$ is the sum of the series \eqref{defoperdel}, it follows that
\begin{equation}\label{eq:nordel}
\|\Delta^{(k)}(S(k)\varphi)\|_{\sup/\Gamma^{(k)}_s}\leq \kappa \sqrt{d} C'W_k\|\varphi'\|_\lv.
\end{equation}

\medskip
\noindent {\it Step 2. Estimates of ${P}^{(k)} \varphi$.} We can now estimate ${P}^{(k)}:\ac(\sqcup_{\alpha\in \mathcal{A}}
I^{(k)}_{\alpha})\to \ac(\sqcup_{\alpha\in \mathcal{A}}
I^{(k)}_{\alpha})/\Gamma^{(k)}_{s}$ recalling that ${P}^{(k)}={U}^{(k)}\circ P_0^{(k)}-\Delta^{(k)}$.
As $\|U^{(k)}\|=1$, in view of \eqref{szaflat}, if $\varphi\in \ac^{\mathfrak{h}}(\sqcup_{\alpha\in \mathcal{A}} I_{\alpha})$ and $\partial_{\pi}(\varphi)=0$ then
\begin{align*}
\|U^{(k)}\circ P_0^{(k)}( S(k)\varphi)\|_{\sup/\Gamma^{(k)}_s}&\leq\|P_0^{(k)}\circ S(k)(\varphi)\|_{\sup}
\leq  C'|I^{(k)}|C'_k(T)\|\varphi'\|_\lv
\leq  C'W_k\|\varphi'\|_\lv.
\end{align*}
Together with \eqref{eq:nordel}, this gives the desired estimate and proves the Lemma. %gives \eqref{szapk1}.
\end{proof}

\begin{lemma}\label{thmcorrecver}
Under the assumptions of Theorem~\ref{thmcorrecverexp}, for any $k\geq 0$,
%Assume that  $T$ satisfies the \ref{UDC}. Suppose that $\varphi\in
%\ac^{\mathfrak{h}}(\sqcup_{\alpha\in \mathcal{A}} I_{\alpha})$,  $\partial_{\pi}(\varphi)=0$ and $\mathfrak{h}(\varphi)=0$. Then
\[\|S(k){\varphi}\|_{\sup}
\leq
 C\big(\|\varphi'\|_\lv \, V_k + \|Q_s(k)\|\|\varphi\|_{\sup}\big),\]
 where $V_k$ is given by the following series
\begin{align*}
V_k=\sum_{0\leq l\leq k}\|Q_s(l,k)\|\big(W_l+\|Z(l)\|W_{l-1}\big),
%,\\ \hat{V}_k&=4\sum_{0\leq l\leq k}|I^{(l)}|\|Q_s(l,k)\|\big(\hat{W}_l^2+\|Z(l)\|\hat{W}_{l-1}^2\big).
\end{align*}
in which $W_{-1}:=0$ by convention and the series $W_l$ for $l\geq 0$ is defined in \eqref{szapk1} of Lemma~\ref{lemma:tildep}.
\end{lemma}

\begin{proof}
By the definition of the operator $\mathfrak{h}$ (see \eqref{def:oph}),  since $\mathfrak{h}(\varphi)=0$, we have that  $U^{(0)}(\varphi)=P^{(0)}(\varphi)$.
In view of the equivariance described by Lemma~\ref{commutes}, it follows that
\[{U}^{(k)}\circ
S(k){\varphi}=S_{\flat}(k)\circ
{U}^{(0)}{\varphi} =S_{\flat}(k)\circ
{P}^{(0)}\varphi={P}^{(k)}\circ S(k)\varphi.\]
Therefore, by Lemma~\ref{lemma:tildep}, %(see \eqref{szapk1})
we have
\[\|{U}^{(k)}\circ
S(k){\varphi}\|_{\sup/\Gamma^{(k)}_s}=\|P^{(k)}(S(k)\varphi)\|_{\sup/\Gamma^{(k)}_s}\leq C W_k \|\varphi'\|_\lv.\]
It follows that for every $k\geq 0$ there exists $\varphi_k\in
\ac (\sqcup_{\alpha\in \mathcal{A}} I^{(k)}_{\alpha})$ and
$s_k\in\Gamma^{(k)}_{s}$ such that
\begin{align}\label{szcsk2}
S(k){\varphi}=\varphi_k+s_k \text{ and }\|\varphi_k\|_{\sup}\leq  C W_k \|\varphi'\|_\lv.
\end{align}
Setting $s_0:=\Delta s_0$ and  $\Delta s_{k+1}=s_{k+1}-Z(k+1)s_k$ for any $k\geq 1$, since for $s_k\in \Gamma^{(k)}$ we have that $S(k,k+1)s_k= Z(k+1)s_k$, we get
\[\Delta s_{k+1}= s_{k+1}- S(k,k+1)s_k= (S(k+1)\varphi-\varphi_{k+1})-S(k,k+1)(S(k) \varphi-\varphi_k)=-\varphi_{k+1}+S(k,k+1)\varphi_k.\] Therefore, by
\eqref{szcsk2},
\begin{align*}
\|\Delta
s_{k+1}\|_{\sup}&=\|\varphi_{k+1}-S(k,k+1)\varphi_k\|_{\sup}\leq\|\varphi_{k+1}\|_{\sup}+
\|S(k,k+1)\varphi_k\|_{\sup}\\
&\leq \|\varphi_{k+1}\|_{\sup}+
\|Z(k+1)\|\|\varphi_k\|_{\sup}\leq C (W_{k+1}+\|Z(k+1)\|W_k) \|\varphi'\|_\lv
\end{align*}
and, since by definition $\Delta s_{0}=s_0= \varphi -\varphi_0$,
\[\|\Delta s_{0}\|_{\sup}=\|\varphi-\varphi_0\|_{\sup}\leq
\|\varphi\|_{\sup}+C W_{0} \|\varphi'\|_\lv.\]
Since $s_k=\sum_{0\leq l\leq k}Q(l,k)\Delta s_{l}$ and $\Delta
s_l\in\Gamma^{(l)}_{s}$, setting $W_{-1}=0$, we have
\begin{align*}
\|s_k\|_{\sup}&\leq\sum_{0\leq l\leq k}\|Q(l,k)\Delta s_{l}\|_{\sup}\leq\sum_{0\leq l\leq k}\|Q_s(l,k)\|\|\Delta s_{l}\|_{\sup}\\
&\leq \|Q_s(k)\|\|{\varphi}\|_{\sup}+ C\sum_{0\leq l\leq k}\|Q_s(l,k)\|(W_{l}+\|Z(l)\|W_{l-1}) \|\varphi'\|_\lv.
\end{align*}
In view of (\ref{szcsk2}), it follows that
\begin{align*}
\|S(k){\varphi}\|_{\sup}&\leq\|{\varphi}_k\|_{\sup}+\|s_k\|
\leq \|Q_s(k)\|\|{\varphi}\|_{\sup}+ 2C\sum_{0\leq l\leq k}\|Q_s(l,k)\|(W_{l}+\|Z(l)\|W_{l-1}) \|\varphi'\|_\lv,
\end{align*}
which completes the proof.
\end{proof}

\begin{lemma}\label{lem:vk}
Suppose that $T$ satisfies the \ref{UDC}. Then for every $0<\tau<(\lambda_1-\lambda)/5$ we have
$V_k=O(e^{-\lambda k})$.
\end{lemma}

\begin{proof}
In view of \eqref{eq:C'} in Proposition~\ref{prop:estKC}  and $|I^{(r)}|\leq\kappa\|Q(r)\|^{-1}=O(e^{-\lambda_1 r})$ (see \eqref{eq:lg} and \eqref{def:UDC-c}), we have
\begin{align*}
W_k&=O\Big(\sum_{r\geq k} \|Q_s(k,r+1)\|\|Z(r+1)\||I^{(r)}|C'_{r}(T)\Big)\\
&=O\Big(\sum_{r\geq k}e^{-\lambda(r+1-k)}e^{4\tau r}e^{-\lambda_1 r}\Big)
=O\Big(e^{-(\lambda_1-4\tau) k}\sum_{r\geq k}e^{-(\lambda+\lambda_1-4\tau)(r-k)}\Big)=O\big(e^{-(\lambda_1-4\tau) k}).
\end{align*}
By the definition of $V_k$, it follows that,
\begin{align*}
V_k&=O\Big(\sum_{0\leq l\leq k} \|Q_s(l,k)\|\big(e^{-(\lambda_1-4\tau)l}+\|Z(l)\|e^{-(\lambda_1-4\tau)(l-1)} \big)\Big)\\
&=O\Big(\sum_{0\leq l\leq k} e^{-\lambda(k-l)}e^{\tau l}e^{-(\lambda_1-4\tau)l}\Big)
=O\Big(e^{-\lambda k}\sum_{0\leq l\leq k} e^{-(\lambda_1 - \lambda-5\tau)l}\Big)=O\big(e^{-\lambda k}\big).
\end{align*}
\end{proof}

\begin{proof}[Proof of Theorem~\ref{thmcorrecverexp}]
The proof follows immediately by combining Lemma~\ref{thmcorrecver} and Lemma~\ref{lem:vk}, which show that %for every $0<\tau<(\lambda_1-\lambda)/5$ we have
\[\|S(k){\varphi}\|_{\sup}\leq C'\big(\|\varphi'\|_\lv \, e^{-\lambda k} + \|Q_s(k)\|\|\varphi\|_{\sup}\big),\]
Since also $ \|Q_s(k)\|= O(e^{-\lambda k})$  by the \ref{UDC} (see  \eqref{def:UDC-a} of Definition~\ref{def:UDC}), we get that $\|S(k){\varphi}\|_{\sup}=O(e^{-\lambda k})$.
\end{proof}

We can now also prove the cohomological reduction.
\begin{proof}[Proof of Theorem~\ref{theorem:coblog}]
Let us assume that $T$ satisfies the \ref{UDC} and that
$\varphi\in\ac(\sqcup_{\alpha\in \mathcal{A}} I_{\alpha})$ and
$\varphi'\in{\ol}(\sqcup_{\alpha\in \mathcal{A}}
I_{\alpha})$.
% and $\mathcal{AS}(\varphi')=0$.
Fix any $0<\tau<\min\{(\lambda_1-\lambda)/5,\lambda\}$.

\smallskip
\noindent {\it Step 1. First correction for the derivative to be in the kernel of $\mathfrak{h}$.}
Let $h_1=\mathfrak{h}(\varphi')\in H(\pi)$ and take any piecewise linear $\underline{\varphi}\in\ac(\sqcup_{\alpha\in \mathcal{A}} I_{\alpha})$
such that $\underline{\varphi}'=h_1$. Then $\mathfrak{h}((\varphi-\underline{\varphi})')=0$.

\smallskip
\noindent {\it Step 2. Correction to be in $\ac^{\mathfrak{h}}(\sqcup_{\alpha\in \mathcal{A}} I_{\alpha})$.} Since $\mathfrak{h}((\varphi-\underline{\varphi})')=0$ by Step 1,   Corollary~\ref{cor:zeroint} shows that the sum of jumps $s(\varphi-\underline{\varphi})=\int_I(\varphi-\underline{\varphi})'(x)\,dx=0$. By \eqref{eq:sofj},
it follows that
\[\sum_{\mathcal{O}\in\Sigma(\pi)}(\partial_\pi (\varphi-\underline{\varphi}))_\mathcal{O}=0.\]
Since the image of $\partial_\pi$ consists of all vectors $(x_{\mathcal{O}})_{\mathcal{O}}$ such that $\sum_{\mathcal{O}\in\Sigma(\pi)}x_{\mathcal{O}}=0$ (see \eqref{eq:pariso}),  there exists $h_2\in \Gamma$ such that $\partial_\pi(h_2)=\partial_\pi (\varphi-\underline{\varphi})$. We claim that $\varphi-\underline{\varphi}-h_2$  belongs to $\ac^{\mathfrak{h}}(\sqcup_{\alpha\in \mathcal{A}} I_{\alpha})$. To see this, notice first that  $\varphi-\underline{\varphi}-h_2\in\ac(\sqcup_{\alpha\in \mathcal{A}} I_{\alpha})$ and that $(\varphi-\underline{\varphi}-h_2)'=\varphi'-h_1\in\ol(\sqcup_{\alpha\in \mathcal{A}} I_{\alpha})$. Furthermore
\begin{gather*}%\mathcal{AS}((\varphi-\underline{\varphi}-h_2)')=\mathcal{AS}(\varphi'-h_1)=\mathcal{AS}(\varphi')=0,\\
\partial_\pi (\varphi-\underline{\varphi}-h_2)=0,\quad \mathfrak{h}((\varphi-\underline{\varphi}-h_2)')=\mathfrak{h}((\varphi-\underline{\varphi})')=0,
\end{gather*}
so $\varphi-\underline{\varphi}-h_2\in\ac^{\mathfrak{h}}(\sqcup_{\alpha\in \mathcal{A}} I_{\alpha})$.

\smallskip
\noindent {\it Step 3. Last  correction to be in the kernel  of  $\mathfrak{h}$.}
Let $h_3=\mathfrak{h}(\varphi-\underline{\varphi}-h_2)\in H(\pi)$ and set
$$\widetilde{\varphi}:=\varphi-\underline{\varphi}-h_2-h_3.$$
Then $\widetilde{\varphi}\in\ac^{\mathfrak{h}}(\sqcup_{\alpha\in \mathcal{A}} I_{\alpha})$ with
%$\mathcal{AS}(\widetilde{\varphi})=\mathcal{AS}((\varphi-\underline{\varphi}-h_2)')=0$,
$\mathfrak{h}(\widetilde{\varphi})=0$ and $\partial_\pi (\widetilde{\varphi})=\partial_\pi (\varphi-\underline{\varphi}-h_2)-\partial_\pi (h_3)=0$.

\smallskip
\noindent {\it Step 4. Proof that $\widetilde{\varphi}$ is a coboundary.}
Given any every bounded function $\varphi:I\to\R$ and $n>0$, by decomposing the Birkhoff sums $\varphi^{(n)}$ into special Birkhoff sums  (see for example  \cite[\S~2.2.3]{Ma-Mo-Yo}), we can get the estimate
\begin{equation}\label{szacsn1}
\|\varphi^{(n)}\|_{\sup}\leq
2\sum_{l\in \N}\|Z(l+1)\|\|S(l)\varphi\|_{\sup},
\end{equation}
 As  $0<\tau<\lambda$,
in view of the \ref{UDC} (in particular the estimate of $\|Z(l)\|$) and Theorem~\ref{thmcorrecverexp}, which gives that $\|S(l)\widetilde{\varphi}\|_{\sup}=O(e^{-\lambda l})$,  it follows that
\[
\|\widetilde{\varphi}^{(n)}\|_{\sup}= O\big(\sum_{l\in \N}\|Z(l+1)\|e^{-\lambda l}\big)=O\big(\sum_{l\in \N}e^{-(\lambda-\tau)l}\big)=O(1).\]
Applying  Gottschalk-Hedlund type arguments (see \cite[\S 3.4]{Ma-Mo-Yo}), we obtain that
 $\widetilde{\varphi}$ is a coboundary with a bounded transfer map.

\smallskip
\noindent {\it Step 5. Conclusive arguments.}
Let us now define $\psi:={\varphi}-\widetilde{\varphi}$. By Step 4, $\psi$ and ${\varphi}$ differ by a coboundary, so they are cohomologous. Furthermore, since by definition of $\psi$ and of $\widetilde{\varphi}$ (see Step 3)
\[\psi={\varphi}-\widetilde{\varphi}=\underline{\varphi}+h_2+h_3, \]
and $\underline{\varphi}$, $h_2$ and  $h_3$ are all piecewise linear (actually piecewise constant in the case of $h_2$ and $h_3$) functions (by construction, see Step 1 and Step 2),
we see that $\psi $ is piecewise linear. Furthermore, since by construction $\mathfrak{h}(\widetilde{\varphi})=0$ and  $\partial_\pi(\widetilde{\varphi})$ (in view of Step 3), we have that
\[\mathfrak{h}(\psi)=\mathfrak{h}({\varphi})-\mathfrak{h}(\widetilde{\varphi})=\mathfrak{h}({\varphi}) \quad \text{and} \quad
\partial_\pi(\psi)=\partial_\pi(\varphi)-\partial_\pi(\widetilde{\varphi})=\partial_\pi(\varphi).\]
Finally,  Theorem~\ref{thmcorrecverexp} shows that $\|S(k)({\varphi}-\psi)\|_{\sup} =\|S(k)\widetilde{\varphi}\|_{\sup}$ decay exponentially.
This  completes the proof.
\end{proof}

\subsection*{Acknowledgements}
We are indebted to Giovanni Forni for many interesting discussions and for pushing us to improve the main result on deviations by introducing singular cocycles. A thank { goes also to P.~Berk, S.~Ghazouani} and F.~Trujillo for their comments on an earlier version of the introduction.  C.~U. is {currently partially supported} by the Swiss National Science Foundation, Grant No.~\verb+200021_188617/1+. She also acknowledges the past support received from the European Research Council under the European Union
Seventh Framework Programme (FP/2007-2013) via the ERC Starting Grant \emph{ChaParDyn} (ERC Grant Agreement n.~335989), which supported initial investigations towards this result. Research was partially supported by  the Narodowe Centrum Nauki Grant 2017/27/B/ST1/00078.

\end{document}